\documentclass[11pt,letterpaper]{amsart}
\usepackage{commath}
\usepackage[normalem]{ulem}
\usepackage{cancel}
\usepackage[dvipsnames]{xcolor}
\usepackage{amsmath, nccmath}
\usepackage{amsfonts}
\usepackage{amssymb}
\usepackage{bbm}
\usepackage{graphicx,verbatim}
\usepackage{tikz}
\usepackage{enumitem}
\newtheorem{lthm}{Theorem}

\usepackage[pagebackref=false,colorlinks=true, pdfstartview=FitV, linkcolor=violet,citecolor=red, urlcolor=blue]{hyperref}
\newtheorem{theorem}{Theorem}
\numberwithin{theorem}{section}

\newtheorem{corollary}[theorem]{Corollary}

\newtheorem{lemma}[theorem]{Lemma}
\newtheorem*{theorem*}{Theorem}

\newtheorem{proposition}[theorem]{Proposition}

\theoremstyle{definition}
\newtheorem{example}[theorem]{Example}
\newtheorem{remark}[theorem]{Remark}
\newtheorem*{notation}{Notation}
\newtheorem{definition}[theorem]{Definition}

\makeatletter
\@namedef{subjclassname@2020}{\textup{2020} Mathematics Subject Classification}
\makeatother

\usepackage{graphicx,verbatim}
\usepackage{tikz}
\usepackage{caption}
\usetikzlibrary{cd}

\def\Q{\ensuremath {{\mathbb Q}}}
\def\D{\ensuremath {{\mathbb D}}}
\def\R{\ensuremath {{\mathbb R}}}
\def\A{\ensuremath {{\mathbb A}}}
\def\C{\ensuremath {{\mathbb C}}}
\def\Z{\ensuremath {{\mathbb Z}}}
\def\I{\ensuremath {{\mathbb I}}}
\def\fp{\ensuremath {{\mathfrak p}}}
\def\fr{\ensuremath {{\mathfrak r}}}
\def\fq{\ensuremath {{\mathfrak q}}}
\def\fa{\ensuremath {{\mathfrak a}}}

\def\fd{\ensuremath {{\mathfrak d}}}
\def\cO{\ensuremath {{\mathcal O}}}
\def\cF{\ensuremath {{\mathcal F}}}
\def\cD{\ensuremath {{\mathcal D}}}

\def\fP{\ensuremath {{\mathcal P}}}

\def\sk{\ensuremath {k}}
\def\sq{\ensuremath {{\mathsf q}}}
\def\sp{\ensuremath {{\mathsf p}}}

\DeclareMathOperator{\el}{ell}
\DeclareMathOperator{\spl}{spl}
\DeclareMathOperator{\sgn}{sgn}
\DeclareMathOperator{\ch}{ch}
\DeclareMathOperator{\CH}{CH}
\DeclareMathOperator{\vol}{vol}
\DeclareMathOperator{\val}{val}
\DeclareMathOperator{\Gal}{Gal}

\DeclareMathOperator{\Mat}{Mat}
\DeclareMathOperator{\GL}{GL}
\DeclareMathOperator{\SL}{SL}

\DeclareMathOperator{\fin}{fin}

\DeclareMathOperator{\Tr}{Tr}
\DeclareMathOperator{\Norm}{\mathsf{N}}
\DeclareMathOperator{\siminf}{\overset{\infty}{\;\sim\;}}

\allowdisplaybreaks
\begin{document}

\title[]{Beyond Endoscopy via Poisson summation for $\GL(2,K)$}

\author[M.~Emory]{Melissa Emory}
\address{Department of Mathematics, Oklahoma State University; 401 MSCS, Stillwater, OK 74078 USA}
\email{melissa.emory@okstate.edu}

\author[M.~Espinosa Lara]{Malors Espinosa Lara}
\address{Yau Mathematical Sciences Center, Jingzhai, Tsinghua University, Haidian District, Beijing, China, 100084}
\email{espino41@mail.tsinghua.edu.cn}

\author[D.~Kundu]{Debanjana Kundu}
\address{Department of Mathematics and Statistics, University of Regina; 3737 Wascana Pkwy, Regina, SK S4S 0A2, Canada}
\email{debanjana.kundu@uregina.ca}

\author[T.~A.~Wong]{Tian An Wong}
\address{Department of Mathematics,University of Michigan-Dearborn; 4901 Evergreen Rd, Dearborn, MI 48128 USA}
\email{tiananw@umich.edu}


\date{\today}

\begin{abstract}
In the early 2000's, R.~Langlands proposed a strategy called Beyond Endoscopy to attack the principle of functoriality, which is one of the central questions of present day mathematics.
A first step was achieved by A.~Altu\u{g} who worked with the group $\GL(2,\Q)$ in a series of three papers.
We generalize the first part of this result to a class of totally real number fields.
In particular, we cancel the contribution of the trivial and special representations to the trace formula using an additive Poisson summation on the regular elliptic terms.
\end{abstract}

\maketitle

\tableofcontents
\addtocontents{toc}{\protect\setcounter{tocdepth}{1}}

\section{Introduction}

\subsection{Philosophy}

Langlands' principle of functoriality is one of the central open problems in modern number theory.
Its study is a driving force for the development of number theory, harmonic analysis, representation theory, and the deep connections between them.
More precisely, functoriality conjecturally maps the world of automorphic 
representations across different groups: given a morphism of $L$-groups associated to reductive groups $H$ and $G$, it predicts a corresponding transfer of automorphic representations from $H$ to $G$.
Underpinning this mechanism is a deep analytic compatibility; the transfer must respect the spectral decomposition of $L^2$-spaces on adelic quotients, carrying automorphic representations over while preserving the core analytic properties of their $L$-functions and $\varepsilon$-factors.
The trace formula has been developed as a major tool to study the structure of automorphic representations.
Some pioneering work in this direction include \cite{JacquetLanglands1970, Langlands1980, Langlands1997}.
In recent years, progress has been made to prove particular cases of functoriality such as \cite{AliI, AliII, AliIII, clozel2023nonabelian, feng2023modular, FLN, Lafforgue,Venkatesh},
yet a general proof remains out of reach.
For a comprehensive account of the history of the problem and the current state of progress, we refer the reader to \cite{Art_PoF, arthur2023work}.

\subsection{A strategy to prove the principle of functoriality}

The trace formula is an equality between two distributions, usually called the \emph{geometric side} and the \emph{spectral side}.
For $G = \GL(n)$, over a number field $K$, the main part on the geometric side is 
\[
\sum_{[\gamma]} \vol(\gamma)\cO(\gamma, f),
\]
which is called the \emph{regular elliptic part} of the trace formula; here $[\gamma]$ runs over regular elliptic conjugacy classes in $G$, $\vol(\gamma)$ refers to a certain ad\`{e}lic volume and $\cO(\gamma,f)$ is the orbital integral associated to the test function $f$.
On the other hand, the main term on the \emph{spectral side} is 
\[
\sum_{\pi} m(\pi)\Tr(\pi(f))
\]
which is the contribution of the discrete spectrum to the trace formula.
Here, $f$ is a test function, $\pi$ runs over the automorphic representations occurring in the discrete spectrum, the trace $\Tr(\pi(f))$ is the character of the representation $\pi$ of $G(\A_K)$, and $m(\pi)$ is the multiplicity of $\pi$ in the $L^2$-spectrum of $G$.

In \cite{LanBE04}, R.~Langlands proposed a strategy to prove the principle of functoriality via the construction of a \emph{refined} trace formula.
Roughly, given a representation $r$ of the $L$-group (denoted by $^LG$), the idea is to weight the \emph{spectral side} of the trace formula with a coefficient $\mu_r(\pi)$, whence
\[
\sum_{\pi} \mu_r(\pi)m(\pi)\Tr(\pi(f)).
\]
The factor $\mu_r(\pi)$ should be non-zero if and only if the associated automorphic $L$-function $L(z,\pi,r)$ has a pole at $z=1$. For example, Langlands considered, as we shall do also, $\mu_r(\pi)$ to simply be the order of $L(s,\pi,r)$ at $z=1$. 

When $K=\Q$ and $r = \text{Sym}^k$, Langlands showed there is a family of ad\`{e}lic test functions $f_p$ indexed by the primes $p$ such that
\begin{align*}
\sum_{\pi} \mu_r(\pi)m(\pi)\Tr(\pi(f)) &= 
\lim_{n\rightarrow \infty} \frac{1}{\# \{p \mid p\leq n\}} \sum_{p\leq n} \log(p) \sum_\pi m(\pi)\Tr(r(a_p(\pi)))\Tr(\pi(f))\\
&=\lim_{n\rightarrow \infty} \frac{1}{\# \{p \mid p\leq n\}} \sum_{p\leq n} \log(p) \sum_\pi m(\pi) \Tr(\pi(f_p)),
\end{align*}
where $a_p(\pi)$ is the Satake parameter of $\pi$ at $p$.
Na{\"i}vely, assuming the equality of the main terms on both sides of the trace formula, we see that
\[
\sum_{\pi} \mu_r(\pi)m(\pi)\Tr(\pi(f)) = \lim_{n\rightarrow \infty} \frac{1}{\# \{p \mid p\leq n\}} \sum_{p\leq n} \log(p) \sum_{[\gamma]} \vol(\gamma)\cO(\gamma, f).
\]
For the limit to exist, the contribution of the non-tempered spectrum of each of these trace formulas needs to be removed.
In our case of interest, i.e.~when $G=\GL(2)$ note that the only non-tempered representation in the discrete spectrum is the trivial one and its twists by grossencharacters.
For the moment, let us assume only the trivial representation appears in the residual spectrum.
Therefore,
\begin{multline}
\label{intro: remove tempered equation}
\lim_{n\rightarrow \infty} \frac{1}{\# \{p \mid p\leq n\}} \sum_{p\leq n} \log(p) \sum_{\pi\neq \mathbf{1}} m(\pi) \Tr(\pi(f_p)) \\
= \lim_{n\rightarrow \infty} \frac{1}{\# \{p \mid p\leq n\}} \sum_{p\leq n} \log(p) \left( \sum_{[\gamma]} \vol(\gamma)\cO(\gamma, f) - \Tr(\mathbf{1}(f_p))\right).
\end{multline}
What remains on the spectral side should now converge.
Thus, the geometric side must converge as well.
In \cite{FLN}, it was suggested that an additive Poisson summation-type formula could be applied to the regular elliptic contribution on the geometric side to isolate the trace of the trivial representation and cancel it.

In a series of papers by A.~Altu{\u{g}} \cite{AliI, AliII, AliIII} this cancellation is performed for $G=\GL(2,\Q)$ with $f$ such that for all primes $q\neq p$, its $q$-adic local factor is the indicator of the maximal compact open subgroup of $\GL(2,\mathbb{Q}_q)$ and when the representation $r$ is the $k$-th symmetric power of the standard representation.
For the case $k=1$, and under certain assumptions on $f$, the above limit equality is verified.
In particular, in \cite{AliI} the trace of the trivial representation is isolated from the discrete spectrum{\textemdash}which is the only representation that has to be removed in this case.
Altu\u{g} shows that
\[
\left( \sum_{[\gamma]} \vol(\gamma)\cO(\gamma, f) - \Tr(\mathbf{1}(f_p))\right)
\]
can be manipulated so that the limit on the right hand side of \eqref{intro: remove tempered equation} converges.
This has been regarded as a proof of concept of the strategy proposed in \cite{FLN}, whose hypothetical Poisson summation formula, while highly suggestive, `cannot literally be true;' see \cite[Footnote 6]{problems}.
Until now, no other example of the removal of the non-tempered spectrum via Poisson summation has been established since then.
We refer the reader to Section~\ref{sec: subsequent} for work that was done since this paper was first made public.

\subsection{Goal of this article}
We extend the work of \cite{AliI} to totally real number fields $K$ (under additional technical hypothesis stated at the beginning of Section~\ref{sec 2 results and strategy} which are clarified as we progress through the paper), a problem posed by Arthur in \cite[\S11]{arthur2023work}.
That is to say, we show that 
\[
\left( \sum_{[\gamma]} \vol(\gamma)\cO(\gamma, f) - \displaystyle\sum_{\chi}\Tr(\mathbf{1}\otimes \chi)(f_p))\right)
\]
can be manipulated so that the limit on the right hand side of \eqref{intro: remove tempered equation} converges when $G=\GL(2,K)$.
The sum over $\chi$ runs over all grossencharacters of $K$ unramified outside a finite set $S$.
We emphasize that for this case these are all the representations whose traces must be canceled from the discrete spectrum.
Note that in the case of Altu\u{g} this sum consists only of the trivial character, whence only the trivial representation has to be canceled.
The strategy of proving our main result can be divided into five parts: 
\begin{description}
 \item[Rewriting the regular elliptic part]
 The regular elliptic part of the trace formula is given by a concrete sum that runs over the regular elliptic conjugacy classes.
 For $\GL(2)$, these classes can be parameterized by the trace and determinant of their characteristic polynomial.
 In this step, we explicitly rewrite every term involved in terms of these two parameters.

 \item[Application of the approximate functional equation] The regular elliptic part of the trace formula involves a volume term that can be expressed as the value of certain $L$-functions evaluated at $z = 1$.
 However, these $L$-functions are not absolutely convergent on the $\Re(z) = 1$ line.
 To overcome this difficulty, we use the approximate functional equation{\textemdash}a technique to expand the values within the critical strip in a way that absolute convergence is recovered at the expense of certain weight factors.
 The goal is to modify the $L$-functions appropriately and obtain absolute convergence on the line of interest.

 \item[Problem of completion and Poisson summation]
 We can now perform Poisson summation provided the following two conditions are satisfied:
 \begin{itemize}
 \item the original sum arising from the regular elliptic part of the trace formula can be extended to a sum over a \textit{complete lattice}.
 \item the function referred to above is a smooth function of compact support.
 \end{itemize}
 The second issue is resolved relatively easily by introducing weight factors, i.e., smooth cutoffs, when applying the approximate functional equation.
 The first issue is significantly more deep{\textemdash}the condition of being `regular elliptic' rules out many pairs of integers which cannot be ignored when performing the Poisson summation.
 We circumvent this serious problem by introducing a modified Hilbert symbol.
 We can now perform Poisson summation in the standard way.

 \item[Concrete evaluation of Kloosterman-type sums] Performing Poisson summation yields character sums, which we call Kloosterman-type sums; here the `character' is the symbol introduced above.
 Evaluating these sums is a critical but lengthy \textit{local} process and the final answer is an explicit quotient of zeta functions.

 \item[Isolation of the trivial representation and twists from the discrete spectrum] There were two steps above which introduced integrals; namely, introducing the weight factors during the approximate functional equation and the Fourier transform in the Poisson summation.
 In this stage, we invoke complex analysis{\textemdash}we shift the contours of integration and go over the poles of the weight factors, as well as of those of the quotient of zeta functions coming from the Kloosterman-type sums.
 The original sum gets divided into the sum of many other terms via the residue theorem.
 Some of the terms cancel each other and among those that remain, one is exactly the trace of the sum of trivial representation and its twists by character whereas the others are multiples of the sum of the traces of the special representation and its twists by characters.
 (We follow the terminology of Altu\u{g}, referring to the discrete contribution of the continuous spectrum to the trace formula for GL(2) as the special representation.)
\end{description}

\subsection{Subsequent works}
\label{sec: subsequent}
Since this paper was submitted in 2024, several works related to the strategy of Beyond Endoscopy, as Langlands and Arthur developed it, have appeared.
Following the path of Altu\u{g}, there are the papers \cite{cheng26part1, cheng26part2,cheng26part3,cheng26part4}.
In these, Y.~Cheng extends the results of \cite{AliI, AliII, AliIII} for the case of $\GL(2, \Q)$ by allowing for ramification.
In another direction, the recent paper \cite{dengMEEL2026GL3} proves results analogous to \cite{AliI} for $\GL(3, \Q)$.

\section{Results and detailed outline of the strategy}
\label{sec 2 results and strategy}

The purpose of this section is to discuss the above steps in a little more detail and highlight our main contributions.
In addition to explaining the results we prove in this paper, this section includes a discussion of obstacles that appear when working over a general number field but do not arise over $\Q$.
In \S\ref{new section} we pose some questions and present some speculations of what we think should happen in the more general setting.

For our discussion, once and for all we fix a prime ideal $\fp$ in $K$.
Next, we choose the representation $r$ to be the $k$-th symmetric power of the standard representation for a fixed positive integer $k$.
The choice of test function $f$ is the same as in \cite{AliI} and depends on $k$; see Definition~\ref{defi: ffq(k) defi} for the precise form.
Throughout this paper we assume that $K$ is a totally real field, although several of the facts we discuss are true in general.
In \S\ref{section: blending step} we explain the main reason for invoking this hypothesis.
In fact, for certain steps we need to make a stronger assumption that the system of fundamental units is `totally positive'; see \ref{ass: positivity}.
Finally, we have to assume that $2$ is completely split in $K$ to carry out the evaluation of the Kloosterman-type sums; see \ref{ass: split}.
These hypotheses are technical and were invoked to simplify our calculations.
Our strategy should work without these hypotheses - we will only need to be more careful and precise with some definitions and some calculations will become more cumbersome.
Mathematically, we do not expect there to be any major obstructions in removing these hypotheses.

\subsection{Rewriting of the regular elliptic part}
An immediate and striking challenge of considering a general base field is that if a regular elliptic class contributes a non-zero term to the regular elliptic part, its determinant must generate the ideal $\fp^k$ where $k$ is a positive integer fixed at the time of defining the test function.
This is automatic when the base field is $\mathbb{Q}$ because it has class number $1$, but not always true for a general number field.
Therefore,
\begin{itemize}
\item if we make the `wrong' choice of the prime $\fp$ or integer $k$, both the regular elliptic part on the geometric side and the sum of the traces of the residual spectrum, in the spectral side, equal $0$.
This case requires no cancellation and one could proceed directly to the evaluation of the limit (analogous to \cite{AliII, AliIII}).
\item The `right' choice of $\fp$ and $k$ is then used to extract the unit{\textemdash}we fix $\fp$ and a generator (say $\varepsilon$) of $\fp^k$.
We write $\det(\gamma) = u\varepsilon$, for some unit $u$ in the ring of integers of $K$ (denoted by $\cO_K$).
Altu\u{g} implicitly does this when writing $\det(\gamma) = \pm p^k$; choosing $\varepsilon = p^k$ and $u = \pm 1$.
Throughout this paper, we work under the assumption that this \emph{divisibility hypothesis} is satisfied; see \ref{ass: class number}.
\end{itemize}

The regular elliptic part of the trace formula is written as
\[
\sum_{u\in \cO_K^*} \sum_{\substack{\tau\in \cO_K\\
\tau^2 - 4u\varepsilon\neq\square}} \mbox{vol}(\gamma)\cO(\gamma, f),
\]
where $\cO(\gamma, f)$ is the orbital integral and we have used the irreducible characteristic polynomial of $\gamma$, namely
\[
P_{\gamma}(X) = X^2 - \tau X + u\varepsilon
\]
to classify the regular elliptic conjugacy classes.
The goal of this step is to rewrite the volume and the orbital integral explicitly in terms of the trace $\tau$ and the unit $u$.
Since $\varepsilon$ is fixed, the determinant is in fact determined by the unit $u$.
This calculation is carried out in \S\ref{sec: The Main Setup} and \S\ref{section: Manipulation of the Regular Elliptic Part}.
Our main result is the following which is Theorem~\ref{manipulation}.
\begin{lthm}
\label{Thm A}
Let $\sp=\Norm_K(\fp)$, the absolute norm of $\fp$ from $K$ to $\Q$.
The regular elliptic part of $\GL(2,K)$ can be written as
\[
R(f) = \abs{D_K}^{1/2} \sp^{-k/2} \sum_{u\in\cO_K^*}\sum_{\tau\in L(u)}
 \theta^{\pm}(\tau, u) L(1, \chi_{\gamma})\left( \sum_{\fd \mid S_{\gamma}} \frac{1}{\Norm_K(\fd)} \prod_{\fq\mid (S_{\gamma}/\fd)}\left(1 - \frac{\chi_{\gamma}(\fq)}{\Norm_K(\fq)}\right)\right)
\]
where
\[
L(u) := \{\tau\in \cO_K : \tau^2 - 4u\varepsilon \neq\square \}.
\]
\end{lthm}
The theorem is proved in the following steps.
\subsubsection{Volume term}
The characteristic polynomial $P_{\gamma}(X)$ generates a quadratic field extension $K_{\gamma}/K$ and associated to it, we have the quadratic sign character $\chi_\gamma(\fq)$ for a prime $\fq$ defined by
\[
\chi_{\gamma}(\fq) = \left\{\begin{array}{ccl}
 1 & \text{if }\fq &\textrm{is split}\\
 -1 & \text{if }\fq &\textrm{is inert}\\
 0 & \text{if }\fq &\textrm{is ramified.}\\
 \end{array}\right.
\]
Set $D_{\gamma}$ to denote the absolute discriminant of $K_{\gamma}$.
Proposition~\ref{prop: volume manipulation} shows that the volume term is proportional to
\[
\sqrt{\abs{D_{\gamma}}}\cdot L(1, \chi_{\gamma}).
\]
In other words, up to a constant $\kappa$ which is the residue of the Dedekind zeta function $\zeta_K$ at $z=1$, the volume term is a product of $\sqrt{\abs{D_{\gamma}}}$ and a special value of an $L$-function.

\subsubsection{Orbital integrals}

First, we prove the existence of an integral ideal $S_{\gamma}$ satisfying
\[
(\tau^2 - 4u\varepsilon) = S_{\gamma}^2\Delta_{\gamma},
\]
where $\Delta_{\gamma}$ is the relative discriminant of $K_{\gamma}/K$.
When the base field is $\mathbb{Q}$, this coincides with the fundamental discriminant and the above equation can be taken as an equality of numbers as opposed to ideals.
In Proposition~\ref{Multiplicative Formula Langlands} we show that the product of orbital integrals at finite primes satisfy
\[
\cO(f_{\fin},\gamma) = \Norm_K(\fp)^{-k/2} \sum_{\fd \mid S_{\gamma}}\Norm_K(\fd) \prod_{\fq \mid \fd}\left(1 - \frac{\chi_{\gamma}(\fq)}{\Norm_K(\fq)}\right).
\]
On the other hand at the archimedean places the test function is
\[
f_{\infty} = f_{\nu_1}\times\cdots\times f_{\nu_n},
\]
where $n = [K:\Q]$.
For each real place (which we generically denote by) $\nu$, we follow \cite{AliI} and expand the local orbital integral using germ expansions.
This is carefully explained in \S\ref{sec: arch orb int}.

\begin{remark}
The fact that we have such a concrete formula is essential for the success of the strategy we adopt.
\end{remark}

\subsubsection{Units}
If we have to isolate the step whose generalization to the number field setting was the most non-trivial, it would be in computations involving the units.
To deal with the units, we include them as a part of a proper subset of a lattice, over which we later perform Poisson summation.
This part of the lattice is an artificial one but, its presence is crucial for matching the two sides of the trace formula.
We will return to this point in the last step; see \S\ref{intro - isolating the contribution}.

A crucial observation for dealing with the units is that via the Dirichlet Unit Theorem, the pairs $(\tau, u)$ indexing the regular elliptic part of the trace formula can be identified with a full lattice inside of ${\pm}\times\mathbb{R}^{2n-1}$.
This technical construction is carried out in \S\ref{section: blending step}.
More concretely, we define the \emph{interpolation function} $\theta^\pm(x,y)$ from $\R^{2n-1}$ to $\C$ in Definition~\ref{def:thetaplusminus} where $x$ is viewed as a vector in $\R^n$ and $y$ as a vector in $\R^{n-1}$.
This is a generalization of $\theta^{\pm}(x)$ used in \cite{AliI}, except that the function appearing in \emph{op.~cit.} does not have two coordinates.
In our case, the main reason for the difficulty arises from the presence of fundamental units.
It is in handling these fundamental units that we require that the number field is totally real and satisfies the condition of `positivity'.
Settling the issue with the units is particularly challenging because there many choices of how to handle the units without any apparent reason as to why one is better than the other.
Our choice of definition of $\theta^{\pm}$ is essentially dictated by our understanding of how to generalize Langlands' calculations for the special representation when working with a general number field rather than over $\Q$.
In other words, once we developed the strategy for the calculations in \S\ref{sec: Computation of the trace of the trivial and the special representations}, we could make the appropriate choice for the definition of the interpolation function.

The interpolation function resembles the germ expansion of the archimedean orbital integral.
Indeed, the discriminant term arising from studying the volume decomposes into local factors (Proposition~\ref{relation discriminants}).
Distributing each local factor to its corresponding orbital integral yields the form of the interpolation function we have defined (in coordinates) at the archimedean places, while at the finite places it amounts to changing the variable $\fd$ to $S_{\gamma}/\fd$.
Substituting the above analysis into the expansion of the regular elliptic part yields Theorem~\ref{Thm A}.

\subsection{Application of the approximate functional equation}

The Dirichlet series that defines $L(z,\chi_\gamma)$ does not converge absolutely on the real line $\Re(z) = 1$ (which is where it is being evaluated).
To continue the analysis, we show that the regular elliptic part of the trace formula can be written as
\[
\sum_{u\in\cO_K^*}\sum_{\tau\in L(u)}\sum_{\fd\mid S_{\gamma}}\sum_{\fa}\chi_{\fd}(\fa)\left(\frac{\Phi_{\pm}(\tau, u)}{\Norm_K(\fd)\Norm_K(\fa)} + \Norm_K(\fd)\Psi_{\pm}(\tau, u)\right),
\]
where $\chi_{\fd}(\fa)$ is a quadratic character whose primitivization yields $\chi_{\gamma}$, and $\Phi_{\pm}(\tau, u), \Psi_{\pm}(\tau, u)$ are certain smooth functions with compact support arising from smoothing the interpolation function $\theta^{\pm}$.
The sum over $\fa$ comes from the Dirichlet series expansion of the $L$-function of $\chi_{\fd}(\fa)$.
This verification is carried out in \S\ref{sec: Approx Funct eqn} and \S\ref{sec: Smoothing of the Singularities}.

To prove this, we have to apply the approximate functional equation.
In particular, we need to verify that the hypothesis of the approximate functional equation is satisfied.
Recall that the approximate functional equation applies to $L$-functions whose functional equation is already known.
For us, there are two $L$-functions of interest, namely
\[
L(z, \gamma) := \Norm_K(S_{\gamma})^z \sum_{\fd\mid S_{\gamma}}\Norm_K(\fd)^{(1 - 2z)} \prod_{\fq\mid (S_{\gamma}/\fd)}\left(1 - \frac{\chi_{\gamma}(\fq)}{\Norm_K(\fq)^z}\right), 
\]
and its completion
\[
\Lambda(z, \gamma) = \Lambda(z, \chi_{\gamma})L(z, \gamma).
\]
Here, $\Lambda(z, \chi_{\gamma})$ is the standard completion of the $L$-function attached to the quadratic sign character.
It is an entire function with a functional equation.
In Theorem~\ref{functional equation of imprimitive Hecke functions}, we prove the functional equation
\[
\Lambda(z, \gamma) = \Lambda(1-z, \gamma),
\]
which follows from verifying that $L(z, \gamma)$ admits a functional equation.
Our proof\footnote{Another way to prove the functional equation of this (Zagier) zeta function is via local methods; see \cite[Theorem~4]{malors21}.
This local approach is more suited for studying the nature of this function as conjectured in \cite{artstrat}.
} involves an algebraic manipulation similar to the one in \cite{zagier2006modular}.
In Proposition~\ref{prop: reformulation mult formula of Langlands} we verify that the product of the finite orbital integrals coincides with the value of an appropriate $L$-function (which admits a functional equation) when evaluated at $z = 1$.
This allows us to now apply the approximate functional equation.

The step of applying the approximate functional equation introduces two functions; namely, an odd Bessel function (denoted by $F$) whose residues we understand precisely and another related function $H_{\gamma}$ for which we know the poles and residues.
Furthermore, similar to the interpolation functions above we can extend $H_{\gamma}$ to a function $H_{(x, y)}$.

Since there are several real places, there are several Gamma factors that need to be evaluated.
The main technical estimate is recorded in Lemma~\ref{lemma:boundH}.
In Theorem~\ref{theorem:approximate functional equation} we write the expansion we get from the approximate functional equation, as a function of a complex variable $z$.
More precisely, 
\begin{lthm}
\label{Thm B}
The following expression holds for any choice of real $A>0$,
\begin{small}
\begin{multline*}
L(z, {\gamma}) = \sum_{\fd\mid S_{\gamma}}\frac{1}{\Norm_K(\fd)^{2z - 1}}\sum_{\fa} \frac{\chi_{\fd}{(\fa)}}{\Norm_K(\fa)^{z}}F\left(\frac{\Norm_K(\fd)^2 \Norm_K(\fa)}{A}\right)\\
+ \left(\Norm_K(S_{\gamma})\abs{D_K}^{1/2}\Norm_K(\Delta_{\gamma})^{1/2}\right)^{1 - 2z}
\sum_{\fd\mid S_{\gamma}}\frac{1}{\Norm_K(\fd)^{1 - 2z}}\sum_{\fa}\frac{\chi_{\fd}(\fa)}{\Norm_K(\fa)^{1 - z}}H_{\gamma}\left(z, \frac{\Norm_K(\fd)^{2}\Norm_K(\fa)A}{\Norm_K(S_{\gamma})^{2}\abs{D_K}\Norm_K(\Delta_{\gamma})}\right),
\end{multline*}
\end{small}
\end{lthm}

\subsubsection{Recap: the situation over $\Q$}
Recall that at this stage (see \cite[Proposition~3.4]{AliI}) Altu\u{g} obtains 
\begin{footnotesize}
\[
\sum_{d^2\mid N}^{'}\frac{1}{d^{2z - 1}} \sum_{a = 1}^{\infty}\frac{1}{a^s}\left(\frac{N/d^2}{a}\right)F\left(\frac{ad^2}{A}\right)
 +\left(\frac{\abs{N}}{\pi}\right)^{1/2 - z} \sum^{'}_{d^2\mid N}\frac{1}{d^{1 - 2z}} \sum_{a = 1}^{\infty}\frac{1}{a^{1-z}}\left(\frac{N/d^2}{a}\right)H_N\left(z, \frac{ad^2A}{\abs{N}}\right).
\]
\end{footnotesize}
Here, $N = \tau^2 \pm 4p^k$ and the $'$ over the sum refers to some congruence conditions we discuss later.
A quick comparison convinces us that the two expressions are almost the same{\textemdash}except how the $L$-functions at infinity are displayed.
A crucial difference is the appearance of $\chi_{\fd}(\fa)$ in the expression in Theorem~\ref{Thm B} while Altu\u{g}'s expression has the Kronecker symbol $\left(\frac{N/d^2}{a}\right)$.
These two objects though essentially the same, are constructed (and expressed) differently.
Altu\u{g}'s expression is more explicit because of the availability of the quadratic reciprocity law over $\Q$.

\subsubsection{The situation over a general number field}
At this stage, while rewriting the regular elliptic part of the trace formula, we encounter cancellation.
More precisely, the Euler factors associated to the primes dividing $S_{\gamma}/\fd$ in the expansion of the $L$-function associated to the volume term gets cancelled by the factors arising in $L(z,\gamma)$ above.
What remains is the Euler product expansion of a \emph{non-primitive} character.

When the Kronecker symbol is available (e.g.~over $\Q$) the conductor of the non-primitive character becomes the `new upper entry' of the symbol; otherwise we leave the change of conductor of the `new' character implicit.
Set $\partial = \tau^2 - 4\det(\gamma)$.
After evaluating $L(z,\gamma)$ at $z = 1$ we see that the regular elliptic part can be expressed as
\begin{equation*}
\begin{split}
&{\kappa\abs{D_K}}\sp^{-\sk/2}\sum_{u\in\cO_K^*} \sum_{\tau\in L(u)} \theta^{\pm}(\tau, u) \left(\sum_{\fd\mid S_{\gamma}}\frac{1}{\Norm_K(\fd)} \sum_{\fa} \frac{\chi_{\fd}{(\fa)}}{\Norm_K(\fa)}F\left(\frac{\Norm_K(\fd)^2 \Norm_K(\fa)}{\Norm_K(\partial)^{\alpha}\abs{D_K}^{\alpha}}\right)\right)\\
 &+ \frac{\kappa \abs{D_K} \sp^{-\sk/2}}{ \Norm_K(\partial)^{1/2}\abs{D_K}^{1/2}}
 \sum_{u\in\cO_K^*}\sum_{\tau\in L(u)}\sum_{\fd\mid S_{\gamma}}\frac{1}{\Norm_K(\fd)^{-1}} \sum_{\fa}\chi_{\fd}(\fa)H_{\gamma}\left(1, \frac{\Norm_K(\fd)^{2}\Norm_K(\fa)}{\Norm_K(\partial)^{1 - \alpha}\abs{D_K}^{1 - \alpha}}\right).
\end{split}
\end{equation*}
For the final step, we define two smooth functions $\Phi_{\pm}(x, y)$ and $\Psi_{\pm}(x, y)$ with compact support in $\R^{2n-1}$ in terms of the interpolation function $\theta^{\pm}$, the Bessel function $F$, and $H_\gamma$.
Using Proposition~\ref{prop: precise value of A} we conclude with the statement that (up to constants depending on the field $K$) the regular elliptic part can be expressed as
\[
{\kappa\abs{D_K} \sp^{-\sk/2}} \sum_{u\in\cO_K^*}\sum_{\tau\in L(u)}\sum_{\fd\mid S_{\gamma}}\sum_{\fa}\chi_{\fd}(\fa)\left(\frac{\Phi_{\pm}(\tau, u)}{\Norm_K(\fd)\Norm_K(\fa)} + \Norm_K(\fd)\Psi_{\pm}(\tau, u)\right).
\]
The approximate functional equation yields series and integrals that are absolutely convergent, and so they can be appropriately swapped when required.
At this stage the analytic issues have been resolved.

\subsection{Problem of completion and Poisson summation}

\subsubsection{Recap: problem of completion over $\Q$}
To understand the gravity of the challenge this step poses, we recall how the issue is resolved in \cite{AliI}.
Over $\Q$, the regular elliptic part can be rewritten as
\begin{footnotesize}
\[
\sum_{\mp} \sum_{\substack{\tau\in\mathbb{Z}\\ \tau^2-4p^k\neq\square}} \sum^{'}_{d^2 \mid \tau^2\pm 4p^k} \sum_{a = 1}^{\infty} \theta^{\mp}_{\infty} \left(\frac{\tau}{2p^{k/2}}\right)\frac{1}{ad}\left(\frac{(\tau^2\pm 4p^k)/f^2}{a}\right) \left(F\left(\frac{ad^2}{A}\right) + \frac{ad^2}{\sqrt{\abs{\tau^2 \pm 4p^k}}}H\left(\frac{ad^2A}{\abs{\tau^2 \pm 4p^k}}\right)\right).
\]
\end{footnotesize}
To perform Poisson summation, one needs a lattice.
In this case, the lattice required is $\mathbb{Z}$, but since several integers are ruled out by imposing $\tau^2-4p^k\neq\square$ (non-square), we do not have a lattice.
To have the lattice, we add the missing integers; i.e.~we consider the sum
\begin{footnotesize}
\[
\sum_{\mp} \sum_{\substack{\tau\in\mathbb{Z}\\ \tau^2-4p^k=\square}} \sum^{'}_{d^2 \mid \tau^2\pm 4p^k} \sum_{a = 1}^{\infty} \theta^{\mp}_{\infty} \left(\frac{\tau}{2p^{k/2}}\right)\frac{1}{ad} \left(\frac{(\tau^2\pm 4p^k)/f^2}{a}\right) \left(F\left(\frac{ad^2}{A}\right) + \frac{ad^2}{\sqrt{\abs{\tau^2 \pm 4p^k}}}H\left(\frac{ad^2A}{\abs{\tau^2 \pm 4p^k}}\right) \right).
\]
\end{footnotesize}
Summing both these expressions yields a sum over $\tau \in \Z$.
After relatively straight forward manipulations, one obtains a sum of the values of a smooth function with compact support over a complete lattice.
This allows for performing Poisson summation.

\subsubsection{Problem of completion over a general number field}
Na{\"i}vely, one would hope that one can mimic the same idea as above.
In other words, we have to consider the following piece to complete the lattice
\[
\kappa \abs{D_K} \sum_{u\in\cO_K^*}\sum_{\tau^2- 4u\varepsilon=\square}\sum_{\fd\mid S_{\gamma}}\sum_{\fa}\chi_{\fd}(\fa)\left(\frac{\Phi_{\pm}(\tau, u)}{\Norm_K(\fd)\Norm_K(\fa)} + \Norm_K(\fd)\Psi_{\pm}(\tau, u)\right).
\]
We immediately have a problem{\textemdash}what do $S_{\gamma}$ and $\chi_{\fd}(\fa)$ mean?
These were constructed from the knowledge of the quadratic extension associated to $\gamma$.
Without such a quadratic extension, there is no quadratic character or an ideal $S_{\gamma}$; thus, the terms have no meaning.
Over $\Q$, both these problems do not arise and the terms involved can be defined regardless of whether their parameters define a regular elliptic class or not because the Kronecker character and the congruence conditions always make sense.

Concretely, we have to settle the following two issues:
\begin{description}
 \item[Extension of the quadratic symbol]
 For any unit $u$, any integer $\tau$, and any integral ideals $\fa,\fd$, define a symbol that coincides with $\chi_{\fd}(\fa)$ whenever there exists a regular elliptic matrix $\gamma$ parameterized by $(\tau,u)$ and $\fd$ divides $ S_{\gamma}$, but that makes sense for any choice of parameters.
 
 \item[Equivalent explicit conditions for $\fd\mid S_{\gamma}$] Find an alternative description of the indexing set $\{\fd : \fd\mid S_{\gamma}\}$ that depends on the parameters $\tau, u$ alone and that makes sense when $\tau^2- 4u\varepsilon=\square$.
\end{description}

The second problem is addressed by the second named author in \cite{malors21} which is discussed in \S\ref{section: the problem of completion}.
We state the relevant result below but refer the reader to the original article for the proof.
\begin{lthm}
For a regular elliptic matrix $\gamma$, the following two conditions are equivalent:
\begin{enumerate}[label = \textup{(\roman*)}]
 \item $\fd\mid S_{\gamma}$,
\item $\fd^2\mid \tau^2 - 4\det(\gamma)$ and (in $\cO_{K_{\fp_i}}$) the following congruence holds
 \[
 \frac{\tau^2 - 4\det(\gamma)}{\pi_{\fp_i}^{2\val_{\fp_i}(\fd)}} \equiv \square \pmod{\pi_{\fp_i}^2} \ \text{ for each }i = 1, \dots, n.
 \]
\end{enumerate}
\end{lthm}

Note that over $\Q$, the second condition simplifies and becomes
\[
\frac{\tau^2 - 4\det(\gamma)}{4d^2}\equiv 0, 1\pmod{4},
\]
which is precisely the congruence conditions we have been referring to so far.
This congruence condition over $\Q$ was already available to Altu\u{g} from the work of D.~Zagier.

\subsubsection{Modified Hilbert symbol}
Another major obstacle was in choosing the right symbol which could uncouple the variables $\tau$, $u$, $\fd$, and $\fa$.
None of the obvious choices such as the Artin symbol, the power residue symbol, the Frobenius symbol, and the Hilbert symbol could do this job.
We define a \emph{modified Hilbert symbol} in terms of the restricted Hilbert symbol; see Definitions~\ref{defi: restricted Hilbert} and \ref{defi: our symbol}.
This new symbol recovers the quadratic sign character and upon appropriate twisting also recovers the imprimitive character; see Theorem~\ref{thm: 7-9}.

The vanishing conditions defining the \textit{restricted} Hilbert symbol guarantee that we work with units (whereas the general Hilbert Symbol is defined for all non-zero entries).
A crucial property for our symbol is that it makes explicit the dependence on $\tau$ for \textit{all} the imprimitive characters that can be obtained from \textit{all} the quadratic sign characters of quadratic extensions of $K$ \textit{simultaneously}.
For the Kronecker symbol, this property is immediate since
\[
\chi_{d}(a) = \left(\frac{(\tau^2\pm 4p^k)/d^2}{a}\right).
\]
It is not possible to replace the Kronecker symbol with one of the other choices mentioned above because there is incompatibility with the division at this stage.
The quantity $\tau^2 - 4u\varepsilon$ is an algebraic integer, and must remain so because any conjugacy class is parameterized by a pair of algebraic integers.
Over $\Q$, the class number is $1$ and there is only one unit which is a square, so the problem gets resolved immediately.

\begin{remark}
To prove the desired properties of the modified Hilbert symbol, we needed to appeal to Artin's reciprocity law despite our best efforts.
We see this as a positive sign because it suggests that the strategy is sensitive to reciprocity laws which are in great part a motivation for the work in the area.
\end{remark}

After resolving both these issues, the \textit{completed regular elliptic part} of the trace formula has the following shape,
\[
\kappa \abs{D_K} \sp^{-\sk/2} \sum_{u\in \cO_K^*}\sum_{\tau\in\cO_K}\sum_{\fd^2\mid \tau^2-4u\varepsilon}^{'}\sum_{\fa} \binom{\tau^2 - 4u\varepsilon, \fd}{\fa}\left(\frac{\Phi_{\pm}(\tau, u)}{\Norm_K(\fd)\Norm_K(\fa)} + \Norm_K(\fd)\Psi_{\pm}(\tau, u)\right).
\]

\subsubsection{Poisson summation}

We show that for any unit $v$
\[
\binom{\tau^2 - 4u\varepsilon, \fd}{\fa} = \binom{(\tau v)^2 - 4uv^2\varepsilon, \fd}{\fa}.
\]
Exploiting this, we divide the units $\cO_K^*$ into their cosets modulo $\cO_K^{*, 2}$ and rearrange the sum over the units as a double sum.
The character depends on the coset representative of the outer sum (called $Y$) and is independent of the coset representative of the inner sum (a parameter denoted by $y$).
The additive periodicity modulo $4\fa\fd^2$ allows us to divide the $\tau$-sum according to the congruence $\mu \pmod{4\fa\fd^2}$.
In Theorem~\ref{thm: after Poisson} we show that the completed regular elliptic part becomes
\begin{multline*}
\frac{\kappa \abs{D_K}^{1/2} \sp^{-\sk/2}}{2^{3n-1} } \sum_{\pm}\sum_{\fa}\sum_{\fd}\sum_{Y\in\mathbb{F}_2^{n-1}}\sum_{\mu\in \cO_K/4\fa\fd^2}^{'}\binom{\mu^2 - 4 {u(Y)}\varepsilon, \fd}{\fa} \\
\sum_{\xi\in \mathbb{Z}^{n}}\sum_{\eta\in\mathbb{Z}^{n-1}} \frac{1}{\Norm_K(\fa)\Norm_K(\fd)^2}\left(\frac{\Phi_{\pm}^{1}(\xi, \eta)}{\Norm_K(\fd)\Norm_K(\fa)} + \Norm_K(\fd)\Psi_{\pm}^{1}(\xi, \eta)\right).
\end{multline*}

Next, we need to show that the sum coming from the units but not contributing to the Poisson summation (i.e.~the sum over $\pm Y$) is a finite sum.
We can (and will) apply Poisson summation to the inner (double) sum.
The main contribution to the spectral side of Poisson summation is when $(\xi, \eta) = (0, 0)$ since there is no exponential factor mitigating its growth.
Restricting to only this part proves the main theorem of this section which is Theorem~\ref{prop:dominanttermwith(tau,u)=0}.
\begin{lthm}
\label{Thm D}
The dominant part of the spectral side of the Poisson summation is
\[
\frac{ \kappa \abs{D_K}^{1/2}\sp^{-\sk/2}}{2^{3n-1} } \sum_{\pm}\sum_{Y\in\mathbb{F}_2^{n-1}}\sum_{\fa}\sum_{\fd}\left(\frac{\Phi_{\pm}^1(0, 0)}{\Norm_K(\fd)^3\Norm_K(\fa)^2} + \frac{\Psi_{\pm}^1(0, 0)}{\Norm_K(\fa)\Norm_K(\fd)}\right)K_{\fa, \fd}(\pm u(Y)).
\]
where $\Phi_{\pm}^1$ and $\Psi_{\pm}^1$ are the modified Fourier transform in $\mathbb{R}^{2n-1}$ of $\Phi_{\pm}$ and $\Psi_{\pm}$.
Given ideals $\fa, \fd$ of $\cO_K$ and a unit $u\in \cO_K^*$, the function $K_{\fa, \fd}(u)$ is a \emph{Kloosterman-type sum}.
\end{lthm}

\subsection{Evaluation of Kloosterman-type sums}

This is the content of \S\ref{sec: Kloosterman} and the proof is presented in Appendix ~\ref{Appendix_A}.
In Theorem~\ref{main theorem of the kloosterman section} we show that the Dirichlet series of a complex parameter $z = \sigma + it$ with $\sigma > 1$ defined as
\[
\D(z,u):= \sum_{\fd}\frac{1}{\Norm_K(\fd)^{2z+1}} \sum_{\fa}\frac{K_{\fa,\fd}(u)}{\Norm_K(\fa)^{z+1}}
\]
admits an analytic continuation to a meromorphic function in the whole complex plane with a pole at $z=\frac12$.
Very crucially, we show that this Dirichlet series is in fact independent of $u$.
More precisely,
\[
\D(z,u)=4^n\frac{\zeta_K(2z)}{\zeta_K(z+1)}\cdot\frac{1-1/\sp^{z(\sk+1)}}{1-1/\sp^{z}}.
\]

\begin{remark}
When $z = 1$, the above expression becomes 
\[
\D(z,u)=4^n\cdot\frac{1-1/\sp^{(\sk+1)}}{1-1/\sp}.
\]
This will recover the $\fp$-adic contribution of the trace of the trivial representation (the $4^n$ gets canceled with other terms).
At $z = 0$ the residue can be computed explicitly and is connected to the value of the special representation.
This concrete expression gives us exactly the contributions of both expected representations.
However, the proof of the above theorem is computational and does not provide insight of \emph{why} should this happen.
\end{remark}

Using the \textit{Fundamental Theorem of Arithmetic} and the local nature of the modified Hilbert symbol we factor the above sum into its local pieces.
These local factors look slightly different for primes above 2 and the odd primes.
This motivates the introduction (see Lemma~\ref{lemma:kl1}) of the local Dirichlet series $\D_{\fq}(z,u)$ satisfying 
\[
\D(z,u)= \prod_{\fq}\D_{\fq}(z,u).
\] 

The proof of evaluating $\D(z,u)$ proceeds by evaluating explicitly $\D_{\fq}$ for each prime $\fq$.
The evaluation of the $\D_{\fq}$ is a long computation which we do not discuss here.
We limit our discussion to some highlights:
\begin{itemize}
 \item the Kloosterman-type sum definition involves the restricted Hilbert symbol but after evaluating the Dirichlet series, this dependence disappears.
 We know this \textit{after} the computation; this begs the question, why does the symbol disappear?
 \item when $\fq\nmid 2$, evaluating $\D_{\fq}$ is precisely counting the number of solutions of a certain quadratic form modulo $\Norm_K(\fq)^b$ for varying values of $b$.
 Hensel's Lemma applies and it suffices to understand the situation modulo $\fq$, which is a field and lift the solutions appropriately.
 \item when $\fq\mid 2$, to apply Hensel's Lemma we need to consider solutions modulo $\fq^{2e_{\fq} + 1}$, where $e_{\fq}$ is the ramification degree of $\fq$ over $2$.
 Also, the counting problem is complicated because the local Kloosterman-type sum at primes above $2$ have a congruence condition modulo $4$ for a certain quotient.
 These two issues together make the computation of the corresponding counting problem quite challenging.
 Despite our best efforts to tame these Kloosterman-type sums, we were unable to do so.
 It is for avoiding this calculation we make the hypothesis that the prime 2 splits in $K$; see \ref{ass: split}.
\end{itemize}

These complications are already visible over $\mathbb{Q}$.
Here, the counting problem is to determine the number of solutions to a system of congruences.
When $\fq\nmid 2$, the quadratic character \textit{is} the modified Hilbert symbol; hence, the nature of one implies the behaviour of the other.
This is not the case for primes above $2${\textemdash}not even for the rational numbers.
The quadratic character assigns the value $-1$ to $3, 5, 7 \pmod{8}$ and the value $1$ to $1 \pmod{8}$.
The Kronecker symbol on the other hand assigns the value $-1$ to $3, 5 \pmod{8}$ and the value $1$ to $1,7\pmod{8}$.
This discrepancy at the value of $7$ makes the reciprocity laws work for the Kronecker symbol but not recover the quadratic character.
As a consequence, the behaviour of the solutions to the system is erratic as the two symbols contribute differently.

For more general number fields, this problem at primes above $2$ grows.
The behaviour of the square class is now dictated modulo $2e_{\fq} + 1$, but its behaviour has influence from both ramification degree as well as the inertia degree.
Also the formulas of the local Hilbert symbol at primes above $2$ are significantly more complicated.
It is unclear how to proceed to get closed formulas for these counting problems.
Nevertheless, the fact that we know what to expect should hint at the possibility that better methods should be available for this evaluation.

\begin{remark}
We wish to emphasize that these discussions and computations do not appear in \cite{altug2013beyond, AliI} where they are left as exercises for the reader.
This might create the false impression that evaluating Kloosterman-types sums at the primes above $2$ and at the odd primes have the same level of difficulty.
An important part of our contribution here lies in elucidating that this is not the case.
We are able to do it in the case of the rationals because of the explicit knowledge of how to evaluate the Kloosterman-type sums.
We can also tackle a few other cases{\textemdash}but it is still not in sufficient generality to get the desired formula over a general number field.
\end{remark}

\subsection{Isolation of the contribution of the trivial representation}
\label{intro - isolating the contribution}
In the last two sections of the paper, we apply tools from complex analysis and calculate the residues of the two integrals arising from the approximate functional equation and the Fourier transform in the Poisson summation formula.
We first isolate the contribution of the \textit{trivial representation}.
Concretely, the value of this representation is computed in Theorem~\ref{main theorem section 9}\ref{prop: trivial rep value} using the Weyl integration formula and it turns out to be 
\begin{equation}
\label{another eqn}
{\sp}^{-\sk/2} R_K \frac{1-{\sp}^{-(\sk+1)}}{1-{\sp}^{-1}} \sum_{\pm} \int_{\R^{2n -1}} \theta^{\pm}(x, y) dxdy,
\end{equation}
Here, $x$ (resp.~$y$) is viewed as a vector in $\R^n$ (resp.~$\R^{n-1}$) and $R_K$ is the regulator of $K$.
As in the case of \cite{AliI}, we can also find the trace of the special representation; this is done in Theorem~\ref{main theorem section 9}\ref{thm: Weyl Int special functions}.

We now explain the strategy of our proof which is a generalization of calculations done by Langlands.
However, there are some subtle differences and pitfalls to which we would like to draw the reader's attention.
Set $G_{\infty} = \GL(2, \R)\times\cdots\times \GL(2, \R)$ and set $Z_+$ to be the central subgroup of scalar matrices with positive entries.
As a Lie group, this is one dimensional; hence, $Z_+\setminus G_{\infty}$ has dimension $4n - 1$.
Let $f\in C_c(Z_+\setminus G_{\infty})$.
Since the group $Z_+$ is not a product group we cannot consider the action of $Z_{+}$ on $G_\infty$ component-wise.
In other words, it is not enough to compute the value of the trace at the archimedean places by doing local computations in each real embedding.

A crucial observation is that the dimensional reduction attributed to the $Z_+$-action on the spectral side corresponds to the fact that rank of the units is $n - 1$.
The remaining $n$ dimensions come from the trace lattice.

After performing the Weyl integration formula for our choice of the (archimedean) test function, the innermost integral in the Weyl formula becomes the orbital integral while the outer sum over tori gives a partition of $\mathbb{R}^{2n - 1}$.
Our careful computations show that as $T$ varies over the maximal tori, the inner integral varies over a region $S_T$ which (up to a set of measure 0) partitions $\{\pm 1\}\times \R^{2n-1}$.
For example, for the trivial representation, we show in Section~\ref{subsec: The Weyl integration formula and the trace of the Trivial Representation} that the trace $\Tr({\bf1}(f))$ is given by \eqref{another eqn}.
An analogous result for the case of the special representation is shown in Section~\ref{subsec: The trace of the special representation at the archimedean places}.
We emphasize that if the units were not taken into account by the Poisson summation and were fixed throughout the process by having an outer sum that runs over the units, we would be unable to arrive at this desired conclusion.
This is because the sum over the units might be indexed by an infinite set while the sum on the spectral side is indexed by a finite one.
Our approach ensures that all that remains from the units in the indexing set is the sign.

In \S\ref{sec: Isolation of the Contribution of the Trivial and Special Representations}, we identify the contribution of the trivial representation and its twists as well as the contribution of the special representation and its twists by grossencharacters on the geometric side.
To complete the proof of the main theorem (see Theorem~\ref{Thm E} below) we then isolate the contribution of the special representation.
The innermost term in the expression in Theorem~\ref{Thm D} consists of two sums, one with $\Psi_{\pm}^1(0, 0)$ and another with $\Phi_{\pm}^1(0, 0)$.
We explicitly write down the (modified) Fourier transform of $\Psi_{\pm}^1(0, 0)$ and $\Phi_{\pm}^1(0, 0)$, invoke absolute convergence, and change the order of integrals with the sums over $\fa, \fd$.
More precisely, in \eqref{eq:contour1new} we get an expression with three integrals for the sum corresponding to $\Phi_{\pm}^1(0, 0)$: the two outer ones come from the Fourier transform{\textemdash}one of which corresponds to the variable $x$ and the other to the variable $y${\textemdash}together they correspond to an integral over $\mathbb{R}^{2n - 1}$.
The complex line integral comes from the definition of $F$ in the approximate functional equation.
We can shift the contour of integration to the left and pass over the poles introduced by the zeta functions and by the Mellin transform of $F$.
This introduces three factors: two residues and the new contour integral.
In Lemma~\ref{lem:Contribution1new} we calculate these values precisely and the first of the three pieces is 
\[
2^{n-1}h_KR_K \sp^{-\sk/2}\frac{1-\sp^{-(\sk+1)}}{1-{\sp^{-1}} } \sum_{\pm} \int \int \theta^{\pm}(x,y) dy dx.
\]
A similar analysis for $\Psi_{\pm}^1(0, 0)$ is done in Lemma~\ref{lemma 9.2}.
After appropriate cancellations we at last obtain our main theorem which is Theorem~\ref{THE THEOREM} in the main text.

\begin{lthm}
\label{Thm E}
Let $f$ be the fixed test function.
Writing $\mathbf{1}$ to denote the trivial representation and $\xi_0$ to denote the special representation, the dominant part of the spectral side of the Poisson summation is equal to
\[
\sum_{\chi} \Tr((\boldsymbol{1}\otimes\chi)(f))
- 2 \sum_{\chi} \Tr((\xi_0\otimes\chi)(f)) + \textrm{remainder terms}.
\]
\end{lthm}

\begin{remark}
This is the exact generalization of Altu\u{g}'s result.
To see this, observe that the final expression obtained in \cite[Theorem 6.1]{AliI} is substituted using the spectral terms that are computed in \cite[Lemma 6.2]{AliI}.
Altu\u{g}'s Theorem 6.1, in the notation used there, can be expressed as $(13)_{\xi_0} = \Tr({\bf{1}}(f^{p,k})) - 2\Tr(\xi_0(f)^{p,k}) +$ remainder terms.
The main theorem \cite[Theorem 1.1]{AliI} is then simply a restatement of this.

From the perspective of Langlands' original Beyond Endoscopy proposal, the key result here is the cancellation of the contribution of the trivial representation which in this case is the only non-tempered representation of $\GL(2)$ in the discrete spectrum assuming the Ramanujan conjecture.
That the second term appears with the same constant multiple as Altu\u{g}'s is a surprise (as it was in the original case), but is no obstruction in itself as the distribution is tempered.
Though as noted by Arthur \cite[\S3.V]{problems}, in higher rank the associated terms can indeed be non-tempered, so being able to identify it on the geometric side is valuable for further studies.
\end{remark}

\subsection{Future Directions and Motivation for Further Research}
\label{new section}

Our intention of writing this long introduction is to convince the reader that the quest of generalizing \cite{AliI} to general number fields was full of surprises and that the mathematics that lurks behind all of this is beautiful.
We made our best efforts to encapsulate these ideas in the questions we propose.
We hope they are the right questions.

\subsubsection{Formula for the non-archimedean orbital integral}
This precise formula for the product of non-archimedean orbital integrals is a fundamental problem when rewriting the regular elliptic part of the trace formula.
Obtaining such a precise formula for $\GL(n)$ when $n\geq 3$ and for other groups is an open problem.
In \cite{artstrat}, J.~Arthur put forth a conjecture regarding what it must be in the case of $\GL(n)$.
For the case of $\GL(2)$, this conjecture was verified in two papers by the second named author, see \cite{malors21, espinosa2023zeta} and was crucially used in this work.

\subsubsection{Relaxing the hypothesis on the number field being being totally real and positive}
As explained before, we made a technical assumption that the number field $K$ is totally real and positive so as to define an interpolation function.
We are optimistic that our results can be proved without this hypothesis.
However, we used both these conditions in significant ways and the process of its removal will enlighten us on how the number theoretic aspects of $K$ inform the representation theory at the archimedean places.
It appears that settling the issue on the existence of an appropriate interpolating function is a problem of independent interest.
This depends on the field $K$ and is unrelated to its application to the strategy of Beyond Endoscopy.

\subsubsection{Reinterpreting the congruence conditions to complete the lattice}
When $G=\GL(n)$ for $n\geq 3$, it is not clear to the authors how to rewrite the condition $\fd\mid S_\gamma$ in an explicit way so that it depends on the parameters that produced the conjugacy class.
We remind the reader that establishing such an explicit dependence is inherently essential if we want to complete the lattice for performing Poisson summation.

It would be interesting if the answer to the question is given as
congruence conditions of some sort as in the case of $\GL(2)$.
In fact, Langlands already hints at the existence of the congruence
conditions for $\GL(2, \Q)$ in \cite{LanBE04}.
He writes
``The sum over $r\in \Z$ that occurs in the elliptic term will be replaced by sums over $r$ satisfying a congruence condition.
This will entail that whatever behaviour we find for $S = \{\infty \}$ should remain valid when congruence conditions are imposed.
Such an assertion, which implies a greater theoretical regularity that may make the proofs easier to come by, has to be tested further, but these are the principles that justify confining myself at first to $m = 1$ and to representations unramified at all finite places.”

The congruence conditions referred to by Langlands are precisely those given by the congruence conditions we have discussed earlier.
The regularity is manifested in the periodicity which is intrinsic to the symbol.
Our investigations show that such congruence conditions still hold for all symmetric powers (i.e.~all choices of $k$, which Langlands calls $m$).
However, if class field theory is to be any guide, we think that the congruence conditions might not be the appropriate way to describe the equivalent form of the divisibility conditions.
As we know, the splitting behaviour of primes in non-abelian extensions cannot be described by congruences alone.
In essence, we might be looking for more complicated representation theory conditions which, in the abelian case, collapse to congruences.

\subsubsection{Relaxing the hypothesis involving splitting of 2 and more general Kloosterman-like sums}
The condition on the prime 2 is a technical hypothesis.
We are optimistic that it should be possible to compute the value of the Kloosterman-like sums for the primes above 2 when the ramification or inertia degree are not both 1.

Another question to ask is what are these Kloosterman-like sums for other cases? 
Giving a complete description in the general case is beyond our reach since the shape that the regular elliptic part of the trace formula will take is still unknown to us.
However, we might succeed in making educated guesses for situations when the base field $K$ is well behaved (e.g.~a Kummer extension of $\mathbb{Q}$) and the prime splitting is well-understood.

Finally, we reiterate that the value of $\D(z, u)$ does not depend on the unit $u$.
This makes us believe that the computations might be doable in a better way that elucidates why this parameter disappears.

\subsubsection{The appearance of \texorpdfstring{$-2$}{} in Theorem~\ref{Thm E}}
At the time of writing this article, the authors are perplexed by the appearance of $-2$ in Theorem~\ref{Thm E} as it matches Altu\u{g}'s calculations exactly.
We note that the negative sign is expected because it appears in the trace formula computations.
However, the appearance of 2 in our result and in \cite{AliI} is surprising.
A common feature in the trace formula theory is to build objects with `nicer' properties out of others which lack the property (for example, stable orbital integrals from invariant ones).
In this process, each of the building blocks appears with some multiplicity.
Is this $-2$ a foreshadowing of a construction of a new trace formula with `nicer' properties?
If so, what is it that we are constructing that requires the special representation to have multiplicity $-2$ as opposed to $-1$?

\section*{Notation}
For the ease of reading this paper, we now summarize notation that is used throughout.

\subsubsection*{Terms related to a number field/$p$-adic local field} 

\begin{itemize}
\item $K/\Q$ is a totally real number field of degree $n$.
\item $\A_K$ is the ring of ad\'eles of $K$.
\item For a general number field $F$, $\sigma$ denotes the real embeddings and $\eta$ denotes the complex embeddings.
There are $r_1$-many real embeddings $r_2$-many pairs of complex embeddings.
\item For a general number field, the map $\sigma_i: F \rightarrow \R$ denotes the $i$-th real embedding with $1\leq i \leq r_1$ and $\eta_j$ are the complex embeddings with $1\leq j \leq r_2$.
For an element $\alpha\in K$, we sometimes write $\alpha_i = \sigma_i(\alpha)$ when there is no scope of confusion.
\item For a general number field $F$, the ring of integers is $\cO_F$, the class number is $h_F$, the regulator is $R_F$, the number of units in $F$ is $\omega_F$, and $D_F$ is the discriminant.
\item $\zeta_K(z)$ is the Dedekind zeta function of $K$ and $\xi_K(z)$ is the completed Dedekind zeta function of $K$.
\item $\kappa$ is the residue of $\zeta_K$ at $z=1$.
\item ideals of $K$ will be denoted by gothic letters $\fa,\mathfrak{b},\fp,\fq$; the prime ideal $\fp$ is fixed throughout the paper.
\item $\nu$ is an infinite place of $K$.
\item For a prime (ideal) $\fq$ of $K$, the $\fq$-adic completion of $K$ is denoted by $K_{\fq}$.
Its uniformizer is $\pi_\fq$.
\item $\Norm_K(\fa)$ is the absolute norm of $\fa$ from $K$ to $\Q$; $\Norm_K(\fp) = \sp$ and $\Norm_K(\fq)=\sq$.
\item $h_{\fp}$ is the order of $\fp$ in the class group of $K$.
\item $\rho_{\fp} = \rho$ is the generator of $\fp^{h_{\fp}}$.
\item For fixed positive integer $\sk$ divisible by $h_{\fp}$, set $\sk' = \sk/h_{\fp}$ and $\varepsilon = \rho^{\sk/h_{\fp}} = \rho^{\sk'}$.
\item $u,v$ are units in $\cO_K^*$.
\item $\beta_1, \dots, \beta_{n-1}$ is a system of fundamental units in $\cO_K^*$ realizing the isomorphism in Dirichlet's Unit Theorem.
\item $u=u(\pm,y) = \pm \beta_1^{y_1}\cdots \beta_{n-1}^{y_{n-1}}$ for a vector $y=(y_1, \dots, y_{n-1})\in \Z^{n-1}$.
If the sign is fixed, we write $u=u(y)$.
\item $L(u) = \{\tau\in \cO_K : \tau^2 - 4u\varepsilon \neq\square \}$.
\end{itemize}

\subsubsection*{Terms related to the regular elliptic element}
\begin{itemize}
\item $R$ is any ring
\item $G(R)=\GL(2,R)$ is the algebraic group of invertible $2\times 2$ matrices, over $R$.
\item $\gamma$ is the regular elliptic element.
\item $G_{\gamma}(R)$ is the group centralizer where $R\in \{K, K_\fq, \cO_{K_\fq}, K_\nu, \A_K\}$.
\item $\vol(\gamma)$ is the volume of the regular elliptic element.
\item $L=K(\gamma)$ is a quadratic extension of $K$.
\item $P_{\gamma}(X)=P(X)$ is the characteristic polynomial of $\gamma$ with roots $\lambda_1,\lambda_2$.
\item $D(\gamma)=D(\lambda_1,\lambda_2)$ is the Weyl discriminant.
\item $\tau=\Tr(\gamma),\ \det(\gamma)$ are the trace and determinant of $\gamma$.
Set $\partial = \tau^2 - 4\det(\gamma)$.
\item $\ch: G(\R) \longrightarrow \R^2$ is a map defined via $\gamma \mapsto ( \Tr(\gamma),4\det(\gamma))$.
\item $\CH$ is the map $\ch \otimes \cdots \otimes \ch: G_{\infty} \longrightarrow \R^2\times\cdots\times \R^2$
\item $T_{\el}$ is the elliptic torus.
\item $T_{\spl}$ is the split torus.
\item $f_{\fq}$ is a test function at $\fq$, $f_\infty$ is a test function at the real place(s), and $f=\prod_{\fq}f_{\fq}\times f_{\infty}$.
\item $\cO(f_\fq,\gamma)$ is the local orbital integral at $\fq$ and $\cO(f_\infty,\gamma)$ is the local orbital integral at the archimedean place(s).
\item $\cO(f,\gamma)$ is the ad\'elic orbital integral.
\item $R(f)$ is the regular elliptic part of the trace formula for $\GL(2,K)$.
\item $\overline{R(f)}$ is the completed regular elliptic part of the trace formula for $\GL(2,K)$.
\item $\Sigma(\square) = \overline{R(f)}-R(f)$.
\item $\overline{\Sigma_0}(f)$ = completed regular elliptic part with $\xi=\eta=0$.
\item $\chi_\gamma$ is the quadratic sign character associated to the quadratic extension $L=K(\gamma)$ over $K$.
\item $D_{L/K}(\gamma)$ is the discriminant of the element $\gamma$.
\item $\Delta_{\gamma}$ is the relative discriminant of the quadratic extension $L/K$.
\item $S_\gamma$ is an ideal defined via
\[S_{\gamma}^2 = \frac{D_{L/K}(\gamma)}{\Delta_{\gamma}}.\]
\item $\chi_{\fd}$ is an imprimitive quadratic Hecke character and $L(z,\chi_{\fd})$ is the associated $L$-function.
\item $\theta^{\pm}: \mathbb{R}^{2n-1} \rightarrow \mathbb{C}$ is the interpolation function.
\item $\mathcal{P}^{\pm}: \mathbb{R}^{2n-1} \rightarrow \mathbb{R}$ is the discriminant function.
\end{itemize}

\subsubsection*{Terms related to the approximate functional equation}
\begin{itemize}
\item For a complex number $z$ and non-square $\delta\equiv 0,1\pmod{4}$, write $L(z, \delta)$ to denote the Zagier zeta function.
\item $L(z,\gamma)$ is the generalized multiplicative formula of Langlands for $\GL(2,K)$.
\item $L_{\nu}(z,\chi_\gamma) = L_\R(z)$ is the $L$-factor at the archimedean place $\nu$ for the character $\chi_\gamma$.
\item $\Lambda(z,\gamma)$ is the completed multiplicative formula of Langlands.
\item $F(x)$ is the smooth cut-off function on $\R^+$ with Mellin transform $\widetilde{F}(x)$.
\end{itemize}

\subsubsection*{Symbols}
\begin{itemize}
\item The Kronecker symbol is denoted by $\binom{\cdot}{\cdot}_K$.
\item The Hilbert symbol is denoted by $ \binom{\cdot, \cdot}{\bullet}_H$, where $\cdot$ is a number and $\bullet$ is an ideal.
\item The restricted Hilbert symbol is denoted by $ \binom{\cdot}{\bullet}_{rH}$.
\item The modified Hilbert symbol is denoted by $ \binom{\cdot, \bullet}{\bullet}$.
\end{itemize}

\subsubsection*{Terms related to the Kloosterman-type sum}
\begin{itemize}
\item $K_{\fa, \fd}(u)$ denotes the Kloosterman-type sum.
\item $\D(-,u)$ denotes the auxiliary Dirichlet series associated to $K_{\fa, \fd}(u)$.
\item For $\upsilon,r\in \Z_{\geq 0}$, the local Kloosterman-type sum at $\fq$ are written as $\widetilde{K}_{\fq^\upsilon, \fq^r}(u)$.
\item $\D_{\fq}(-,u)$ denotes the local auxiliary Dirichlet series associated to the local Kloosterman-type sum.
\end{itemize}

\subsubsection*{Terms related to the representation theory side}
\begin{itemize}
\item $\mathbf{1}(f)$ is the trivial representation for the test function $f$.
\item $\xi_0(f)$ is the special representation for the test function $f$.
\item For any compact connected Lie group $G$ with maximal torus $T$, write $W(G,T)$ to denote the Weyl group of $G$ with respect to $T$.
\end{itemize}

\subsection*{Acknowledgments}The authors would like to thank the American Institute of Math for sponsoring the workshop ``Functoriality and the trace formula" in December 2017 and for providing a productive and enjoyable environment for our initial discussions on this work.
We would especially like to thank the organizers Ali Altu\u{g}, James Arthur, Bill Casselman, and Tasho Kaletha for making the workshop and this collaboration possible.
We would also like to thank James Arthur, Bill Casselman,  Jayce Getz, Julia Gordon, and Bin Xu for helpful conversations around this work.
Part of this work was carried out during the NUS program ``On the Langlands Program: Endoscopy and Beyond" in December 2018--January 2019 and we thank them for their hospitality.
We also thank the Fields Institute for hosting the Beyond Endoscopy Mini Conference in April 2023 and providing us with a wonderful working atmosphere.
ME was partially supported by an AMS-Simons Travel Award and the NSF grant DMS-2002085.
DK was partially supported by a PIMS Postdoctoral Fellowship during the preparation of this article.
She acknowledges the support of NSERC Discovery Grant RGPIN-2026-07384 and DGECR-2026-00254 during the revision stages of the article.
TAW was partially supported by NSF grant DMS-2212924.
We thank the referee for a careful reading of the article and for making several suggestions which helped improve the paper.

\section{The Main Setup}
\label{sec: The Main Setup}

This section is preliminary in nature.
The main goal is to set the stage, give precise definitions, and describe the main setup of our work.
We clarify our first main hypothesis on a certain divisibility condition that must be satisfied.

\subsection{Definition of the test function and its implications}
\label{sec: Definition of the test function and its implications}

For any ring $R$, set $G(R)=\GL(2,R)$ to denote the algebraic group of $2\times 2$ invertible matrices over $R$.
Throughout this paper, $K$ is a totally real number field\footnote{Many statements we prove in this paper do not require that the number field is totally real.
However, this is a technical hypothesis we need to impose for certain calculations which will be clarified in \S\ref{section: blending step}.}.

\begin{definition}
An element $\gamma\in G(K)$ is called a \emph{regular elliptic element} if its characteristic polynomial is irreducible over $K$.
Given such a $\gamma$, define its \emph{characteristic polynomial} as
\[
P_{\gamma}(X) = P(X) = X^2 - \Tr(\gamma) X + \det(\gamma) 
\] 
where $\Tr(\gamma)$ and $\det(\gamma)$ are the trace and determinant of $\gamma$, respectively.
\end{definition}

\begin{notation}
We henceforth write $\tau=\Tr(\gamma)$.
\end{notation}

This polynomial $P_{\gamma}(X)$ generates a quadratic extension $L=K(\gamma)$ of $K$.

\begin{proposition}
\label{prop kot 5.9}
Let $\gamma\in G(K)$ be a regular elliptic element.
\begin{enumerate}[label = \textup{(\roman*)}]
\item The algebra centralizer
\[
{\Mat}_{\gamma}(K) = \{X\in {\Mat}_{2\times 2}(K) \mid \ X\gamma = \gamma X\}
\]
is (field) isomorphic to $L$.
Also, the diagonal matrices are identified with $K$.
\item The group centralizer $G_{\gamma}(K) = \GL(2, K)_{\gamma}\subset {\Mat}_{\gamma}(K)$ coincides with the invertible elements $L^*$ of $L$.
\item Under the previous identifications, $\gamma$ is a root of its characteristic polynomial, allowing identifying the matrix $\gamma$ with an element of $L$.
\end{enumerate}
\end{proposition}

\begin{proof}
See \cite[\S 5.9]{Art05_Harmonic-Analysis}.
\end{proof}

\begin{remark}{\leavevmode}
\begin{enumerate}[label = (\roman*)]
\item Even though $\gamma$ is a priori a matrix, Proposition~\ref{prop kot 5.9} allows us to view it as an element of $L$.
We will shift between both perspectives without specifying when there is no possibility of confusion.
\item When $\tau$ and $\det(\gamma)$ are algebraic integers of $K$, the regular elliptic element $\gamma$ is an algebraic integer of $K(\gamma)$.
\end{enumerate}
\end{remark}

Throughout this paper, \textbf{we fix a prime ideal $\fp$ in $K$}.
Let $\fq$ be any prime ideal of $\cO_K$, which depending on the situation may or may not be equal to $\fp$.
The $\fq$-adic completion of $K$ is denoted by $K_{\fq}$.
A regular elliptic element $\gamma\in G(K)$ is an element of $G(K_{\fq})=\GL(2, K_{\fq})$.
It generates a quadratic extension of reduced $K_{\fq}$-algebras and there are three possibilities for these extensions:
\begin{enumerate}[label = (\alph*)]
\item the split case (i.e.~the extension is $K_{\fq}\times K_{\fq}$),
\item the non-split unramified case (i.e.~the extension is unramified),
\item the non-split ramified case (i.e.~the extension is ramified).
\end{enumerate}
This behaviour parallels the splitting behaviour of $\fq$ in the quadratic extension $L/K$.
In particular, $\gamma$ can be identified with an element of the (local) extension of $K_{\fq}$.

\subsubsection{Choice of Measures}
\label{assumption: measure}
For ease of calculations, in this section we follow the choices of measures made by Langlands, and we explain them below: 
\begin{itemize}
\item At a finite place $\fq$, we equip $G(K_{\fq})$ with the Haar measure such that the maximal compact set $G(\cO_{K_{\fq}})$ has measure 1.
 For a regular elliptic element $\gamma$, the centralizer $G_{\gamma}(K_{\fq})$ is also endowed with a Haar measure such that $G_{\gamma}(\cO_{K_{\fq}})$ has measure 1.
 \item At an infinite place $\nu$, we equip $G(K_{\nu})$ with the standard Lebesgue measure.
For a regular elliptic element $\gamma$, the centralizer $G_{\gamma}(K_{\nu})$ is endowed with a measure that depends on whether this torus is split or elliptic.
Representatives of these two tori are, respectively,
\[
\left\{\begin{pmatrix}
 \lambda_1 & 0\\
 0 & \lambda_2\\
 \end{pmatrix}\mid 
 \lambda_1, \lambda_2 \neq 0 \right\} \ \textrm{ and } \
 \left\{\begin{pmatrix}
 r\cos(\theta) & r\sin\theta\\
 -r\sin(\theta) & r\cos\theta\\
 \end{pmatrix}
 \mid r > 0,\; 0\le \theta \le \pi \right\}.
\]
The elliptic torus can also be described as
\[
\left\{\begin{pmatrix}
 a & -b\\
 b & a\\
 \end{pmatrix}
 \mid a, b\in\R,\; a^2 + b^2 \neq 0 \right\}.
\]
For each case our choice of measure is as follows:
\begin{itemize}
\item \emph{split case:} we endow it with the measure $\frac{d\lambda_1d\lambda_2}{\abs{\lambda_1\lambda_2}}$, where $\lambda_1$ and $\lambda_2$ are the eigenvalues of elements in the split torus $G_{\gamma}(K_{\nu})$.
\item \emph{elliptic case:} we parameterize letting $0 < r < \infty$; and the invariant Haar measure is $2drd\theta/r$.
This measure coincides, in the other coordinates, with $2dadb/(a^2 + b^2)$.
\end{itemize}
\end{itemize}

In each situation, it follows from the general theory of reductive groups that the space $G_{\gamma}\setminus G$ admits a unique right invariant measure that makes the integration in stages formula true without constants.

\subsubsection{Definition of the Test Function}
As mentioned in the introduction, we need to manipulate the regular elliptic part for a specific choice of test function.
We now explain its construction.

\begin{definition}
\label{defi: ffq(k) defi}
Let $\fa$ be any ideal of $\cO_K$ and write $\Norm_K(\fa)$ to denote the absolute norm of $\fa$ from $K$ to $\Q$.
Let $k$ be a fixed non-negative integer and let $\fq$ be a prime ideal of $K$ (possibly equal to $\fp$).
Write $\mathbbm{1}$ to denote the indicator function of the maximal compact subgroup $G(\cO_{K_{\fq}})$.
For a non-negative integer $k$ define
\begin{small}
\begin{align*}
f_{\fq}^{(k)} &:= \mathbbm{1}\left(X\in\Mat(\cO_{K_{\fq}})\mid\; \abs{\det(X)}_{\fq} = \Norm_{K}(\fq)^{-k}\right) \text{ and }\\
f_\fq &:= \begin{cases}
f_{\fq}^{(0)}
 = \mathbbm{1}\left(X\in\Mat(\cO_{K_{\fq}})\mid \abs{\det(X)}_{\fq} = 1\right) & \text{ if }\fq\neq\fp\\
\Norm_{K}(\fp)^{-k/2}f_{\fp}^{(k)}
 =\Norm_{K}(\fp)^{-k/2}\mathbbm{1}\left(X\in\Mat(\cO_{K_{\fp}})\mid\; \abs{\det(X)}_{\fp} = \Norm_{K}(\fp)^{-k}\right) & \text{ if }\fq=\fp.
\end{cases}
\end{align*}
\end{small}
\end{definition}

At the archimedean places, we denote the test functions (at each real place) by $f_{\nu_i}$.
We also define
\[
f_{\infty} = \prod_{i = 1}^{n} f_{\nu_i}.
\]
The \emph{only} conditions we impose on the test functions is that their orbital integrals must have compact support as functions of the trace and the determinant (see Section~\ref{sec: arch orb int} for a more detailed discussion on what this means) and that
\[
f_{\infty}\in C_c(Z_+\backslash G_{\infty}).
\]
With this, we now construct our test function.
\begin{definition}
\label{defi: k fixed here}
Let $K/\Q$ be a totally real number field of degree $n$.
Fix a positive integer $k$ as in Definition~\ref{defi: ffq(k) defi}.
The test function used throughout this paper is the following
\[
f : = \prod_{\fq} f_{\fq} \times f_{\infty} = \Norm_{K}(\fp)^{-k/2}f_{\fp}^{(k)} \times \prod_{\fq\neq \fp} f_{\fq}^{(0)} \times f_{\infty}.
\]
Set
\[
f_{\textrm{fin}} = \prod_{\fq} f_{\fq} = \Norm_{K}(\fp)^{-k/2}f_{\fp}^{(k)} \times \prod_{\fq\neq \fp} f_{\fq}^{(0)}.
\]
\end{definition}

Since this function can be expressed as a product, its ad\`elic orbital integral can be split into the product of the local orbital integrals.
More precisely, let $\gamma$ represent a regular elliptic element in $G(K)$; then
\[
\cO(f, \gamma) = \prod_{\fq} \cO(f_{\fq}, \gamma)\times \cO(f_{\infty}, \gamma) = \cO(f_{\fin}, \gamma)\times \cO(f_{\infty}, \gamma).
\]

\subsubsection{Definition of the orbital integral}
In view of our choice of measures, the orbital integrals are defined as follows: 
\begin{align*}
\cO(f_{\fq}, \gamma) &:= \int_{G_{\gamma}(K_{\fq})\setminus G(K_{\fq})} f_{\fq}(x^{-1}\gamma x) dx,\\
\cO(f_{\infty}, \gamma) &:= \int_{G_{\gamma}(\R)\setminus G(\R)} f_{\infty}(x^{-1}\gamma x) dx.
\end{align*}
These orbital integrals are invariant under conjugation by $\gamma$.
If this integral is not $0$, then at some point the integrand must not be $0$ itself.
Hence, there exists $x\in G(K_{\fq})$ such that
\[
x^{-1}\gamma x \in G(\cO_{K_{\fq}})
\]
and satisfies the condition of the indicator function appearing in the definition of $f_{\fq}$.
Since the determinant is invariant under conjugation, we conclude that
\[
\abs{\det(x^{-1}\gamma x)}_{\fq} = \abs{\det(\gamma)}_{\fq} =\left\{
\begin{array}{cc}
 \Norm_{K}(\fp)^{-k} & \text{if }\fq= \fp\\
 1 & \text{otherwise}.
\end{array}
\right.
\]
This specifies the value of the $\fq$-adic norm of $\gamma$ at every finite place $\fq$.
The following result is now almost immediate.

\begin{proposition}
\label{prop 3-7}
Let $\gamma\in G(K)$ be a regular elliptic element such that $\cO(f_{\fin}, \gamma) \neq 0$.
Then $\det(\gamma)\in \cO_K$ and $\tau\in \cO_K$.
Fix an integer $k$ as in Definition~\ref{defi: k fixed here}.
As ideals,
\[
(\det(\gamma)) = \fp^k.
\]
\end{proposition}

\begin{proof}
The only part that requires a proof is the claim that $\tau\in \cO_K$.
Let $\lambda_1, \lambda_2$ be the roots of $P_{\gamma}(X)$ in $L/K$.
Since they are Galois conjugates, they have the same norm (in $L$).
If $\fq_L\mid \fq$ is a prime ideal in $L$, the $\fq$-adic norm can be extended to a $\fq_L$-adic norm.
\[
\abs{\lambda_1 + \lambda_2}_{\fq_L} \le \max\left\{\abs{\lambda_1}_{\fq_L}, \abs{\lambda_2}_{\fq_L}\right\} = \abs{\lambda_1}_{\fq_L} = \abs{\lambda_1\lambda_2}_{\fq_L}^{1/2} = \abs{\det(\gamma)}_{\fq_L}^{1/2}\le 1.
\]
We deduce that $\tau = \lambda_1 + \lambda_2 \in \cO_L$.
However, $\tau$ is an element of $K$, so $\tau\in \cO_K$.
\end{proof}

\begin{definition}
Conjugacy classes with $\cO(f_{\fin}, \gamma) \neq 0$ will be referred to as \emph{contributing conjugacy classes}.
\end{definition}

\begin{corollary}
If $\gamma$ represents a contributing conjugacy class, then it is an algebraic integer in $L=K(\gamma)$ and in $L_\fq=K_{\fq}(\gamma)$ for all primes $\fq$.
\end{corollary}

\begin{definition}
\label{defi: volume}
Set $\A_K$ to denote the ring of ad\`eles of $K$.
Let $Z_+$ be the connected central subgroup of $G_{\gamma}(K)$.
The \emph{volume} of the regular elliptic element $\gamma$ is defined as
\[
\vol(\gamma) := \vol\left(Z_{+}G_{\gamma}\left(K\right)\setminus G_{\gamma}\left(\A_{K}\right)\right).
\] 
\end{definition}

Our object of study is the regular elliptic part of the trace formula for $\GL(2, K)$.

\begin{proposition}
\label{prop: using arthur to define RE part}
The regular elliptic part of the trace formula for $\GL(2, K)$ is 
\[
R(f) = \sum_{\gamma} \vol(\gamma)\cO(f,\gamma),
\]
where the sum runs over all conjugacy classes in $\GL(2, K)$ of regular elliptic elements.
\end{proposition}

\begin{proof}
Technically, this statement follows immediately from definitions.
We make references to the literature so that the reader can chase the appropriate definitions for a general number field.

\textbf{Within this short proof we follow the notation of \cite{Art-intro}.}
In (27.1) 
it is stated that the regular elliptic part of the invariant trace formula, for a general reductive group $G$ over $F$, is
\[
\sum_{\gamma\in\Gamma_{\textrm{reg, ell}}(G)} a^G(\gamma)f_G(\gamma).
\]
Here, $a^G(\gamma)$ stands for the volume we defined in Definition~\ref{defi: volume} and $f_G(\gamma) = \cO(f, \gamma)$.

As explained in the line above this equation in \emph{op.~cit.}, $\Gamma_{\textrm{reg}, {\el}}(G)$ consists of the equivalence classes of regular elliptic conjugacy classes in $\Gamma(G)_S$, where $S$ denotes the set of archimedean places.

In (22.1), it is defined that
\[
\Gamma(G)_S = (G(F))_{G, S},
\]
and the notation on the left consists of conjugacy classes under $(G, S)$-equivalence.
This concept is defined in p.~113, and consists of two conditions that must be satisfied (the second of which refers to a definition on unipotent elements given on p.~110).
These conditions \textit{never} identify two different regular elliptic conjugacy classes.
Hence, each regular conjugacy class contributes a different term to the regular elliptic part of the trace formula.
The result follows.
\end{proof}

Now we can prove the following result.
\begin{theorem}
\label{thm: condn on k and h}
Let $f$ be a test function as before.
If the order of $\fp$ in the class group of $K$ \emph{does not} divide the non-negative integer $k$, then the regular elliptic part of the trace formula for the test function $f$ vanishes.
Equivalently,
\[
\sum_{\gamma}\vol(\gamma)\cO(f, \gamma) = 0.
\]
\end{theorem}

\begin{proof}
If the regular elliptic part of the trace formula for the function $f$ does not vanish, then Proposition~\ref{prop 3-7} asserts that $(\det(\gamma)) = \fp^k$.
Hence, $\fp^k$ is a principal ideal and the result follows.
\end{proof}

\begin{remark}{\leavevmode}
\begin{enumerate}[label = \textup{(\roman*)}]
\item Even when $\gamma$ is a contributing conjugacy class, it is possible that the archimedean contribution of the orbital integral is zero.
In other words, when $\gamma$ is a contributing conjugacy class, the summand $\vol(\gamma)\cO(f,\gamma)$ can vanish precisely when the infinite orbital integral does.
\item Set $h_{\fp}$ to denote the order of $\fp$ in the class group of $K$.
In view of Theorem~\ref{thm: condn on k and h}, we assumpe $h_{\fp}\mid k$ throughout this paper.
This allows manipulating the regular elliptic part and finding the contribution of the trivial representation in it.
\end{enumerate}
\end{remark}

\begin{equation}
\tag*{\textup{\textbf{(Hyp-div)}}}
\label{ass: class number}
\begin{minipage}{0.8\textwidth}
With notation as above, we assume $h_{\fp}\mid k$ and refer to this as (Hyp-div).
\end{minipage}
\end{equation}

\begin{notation}
Fix a positive integer $\sk$ satisfying \ref{ass: class number}.
Since $\fp^{h_{\fp}}$ is principal, we fix a generator $\rho_{\fp} = \rho$ of $\fp^{h_{\fp}}$
and suppose that $\sk = {h_{\fp}}\sk'$.
Then 
\begin{equation}
\label{eqn: det gamma and rho}
\left(\det\left(\gamma\right)\right) = \fp^{\sk} = \fp^{{h_{\fp}}\sk'} = \left(\rho\right)^{\sk'} = \left(\rho^{\sk'}\right).
\end{equation}
Set $\varepsilon = \rho^{\sk'}$.
Note that $\rho$ depends on the prime $\fp$, up to units of $\cO_K$, and $\sk'$ depends on $\sk$ and $\fp$.
There is no dependence on $\gamma$.
In particular, it is possible to define and fix $\varepsilon$ without any reference to $\gamma$.
\end{notation}

\subsection{The analytic class number formula}

So far we have shown that
\[
\sum_{\gamma} \vol\left(\gamma\right)\cO\left(f, \gamma\right) = \sum_{\left(\det\left(\gamma\right)\right) = \fp^{\sk}} \vol\left(\gamma\right)\cO\left(f, \gamma\right).
\]
Next, we simplify the volume term and express it in terms of arithmetic quantities associated with the number fields $K$ and $L=K(\gamma)$.
We first quote the following well-known result.

\begin{proposition}
\label{centralizer idele equation}
For a regular elliptic element $\gamma$, there exists a measure preserving isomorphism
\[
Z_{+}G_{\gamma}(K)\setminus G_{\gamma}(\A_{K}) \cong L^*\setminus \I_{L}^1,
\]
where $\I_{L}^1$ are the unit id{\`e}les.
\end{proposition}

\begin{proof}
See {\cite[Proposition~26.1]{KL06}}.
\end{proof}

The right hand side of the above isomorphism is the object of study when one proves the analytic class number formula via id{\`e}lic methods.
The following proposition calculates its measure.
Thus, by Proposition~\ref{centralizer idele equation} we know its precise value.

\begin{proposition}
\label{prop: volume manipulation}
Let $\gamma$ be a regular elliptic element and let $\chi_{\gamma}$ denote the quadratic sign character associated the quadratic extension $L/K$ generated by $\gamma$.
Then, by transferring the id{\`e}lic measure from Proposition~\ref{centralizer idele equation}, we have
\[
\vol(\gamma) = \frac{2^{r_1}h_KR_K}{\sqrt{\abs{D_K}} \omega_K} L(1, \chi_{\gamma}) \sqrt{\abs{D_L}},
\]
where $h_K$ is the class number of $K$, $R_K$ is the regulator of $K$, and $D_F$ is the discriminant of the number field $F\in \{K,L\}$.

\end{proposition}

\begin{proof}
Let $r_1^\prime$ (resp.~$r_2'$) denote the number of real (resp.~complex) embeddings of $L$.
Therefore,
\begin{align*}
\vol(\gamma) 	&:= \vol\left(Z_+G_{\gamma}\left(K\right)\setminus G_{\gamma}\left(\A_K\right)\right)\\
&= 	\vol\left(L^*\setminus \I_L^1\right) \qquad \textrm{by Proposition~\ref{centralizer idele equation}}\\
&= 	\frac{2^{r_1^\prime} (2\pi)^{r_2^\prime} h_L R_L}{\omega_L \sqrt{\abs{D_L}}}\sqrt{\abs{D_L}}\\
&= 	\lim_{z\rightarrow 1}(z - 1)\zeta_L(z) \sqrt{\abs{D_L}}\\
&= 	\lim_{z\rightarrow 1}(z - 1)\zeta_K(z) L(z, \chi_{\gamma}) \sqrt{\abs{D_L}} \qquad \textrm{by \cite[Theorem 8.15]{IwanamiSeries}}\\ 
&= 	\frac{2^{r_1} h_KR_K}{\sqrt{\abs{D_K}}\omega_K} \lim_{z\rightarrow 1} L(z, \chi_{\gamma}) \sqrt{\abs{D_L}}\\
&=	\kappa L(1, \chi_{\gamma}) \sqrt{\abs{D_L}},
\end{align*}
where $\kappa = \frac{2^{r_1} h_KR_K}{\sqrt{\abs{D_K}}\omega_K}$ is the residue of $\zeta_K$ at $z=1$.
The second equality follows from Tate's computation of the measure of $L^*\backslash\I^1_L$ \cite[Theorem~4.3.2]{Tate} (see also \cite[\S 7.5]{RV}), adjusted for our normalization of the local measures: we take $\vol(\cO_{\fq}^*)=1$ at the finite places, whereas J.~Tate uses the self-dual normalization with $\vol(\cO_{\fq}^*)=(\Norm\fd_{\fq})^{-1/2}$, and the global discrepancy $\prod_{\fq}(\Norm\fd_{\fq})^{1/2}=\sqrt{\abs{D_L}}$ accounts for the factor of $\sqrt{\abs{D_L}}$.
\end{proof}

Since $K$ is totally real of degree $n$, we have $r_1 = n$ and $\omega_K = 2$.
Consequently,
\[
\kappa = \frac{2^{n-1} h_KR_K}{\sqrt{\abs{D_K}}}.
\]
This value will be repeatedly used in our explicit calculations throughout the paper.
The trace formula becomes 
\[
\sum_{\gamma} \vol\left(\gamma\right)\cO\left(f, \gamma\right) = \kappa\sum_{\left(\det\left(\gamma\right)\right) = \fp^{\sk}} L(1, \chi_{\gamma})\sqrt{\abs{D_L}}\cO\left(f, \gamma\right).
\]

From now on, we will work with the right hand side of the above equation.

\subsection{$\fq$-adic Orbital Integrals}

We begin by recording a result of Langlands from \cite[Lemma 1]{LanBE04} which involves computing orbital integrals.
His proof is explained over $\Q$ but it can be extended to number fields in general.
Throughout, we will assume that $\gamma$ is a contributing conjugacy class.

\begin{lemma}
\label{behaviour O}
Fix a positive integer\footnote{This lemma is true for a general positive integer $k$ which need not necessarily satisfy \ref{ass: class number}.} $k$.
Let $\fq$ a prime ideal of $\cO_K$ and let $K_{\fq}$ denote its $\fq$-adic completion.
Set $\sq = \Norm_K(\fq)$.
Then
\[
\cO(f_{\fq}^{(k)}, \gamma) = 0,
\]
unless $\abs{\det(\gamma)}_{\fq} = \sq^{-k}$.
In the latter case, there exists a non-negative integer $n_{\gamma} = n_{\gamma,\fq}$ such that
\begin{enumerate}[label = \textup{(\roman*)}]
\item if the extension $K_{\fq}(\gamma)/K_{\fq}$ \textbf{splits}, then
\[
\cO(f_{\fq}^{(k)}, \gamma) = \sq^{n_{\gamma}}.
\]
\item if the extension $K_{\fq}(\gamma)/K_{\fq}$ does not split and is \textbf{unramified} then
\[
\cO(f_{\fq}^{(k)}, \gamma) = \sq^{n_{\gamma}}\left(\frac{\sq + 1}{\sq - 1}\right) - \frac{2}{\sq - 1}.
\]
\item if the extension $K_{\fq}(\gamma)/K_{\fq}$ does not split and is \textbf{ramified} then
\[
\cO(f_{\fq}^{(k)}, \gamma) = \left(\frac{\sq^{n_{\gamma}+1}}{\sq - 1}\right) - \frac{1}{\sq - 1}.
\]
\end{enumerate}
In particular, when $n_{\gamma}=0$ the orbital integral $\cO(f_{\fq}, \gamma)=1$ is independent of the splitting behaviour.
\end{lemma}

\begin{proof}
See \cite{LanBE04} for an approach counting lattices in the building of $\SL(2, K_{\fq})$.
For a different approach see \cite[Proposition~5]{malors21}.
\end{proof}

The contributing conjugacy classes are precisely the ones where $\cO(f_{\sq}^{(\sk)},\gamma)\neq 0$.
This is because $(\det(\gamma)) = \fp^{\sk}$.
At the fixed prime $\fp$, we have that $\abs{\det(\gamma)}_{\fp} = \Norm_K(\fp)^{-\sk}$ and at primes $\fq\neq \fp$, we have that $\abs{\det(\gamma)}_{\fq} = \Norm_K(\fq)^{0}=1$.

\begin{notation}
As in the above lemma, we often write $\Norm_K(\fp)=\sp$ and $\Norm_K(\fq)=\sq$.
\end{notation}

We now define a new ideal of $\cO_K$.
\begin{definition}
\label{def:sgamma}
Consider the regular elliptic element $\gamma$ (viewed as an element of $L$).
Write $D_{L/K}(\gamma)$ to denote the discriminant of the element $\gamma$ and $\Delta_{\gamma}$ to denote the relative discriminant of the quadratic extension $L/K$.
Define the ideal $S_{\gamma}$ as follows
\[
S_{\gamma}^2 := \frac{D_{L/K}(\gamma)}{\Delta_{\gamma}}.
\]
\end{definition}

The ideal $S_\gamma$ encodes ramification behaviour and in turn is related to the orbital integral.
More precisely,

\begin{lemma}
\label{behaviour k}
Let $\fq$ be a any prime ideal of $\cO_K$.
Then,
\[
n_{\gamma, \fq} = \val_{\fq}(S_{\gamma}).
\]
\end{lemma}

\begin{proof}
See \cite[Proposition~4]{malors21}.
\end{proof}

The next result follows immediately from the above discussion.

\begin{proposition}
\label{Multiplicative Formula Langlands}
Let $\gamma$ be a contributing matrix.
Then,
\begin{align*}
\cO(f_{\fin}, \gamma) & = \sp^{-\sk/2}\prod_{\fq \mid S_{\gamma}} \cO(f_{\fq}, \gamma) \\
& = \sp^{-\sk/2}\sum_{\fd\mid S_{\gamma}}\Norm_K(\fd) \prod_{\fq\mid \fd}\left(1 - \frac{\chi_{\gamma}(\fq)}{\Norm_K(\fq)}\right) \\
&= \sp^{-\sk/2}\Norm_K(S_{\gamma}) \sum_{\fd\mid S_{\gamma}}\frac{1}{\Norm_K(\fd)} \prod_{\fq\mid \fd'}\left(1 - \frac{\chi_{\gamma}(\fq)}{\Norm_K(\fq)}\right), 
\end{align*}
where $\fd':= S_{\gamma}/\fd$.
\end{proposition}

\begin{proof}
See \cite[Theorem~3]{malors21}.
The presence of the factor of $\sp^{-\sk/2}$ is due to the slight difference in the definition of the test function; see \cite[line above (24)]{malors21}.
\end{proof}

\subsection{Extracting the unit sum.}

Recall that we have fixed a prime ideal $\fp$ and an integer $\sk$ satisfying \ref{ass: class number} such that $\fp^{h_{\fp}}$ is generated by $\rho$ and $\varepsilon = \rho^{\sk'}$ is a generator of $\fp^{\sk}$.
We now manipulate the regular elliptic part.

\begin{proposition}
\label{prop 3-24}
A contributing conjugacy class $\gamma$ is determined precisely by a pair $(\tau, u)$, where $\tau$ is in $\cO_K$, $u$ is in $\cO^*_K$, and $\tau^2 - 4u\varepsilon$ is not a square in $K$.
\end{proposition}

\begin{proof}
Given $\gamma$, we view it as a matrix with $\Tr(\gamma)=\tau$.
By Proposition~\ref{prop 3-7}, we know that $\tau$ is an integer in $K$.
Since $(\det(\gamma)) = \fp^{\sk} = (\varepsilon)$, we obtain that $\det(\gamma) = u\varepsilon$ for some unit $u$.
Since $\gamma$ is a regular elliptic element, we also know that $\tau^2 - 4u\varepsilon$ is not a square; otherwise, the roots of the characteristic polynomial would be in $K$.

Conversely, given such a $\tau$ and $u$, consider the matrix
$\begin{pmatrix}
0 & 1\\
-u\varepsilon & \tau
\end{pmatrix}\in G(K)$.
The associated characteristic polynomial is $X^2 - \tau X + u\varepsilon$, which is irreducible over $K$.
Hence, this represents a regular elliptic class, which is in fact a contributing class since $u\varepsilon$ generates $\fp^{\sk}$.
 \end{proof}

\begin{proposition}
With notation and assumptions introduced above, define the set
\[
L(u) = \{\tau\in \cO_K : \tau^2 - 4u\varepsilon \neq\square \}.
\]
Then
\[
R(f) = \sum_{\gamma} \vol(\gamma)\cO(f, \gamma) = \sum_{u\in \cO_K^*} \sum_{\tau\in L(u)} \vol\left(
 \begin{pmatrix}
 0 & 1\\
 -u\varepsilon & \tau
 \end{pmatrix}
 \right) \cO\left(f,
 \begin{pmatrix}
 0 & 1\\
 -u\varepsilon & \tau
 \end{pmatrix}
 \right).
\]
\end{proposition}
\begin{proof}
We have seen that 
\[
 \sum_{\gamma} \vol\left(\gamma\right)\cO\left(f, \gamma\right) = \sum_{\left(\det\left(\gamma\right)\right) = \fp^{\sk}} \vol\left(\gamma\right)\cO\left(f, \gamma\right).
\]
Using Proposition~\ref{prop 3-24} and rearranging the sum, we obtain
\[
\sum_{\gamma} \vol(\gamma)\cO(f, \gamma) = \sum_{u\in \cO_K^*} \sum_{\tau\in L(u)} \vol\left(
\begin{pmatrix}
0 & 1\\
-u\varepsilon & \tau
\end{pmatrix}
\right) \cO\left(f,
\begin{pmatrix}
0 & 1\\
-u\varepsilon & \tau
\end{pmatrix}
\right).
\]
This re-grouping can be done because this part of the trace formula converges absolutely.
\end{proof}

\begin{notation}
We continue to write 
\[
 \sum_{u\in \cO_K^*} \sum_{\tau\in L(u)} \vol(\gamma)\cO(f, \gamma)
\]
instead of 
\[
 \sum_{u\in \cO_K^*} \sum_{\tau\in L(u)} \vol\left(
 \begin{pmatrix}
 0 & 1\\
 -u\varepsilon & \tau
 \end{pmatrix}
 \right) \cO\left(f,
 \begin{pmatrix}
 0 & 1\\
 -u\varepsilon & \tau
 \end{pmatrix}
 \right).
\]
But, we keep in mind that $\gamma$ is now a specific conjugacy class that depends on $u$ and $\tau$.
\end{notation}

\subsection{Archimedean Orbital Integrals}
\label{sec: arch orb int}
We begin with a brief review of what happens in $G(\R)=\GL(2, \R)$, so that our next computations are clearer.

\subsubsection{Overview of Germ Expansion}

Let $h$ be a smooth function of compact support in $G(\R)$ and let $\gamma$ be a regular elliptic element.
Its orbital integral defined as 
\[
\cO(h,\gamma) := \int_{G_{\gamma}(\R)\setminus G(\R)} h(x^{-1}\gamma x) dx
\]
is a class function, i.e.~it only depends on the conjugacy class of $\gamma$.
Following Langlands, we define the map
\begin{equation*}
\begin{split}
\ch: G(\R) &\longrightarrow \R\times \R^*\\
\gamma &\mapsto ( \Tr(\gamma),4\det(\gamma)),
\end{split}
\end{equation*}
and use the coordinates $(r,N)$ for the codomain.
The center of $G(\R)$ becomes the parabola $N = r^2$ minus the origin, which divides the plane into two disjoint regions.
Away from the parabola, the map $\ch:G(\R)\rightarrow \R^2$ is a submersion and its fibers are the regular semisimple conjugacy classes.
A matrix in $G(\R)$ maps to a point in the interior of the parabola precisely when the matrix corresponds to an element of the elliptic torus (denoted by $T_{\el}$).
Whereas, the matrix maps to a point in the exterior of the parabola when it corresponds to an element of the split torus (denoted by $T_{\spl}$).

These tori are defined as follows:
\[
T_{\spl} = \left\{\begin{pmatrix}
 \lambda_1 & 0\\
 0 & \lambda_2\\
 \end{pmatrix}\mid 
 \lambda_1, \lambda_2 \neq 0 \right\} \ \textrm{ and } \
 T_{\el} = \left\{\begin{pmatrix}
 r\cos(\theta) & r\sin\theta\\
 -r\sin(\theta) & r\cos\theta\\
 \end{pmatrix}
 \mid r > 0,\; 0\le \theta \le \pi \right\}.
\]
Depending on the nature of the regular elliptic element $\gamma$, its centralizer $G_{\gamma}(\R)$ is either conjugate to $T_{\spl}$ or $T_{\el}$ (often denoted as $A,B$ in the literature).
\textbf{In what follows, we follow \cite{shel} closely and use her notation.}

Recall that the unipotent matrices are the ones of the form
\[
\begin{pmatrix}
 1 & x\\
 0 & 1\\
 \end{pmatrix}
\]
with standard Haar measure $dn = dx$.
Let $K_0$ be the circle group of matrices
\[
\begin{pmatrix}
 \cos(\theta) & \sin\theta\\
 -\sin(\theta) & \cos\theta\\
 \end{pmatrix},
\]
with Haar measure $dk = d\theta$.
Further define the integrals
\begin{align*}
 h_0(x) &:= \int_{K_0} h(kxk^{-1})dk,\\
 H_{r}(u, v) &:= \int_{K_0}h\left(r k\exp\begin{pmatrix}
 0 & u\\
 -v& 0\\
 \end{pmatrix}k^{-1}\right)dk.
\end{align*}

Let $\lambda_1,\lambda_2$ be the two roots of the polynomial $P_{\gamma}$.
Define the \emph{Weyl discriminant}
\[
D(\gamma) = D(\lambda_1, \lambda_2)= (1-\lambda_1/\lambda_2)(1-\lambda_2/\lambda_1)
=\frac{-(\lambda_1-\lambda_2)^2}{\lambda_1\lambda_2}.
\]
Then, D.~Shelstad \cite{shel}
has shown that \begin{align*}
\abs{D(\gamma)}^{\frac{1}{2}}\cO_{T_{\spl}}(\gamma) & = \frac{1}{2}\abs{\frac{\lambda_1}{\lambda_2}}^{\frac{1}{2}} \int_N h_0(n\gamma) dn \text{ and}\\
\cO_{T_{\el}}\left(\begin{pmatrix}
 r\cos(\theta) & r\sin\theta\\
 -r\sin(\theta) & r\cos\theta\\
 \end{pmatrix}\right)
&= \frac{1}{4} \int_{0}^{\infty}(H_{r}(e^t\theta, e^{-t}\theta) + H_{r}(-e^t\theta, -e^{-t}\theta))(e^t - e^{-t}) dt.
\end{align*}

Since the orbital integrals are class functions, there exist two functions $g_1, g_2$ in the coordinates $(r, N)$ which depend on $h$, are defined everywhere away from the parabola $N = r^2$ and away from $N = 0$ (which would correspond to non-invertible matrices).
Moreover, both $\cO_{T_{\spl}}$ and $\cO_{T_{\el}}$ can be written in terms of these two smooth functions $g_1$ and $g_2$.
More precisely,
\begin{align*}
 \cO_{T_{\spl}}(r,N) &= \frac{1}{2}\abs{\frac{r^2}{N}-1}^{-1/2} g_2(r,N) \text{ and}\\
 \cO_{T_{\el}}(r,N) &= g_1(r,N) + \frac{1}{2}\abs{\frac{r^2}{N}-1}^{-1/2} g_2(r,N).
\end{align*}
Since $h$ has compact support, the supports of $g_1$ and $g_2$ are relatively compact in $\mathbb{R}^2$.
Also, $g_2$ is supported in the elliptic torus (i.e.~inside the parabola).

For consistency in notation, we write the right hand side of both $\cO_{T_{\spl}}$ and $\cO_{T_{\el}}$ as
\[
g_1(r,N) + \frac{1}{2}\abs{\frac{r^2}{N}-1}^{-1/2} g_2(r,N),
\]
with the understanding that $g_1 \equiv 0$ outside the parabola.
The singularity of the orbital integral is detected by
\[
\abs{\frac{r^2}{N}-1}^{-1/2}
\]
which diverges as we approach the `singular set' given by the parabola.

Consider the orbital integral $\cO(h,\gamma)$.
After performing an inner automorphism the centralizer of $\gamma$ becomes either $T_{\spl}$ or $T_{\el}$.
These automorphisms do not change the class $\ch(\gamma)$.
We conclude that as a function in the $(r, N)$ coordinates
\begin{equation}
\label{OhNr}
 \cO(h,(r, N)) = g_1(r,N) + \frac{1}{2}\abs{\frac{r^2}{N}-1}^{-1/2} g_2(r,N).
\end{equation}

\subsubsection{Calculations over \texorpdfstring{$\Q$}{}: Langlands and Altu\u{g}}

The standing assumption that Langlands and Altu\u{g} have is that the function $h$ is invariant under $Z_+$.
This invariance property is inherited by the orbital integral, and by the functions $g_1$ and $g_2$.
In the $(r, N)$ coordinates it means that for positive $\lambda$
\[
g_m(\lambda r, \lambda^2N) = g_m(r, N), \ \text{ for } m= 1, 2.
\]
Setting $\lambda = \frac{1}{\sqrt{\abs{N}}}$, \eqref{OhNr} becomes
\begin{align*}
 \cO_h(r,N) 
 &= g_1(r, N) + \frac{1}{2}\abs{\frac{r^2}{N}-1}^{-1/2} g_2(r, N)\\
 &= g_1\left(\frac{r}{\sqrt{\abs{N}}}, \sgn(N)\right) + \frac{1}{2}\abs{\frac{r^2}{N}-1}^{-1/2} g_2\left(\frac{r}{\sqrt{\abs{N}}}, \sgn(N)\right).
\end{align*}
Thus, we are led to four functions: $g_1(r, 1)$ , $g_1(r, -1)$, $g_2(r, 1)$ and $g_2(r, -1)$.
Looking at the lines $N = 1$ (blue) and $N = -1$ (orange) in the $(r,N)$ plane in Figure~\ref{fig: drawing for Q} we conclude:

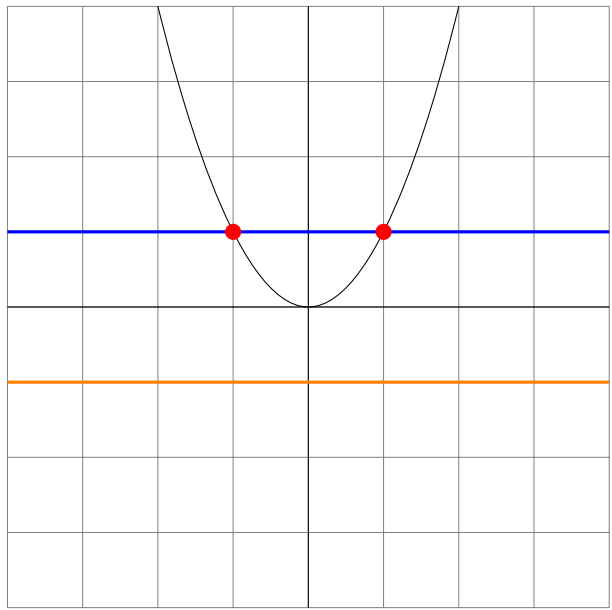
\begin{figure}[h]
 \centering
 \begin{tikzpicture}
 \draw[step=1 cm,gray,very thin] (-4,-4) grid (4, 4);
 \draw (-4,0) -- (4,0);
 \draw (0,-4) -- (0,4);
 \draw (-2,4) parabola bend (0,0) (2, 4); 
 \draw[orange, very thick] (-4,-1)--(4, -1);
 \draw[blue, very thick] (-4,1)--(4, 1);
 \draw[fill, red] (1,1) circle (.1cm);
 \draw[fill, red] (-1,1) circle (.1cm);
 \end{tikzpicture}
 \captionof{figure}{The case of $K = \Q$.}
 \label{fig: drawing for Q}
\end{figure}

\begin{enumerate}[label = (\alph*)]
 \item the orbital integrals are functions of the trace and depend only on the four functions $g_m(r,\pm 1)$ with $m=1,2$.
 \item the functions $g_m(r, -1)$ are smooth and of compact support since their support lies entirely on the line $N = -1$ which is in the split torus.
 On the other hand, the functions $g_m(r, 1)$ are also compactly supported but might not be smooth as they might have singularities as they approach the value $r = \pm 1$, precisely because they change from the elliptic torus to the split one at those points.
\end{enumerate}

\subsubsection{Totally Real Fields}
\label{germ expansion section for us}

Our archimedean test function is
\[
f_{\infty} = f_{\nu_1}\times\cdots\times f_{\nu_n},
\]
where $n = [K:\Q]$.
For each real place (which we generically denote by) $\nu$, we can perform the analysis as before and conclude the existence of two functions $g_{\nu,1}$ and $g_{\nu,2}$ in the coordinates $(r, N)$ with
\[
\cO(f_{\nu}, (r, N)) = g_{\nu,1}(r, N) + \frac{1}{2}\abs{\frac{r^2}{N}-1}^{-1/2} g_{\nu,2}(r, N).
\]
Picking a unit $u\in\cO_K^*$ and $\tau\in L(u)$ fixes $\gamma$.
We can then write the archimedean orbital integral as
\[
\cO(f_{\infty}, \gamma(\tau,u)) = \prod_{i = 1}^n\left(g_{\nu_i,1}(\tau_i, 4u_i\varepsilon_i) + \frac{1}{2}\abs{\frac{\tau_i^2}{4u_i\varepsilon_i}-1}^{-1/2} g_{\nu_i,2}(\tau_i, 4u_i\varepsilon_i)\right),
\]
where $\tau_i, u_i, \varepsilon_i$ stand for the embedding of $\tau, u, \varepsilon$ into the $i$-th real place via $\sigma_i:K\rightarrow\R$, respectively.
We can conclude this section by recording the following result.

\begin{theorem}
\label{prop: regular elliptic part}
The regular elliptic part of the trace formula is
\begin{multline*}
 R(f) =\sum_{u\in \cO_K^*}\sum_{\tau\in L(u)}\vol(\gamma(\tau,u))\cO(f_{\fin}, \gamma(\tau,u)) \times \\ \prod_{i = 1}^n\left(g_{\nu_i,1}(\tau_i, 4u_i\varepsilon_i) + \frac{1}{2}\abs{\frac{\tau_i^2}{u_i\varepsilon_i}-1}^{-1/2} g_{\nu_i,2}(\tau_i, 4u_i\varepsilon_i)\right).
\end{multline*}
\end{theorem}

\begin{remark}
The functions $g_{\nu_i, 1}, g_{\nu_i, 2}$ depend only on $f_{\nu_i}$.
Contrary to what was done in \cite{AliI}, we do not make the change of variables to transform this into a function that only depends on the trace.
This is because $Z_{+}$ affects all entries in \textit{the same way} but the change of variable would require a different scaling factor at each place.
\end{remark}

Our objective in the next sections is to analytically study these expressions.

\subsection{A partition of \texorpdfstring{$\mathbb{R}^{2n-1}$}{}}
\label{subsec: A partition of R 2n-1}

To make the manipulation clearer at later stages, we partition $\mathbb{R}^{2n}$ into different subsets.
Recall from \S\ref{sec: arch orb int} that in the $(r, N)$ coordinates the $\R^2$ plane is partitioned into different regions by the parabola $r^2 = N$.
For calculations in this section, we consider the following partition of $\mathbb{R}^2$.

\begin{definition}
\label{def: regions S}
Define the subsets of $\mathbb{R}^2$ by
\begin{align*}
 S_0 &:= \{(r, N) \mid N = 0 \mbox{ or } r^2 - N = 0\},\\
 S_{\spl}&:= \{(r, N) \mid N \neq 0 \mbox{ and } r^2 - N > 0\},\\
 S_{\ell}&:= \{(r, N) \mid N \neq 0 \mbox{ and } r^2 - N < 0\}.
\end{align*}
Let $j_1, \ldots, j_n \in \{0, {\spl}, {\el}\}$ and $J = (j_1, \ldots , j_n)$.
Define
\[
S_J: = S_{j_1}\times S_{j_2}\times\cdots\times S_{j_n}\subset\mathbb{R}^{2n}.
\]
\end{definition}

\begin{remark}
The regions $S_{\el}$ and $S_{\spl}$ get their subindex notation because they are the images under $\ch$ of the elliptic and split torus, respectively.
The partition they form, together with $S_0$, is shown in Figure~\ref{fig: subsets of R2}.
\end{remark}

\begin{figure}[h]
\centering
 \begin{tikzpicture}
 \fill[green!30] (-4,-4) rectangle (4,4);

 \fill[blue!30]
 (-2,4) parabola bend (0,0) (2,4) -- (2,4) -- (-2,4) -- cycle;

 \draw[dashed, thick, red] (-2,4) parabola bend (0,0) (2,4);
 
 \draw[step=1 cm,gray,very thin] (-4,-4) grid (4, 4);
 \draw [dashed, thick, red](-4,0) -- (4,0);
 \draw (0,-4) -- (0,4);
 \draw (0,-2) node[fill=green!30, minimum size=5mm] {\textbf{$S_{\spl}$}};
 \draw (0,2) node[fill=blue!30, minimum size=5mm] {$S_{\el}$};
 \end{tikzpicture}
 \caption{The subsets $S_j$ of $\mathbb{R}^2$.}
 \label{fig: subsets of R2}
\end{figure} 

\begin{remark}
Notice that $S_J$ has measure zero if and only if one of its $j_i = 0$ (i.e.~if in one of the coordinates is a set of measure zero).
In particular, the integral of any integrable function over such a set is zero.
\end{remark}

\begin{notation}
Set $G_{\infty} := \GL(2, \R)\times\cdots\times \GL(2, \R)$, where there are $n$ factors $\GL(2, \R)$.
\end{notation}

Let $z$ be a positive real number.
The $Z_+$ action on $G_{\infty}$ is given by
\[
z\cdot(\gamma_1, \dots, \gamma_n) = (z\gamma_1, \dots, z\gamma_n).
\]
This action corresponds to the action of $\R^+$ on $(r_1, N_1,\ldots, r_n, N_n)$ given by
\[
z\cdot (r_1, N_1,\ldots, r_n, N_n) = (zr_1, z^2N_1,\ldots,zr_n, z^2N_n).
\]
In other words, there is a commutative map for every $z>0$
\begin{center}
 \begin{tikzcd}
 G_{\infty} \arrow[r, "z\cdot"]\arrow[d, "\CH"] & G_{\infty} \arrow[d, "\CH"]\\
 \R^{2n} \arrow[r, "z\cdot"] & \R^{2n}
 \end{tikzcd}
\end{center}

The following lemma is straightforward.
\begin{lemma}
Keep the notation introduced above.
\label{lemma: regions are invariant}
\begin{enumerate}[label = \textup{(\roman*)}]
 \item Let $j_1, \ldots, j_n \in \{0, {\spl}, {\el}\}$ and $J = (j_1, \ldots , j_n)$.
 Then $S_J$ is invariant under the action of $Z_+$.
 \item The quotient space $Z_+ \setminus G_{\infty}$ is a Lie group with $2^n$ connected components.
 \item $Z_+\backslash S_J$ form a partition of $Z_+ \setminus G_{\infty}$.
\end{enumerate}
\end{lemma}

\subsection{Splitting behaviour at infinity}
\label{subsec: spl behaviour at inf}

We now define the following function:
\begin{definition}
Let $\CH:= \ch \otimes \cdots \otimes \ch: G_{\infty} \longrightarrow \R^2\times\cdots\times \R^2$ be given by
\[
\CH(\gamma_1,\ldots, \gamma_n) = (r_1, N_1, \dots, r_n, N_n).
\]
\end{definition}
We now define an equivalence relation in points of $\mathbb{R}^{2n}$ that will simplify our computations and explanations significantly.

\begin{definition}
Let $v_1$ and $v_2$ be two points in $\mathbb{R}^{2n}$.
They are said to be \textit{equivalent at infinity} if $v_1, v_2$ belong to the same region $S_J$.
The notation is 
\[
v_1\siminf v_2.
\]
The point $v \in \mathbb{R}^{2n}$ is \textit{completely split} at infinity or \textit{not split at infinity} according as $v\in S_{\spl}\times\cdots\times S_{\spl}$ or not.
\end{definition}

By Lemma~\ref{lemma: regions are invariant} the equivalence class of $v$ under the above equivalence class is well defined on $Z_+\backslash G_{\infty}$.
Thus, we extend the definition to the quotient space $Z_+\backslash G_{\infty}$.

Finally, we extend the definitions to the regular elliptic matrices.
\begin{definition}
Let $\gamma_1$ and $\gamma_2$ be two regular elliptic matrices in $\GL(2, K)$.
They are called \emph{equivalent at infinity} when 
\begin{equation*}
 \CH(\sigma_1(\gamma_1), \ldots ,\sigma_n(\gamma_1))\siminf \CH(\sigma_1(\gamma_2), \ldots ,\sigma_n(\gamma_2)).
\end{equation*}
The matrix $\gamma_1$ is said to be \emph{completely split} or \textit{not split} at infinity according as 
\begin{equation*}
 \CH(\sigma_1(\gamma_1), \ldots ,\sigma_n(\gamma_1))
\end{equation*}
is \emph{completely split} or \emph{not split} at infinity.
\end{definition}

The following will be useful later for explicit computations.

\begin{definition}
\label{indicator function}
Write $C: \mathbb{R}^{2n}\longrightarrow \{0, 1\}$ to be the indicator function of
\[
S_{\spl}\times\cdots\times S_{\spl}.
\]
By abuse of notation, using the fact that this set is $Z_+$ invariant, we also set $C$ to denote the indicator function after quotienting by $Z_+$.
\end{definition}

We provide an example when $n=2$ to get a feeling for the subsets defined here.
We will consider a totally real quadratic field $K$, but we do not impose the condition on totally positive units for this example.

\begin{example}
\label{ex: example of regions}
Let $K = \Q(\sqrt{3})$.
Its two real embeddings are
\begin{align*}
\sigma_1(a + b\sqrt{3}) = a + b\sqrt{3},\\
\sigma_2(a + b\sqrt{3}) = a - b\sqrt{3}.
\end{align*}
This number field has class number $1$; we can pick $k = 1$ and an ideal $\fP = (1 + 2\sqrt{3})$.
This ideal is prime and its norm is $p = 13$.
Set $\varepsilon = 1 + 2\sqrt{3}$.
 
Consider the matrix
\[
\gamma = \begin{pmatrix}
 3\sqrt{3} & 1 - 3\sqrt{3}\\
 1 + 6\sqrt{3} & 1 - 2\sqrt{3}\\
 \end{pmatrix}\in\GL(2, \cO_K).
\]
One can compute
\[
\det(\gamma) = -1-2\sqrt{3} = -\varepsilon.
\]
Thus, $\gamma$ will be a contributing matrix as long as it is regular elliptic.
The discriminant of its characteristic polynomial is
\[
\Tr(\gamma)^2 - 4\det(\gamma) 
 = (1+\sqrt{3})^2 + 4(1 +2\sqrt{3})
 = 2(4 + 5\sqrt{3}).
\]
This is not a square because $4 + 5\sqrt{3}$ is a prime in $\Q(\sqrt{3})$ relatively prime to $2$.
By Proposition~\ref{prop 3-24}, we conclude $\gamma$ is a contributing matrix that corresponds to the unit $u = -1$ and $\tau = 1 + \sqrt{3}$.
 
The trace-determinant coordinates we have defined $(r_1, N_1, r_2, N_2)$ for this matrix are
\[
(1 + \sqrt{3}, -1-2\sqrt{3}, 1-\sqrt{3}, -1+2\sqrt{3}).
\]
Computing explicitly gives
\[
N_1N_2 = -13.
\]
Recall that $z$ is positive by definition, so we conclude that the coordinates of this point in $\{\pm 1\} \times \R^3$ are precisely
\[
(-1) \times (1 + \sqrt{3}, 1-\sqrt{3}, -1+2\sqrt{3}).
\]

Let us now understand where does this point lie in the subsets.
We must understand whether the embeddings $\sigma_1$ and $\sigma_2$ send the discriminant to positive or negative numbers; this will determine whether the matrix has two distinct real eigenvalues or a complex pair of eigenvalues.
We have
 \begin{align*}
 \sigma_1(\Tr(\gamma)^2 - 4\det(\gamma)) &= \sigma_1(8 + 10\sqrt{3}) = 8 + 10\sqrt{3} > 0,\\
 \sigma_2(\Tr(\gamma)^2 - 4\det(\gamma)) &= \sigma_2(8 + 10\sqrt{3}) = 8 - 10\sqrt{3} < 0.
 \end{align*}
Under the first embedding the matrix is split, whereas under the second embedding it is elliptic.
Consequently, this point is in $S_{\spl}\times S_{\el}$.
\end{example}

\section{Manipulation of the Regular Elliptic Part}
\label{section: Manipulation of the Regular Elliptic Part}

In this section the goal is to rewrite the expression of the regular elliptic part of the trace formula obtained in Theorem~\ref{prop: regular elliptic part} in a way so that we can apply the Approximate Functional Equation.
We carry out the manipulation by substituting the volume term with the expression obtained in Proposition~\ref{prop: volume manipulation}, the finite orbital integral with the expression obtained in Proposition~\ref{Multiplicative Formula Langlands}, and finally getting a better handle on the infinite orbital integral by studying the units.

A key difficulty in generalizing Altu\u{g}'s result over $\Q$ to the general number field case was in handling the units.
The first step towards solving this issue is done by defining an appropriate \emph{interpolation function} in \S\ref{section: blending step}.
It is because of how we define this function that we crucially need to assume that our base field is totally real.

\subsection{The First Simplification}

Let $\gamma$ be a contributing conjugacy class.
By earlier discussion we know that $\left(\det(\gamma)\right) = \fp^{\sk}$.
By the product formula, 
\[
 \prod_{\nu} \abs{\det(\gamma)}_{\nu} = 1.
\]
Since the only finite prime contributing to its factorization is $\fp$, we have 
\[
\prod_{\nu\mid \infty} \abs{\det(\gamma)}_{\nu} = \sp^{\sk}.
\]
Given two different complex numbers $\alpha_1, \alpha_2$, recall that we have defined
\[
\abs{D(\alpha_1, \alpha_2)}^{-1/2} = \frac{\abs{\alpha_1\alpha_2}^{1/2}}{\abs{\alpha_1 - \alpha_2}}.
\]
This factor determines the singularity of the orbital integrals.
It is also the inverse of the Weyl discriminant that appears in the Weyl integration formula.

We prove the next result only for totally real number fields.
The proof remains unchanged in the general case, except the need to account for the complex embeddings.

\begin{proposition}
\label{relation discriminants}
Let $\lambda_1$ and $\lambda_2$ be the roots of $P_\gamma$ and define $\partial = \tau^2 - 4\det(\gamma)$.
Then,
\[
\sp^{\sk/2} = \sqrt{\Norm_K\left(\partial\right)} \prod_{i = 1}^{r_1} \abs{D\left(\sigma_i\left(\lambda_1\right), \sigma_i\left(\lambda_2\right)\right)}^{-1/2},
\]
where $\sigma_i$ denotes the real embeddings of $K$.
\end{proposition}

\begin{proof}
Using the product formula and the definition of the norms given by the archimedean valuations,
\begin{equation}
\label{equation qk}
\begin{split}
\sp^{\sk}	&= \prod_{\nu\mid \infty}\abs{\det(\gamma)}_{\nu}\\
&= \prod_{i = 1}^{r_1} \abs{\sigma_i(\det(\gamma))} \\ 
&= \prod_{i = 1}^{r_1} \abs{\sigma_i(\lambda_1\lambda_2)} \\ 
&= \prod_{i = 1}^{r_1} \abs{\sigma_i(\lambda_1)\sigma_i(\lambda_2)}.
\end{split}
\end{equation}
Moreover, in $L$ we have
\[
\partial = \tau^2 - 4\det(\gamma) = (\lambda_1 + \lambda_2)^2 - 4\lambda_1\lambda_2 = (\lambda_1 - \lambda_2)^2.
\]
Applying any real embedding $\sigma$ to $\partial$ and taking the absolute value, we deduce
\[
\abs{\sigma(\partial)} = \abs{\sigma(\lambda_1) - \sigma(\lambda_2)}^2.
\]
For a (real) place this implies
\[
\abs{\sigma(\lambda_1)\sigma(\lambda_2)} = D\left(\sigma\left(\lambda_1\right), \sigma\left(\lambda_2\right)\right)^{-1} \abs{\sigma\left(\partial\right)}.
\]
Substituting this in \eqref{equation qk}, we obtain
\begin{align*}
\sp^{\sk}	
&= \prod_{i = 1}^{r_1} \abs{\sigma_i\left(\lambda_1\right)\sigma_i\left(\lambda_2\right)} \\ 
&= \prod_{i = 1}^{r_1} D\left(\sigma_i\left(\lambda_1\right), \sigma_i\left(\lambda_2\right)\right)^{-1}\abs{\sigma_i(\partial)} \\ 
&= \Norm_{K}(\partial)\prod_{i = 1}^{r_1} \abs{D\left(\sigma_i\left(\lambda_1\right), \sigma_i\left(\lambda_2\right)\right)}^{-1}.
\end{align*}
Since $\partial\in K$, we have used the product formula in $K$ for the above calculation.
Taking square root,
\[
 \sp^{\sk/2} = \sqrt{\Norm_K\left(\partial\right)} \prod_{i = 1}^{r_1} \abs{D\left(\sigma_i\left(\lambda_1\right), \sigma_i\left(\lambda_2\right)\right)}^{-1/2}.
 \qedhere
\]
\end{proof}

\subsubsection{}
Fix a contributing regular elliptic matrix $\gamma$ corresponding to the unit $u$ and trace $\tau$.
By Theorem~\ref{prop: regular elliptic part}, the regular elliptic part of the trace formula can be written as
\begin{multline*}
R(f)=\sum_{u\in \cO_K^*}\sum_{\tau\in L(u)}\vol(\gamma(\tau,u))\cO(f_{\fin}, \gamma(\tau,u)) \times \\ 
\prod_{i = 1}^n\left(g_{\nu_i,1}(\tau_i, 4u_i\varepsilon_i) + \frac{1}{2}\abs{\frac{\tau_i^2}{4u_i\varepsilon_i}-1}^{-1/2} g_{\nu_i,2}(\tau_i, 4u_i\varepsilon_i)\right).
\end{multline*}
Substituting the values we have obtained Proposition~\ref{prop: volume manipulation}, the above expression equals
\begin{multline*}
\kappa \sum_{u\in \cO_K^*}\sum_{\tau\in L(u)}L(1, \chi_\gamma)\sqrt{\abs{D_L}}\cO(f_{\fin}, \gamma(\tau,u)) \times \\ \prod_{i = 1}^n\left(g_{\nu_i,1}(\tau_i, 4u_i\varepsilon_i) + \frac{1}{2}\abs{\frac{\tau_i^2}{4u_i\varepsilon_i}-1}^{-1/2} g_{\nu_i,2}(\tau_i, 4u_i\varepsilon_i)\right).
\end{multline*}
We will now rewrite $\sqrt{\abs{D_L}}$ in a way that will be more useful.
Recall that
\[
\sqrt{\abs{D_L}} = \abs{D_K}\sqrt{\Norm_K\left(\Delta_{\gamma}\right)}.
\]
Using the definition $\Delta_{\gamma}S_{\gamma}^2 = (\partial)$ and Proposition~\ref{relation discriminants}, we conclude
\[
\sqrt{\abs{D_L}} 
= \frac{ \abs{D_K}\sqrt{\Norm_K\left(\partial\right)}}{\Norm_K\left(S_{\gamma}\right)}
= \frac{\sp^{\frac{\sk}{2}} \prod_{i = 1}^{n} D\left(\sigma_i\left(\lambda_1\right), \sigma_i\left(\lambda_2\right)\right)^{1/2} \abs{D_K}}{\Norm_K\left(S_{\gamma}\right)}.
\]
Hence, each term arising in the regular elliptic part of the trace formula is of the form
\begin{multline*}
\sp^{\frac{\sk}{2}}\abs{D_K}\frac{L(1, \chi_\gamma)\cO(f_{\fin}, \gamma(\tau,u)) \prod_{i = 1}^{n} D\left(\sigma_i\left(\lambda_1\right), \sigma_i\left(\lambda_2\right)\right)^{1/2} }{\Norm_K\left(S_{\gamma}\right)} \times \\
\prod_{i = 1}^n\left(g_{\nu_i,1}(\tau_i, 4u_i\varepsilon_i) + \frac{1}{2}\abs{\frac{\tau_i^2}{4u_i\varepsilon_i}-1}^{-1/2} g_{\nu_i,2}(\tau_i, 4u_i\varepsilon_i)\right).
\end{multline*}

\subsubsection{}
For a fixed unit $u\in\cO_K^{*}$, we now study 
\[
\label{orbitalsone}
\prod_{i = 1}^{n} D\left(\sigma_i\left(\lambda_1\right), \sigma_i\left(\lambda_2\right)\right)^{1/2} \prod_{i = 1}^n\left(g_{\nu_i,1}(\tau_i, 4u_i\varepsilon_i) + \frac{1}{2}\abs{\frac{\tau_i^2}{4u_i\varepsilon_i}-1}^{-1/2} g_{\nu_i,2}(\tau_i, 4u_i\varepsilon_i)\right).
\]
Recall that $\lambda_1, \lambda_2$ are the roots of $P_\gamma(X)$.
Notice that by definition
\[ 
\abs{D\left(\sigma_i\left(\lambda_1\right), \sigma_i\left(\lambda_2\right)\right)}^{1/2} = 2\abs{\frac{\tau_i^2}{4u_i\varepsilon_i}-1}^{1/2}.
\]
By substituting these terms and by putting $D\left(\sigma_i\left(\lambda_1\right), \sigma_i\left(\lambda_2\right)\right)^{1/2}$ with the $i$-th real orbital integral, we deduce that
\begin{multline*}
R(f) =\kappa \abs{D_K}\sum_{u\in \cO_K^*}\sum_{\tau\in L(u)} \sp^{\frac{\sk}{2}}\frac{L(1, \chi_\gamma)\cO(f_{\fin}, \gamma(\tau,u))}{\Norm_K\left(S_{\gamma}\right)}
\times \\
\prod_{i = 1}^n\left(2\abs{\frac{\tau_i^2}{4u_i\varepsilon_i}-1}^{1/2} g_{\nu_i,1}(\tau_i, 4u_i\varepsilon_i) + g_{\nu_i,2}(\tau_i, 4u_i\varepsilon_i)\right).
\end{multline*}

\subsection{Interpolation function \texorpdfstring{$\theta^{\pm}$}{}}
\label{section: blending step}

We begin by defining two maps: the first will allow us to handle the traces $\tau$ in the regular elliptic part of the trace formula, while the second will give a handle on the units $u$.

Define the map 
\begin{align*}
T:K&\longrightarrow\R^n \\
\alpha &\mapsto \left(\sigma_1(\alpha), \dots, \sigma_{n}(\alpha)\right).
\end{align*}
We know that $T(\cO_K)$ is a lattice in $\R^n$.
By Dirichlet's Unit Theorem (for a totally real field) we have,
\[
\cO_K^* \simeq \{\pm 1\}\times \Z^{n-1}.
\]
Let $W\subset\R^n$ be the $n-1$-dimensional subspace with
\[
x_1+\cdots+x_n = 0.
\]
The embedding $S:\cO_K^*\rightarrow \R^n$ is given by
\[
S(x) = (\log{\abs{x}_1}, \dots, \log{\abs{x}_{n}}),
\]
where $\abs{x}_i$ is the absolute value with respect to the $i$-th archimedean place.
Let us fix once and for all a system of fundamental units $\beta_1, \dots, \beta_{n-1}$ in $\cO_K^*$ that realizes the above isomorphism.
For all $1\leq i \leq n-1$, we see from the proof of the Dirichlet unit theorem that $S(\beta_i)$ is a basis for a lattice in $W$.
We write the coordinates of $W$ with respect to these generators as $y = (y_1, \dots,y_{n-1})$ by which we mean that a unit $u\in \cO^*_K$ can be written as 
\[
u = \pm \beta_1^{y_1} \cdots \beta_{n-1}^{y_{n-1}}.
\]

To proceed with the calculations, we are required to make an assumption here, namely that the system of fundamental units is `totally positive'.
\begin{equation}
\tag*{\textup{\textbf{(Hyp-pos)}}}
\label{ass: positivity}
\begin{minipage}{0.8\textwidth}
Let $K$ be a totally real number field that has a totally positive system of fundamental units.
\end{minipage}
\end{equation}
Recall that a system of fundamental units is called totally positive if $T(\beta_i)>0$ for all $1\leq i \leq n-1$.
This is essentially the only place in the paper where we require that the base field $K$ is totally real.
If $K/\Q$ is a totally real \emph{quadratic field} then it was observed by H.~Stark that a density of $100\%$ have a system of totally positive fundamental units; see \cite[p.~684]{DV18}.
When $[K:\Q]>2$, there are fewer fields satisfying this property, but such fields exist.

\begin{definition}
\label{def:thetaplusminus}
Define the \emph{interpolation function} $\theta^\pm$ from $\R^{2n-1}$ to $\C$ as 
\begin{small}
\begin{multline*}
\theta^{\pm}(r_1, \ldots , r_n, y_1, \dots, y_{n-1}) = \\ 
\sp^{\frac{\sk}{2}}\prod_{i = 1}^n\left(2\abs{\frac{r_i^2}{\pm4\beta_{1,i}^{y_1}\cdots \beta_{n-1,i}^{y_{n-1}}\varepsilon_i}-1}^{1/2} g_{\nu_i,1}(r_i, \pm4\beta_{1,i}^{y_1}\cdots \beta_{n-1,i}^{y_{n-1}}\varepsilon_i) + g_{\nu_i,2}(r_i, \pm4\beta_{1,i}^{y_1}\cdots \beta_{n-1,i}^{y_{n-1}}\varepsilon_i)\right),
\end{multline*}
\end{small}
where $\beta_{t,i} = \sigma_i(\beta_t)$ and the $\varepsilon_i = \sigma_i(\varepsilon)$.
\end{definition}

The interpolation function is only defined on a dense subset of $\R^{2n-1}$.
However, for succinctness of exposition we say that 
Moreover, its support is relatively compact in $\mathbb{R}^{2n-1}$ because the $g_{\nu,j}$ have relatively compact support in $\mathbb{R}^2$.

\begin{notation}
To simplify notation, we follow the following conventions:
\begin{itemize}
\item we write $\theta^{\pm}(r, y)$ instead of $\theta^{\pm}(r_1, \ldots , r_n, y_1, \dots , y_{n-1})$.
\item for $y=(y_1,\dots,y_{n-1})\in \mathbb{Z}^n$ define $\beta^y=\beta_1^{y_1} \cdots \beta_{n-1}^{y_{n-1}}$.
The units of $\cO_K$ are $\pm \beta^y$.
\item we continue to write $u$ to denote a unit in $\cO_{K}^*$ with the understanding that if we are evaluating the function $\theta^{\pm}(x, y)$ we are referring to its $\beta$-component under the isomorphism induced by the Dirichlet unit theorem.
\item we identify $\tau$ and $T(\tau)$.
\end{itemize}
\end{notation}

The regular elliptic part of the trace formula can then be written as
\begin{align*}
R(f) &= \kappa \abs{D_K} \sum_{u\in\cO_K^*}\sum_{\tau\in L(u)}
 \frac{L(1, \chi_\gamma)\cO(f_{\fin}, \gamma(\tau,u))}{\Norm_K\left(S_{\gamma}\right) }\theta^{\pm}(\tau, u) \\
& = \kappa \abs{D_K} \sum_{\pm} \sum_{\beta}\sum_{\tau\in L(u)}
 \frac{L(1, \chi_\gamma)\cO(f_{\fin}, \gamma(\tau,u))}{\Norm_K\left(S_{\gamma}\right) }\theta^{\pm}(\tau_1,\dots,\tau_n, y_1, \dots, y_{n-1}).
\end{align*}
Note that $\tau \in \cO_K$ can be viewed as an element of $\R^n$, so it makes sense to set $r=\tau$ as an input for $\theta^{\pm}$.
On the other hand, when we write $y=u$, we mean that $y=(y_1, \dots, y_{n-1})$ corresponding to the unit $u = \beta^y$.

We wish to perform Poisson summation on the variables $(\tau,u)$; but we need to address the following issues: 
\begin{enumerate}[label = (\alph*)]
\item the function $\theta^{\pm}$ has singularities.
\item the lattice on which we will perform the Poisson Summation lives in $\R^{2n-1}$; but it is unclear if the factor $\frac{L(1, \chi_\gamma)\cO(f_{\fin}, \gamma(\tau,u))}{\Norm_K\left(S_{\gamma}\right) }$ in front of the interpolation function $\theta^{\pm}$ -- which should be treated as an arithmetic weight -- can itself be interpolated by a smooth function in $\R^{2n - 1}$.
\item $L(u)$ is not a lattice.
\end{enumerate}
This will be done in \S\ref{sec: Smoothing of the Singularities} and \S\ref{section: the problem of completion}.

We conclude this section with the following result
\begin{theorem}
\label{manipulation}
With notation and assumptions introduced so far, the regular elliptic part of $G(K)$ can be written as
\[
R(f) = \kappa \abs{D_K} \sp^{-\sk/2} \sum_{u\in\cO_K^*}\sum_{\tau\in L(u)}
 \theta^{\pm}(\tau, u) L(1, \chi_{\gamma})\left( \sum_{\fd \mid S_{\gamma}} \frac{1}{\Norm_K(\fd)} \prod_{\fp\mid (S_{\gamma}/\fd)}\left(1 - \frac{\chi_{\gamma}(\fp)}{\Norm_K(\fp)}\right)\right).
\]
\end{theorem}

\begin{proof}
This follows from substituting Theorem~\ref{Multiplicative Formula Langlands} into our previous expansion of the regular elliptic part.
\end{proof}

In the next part, we carry out a detailed study of the sum over incomplete $L$-functions and the sum over the finite orbital integrals via the \textit{Approximate Functional Equation}.

\section{The Approximate Functional Equation}
\label{sec: Approx Funct eqn}

Let $z$ be a complex number and consider a non-square integer $\delta \equiv 0, 1\pmod 4$.
In this setting, the \textit{Zagier zeta function} introduced in \cite{zagier2006modular} is defined as follows
\[
 L(z, \delta) := \sum_{d^2\mid \delta}'\frac{1}{d^{2z - 1}}L\left( z, \binom{ \frac{\delta}{d^2}}{\cdot}_K \right).
\]
Here, we use the notation $\binom{ \frac{\delta}{d^2}}{\cdot}_K$ to denote the Kronecker symbol.
The sum is taken only over those $d$ such that $\delta/d^2 \equiv 0, 1\pmod 4$.
This congruence condition is indicated by the $\prime$ on top of the summation sign.

The Zagier zeta function has the following two pivotal roles in \cite{AliI}:
\begin{enumerate}[label = (\alph*)]
\item it recovers the values of the finite orbital integrals when evaluated at $z = 1$.
\item an appropriate completion of the Zagier zeta function satisfies a functional equation to which it is possible to apply the approximate functional equation.
\end{enumerate}
This allows us to fix the issues preventing us from performing Poisson summation.

\begin{remark}
For a general number field $F$, Arthur conjectures in \cite[\S3]{artstrat} that the generalization of the Zagier zeta function for GL$(n,F)$ is obtained via the work of Z.~Yun, see \cite{yun2013orbital}.
This is verified by the second named author in \cite{malors21} in the case $n=2$.
We make use of this construction below.
For more details about the Zagier zeta function over a general number field, we refer the reader to \cite{malors21}.
\end{remark}

\subsection{Multiplicative formula of Langlands} 
\label{subsub: MF of Langlands}
In this section, we define the Zagier zeta function for a totally real number field and prove a functional equation.
As before, let $\chi_{\gamma}$ be the quadratic sign character associated to the extension generated by the regular elliptic element $\gamma$.

\begin{definition}
Let $K$ be a totally real number field\footnote{The definition does not require that $K$ is totally real and can be made for all number fields.} and let $z$ be a complex parameter with $\Re(z) > 1$.
Set
\[
 \cO(z, \gamma) := \Norm_K(S_{\gamma})^z \sum_{\fd\mid S_{\gamma}}\Norm_K(\fd)^{(1 - 2z)} \prod_{\fq\mid \fd'}\left(1 - \frac{\chi_{\gamma}(\fq)}{\Norm_K(\fq)^z}\right), 
\]
where $\fd' := \frac{S_{\gamma}}{\fd}$ for any ideal $\fd$ such that $\fd\mid S_{\gamma}$.
Writing $L(z,\chi_\gamma)$ for the Hecke $L$-function of the character $\chi_\gamma$, define
\[
L(z, \gamma) := L(z, \chi_{\gamma}) \sum_{\fd\mid S_{\gamma}}\Norm_K(\fd)^{(1 - 2z)} \prod_{\fq\mid \fd'}\left(1 - \frac{\chi_{\gamma}(\fq)}{\Norm_K(\fq)^z}\right).
\]
This is the \emph{(generalized) multiplicative formula of Langlands} for $\GL(2, K)$.
Next, we define the completion of $\cO(z, \gamma)$ as 
\[
{\Lambda}(z, \gamma) 
:=\Lambda(z, \chi_{\gamma})\cO(z, \gamma)
\]
where $\Lambda(z, \chi_{\gamma})$ is defined below.
This is the \emph{(generalized) completed multiplicative formula of Langlands} for $\GL(2, K)$.
\end{definition}

For a totally real field $K$, the completed Dedekind zeta function is defined as
\[
\xi_K(z) := \abs{D_K}^{z/2}L_{\R}(z)^n\zeta_K(z).
\]
We define
\[
\Lambda(z, \chi_{\gamma}):= \abs{D_K}^{z/2}\Norm(\Delta_{\gamma})^{z/2}L(z, \chi_{\gamma})\prod_{\nu \mid \infty}L_{\mathbb{R}}(z + \delta_{\nu, \gamma}).
\]
Also, the $L$-factors at the archimedean places for $\chi_{\gamma}$ are given by
\[
L_{\nu}(z,\chi_\gamma) = L_{\R}(z + \delta_{\nu, \gamma}) = \pi^{(-z +\delta_{\nu, \gamma})/2}\Gamma((z+\delta_{\nu, \gamma})/2)
\]
with $\delta_{\nu, \gamma} = 1$ if $\chi_\gamma$ is ramified at $\nu$ and 0 otherwise.

We first show that $\Lambda(z, \gamma)$ behaves nicely.
More precisely,

\begin{theorem}
\label{functional equation of imprimitive Hecke functions}
The completed multiplicative formula of Langlands ${\Lambda}(z, \gamma)$, admits an analytic continuation to an entire function and satisfies the functional equation
\[
{\Lambda}(1 - z, \gamma) = {\Lambda}(z, \gamma).
\]
\end{theorem}

\begin{proof}
Writing $\mu(\cdot)$ to denote the M{\"o}bius function of an ideal of $\cO_K$, we have
\begin{equation}
\label{Rewriting product as dirichlet sum}
\prod_{\fq\mid \fd'}\left(1 - \frac{\chi_{\gamma}(\fq)}{\Norm_K(\fq)^s}\right) = \sum_{\fr\mid \fd'}\frac{\chi_{\gamma}(\fr)\mu(\fr)}{\Norm_K(\fr)^s}.
\end{equation} 
On the left hand side, the primes are not repeated in the product.
So, we only obtain terms with different prime factors whose sign depends on the number of such factors - this is controlled by the M{\"o}bius function on the right.

Substituting \eqref{Rewriting product as dirichlet sum} into definition of the completed multiplicative formula yields, 
\begin{align*}
{\Lambda}(z, \gamma) &= \Norm_K(S_{\gamma})^z \Lambda(z, \chi_{\gamma}) \sum_{\fd\mid S_{\gamma}}\frac{1}{\Norm_K(\fd)^{2z - 1}} \sum_{\fr\mid \fd'}\frac{\chi_{\gamma}(\fr)\mu(\fr)}{\Norm_K(\fr)^z}\\
&= \Norm_K(S_{\gamma})^z\Lambda(z, \chi_{\gamma}) \sum_{\fd\mid S_{\gamma}, \fr\mid \fd'}\frac{1}{\Norm_K(\fd)^{2z - 1}} \frac{\chi_{\gamma}(\fr)\mu(\fr)}{\Norm_K(\fr)^z}\\
&= \Norm_K(S_{\gamma})^z\Lambda(z, \chi_{\gamma}) \sum_{\fr\mid S_{\gamma}, \fd\mid \fr^\prime}\frac{1}{\Norm_K(\fd)^{2z - 1}} \frac{\chi_{\gamma}(\fr)\mu(\fr)}{\Norm_K(\fr)^z} \qquad \text{ where }\fr'=\frac{S_\gamma}{\fr}\\
&= \Norm_K(S_{\gamma})^z\Lambda(z, \chi_{\gamma}) \sum_{\fr\mid S_{\gamma}}\frac{\chi_{\gamma}(\fr)\mu(\fr)}{\Norm_K(\fr)^z} \sum_{\fd\mid \fr^\prime}\frac{1}{\Norm_K(\fd)^{2z - 1}}\\
&= \Norm_K(S_{\gamma})^z\Lambda(z, \chi_{\gamma}) \sum_{\fr\mid S_{\gamma}}\frac{\chi_{\gamma}(\fr)\mu(\fr)}{\Norm_K(S_{\gamma})^z}\Norm_K(\fr^\prime)^z \sum_{\fd\mid \fr^\prime}\frac{1}{\Norm_K(\fd)^{2z - 1}}\\
&= \Lambda(z, \chi_{\gamma}) \sum_{\fr\mid S_{\gamma}}\frac{\chi_{\gamma}(\fr)\mu(\fr)}{\sqrt{\Norm_K(S_{\gamma})}}\sqrt{\Norm_K(S_{\gamma})}\Norm_K(\fr^\prime)^{z} \sum_{\fd\mid \fr^\prime}\frac{1}{\Norm_K(\fd)^{2z - 1}}\\
&= \Lambda(z, \chi_{\gamma})\Norm_K(S_{\gamma})^{1/2} \sum_{\fr\mid S_{\gamma}}\frac{\chi_{\gamma}(\fr)\mu(\fr)}{\sqrt{\Norm_K(\fr\fr^\prime)}}\Norm_K(\fr^\prime)^{z} \sum_{\fd\mid \fr^\prime}\frac{1}{\Norm_K(\fd)^{2z - 1}}\\
&= \Lambda(z, \chi_{\gamma})\Norm_K(S_{\gamma})^{1/2} \sum_{\fr\mid S_{\gamma}}\frac{\chi_{\gamma}(\fr)\mu(\fr)}{\sqrt{\Norm_K(\fr)}}\Norm_K(\fr^\prime)^{z-1/2} \sum_{\fd\mid \fr^\prime}\frac{1}{\Norm_K(\fd)^{2z - 1}}.
\end{align*}
Shifting our focus to the term
\[
\Norm_K(\fr^\prime)^{z-1/2} \sum_{\fd\mid \fr^\prime}\frac{1}{\Norm_K(\fd)^{2z - 1}},
\]
we see it remains unchanged when we replace $z$ by $1 - z$.
Indeed
\begin{align*}
\Norm_K(\fr^\prime)^{(1 - z) -1/2} \sum_{\fd\mid \fr'}\frac{1}{\Norm_K(\fd)^{2(1 - z) - 1}}
&= \Norm_K(\fr')^{1/2 - z} \sum_{\fd\mid \fr'}\frac{1}{\Norm_K(\fd)^{1 - 2z}}\\
&= \Norm_K(\fr')^{1/2 - z} \sum_{\fd\mid \fr'}\frac{1}{\Norm_K\left(\frac{\fr'}{\fd}\right)^{1 - 2z}}\\
&= \Norm_K(\fr')^{1/2 - z - (1 - 2z)} \sum_{\fd\mid \fr'}\frac{1}{\Norm_K(\fd)^{2z - 1}}\\
&= \Norm_K(\fr')^{z - 1/2} \sum_{\fd\mid \fr'}\frac{1}{\Norm_K(\fd)^{2z - 1}}.
\end{align*}
Hence, when we substitute $1 - z$ for $z$ in the whole expression we obtain
\begin{align*}
{\Lambda}(1 - z, \gamma) &= \Lambda(1 - z, \chi_{\gamma})\Norm_K(S_{\gamma})^{1/2} \sum_{\fr\mid S_{\gamma}}\frac{\chi(\fr)\mu(\fr)}{\sqrt{\Norm_K(\fr)}}\Norm_K(\fr')^{(1 -z)-1/2} \sum_{\fd\mid \fr'}\frac{1}{\Norm_K(\fd)^{2(1 - z) - 1}}\\
&= \Lambda(z, \chi_{\gamma})\Norm_K(S_{\gamma})^{1/2} \sum_{\fr\mid S_{\gamma}}\frac{\chi(\fr)\mu(\fr)}{\sqrt{\Norm_K(\fr)}}\Norm_K(\fr')^{z-1/2} \sum_{\fd\mid \fr'}\frac{1}{\Norm_K(\fd)^{2z - 1}}\\
 &= {\Lambda}(z, {\gamma}).
\end{align*}
This manipulation required using the functional equation for the primitive quadratic Hecke character.
\end{proof}

\begin{remark}
Theorem~\ref{functional equation of imprimitive Hecke functions} is an adaptation of \cite[Equations (4), (5), and Lemma~2.1]{SY13}, where they work over $\mathbb{Z}$.
\end{remark}

The following result connects this construction with our work on the regular elliptic part of the trace formula.

\begin{proposition}
\label{prop: reformulation mult formula of Langlands}
With notation as introduced above,
\[
\cO(f_{\fin}, \gamma) = \sp^{-\sk/2}\cO(1, \gamma).
\]
\end{proposition}

\begin{proof}
This is a reformulation of Proposition~\ref{Multiplicative Formula Langlands}.
\end{proof}

We now define certain quadratic characters that arise naturally in our calculations.

\begin{definition}
\label{average L function definition}
Let $\fd$ be an ideal dividing $S_{\gamma}$.
Define an \emph{imprimitive character} $\chi_{\fd}$, obtained from $\chi_{\gamma}$, by setting $\chi(\fq) =0$ for any $\fq\mid \frac{S_\gamma}{\fd}$.
Associated to this imprimitive quadratic Hecke character is an $L$-function
\[
L(z, \chi_{\fd}) = L(z, \chi_{\gamma}) \prod_{\fq\mid \frac{S_\gamma}{\fd}}\left(1 - \frac{\chi_{\gamma}(\fq)}{\Norm_K(\fq)^z}\right),
\]
where $L(z,\chi_\gamma)$ is the Hecke $L$-function associated to the quadratic character $\chi_{\gamma}$.
\end{definition}

Notice that
\begin{equation}
\label{eq: L(z, gamma) rewrite}
L(z, \gamma) = L(z, \chi_{\gamma})\sum_{\fd \mid S_{\gamma}} \frac{1}{\Norm_K(\fd)^{2z - 1}}\prod_{\fp\mid (S_{\gamma}/\fd)} \left(1 - \frac{\chi_{\gamma}(\fp)}{\Norm_K(\fp)^z}\right) = \sum_{\fd \mid S_{\gamma}} \frac{1}{\Norm_K(\fd)^{2z-1}} L(z, \chi_{\fd}).
\end{equation}
Recall from the expression of $R(f)$ obtained in Theorem~\ref{manipulation} the presence of the factor
\[
L(1, \chi_{\gamma})\left( \sum_{\fd \mid S_{\gamma}} \frac{1}{\Norm_K(\fd)} \prod_{\fp\mid (S_{\gamma}/\fd)}\left(1 - \frac{\chi_{\gamma}(\fp)}{\Norm_K(\fp)}\right)\right),
\]
which is precisely $L(1, \gamma)$.
We now rewrite Theorem~\ref{manipulation} as follows.

\begin{proposition}
The regular elliptic part of the trace formula is
\[
R(f) = \kappa \abs{D_K} \sp^{-\sk/2} \sum_{u\in\cO_K^*}\sum_{\tau\in L(u)} \theta^{\pm}(\tau, u)L(1, \gamma).
\]
\end{proposition}

\begin{remark}
For $K = \mathbb{Q}$ the above decomposition becomes 
\[
\sum_{\pm}\sum_{\tau\in L(\pm 1)}
 \theta^{\pm}(\tau)L(1, \tau^2 \pm 4p^\sk),
\] 
which corresponds to \cite[(unnumbered) equation at the start of \S4]{AliI}.
An important deviation from the work of Altu\u{g} is that in order to define $L(z,\tau^2 \pm 4p^\sk)$, he uses a sum which has some congruence condition as opposed to the condition $d\mid S_{\gamma}$.
We will discuss this at length in \S\ref{section: congruence condition}.
 
When $K=\Q$, the sum over $\pm$ comes from the Dirichlet unit theorem (but there is no $\beta$-sum, i.e.~sum over fundamental units).
In the general case, we have to deal with the extra $\beta$-sum.
The factor $\sp^{-\sk/2}$ that we have but Altu\u{g} does not, is just a manifestation of the slight difference between our interpolation functions.

Another fundamental difference is that $\chi_d$ coincides with the Kronecker symbol which is well-understood and can be easily computed in $\mathbb{Q}$.
For a general number field, calculating the symbols is non-trivial, which we explore in \S\ref{subsection: The Normalized Character}.
\end{remark}

\subsection{Statement and Proof of Approximate Functional Equation}

\subsubsection{}
The approximate functional equation requires the use of a smooth cut-off function.
While it is not necessary to specify this function precisely, it is possible for us to do so as in \cite{AliI}.
For example, consider the smooth function
\begin{equation}
\label{cut-off}
F(x) = \frac{1}{2K_0(2)} \int_x^\infty e^{-y - 1/y}\frac{dy}{y} \textrm{ where }x>0,
\end{equation}
where $K_0$ is the zero-th Bessel function of the second kind.
We know from \cite[p.~257]{IK04} that $0< F(x) < \frac{e^{-x}}{2K_0(2)}$.

Define its Mellin transform,
\[
\widetilde{F}(z) := \int_0^\infty F(x)x^z\frac{dx}{x}.
\]
Moreover, the following bound will be required.

\begin{lemma}
\label{lemma:boundforFtilde}
The Mellin transform $\widetilde{F}(z)$ is meromorphic with a unique simple pole at $z = 0$ with residue $1$.
Furthermore, it is odd and writing $z = \sigma + it$, the implied constant in the following bound is absolute
\[
\widetilde{F}(z) \ll \abs{z}^{\abs{\sigma}-1}e^{-\frac{\pi\abs{t}}{2}}.
\]
\end{lemma}

\begin{proof}
This follows from the properties of $F_1$ in \cite[pp.~257--258]{IK04}.
Note that the residue of $\widetilde{F}_1(z)$ at $z=0$ is $1$, and the bounds \cite[(10.57)]{IK04} also hold for $\widetilde{F}$.
\end{proof}

\subsubsection{}
We are now in a position to state and prove the \emph{Approximate Functional Equation}.

As in \S\ref{subsub: MF of Langlands}, we need to work with the $L$-factors at the infinite places which are defined in terms of local Gamma factors; see, \cite[Lecture~2]{Rohrlich} or \cite[Equation~(5.3)]{IK04}.

\begin{notation}
\label{not: factor infinity}
Let $z$ be a complex number and define 
\[
L_{\infty}(z,\chi_\gamma):= \prod_{\nu \mid \infty}L_{\mathbb{R}}(z + \delta_{\nu, \gamma})
\]
where $\delta_{\nu, \gamma} =1$ if $\chi_\gamma$ is ramified at $\nu$ and 0 otherwise.
\end{notation}

\begin{theorem}[Approximate Functional Equation]
\label{theorem:approximate functional equation}
Let $\gamma$ be a regular elliptic element and let $\chi_{\gamma}$ denote the quadratic sign character associated the quadratic extension $L/K$ generated by $\gamma$.
Let $F$ be a smooth function as defined above.
For any $y\in\R$, further define\footnote{In this definition (and later on in the proof), note that $u$ is a choice of variable and is not to be confused with the unit $u$, which does not appear in this section.}
\begin{equation}
\label{defi: H-gamma}
H_{\gamma}(z, y) := \frac{1}{2\pi i}\int_{\Re(u) = 1} y^{-u}\widetilde{F}(u)\frac{L_{\infty}(1 - z + u, \chi_{\gamma})}{L_{\infty}(z - u, \chi_{\gamma})}du.
\end{equation}
Let $A$ be any positive real number.
Then $L(z, {\gamma})$ equals
\begin{small}
\begin{multline*}
\sum_{\fd\mid S_{\gamma}}\frac{1}{\Norm_K\left(\fd\right)^{2z - 1}}\sum_{\fa} \frac{\chi_{\fd}{(\fa)}}{\Norm_K\left(\fa\right)^{z}}F\left(\frac{\Norm_K\left(\fd\right)^2 \Norm_K(\fa)}{A}\right)\\
+ \left(\Norm_K\left(S_{\gamma}\right)\abs{D_K}^{1/2}\Norm_K(\Delta_{\gamma})^{1/2}\right)^{1 - 2z}
\sum_{\fd\mid S_{\gamma}}\frac{1}{\Norm_K\left(\fd\right)^{1 - 2z}}\sum_{\fa}\frac{\chi_{\fd}\left(\fa\right)}{\Norm_K\left(\fa\right)^{1 - z}}H_{\gamma}\left(z, \frac{\Norm_K\left(\fd\right)^{2}\Norm_K\left(\fa\right)A}{\Norm_K\left(S_{\gamma}\right)^{2}\abs{D_K}\Norm_K(\Delta_{\gamma})}\right),
\end{multline*}
\end{small}
where the sum over $\fa$ runs over all ideals of $\cO_K$.
\end{theorem}

\begin{proof}
Fix $z\in\C$ and a positive real number $\sigma$ such that $\sigma + \Re(z) > 1$ and $\Re(z) - \sigma < 0$.
Define the integral
\begin{equation}
\label{define I}
I := \frac{1}{2\pi i} \int_{\Re(u) = \sigma}{L}(z + u, \gamma)\widetilde{F}(u)A^udu.
\end{equation}
By the choice of $\sigma$, this integral is absolutely convergent.
By equation \eqref{eq: L(z, gamma) rewrite},
\begin{align*}
 {L}(z + u, \gamma) 
 &= \sum_{\fd\mid S_{\gamma}}\frac{1}{\Norm_K(\fd)^{2z + 2u - 1}}L(z + u, \chi_{\fd})\\
 &= \sum_{\fd\mid S_{\gamma}}\frac{1}{\Norm_K(\fd)^{2z + 2u - 1}} \sum_{\fa}\frac{\chi_{\fd}(\fa)}{\Norm_K(\fa)^{z+u}}.
\end{align*}
The last equality follows since we are in the half-plane of absolute convergence of the Dirichlet series for each $\chi_{\fd}$.
Substituting this in \eqref{define I} gives
\begin{align*}
I &= \frac{1}{2\pi i} \int_{\Re(u) = \sigma}{L}(z + u, \gamma)\widetilde{F}(u)A^udu\\
&= \frac{1}{2\pi i} \int_{\Re(u) = \sigma}\left( \sum_{\fd\mid S_{\gamma}}\frac{1}{\Norm_K(\fd)^{2z + 2u - 1}} \sum_{\fa}\frac{\chi_{\fd}(\fa)}{\Norm_K(\fa)^{z + u}}\right)\widetilde{F}(u)A^udu\\
&= \frac{1}{2\pi i} \int_{\Re(u) = \sigma} \sum_{\fd\mid S_{\gamma}} \sum_{\fa}\frac{1}{\Norm_K(\fd)^{2z + 2u - 1}}\frac{\chi_{\fd}(\fa)}{\Norm_K(\fa)^{z + u}}\widetilde{F}(u)A^udu\\
&= \frac{1}{2\pi i} \sum_{\fd\mid S_{\gamma}} \sum_{\fa} \int_{\Re(u) = \sigma}\frac{1}{\Norm_K(\fd)^{2z + 2u - 1}}\frac{\chi_{\fd}(\fa)}{\Norm_K(\fa)^{z + u}}\widetilde{F}(u)A^udu\\
&= \frac{1}{2\pi i} \sum_{\fd\mid S_{\gamma}}\frac{1}{\Norm_K(\fd)^{2z - 1}} \sum_{\fa} \frac{\chi_{\fd}{(\fa)}}{\Norm_K(\fa)^{z}} \int_{\Re(u) = \sigma}\frac{1}{\Norm_K(\fd)^{2u}}\frac{1}{\Norm_K(\fa)^{u}}\widetilde{F}(u)A^udu\\
&= \frac{1}{2\pi i} \sum_{\fd\mid S_{\gamma}}\frac{1}{\Norm_K(\fd)^{2z - 1}} \sum_{\fa} \frac{\chi_{\fd}{(\fa)}}{\Norm_K(\fa)^{z}} \int_{\Re(u) = \sigma}\Norm_K(\fd)^{-2u}\Norm_K(\fa)^{-u}\widetilde{F}(u)A^udu.
\end{align*}
It is possible to interchange the order of integration as the integrand is absolutely convergent.
Using the Mellin inversion formula for the last integral gives
\[
 \frac{1}{2\pi i} \int_{\Re(u) = \sigma}\Norm_K(\fd)^{-2u}\Norm_K(\fa)^{-u}\widetilde{F}(u)A^udu = F\left(\frac{\Norm_K(\fd)^2 \Norm_K(\fa)}{A}\right).
\]
Hence,
\[
I = \sum_{\fd\mid S_{\gamma}}\frac{1}{\Norm_K(\fd)^{2z - 1}} \sum_{\fa} \frac{\chi_{\fd}{(\fa)}}{\Norm_K(\fa)^{z}}F\left(\frac{\Norm_K(\fd)^2 \Norm_K(\fa)}{A}\right).
\]
Shifting the contour from $\Re(u) = \sigma$ to $\Re(u) = \sigma'$ for any $\sigma' < 0$, we get a residue from the pole of $\widetilde{F}(u)$ at $u = 0$.
This pole is simple, so the residue is
$ \lim_{u \rightarrow 0} u \widetilde{F}(u)$.
We deduce the equality
\[
 I = L(z, \gamma) + \frac{1}{2\pi i} \int_{\Re(u) = \sigma'}{L}(z + u, \gamma)\widetilde{F}(u)A^udu,
\]
or equivalently, writing $I'$ to denote the integral running over $\sigma'$ instead of $\sigma$, we have that
\[
 {L}(z, \gamma) = I - I'.
\]

Next we evaluate $I'$.
Making the change of variables $u \rightarrow -u$ and using that $\widetilde{F}$ is odd
\begin{align*}
 I' &= -\frac{1}{2\pi i} \int_{\Re(u) = \sigma'}{L}(z - u, \gamma)\widetilde{F}(-u)A^{-u}du\\
 &= -\frac{1}{2\pi i} \int_{\Re(u) = \sigma'}{L}(z - u, \gamma)\widetilde{F}(u)A^{-u}du.
\end{align*}
Evaluating ${L}(z,\gamma)$ at $1 - (z - u)$, a point in the half plane of absolute convergence, we obtain the Dirichlet series

\[
 {L}(1 - z + u, \gamma) = \sum_{\fd\mid S_{\gamma}}\frac{1}{\Norm_K(\fd)^{1 - 2z + 2u}}L(1 - z + u, \chi_{\fd}).
\]

Substituting this into $I'$ and using the functional equation of the completed Zagier zeta function,
\begin{align*}
 -I' &= \frac{1}{2\pi i} \int_{\Re(u) = \sigma'}{L}(z - u, \gamma)\widetilde{F}(u)A^{-u}du.\\
 &= \frac{1}{2\pi i} \int_{\Re(u) = \sigma'}\Norm_K(S_{\gamma})^{1 - 2(z - u)}\abs{D_K}^{(1 - 2(z - u))/2}\Norm_K(\Delta_{\gamma})^{(1 - 2(z - u))/2} \\ 
 & \hspace{2in}\times\frac{L_{\infty}(1 - z + u, \chi_\gamma)}{L_{\infty}(z - u, \chi_\gamma)}{L}(1 - z + u, \gamma)
 \widetilde{F}(u)A^{-u}du\\
 &= \frac{1}{2\pi i} \int_{\Re(u) = \sigma'}\Norm_K(S_{\gamma})^{1 - 2(z - u)}\abs{D_K}^{(1 - 2(z - u))/2}\Norm_K(\Delta_{\gamma})^{(1 - 2(z - u))/2}\frac{L_{\infty}(1 - z + u, \chi_\gamma)}{L_{\infty}(z - u, \chi_\gamma)}\\
 & \hspace{2.2in} \times\widetilde{F}(u)A^{-u} \sum_{\fd\mid S_{\gamma}}\frac{1}{\Norm_K(\fd)^{1 - 2z + 2u}} \sum_{\fa}\frac{\chi_{\fd}(\fa)}{\Norm_K(\fa)^{1 - z + u}}du \\
 &= \frac{1}{2\pi i} \int_{\Re(u) = \sigma'} \sum_{\fd\mid S_{\gamma}} \sum_{\fa}\frac{\Norm_K(S_{\gamma})^{1 - 2(z - u)}\abs{D_K}^{(1 - 2(z - u))/2}\Norm_K(\Delta_{\gamma})^{(1 - 2(z - u))/2} \widetilde{F}(u)A^{-u}}{\Norm_K(\fd)^{1 - 2z + 2u}}\\
 & \hspace{3in}\times \frac{L_{\infty}(1 - z + u, \chi_\gamma)}{L_{\infty}(z - u, \chi_\gamma)}\frac{\chi_{\fd}(\fa)}{\Norm_K(\fa)^{1 - z + u}}du.
\end{align*} 
As the Dirichlet series converges absolutely and uniformly in a half-plane away from the abscissa of convergence, we may interchange the order of integration.
Shifting the integral to $\sigma = 1$ since there are no poles in between, 
\begin{align*}
 -I' &= \frac{1}{2\pi i} \sum_{\fd\mid S_{\gamma}} \sum_{\fa} \int_{\Re(u) = \sigma'}\frac{\Norm_K(S_{\gamma})^{1 - 2(z - u)}\abs{D_K}^{(1 - 2(z - u))/2}\Norm_K(\Delta_{\gamma})^{(1 - 2(z - u))/2}\widetilde{F}(u)A^{-u}}{\Norm_K(\fd)^{1 - 2z + 2u}}\\
 & \hspace{3in}\times \frac{L_{\infty}(1 - z + u, \chi_\gamma)}{L_{\infty}(z - u, \chi_\gamma)}\frac{\chi_{\fd}(\fa)}{\Norm_K(\fa)^{1 - z + u}}du\\
 &= \frac{\left(\Norm_K\left(S_{\gamma}\right)\abs{D_K}^{1/2}\Norm_K(\Delta_{\gamma})^{1/2}\right)^{1 - 2z}}{2\pi i}
 \sum_{\fd\mid S_{\gamma}}\frac{1}{\Norm_K\left(\fd\right)^{1 - 2z}} \sum_{\fa}\frac{\chi_{\fd}(\fa)}{\Norm_K(\fa)^{1 - z}}\\
 & \hspace{1in}\times \int_{\Re(u) = \sigma'}\frac{\Norm_K(S_{\gamma})^{2u}\abs{D_K}^{u}\Norm_K(\Delta_{\gamma})^u\widetilde{F}(u)A^{-u}}{\Norm_K(\fd)^{2u}\Norm_K(\fa)^{u}}\frac{L_{\infty}(1 - z + u, \chi_\gamma)}{L_{\infty}(z - u, \chi_\gamma)}du.
\end{align*}
Collecting terms inside the last integral, it is possible to write it as
\[
 \int_{\Re(u) = 1}\left(\frac{\Norm_K\left(\fd\right)^{2}\Norm_K\left(\fa\right)A}{\Norm_K\left(S_{\gamma}\right)^{2}\abs{D_K}\Norm_K(\Delta_{\gamma})}\right)^{-u}\frac{L_{\infty}\left(1 - z + u, \chi_\gamma\right)}{L_{\infty}\left(z - u, \chi_\gamma\right)}\widetilde{F}\left(u\right)du.
\]
This motivates the definition in the theorem 
\[
 H_{\gamma}(z, y) := \frac{1}{2\pi i}\int_{\Re(u) = 1} y^{-u}\widetilde{F}(u)\frac{L_{\infty}(1 - z + u, \chi_\gamma)}{L_{\infty}(z - u, \chi_\gamma)}du.
\]
Therefore, we have proven
\begin{footnotesize}
\[
 -I' = \left(\Norm_K\left(S_{\gamma}\right)\abs{D_K}^{1/2}\Norm_K(\Delta_{\gamma})^{1/2}\right)^{1 - 2z}
 \sum_{\fd\mid S_{\gamma}}\frac{1}{\Norm_K\left(\fd\right)^{1 - 2z}} \sum_{\fa}\frac{\chi_{\fd}\left(\fa\right)}{\Norm_K\left(\fa\right)^{1 - z}}H_{\gamma}\left(z, \frac{\Norm_K\left(\fd\right)^{2}\Norm_K\left(\fa\right)A}{\Norm_K\left(S_{\gamma}\right)^{2}\abs{D_K}\Norm_K(\Delta_{\gamma})}\right).
\]
\end{footnotesize}
Hence, putting together our evaluations of $I$ and $I'$ we get
\begin{footnotesize}
\begin{equation*}
\begin{split}
 & {L}(z, \gamma) = \sum_{\fd\mid S_{\gamma}}\frac{1}{\Norm_K\left(\fd\right)^{2z - 1}}\sum_{\fa} \frac{\chi_{\fd}{(\fa)}}{\Norm_K\left(\fa\right)^{z}}F\left(\frac{\Norm_K\left(\fd\right)^2 \Norm_K(\fa)}{A}\right) \\
 & + \left(\Norm_K\left(S_{\gamma}\right)\abs{D_K}^{1/2}\Norm_K(\Delta_{\gamma})^{1/2}\right)^{1 - 2z}
 \sum_{\fd\mid S_{\gamma}}\frac{1}{\Norm_K\left(\fd\right)^{1 - 2z}}\sum_{\fa}\frac{\chi_{\fd}\left(\fa\right)}{\Norm_K\left(\fa\right)^{1 - z}}H_{\gamma}\left(z, \frac{\Norm_K\left(\fd\right)^{2}\Norm_K\left(\fa\right)A}{\Norm_K\left(S_{\gamma}\right)^{2}\abs{D_K}\Norm_K(\Delta_{\gamma})}\right).
 \end{split}
\end{equation*}
\end{footnotesize}
\end{proof}

\subsection{The function $H_{(r, y)}$}

The following result connects the archimedean behavior of the quadratic character $\chi_{\gamma}$ with the regions we defined in \S\ref{subsec: A partition of R 2n-1}.

\begin{theorem}
\label{thm: reinterpretation of C}
Let $\gamma\in\GL(2, K)$ be a regular elliptic matrix and $\nu$ be a real place of $K$.
Then, $\chi_{\gamma}$ is unramified at $\nu$ if and only if $(r_{\nu}, N_{\nu}) \in S_{\spl}$.
\end{theorem}

\begin{notation}
In this proof we write $\chi$ in place of $\chi_{\gamma}$ for ease of notation.
\end{notation}
\begin{proof}
Let $\nu$ denote a real place of $K$ and let $K_{\nu} = \R$ be the completion.
Let $\tilde{\nu}$ be some (and hence any) extension of $\nu$ to $K_{\gamma}$,
and let $\rho_{\gamma,\nu}\colon K_\nu^* \rightarrow \Gal(K_{\gamma,\tilde{\nu}}/K_\nu)$
be the local Artin map, so that $\chi_\nu$ is the character of $K_\nu^*$ attached to
$K_{\gamma,\tilde{\nu}}/K_\nu$.
Recall that
\begin{align*}
\chi \textrm{ is unramified at } \nu
&\Longleftrightarrow \chi_{\nu}(-1) = 1 \textrm{ by definition}\\
&\Longleftrightarrow \rho_{\gamma,\nu}(-1) = 1\\
&\Longleftrightarrow -1 \in \ker(\rho_{\gamma,\nu})
 = \Norm_{K_{\gamma,\tilde{\nu}}/K_{\nu}}(K_{\gamma,\tilde{\nu}}^*)
 \textrm{ by the local reciprocity law.}
\end{align*}
We have $\Norm_{K_{\gamma,\tilde{\nu}}/K_{\nu}}(K_{\gamma,\tilde{\nu}}^*)=\R_{>0}$ if $\nu$ ramifies (i.e.~$K_{\gamma,\tilde{\nu}}=\C$) and $\Norm_{K_{\gamma,\tilde{\nu}}/K_{\nu}}(K_{\gamma,\tilde{\nu}}^*)=\R^*$ if $\nu$ splits
(i.e.~$K_{\gamma,\tilde{\nu}}= \R$).
Thus, $\chi$ is unramified at $\nu$ if and only if $K_{\gamma,\tilde{\nu}}= \R$.
We know that $K_{\gamma,\tilde{\nu}} = K_{\nu}\!\left(\sqrt{\tau_\nu^2 - 4\delta_\nu}\right)$, which equals $\R$ if and only if $\tau_\nu^2 - 4\delta_\nu > 0$.

Finally, recall that in the $(r_\nu, N_\nu)$-coordinates we have
$r_\nu = \tau_\nu$ and $N_\nu = 4\delta_\nu$, so that $\tau_\nu^2 - 4\delta_\nu > 0$ if and only if $r_{\nu}^2 - N_{\nu} > 0$.
Since $\gamma\in\GL(2,K)$ we have $\delta_\nu\neq 0$, whence $N_\nu \neq 0$.
So, $\chi$ is unramified at $\nu$ if and only if $(r_\nu, N_\nu)$ lies in $S_{\spl}$.
\end{proof}

\begin{proposition}
\label{prop: extension H}
Let $\gamma_1, \gamma_2$ be regular elliptic matrices.
If $\gamma_1\siminf\gamma_2$ then $H_{\gamma_1}= H_{\gamma_2}$.
\end{proposition}

\begin{proof}
\label{Prop 5.9}
By definition
\[
H_{\gamma}(z, y) := \frac{1}{2\pi i}\int_{\Re(u) = 1} y^{-u}\widetilde{F}(u)\frac{L_{\infty}(1 - z + u, \chi_{\gamma})}{L_{\infty}(z - u, \chi_{\gamma})}du.
\]
The dependence of $\gamma$ is on the factors at infinity.
But, these factors are determined by the ramification of $\gamma$ at infinity.
Theorem~\ref{thm: reinterpretation of C} precisely states that the ramification behaviour of $\chi_{\gamma}$ remains invariant if $\gamma_1\siminf\gamma_2$.
\end{proof}

\begin{definition}
\label{def: Function H}
Let $(r, y)\in\mathbb{R}^n\times\mathbb{R}^{n-1}$ be such that $(r, \beta^y)\in S_J$ for the multi-index $J = (j_1, \ldots, j_n)$ where each $j_t \in\{ {\spl}, {\el}\}$.
Define $H_{(r, y)}: \mathbb{C}\times\mathbb{R}\longrightarrow\mathbb{C}$ by
\[
H_{(r, y)}(z, t) := H_{\gamma}(z, t)
\]
for any regular elliptic matrix with $\gamma\in\GL(2, K)$ with $\CH(\gamma)\in S_J$.
\end{definition}

\begin{remark}{\leavevmode}
\begin{enumerate}[label = (\roman*)]
\item The above definition is well-defined because all maximal tori of $G_{\infty}$ span all possible splitting behaviour and Proposition~\ref{Prop 5.9} asserts that the function $H_{\gamma}$ is invariant under splitting behaviour.
\item 
Definition~\ref{def: Function H} allows us to extend the definition of $H_{\gamma}=H_{(\tau,u)}$ to \emph{all} points $(r, y)\in \R^{2n-1}$.
[Here we are using the fact that $\gamma$ is characterized by $(\tau,u)$ where $u$ is once again a unit in $\cO_K$.]
\item Implicit in the above definition is the fact that $\gamma_1\siminf \gamma_2$ implies 
\[
\frac{L_{\infty}(1 - z + u, \chi_{\gamma_1})}{L_{\infty}(z - u, \chi_{\gamma_1})} = \frac{L_{\infty}(1 - z + u, \chi_{\gamma_2})}{L_{\infty}(z - u, \chi_{\gamma_2})}.
\]
Indeed, recalling the notation in \S\ref{not: factor infinity} we see that these factors depend on the behavior of $\chi_\gamma$ at infinity which in turn only depends on the subset $S_J$ (as proved in Theorem~\ref{prop: extension H}).
When we want to emphasize the independence we suppress the subscript $\gamma$ and write
\[
\frac{L_{\infty}(1 - z + u, \chi)}{L_{\infty}(z - u, \chi)}.
\]
In particular, we use this notation in Corollary~\ref{no dep on gamma}.
\end{enumerate}
\end{remark}

\subsection{Estimates on $H_{\gamma}(1,x)$}

The difference for each $H_{\gamma}(1,x)$ as defined in Theorem~\ref{theorem:approximate functional equation} comes from the $L_{\infty}$ factors.
We need to calculate the bound when $z=1$.

\begin{lemma}
\label{lemma:boundH}
For $x \geq 1$, 
\[
H_{\gamma}(1,x) \ll \frac{1}{x} e^{-(\sqrt[n+1]{x})}\] where the implied constant is absolute.
\end{lemma}

\begin{notation}
In this proof we write $u$ to denote a variable and it should not be confused with a unit of $\cO_K$.
\end{notation}

\begin{proof}
We begin the proof by recalling Stirling's formula for the Gamma function (see \cite[p.~326]{SS038}), for $|u| \rightarrow \infty$ and $|\arg u|<\pi$ we have 
\[
\Gamma(u)=\sqrt{2\pi} u^{u-1/2}e^{-u}\left(1+O\left(\abs{u}^{-1/2}\right)\right).
\]
Therefore, 
\begin{footnotesize}
\begin{align*}
\Gamma\left((u+\delta_{\nu,\gamma})/2\right) &=\sqrt{2\pi}\left((u+\delta_{\nu,\gamma})/2\right)^{(u+\delta_{\nu,\gamma}-1)/2}e^{(-u-\delta_{\nu,\gamma})/2} \left(1+O\left(\abs{(u+\delta_{\nu,\gamma})/2}^{-1/2}\right)\right) \\
\Gamma\left((1-u+\delta_{\nu,\gamma})/2\right) & =\sqrt{2\pi}\left((1-u+\delta_{\nu,\gamma})/2\right)^{(-u+\delta_{\nu,\gamma})/2}e^{(-1+u-\delta_{\nu,\gamma})/2}\left(1+O\left(\abs{(1-u+\delta_{\nu,\gamma})/2}^{-1/2}\right)\right).
\end{align*}
\end{footnotesize}
Combining the above two expressions,
\begin{small}
\begin{align*}
& \frac{\Gamma\left((u+\delta_{\nu,\gamma})/2\right)}{\Gamma\left((1-u+\delta_{\nu,\gamma})/2\right)} \\
= & e^{-u+1/2} \left(\frac{u+\delta_{\nu,\gamma}}{2}\right)^{u-1/2} \left(\frac{u+\delta_{\nu,\gamma}}{2}\right)^{-u/2+\delta_{\nu,\gamma}/2} \left(\frac{1-u+\delta_{\nu,\gamma}}{2}\right)^{(u-\delta_{\nu,\gamma})/2}\left(1+O\left({\abs{u}^{-1/2}}\right)\right)\\
= & \left(\frac{u+\delta_{\nu,\gamma}}{2e}\right)^{u-1/2}\left(\frac{1-u+\delta_{\nu,\gamma}}{u+\delta_{\nu,\gamma}}\right)^{(u-\delta_{\nu,\gamma})/2}\left(1+O\left({\abs{u}^{-1/2}}\right)\right)\\
= & \left(\frac{u}{2e}\right)^{u-1/2}\left(\frac{u+\delta_{\nu,\gamma}}{u}\right)^{u-1/2} \left(\frac{u}{u+\delta_{\nu,\gamma}}\right)^{(u-\delta_{\nu,\gamma})/2} \left(\frac{1-u+\delta_{\nu,\gamma}}{u}\right)^{(u-\delta_{\nu,\gamma})/2} \left(1+O\left({\abs{u}^{-1/2}}\right)\right)\\
= & \left(\frac{u}{2e}\right)^{u-1/2}\left(1+\frac{\delta_{\nu,\gamma}}{u}\right)^{u/2+(\delta_{\nu,\gamma}-1)/2}\left(\frac{1+\delta_{\nu,\gamma}}{u}-1\right)^{(u-\delta_{\nu,\gamma})/2}\left(1+O\left({\abs{u}^{-1/2}}\right)\right).
\end{align*}
\end{small}
Since $\delta_{\nu,\gamma}\in\{0,1\}$, the two quotients above are precisely the quotients $\Gamma\left(\frac{u}{2}\right)/\Gamma\left(\frac{1-u}{2}\right)$ and $\Gamma\left(\frac{1+u}{2}\right)/\Gamma\left(\frac{2-u}{2}\right)$ that are bounded in the proof of \cite[Lemma~3.6]{AliI}.
The bound is applied only for $u$ on the vertical lines $\Re(u)=\sigma$ with $\sigma\geq 1$ (namely on $\Re(u)=1$ and on the shifted contour $\Re(u)=\max\{1,\sqrt[n+1]{x}\}$ below), and on such a line the M\"obius transformation $u\mapsto (1+\delta_{\nu,\gamma})/u$ maps $\Re(u)=\sigma$ onto the circle through the origin of radius $(1+\delta_{\nu,\gamma})/2\sigma\leq 1$.
Consequently,
\[
\abs{1+\delta_{\nu,\gamma}/u}\leq 2
\qquad\textrm{and}\qquad
\abs{(1+\delta_{\nu,\gamma})/u-1}\leq 1,
\]
which is the observation establishing equation~$(\dagger)$ in the proof of
\cite[Lemma~3.6]{AliI}.
By \emph{loc.~cit.} we conclude that on every line
$\Re(u)=\sigma\geq 1$,
which gives 
\[
\frac{\Gamma\left((u+\delta_{\nu,\gamma})/2\right)}{\Gamma\left((1-u+\delta_{\nu,\gamma})/2\right)}\ll \left(\frac{u}{\sqrt{2}e}\right)^{u-1/2}
\]
where the implied constant is absolute.
Hence we have 
\[
\frac{L_{\infty}(u,\chi_\gamma)}{L_{\infty}(1-u,\chi_\gamma)}= \prod_{j=1}^n\frac{\Gamma\left((u+\delta_{\nu,\gamma})/2\right)}{\Gamma\left((1-u+\delta_{\nu,\gamma})/2\right)}\ll \left(\frac{u}{\sqrt{2}e}\right)^{n(u-1/2)}.
\]
Substituting into the definition for $H_{\gamma}(1,x)$ and using the bound on $\widetilde{F}(x)$ given in Lemma~\ref{lemma:boundforFtilde}, 

\begin{align*}
H_{\gamma}(1,x)& \ll \int_{\Re(u)=1} \abs{\frac{u}{\sqrt{2}e}}^{n(\Re(u)-1/2)}x^{-\Re(u)}\abs{u}^{\Re(u)-1}e^{-\abs{\Im(u)}\pi/2}\\
&\ll \int_{\Re(u)=1}\abs{u}^{(n+1)\Re(u)-(n+2)/2}(e^{n+1}x)^{-\Re(u)}e^{-\pi\abs{\Im(u)}/2}du.
\end{align*}
Shifting the contour to $\Re(u) = \max\{1, \sqrt[n+1]{x} \}$ gives

\begin{align*}
H_{\gamma}(1,x)&\ll \frac{1}{x^{\frac{n+2}{2(n+1)}}}e^{-(\sqrt[n+1]{x})}\\
&\ll \frac{1}{x} e^{-(\sqrt[n+1]{x})}.\qedhere
\end{align*}
\end{proof}
\begin{corollary}
For $t\geq 1$, 
\[
H_{(r, y)}(1, t) \ll \frac{1}{t} e^{-(\sqrt[n+1]{t})}\] where the implied constant is absolute.
\end{corollary}

\subsection{Manipulation using the Approximate Functional Equation}

At this point, we have the following expression for the regular elliptic part of the trace formula
\[
R(f) = \kappa \abs{D_K} \sp^{-\sk/2} \sum_{u\in\cO_K^*}\sum_{\tau\in L(u)} \theta^{\pm}(\tau, u)L(1, \gamma).
\]
Using the approximate functional equation (see Theorem~\ref{theorem:approximate functional equation}) at $z = 1$, we find
\begin{align*}
L(1, \gamma) &= \sum_{\fd\mid S_{\gamma}}\frac{1}{\Norm_K(\fd)} \sum_{\fa} \frac{\chi_{\fd}{(\fa)}}{\Norm_K(\fa)}F\left(\frac{\Norm_K(\fd)^2 \Norm_K(\fa)}{A}\right)\\
&+ \frac{1}{ \Norm_K(S_{\gamma})\abs{D_K}^{1/2}\Norm_K(\Delta_{\gamma})^{1/2}} \sum_{\fd\mid S_{\gamma}}\frac{1}{\Norm_K(\fd)^{-1}} \sum_{\fa}\chi_{\fd}(\fa) H_{\gamma}\left(1, \frac{\Norm_K(\fd)^{2}\Norm_K(\fa)A}{\Norm_K(S_{\gamma})^{2}\abs{D_K}\Norm_K(\Delta_{\gamma})}\right).
\end{align*}
Recall that $\partial = \tau^2 - 4\det(\gamma)$, therefore
\[
\Norm_K(\partial) = \frac{\abs{D_L}\Norm_K(S_{\gamma}^2)}{\abs{D_K}^2} = \Norm_K(\Delta_{\gamma})\Norm_K(S_{\gamma})^2.
\]
Combining these expressions, we obtain that
\begin{equation*}
\begin{split}
 R(f) &= \kappa \abs{D_K} \sp^{-\sk/2} \sum_{u\in\cO^*_K}\sum_{\tau\in L(u)} \theta^{\pm}(\tau, u) \left(\sum_{\fd\mid S_{\gamma}}\frac{1}{\Norm_K(\fd)} \sum_{\fa} \frac{\chi_{\fd}{(\fa)}}{\Norm_K(\fa)}F\left(\frac{\Norm_K(\fd)^2 \Norm_K(\fa)}{A}\right) \right.\\
 & +\left. \frac{1}{\Norm_K(\partial)^{1/2}\abs{D_K}^{1/2}} \sum_{\fd\mid S_{\gamma}}\frac{1}{\Norm_K(\fd)^{-1}} \sum_{\fa}\chi_{\fd}(\fa)H_{\gamma}\left(1, \frac{\Norm_K(\fd)^{2}\Norm_K(\fa)A}{\Norm_K(\partial)\abs{D_K}}\right)\right).
 \end{split}
\end{equation*}

\begin{remark}
At this stage two sums have appeared: one runs over all ideals $\fa$ (this sum arises from the Dirichlet series of the imprimitive character).
The other runs over $\fd\mid S_{\gamma}$ (this arises from the definition of the generalized Multiplicative Formula of Langlands).
\end{remark}

\section{Smoothing the Singularities}
\label{sec: Smoothing of the Singularities}

Throughout this section we fix a totally positive system of fundamental units of $K$ denoted by $\{\beta_1, \ldots, \beta_{n-1}\}$.
Let $\gamma$ be a contributing regular elliptic matrix that corresponds to the unit $u$ and the trace $\tau$.
Recall from \eqref{eqn: det gamma and rho} that $\varepsilon = \rho^{\sk'}$ is a fixed generator for the principal ideal $\left( \det(\gamma)\right) = \fp^{\sk}$ with $\Norm_K(\fp)=\sp$.

\begin{notation}
We write $\varepsilon_i=\sigma_i(\varepsilon)$ and $\beta_{t,i}=\sigma_i(\beta_t)$ for $1\leq t \leq n-1$, where $\sigma_i$ is the $i$-th real embedding of $K$.
\end{notation}

\begin{definition}\label{def: P pm}
Let $r = (r_1, \ldots , r_n) \in \R^n$ and $y=(y_1, \dots, y_{n-1})\in \R^{n-1}$.
Define the discriminant function $\mathcal{P}^{\pm}:\R^{n}\times\R^{n-1}\rightarrow\R$ by
\begin{align*}
\mathcal{P}^{\pm}(r , y) &:= \frac{1}{2^n}\prod_{i = 1}^n \left( \frac{\abs{ r_i^2 \pm 4 \beta_{1,i}^{y_1}\dots \beta_{n-1,i}^{y_{n-1}}\varepsilon_i}^{-1/2}}{\abs{4\beta_{1,i}^{y_1}\dots \beta_{n-1,i}^{y_{n-1}}\varepsilon_i}^{-1/2}} \right).
\end{align*}
\end{definition}

Next, we evaluate the above defined function at $r=\tau$ and $y=u$.
Our notation here will be similar to the one adapted for $\theta^\pm$.
We explain it once again for convenience.
Since $\tau \in \cO_K$, it can be viewed as an element of $\R^n$.
On the other hand, when we write $y=u$, we mean that $y=(y_1, \dots, y_{n-1})$ corresponding to the unit $u = \pm\beta_1^{y_1}\dots \beta_{n-1}^{y_{n-1}}$.
The relevance of the functions $\mathcal{P}^{\pm}$ is clear from the following result.

\begin{proposition}
\label{Prop: restatement}
With notation as above, 
\[
\mathcal{P}^{\pm}(\tau, u) = \frac{\sp^{\sk/2}}{\sqrt{\Norm_K(\partial)}}.
\]
\end{proposition}
\begin{proof}
By hypothesis we are evaluating the functions $\mathcal{P}^{\pm}$ at $r=\tau$ and $y=u$.
Since there exists a regular elliptic matrix corresponding to $(\tau, u)$, for each $i$ we can write
\[
\frac{\abs{r_i^2 \pm 4 \beta_{1,i}^{y_1}\dots \beta_{n-1,i}^{y_{n-1}}\varepsilon_i}^{-1/2}}{\abs{4\beta_{1,i}^{y_1}\dots \beta_{n-1,i}^{y_{n-1}}\varepsilon_i}^{-1/2}} = \abs{\frac{\tau_i}{4u_i\varepsilon_i} \pm 1}^{-1/2} = 2D(\sigma_i(\lambda_1), \sigma_i(\lambda_2))^{-1/2},
\]
where $\lambda_1, \lambda_2$ are the eigenvalues of $\gamma(\tau, u)$.
By Proposition~\ref{relation discriminants}
\[
\mathcal{P}^{\pm}(\tau, u) = \frac{\sp^{\sk/2}}{\sqrt{\Norm_K(\partial)}}.
\qedhere
\]
\end{proof}

Recall that our interpolation function $\theta^{\pm}$ has singularities.
In what follows we define smoothed-out versions of this interpolation function which then appear more naturally in the expression for the regular elliptic part.

\begin{definition}
Let $F$ be the smooth function defined in \eqref{cut-off}.
For a parameter $0< \alpha < 1$, and for $(r, y)\in\mathbb{R}^n\times\mathbb{R}^{n-1}$ such that $(r, \beta^y)$ is in the union of the open regions $S_J$ define the functions
\begin{small}
\begin{align*}
 \Phi_{(\pm, \alpha)}(r, y) &:= \Phi_{\pm}(r, y) \\ 
 & := \theta^{\pm}(r, y)F\left(\frac{\Norm_K(\fd)^2\Norm_K(\fa)\sp^{-\sk\alpha }}{\abs{D_K}^\alpha} \mathcal{P}^{\pm}(r, y)^{2\alpha}\right),\\
\Psi_{(\pm, \alpha)}(r, y) &:= \Psi_{\pm}(r, y)\\
&:= \theta^{\pm}(r, y) \mathcal{P}^{\pm}(r, y)\abs{D_K}^{-1/2}\sp^{-\sk/2}H_{(r, y)}\left(1, \frac{\Norm_K(\fd)^2\Norm_K(\fa)\sp^{-\sk(1-\alpha)}}{\abs{D_K}^{1 - \alpha}}\mathcal{P}^{\pm}(r, y)^{2(1-\alpha)}\right).
\end{align*}
\end{small}
\end{definition}

Suppose that $r = \tau$ and $y = u$.
Then
\begin{equation}
\begin{split}
\label{defn: Phi}
\Phi_{\pm}(\tau, u) 
&= \theta^{\pm}(\tau, u)F\left(\frac{\Norm_K(\fd)^2\Norm_K(\fa)\sp^{-\sk\alpha }}{\abs{D_K}^\alpha} \mathcal{P}^{\pm}(\tau, u)^{2\alpha}\right)\\
&= \theta^{\pm}(\tau, u)F\left(\frac{\Norm_K(\fd)^2\Norm_K(\fa)\sp^{-\sk\alpha }}{\abs{D_K}^\alpha} \left(\frac{\sp^{\sk/2}}{\Norm_K(\partial)^{1/2}}\right)^{2\alpha}\right) \quad \textrm{by Proposition}~\ref{Prop: restatement}\\
&= \theta^{\pm}(\tau, u)F\left(\frac{\Norm_K(\fd)^2\Norm_K(\fa)}{\Norm_K(\partial)^{\alpha}\abs{D_K}^\alpha} \right).
\end{split}
\end{equation}

Recall the definition of the function $H_\gamma(z,y)$ introduced in \eqref{defi: H-gamma}.
A similar calculation as the one above yields
\begin{footnotesize}
\begin{equation}
\begin{split}
\label{defn: Psi}
\Psi_{\pm}(\tau, u)
&= \theta^{\pm}(\tau, u) \mathcal{P}^{\pm}(\tau, u)\abs{D_K}^{-1/2}\sp^{-\sk/2}H_{\gamma}\left(1, \frac{\Norm_K(\fd)^2\Norm_K(\fa)\sp^{-\sk(1-\alpha)}}{\abs{D_K}^{1 - \alpha}}\mathcal{P}^{\pm}(\tau, u)^{2(1-\alpha)}\right)\\
&= \theta^{\pm}(\tau, u) \left(\frac{\sp^{\sk/2}}{\Norm_K(\partial)^{1/2}}\right)\abs{D_K}^{-1/2}\sp^{-\sk/2}H_{\gamma}\left(1, \frac{\Norm_K(\fd)^2\Norm_K(\fa)\sp^{-\sk(1-\alpha)}}{\abs{D_K}^{1 - \alpha}}\left(\frac{\sp^{\sk/2}}{\Norm_K(\partial)^{1/2}}\right)^{2(1-\alpha)}\right)\\
&= \frac{\theta^{\pm}(\tau, u)}{\abs{D_K}^{1/2}\Norm_K(\partial)^{1/2}}H_{\gamma}\left(1,\frac{\Norm_K(\fd)^2\Norm_K(\fa)}{\Norm_K(\partial)^{1-\alpha}\abs{D_K}^{1 - \alpha}}\right),
\end{split}
\end{equation}
\end{footnotesize}
where $\gamma = \gamma(\tau, u)$ is the corresponding regular elliptic matrix.

\begin{proposition}
\label{prop: smoothing}
Let $\alpha>0$ and let $\phi$ be a Schwartz function\footnote{For our purposes, $\phi$ will either be the function $F$ or the function $H_{\gamma}$.} on $\mathbb R$.
Then 
\[
\theta^\pm(r,y)\phi(\mathcal{P}^{\pm}(r,y)^{2\alpha}),\qquad 
\theta^\pm(r,y)\mathcal{P}^{\pm}(r,y)\phi(\mathcal{P}^{\pm}(\tau, u)^{2(1-\alpha)})
\]
are smooth in $(r,y)$ and compactly supported.
\end{proposition}

\begin{proof}
The proof is essentially that of \cite[Theorem 4.1]{GKMPW}, also \cite[Proposition 4.1]{AliI}, in the setting of $\GL(n,\mathbb Q)$ and $\GL(2,\mathbb Q)$ respectively.
Following the discussion of \cite[\S4]{GKMPW}, the singularities of the archimedean orbital integral $\theta^\pm$ in $\mathbb R^{2n-1}$ are precisely the zero locus of the discriminant, which is given by the vanishing of $\mathcal{P}^{\pm}(\tau, u)$, and are at most $\Vert (r, y)\Vert^{-\beta}$ for some $\beta\ge0$.
Then for any singular point $a$ of $\theta^\pm$ and $\alpha>0$ we have $\mathcal{P}^{\pm}(r, y)^{2\alpha}\to \infty$ as $v=(r,y)\to a$, and
\[
\phi(\mathcal{P}^{\pm}(r, y)^{2\alpha}) \ll_M \mathcal{P}^{\pm}(r, y)^{-2\alpha M}
\]
for any $M>0.$ In particular, taking $M > (\beta + 1)/\alpha$ and arguing as in the proof of \cite[Theorem 4.1]{GKMPW}, we see that the product $\theta^\pm(r, y)\phi(\mathcal{P}^{\pm}(\tau, u)^{2\alpha})$ and its partial derivatives vanish as $(r, y)$ approaches $a$.
The argument for the second product $\theta^\pm(r, y)\mathcal{P}^{\pm}(x, y)\phi(\mathcal{P}^{\pm}(v, y)^{2(1-\alpha)})$ follows in a similar manner.

Finally, at the archimedean places, we had imposed the condition that the test functions have compact support and the orbital integrals have relatively compact support.
This ensures that the functions of interest are compactly supported.
\end{proof}

\begin{remark}
The proof shows these functions are now defined in $\mathbb{R}^n\times\mathbb{R}^{n-1}$ and not only in the regular set.
\end{remark}

Upon choosing a precise value of $A$ in Theorem~\ref{theorem:approximate functional equation}, it is possible to ensure that $\Phi_{\pm}$ and $\Psi_{\pm}$ emerge in our manipulation of the regular elliptic part.

\begin{proposition}
\label{prop: precise value of A}
For a regular elliptic class $\gamma$ and a parameter $0 < \alpha <1$,
\begin{align*}
L(1, \gamma) & = \sum_{\fd\mid S_{\gamma}}\frac{1}{\Norm_K(\fd)} \sum_{\fa} \frac{\chi_{\fd}{(\fa)}}{\Norm_K(\fa)}F\left(\frac{\Norm_K(\fd)^2 \Norm_K(\fa)}{\Norm_K(\partial)^\alpha \abs{D_K}^\alpha}\right) \\
&+ \frac{1}{ \sqrt{\Norm_K(\partial)} \abs{D_K}^{1/2}} \sum_{\fd\mid S_{\gamma}}\frac{1}{\Norm_K(\fd)^{-1}} \sum_{\fa}\chi_{\fd}(\fa)H_{\gamma}\left(1, \frac{\Norm_K(\fd)^{2}\Norm_K(\fa)}{\Norm_K(\partial)^{1-\alpha}\abs{D_K}^{1-\alpha}}\right), 
\end{align*}
\end{proposition}

\begin{proof}
This follows from the Approximate Functional Equation upon choosing
\[
A = \Norm_K(\partial)^{\alpha}\abs{D_K}^{\alpha}.
\qedhere
\]
\end{proof}

Using Proposition~\ref{prop: precise value of A} and the definitions of $\Psi_{\pm}$ and $\Phi_{\pm}$, we get
\begin{footnotesize}
\begin{equation}
\begin{split}
 R(f)
 &=\kappa \abs{D_K} \sp^{-\sk/2}\sum_{u\in\cO_K^*}\sum_{\tau\in L(u)} \theta^{\pm}(\tau, u) \left(\sum_{\fd\mid S_{\gamma}}\frac{1}{\Norm_K(\fd)} \sum_{\fa} \frac{\chi_{\fd}{(\fa)}}{\Norm_K(\fa)}F\left(\frac{\Norm_K(\fd)^2 \Norm_K(\fa)}{\Norm_K(\partial)^{\alpha}\abs{D_K}^{\alpha}}\right)\right)\\
 &+ \frac{ \kappa \abs{D_K} \sp^{-\sk/2}}{ \Norm_K(\partial)^{1/2}\abs{D_K}^{1/2}}
 \sum_{u\in\cO_K^*}\sum_{\tau\in L(u)}\sum_{\fd\mid S_{\gamma}}\frac{1}{\Norm_K(\fd)^{-1}} \sum_{\fa}\chi_{\fd}(\fa)H_{\gamma}\left(1, \frac{\Norm_K(\fd)^{2}\Norm_K(\fa)}{\Norm_K(\partial)^{1 - \alpha}\abs{D_K}^{1 - \alpha}}\right)\\
&= {\kappa \abs{D_K} \sp^{-\sk/2}} \sum_{u\in\cO_K^*}\sum_{\tau\in L(u)}\left(\sum_{\fd\mid S_{\gamma}}\frac{1}{\Norm_K(\fd)} \sum_{\fa} \frac{\chi_{\fd}{(\fa)}}{\Norm_K(\fa)}\Phi_{\pm}(\tau, u) + \sum_{u\in\cO_K^*}\sum_{\fd\mid S_{\gamma}}\Norm_K(\fd) \sum_{\fa}\chi_{\fd}(\fa)\Psi_{\pm}(\tau, u)\right)\\
&= \kappa \abs{D_K} \sp^{-\sk/2} \sum_{u\in\cO_K^*}\sum_{\tau\in L(u)}\sum_{\fd\mid S_{\gamma}}\left(\frac{1}{\Norm_K(\fd)} \sum_{\fa} \frac{\chi_{\fd}{(\fa)}}{\Norm_K(\fa)}\Phi_{\pm}(\tau, u)
 + \Norm_K(\fd) \sum_{\fa}\chi_{\fd}(\fa)\Psi_{\pm}(\tau, u)\right)\\
&= \kappa \abs{D_K} \sp^{-\sk/2} \sum_{u\in\cO_K^*}\sum_{\tau\in L(u)}\sum_{\fd\mid S_{\gamma}}\sum_{\fa}\chi_{\fd}(\fa)\left(\frac{\Phi_{\pm}(\tau, u)}{\Norm_K(\fd)\Norm_K(\fa)} + \Norm_K(\fd)\Psi_{\pm}(\tau, u)\right).
\end{split}
\label{eqn: RU in terms of PhiU and PsiU}
\end{equation}
\end{footnotesize}

\section{The Problem of Completion}
\label{section: the problem of completion}

In this section, we perform two main tasks.
First, we replace the sum over ideals $\fd \mid S_\gamma$ arising in the expression of the regular elliptic part by a congruence condition.
Second, we introduce a modified Hilbert symbol which correctly detects the character $\chi_{\fd}(-)$ in our case and clarifies the dependence on $\tau$.
These two steps will then allow us to `complete the lattice' and set the stage for Poisson summation.
 
\subsection{The Congruence Conditions}
\label{section: congruence condition}

So far we have dealt with the condition $\fd\mid S_{\gamma}$.
However, to perform Poisson summation over the variables $u, \tau$ we need to be explicit regarding its role.
We now aim to replace this divisibility condition by a congruence condition as was done by Altu\u{g}.
The main result of this section is stated below.

Let $\fp_1,\dots,\fp_m$ be the (possibly repeated) primes above $2$ in $\cO_K$, that is,
\[
2\cO_K = \fp_1\cdots\fp_m.
\]
We assume $\pi_{\fp_i}$ is a uniformizer.

\begin{theorem}
\label{prop 9 esp26}
Let $\gamma$ be a contributing conjugacy class.
The following two conditions are equivalent
\begin{enumerate}[label = \textup{(\roman*)}]
 \item $\fd\mid S_{\gamma}$,
\item $\fd^2\mid \tau^2 - 4\det(\gamma)$ and for each $i = 1, \dots, m$ the following congruence condition holds in $\cO_{K_{\fp_i}}$,
\[
\frac{\tau^2 - 4\det(\gamma)}{\pi_{\fp_i}^{2\val_{\fp_i}(\fd)}} \equiv \square \pmod{\pi_{\fp_i}^2}.
\]
\end{enumerate}
\end{theorem}

\begin{proof}
This is \cite[Proposition~9]{malors21}.
\end{proof}

Using this result we have
\[
R(f)
= \kappa \abs{D_K} \sp^{-\sk/2} \sum_{u\in \cO_K^*}\sum_{\tau\in L(u)}\sum_{\fd^2\mid \tau^2-4u\varepsilon}^{'}\sum_{\fa}\chi_{\fd}(\fa)\left(\frac{\Phi_{\pm}(\tau,u)}{\Norm_K(\fd)\Norm_K(\fa)} + \Norm_K(\fd)\Psi_{\pm}(\tau,u)\right),
\]
where following Altu\u{g} we use $'$ to denote the congruence conditions given by Theorem~\ref{prop 9 esp26}.

\subsection{The Normalized Character}
\label{subsection: The Normalized Character}

The primary goal of this section is to clarify the dependence of $\chi_{\fd}$ on $\tau$.

\begin{remark}
Over $\Q$, the ideals can be replaced by rational elements and the Kronecker symbol satisfies
\[
\chi_{d}(a) = \binom{\frac{\tau^2 - 4p^{\sk}}{d^2}}{a}_{K},
\]
making the following step unnecessary; see \cite[(4)]{AliI}.
However, we want to provide a uniform proof and obtain an expression similar to the one above for $\chi_{\fd}$.
\end{remark}

\subsubsection{}
We begin by recalling the definition of the Hilbert Symbol in our context.

\begin{definition}
\label{defi: restricted Hilbert}
Let $K$ be a number field and let $\fq$ be a prime of $K$.
The \emph{Hilbert Symbol} of $K_{\fq}$ is defined in $K_{\fq}^*\times K_{\fq}^*$ given by 
\[
 \binom{a, b}{\fq}_H = 
 \left\{
	\begin{array}{ll}
		1 & \text{ if } z^2 = ax^2 + by^2 \text{ has solutions in } \cO_{K_q} \\ 
		-1 & \text{ otherwise}.
	\end{array}
\right.
\]
Fix a uniformizer $\pi_{\fq}$ and define the \emph{restricted Hilbert Symbol} as the function on $\cO_{K_{\fq}}$ given by
\[
\binom{a}{\fq}_{rH} = 
 \left\{
	\begin{array}{ll}
		 \binom{a, \pi_{\fq}}{\fq}_{\tiny{H}}& \text{ if } a\in\cO_{K_{\fq}}^* \\
		0 & \text{ otherwise}
	\end{array}
\right.
\]
and we extend it to $K_{\fq}$ as $0$ for all $a$ with $\abs{a}_{\fq} \neq 1$.
\end{definition}

\begin{remark}
For primes $\fq\nmid 2$, the restricted Hilbert Symbol coincides with the quadratic character of $\cO_{K_{\fq}}^*$, i.e.~the character whose value is $\pm 1$ according as $a$ is a square or not.
It is independent of the choice of the uniformizer $\pi_{\fq}$; see \cite[chapter 4]{Lem_RecLaw}.
For the primes above $2$ this is not the case but the above definition will allow us make explicit the character $\chi_{\fd}$.
\end{remark}

We record the following fact of the Hilbert Symbol that will be useful for us.
We include the proof of (ii) as we were not able to find it written in the literature.

\begin{lemma}\label{periodicity hilbert symbols}
Let $\fq$ be a prime of $K$ and $a_1, a_2$ elements of $\cO_{K_{\fq}}^*$.
Then
\begin{enumerate}[label = \textup{(\roman*)}]
\item \label{case 1 q not divides 2 periodicity}If $\fq\nmid 2$ and $a_1\equiv a_2 \pmod{\pi_{\fq}}$, then
\[
\binom{a_1, \pi_{\fq}}{\fq}_H= \binom{a_2, \pi_{\fq}}{\fq}_H.
\]
\item\label{case 2 q divides 2 rami eq} If $\fq\mid 2$, has ramification degree $e_{\fq}$, and $a_1\equiv a_2 \pmod{\pi_{\fq}^{2e_{\fq} + 1}}$, then
\[
\binom{a_1, \pi_{\fq}}{\fq}_H= \binom{a_2, \pi_{\fq}}{\fq}_H.
\]
\end{enumerate}
\end{lemma}

\begin{proof}
Both these assertions are applications of Hensel's Lemma.
We show (ii) and leave (i) as an exercise for the reader.
If $\fq$ ramifies, then Hensel's lemma asserts that $u\in \cO_{K_{\fq}}^*$ is a square if and only if $u$ is a square modulo $\pi^{2e_{\fq}+1}_{\fq}$.

If $\binom{a_1, \pi_{\fq}}{\fq}_H = 1$ then there exist $X_0, Y_0, Z_0$ (not all equal to 0) satisfying
\[
Z_0^2 = a_1X_0^2 + \pi_{\fq}Y_0^2.
\]
Using an infinite descent argument, we assume without loss of generality that $Z_0, X_0$ are in $\cO_{K_{\fq}}^*$.
In particular,
\[
a_2X_0^2 + \pi_{\fq}Y_0^2 \equiv a_1X_0^2 + \pi_{\fq}Y_0^2 \equiv Z_0^2\pmod{ \pi_{\fq}^{2e_{\fq}+1}},
\]
where the first congruence follows from the assumptions.
In view of Hensel's Lemma, we deduce that $a_2X_0^2 + \pi_{\fq}Y_0^2$ is a square; i.e.~there exists $Z_1\in \cO_{K_{\fq}}^*$ with
\[
Z_1^2 = a_2X_0^2 + \pi_{\fq}Y_0^2.
\]
Thus, the Hilbert symbol on the right is also $1$.
Both sides play symmetric roles so we deduce that both sides are either $\pm 1$.
This concludes the proof.
\end{proof}

\subsubsection{}
Now, we define a symbol which allows us to recover the characters $\chi_{\fd}$.
It also clarifies the role of $\tau$ and $u$.

\begin{definition}
\label{defi: our symbol}
Let $\fa, \mathfrak{b}$ be any ideals of $\cO_K$ and let $\alpha$ be an element of $K$.
For each prime ideal $\fq$ of $K$, let $\pi_{\fq}$ be a uniformizer of $K_{\fq}$.
Define the \emph{modified Hilbert symbol}, with respect to the $\pi_{\fq}$, as
\[
 \binom{\alpha, \fa}{\mathfrak{b}} = \prod_{\fq\mid \mathfrak{b}}\binom{\alpha \pi_{\fq}^{-2\val_{\fq}(\fa)}}{\fq}^{\val_{\fq}(\mathfrak{b})}_{rH}.
\]
\end{definition}

\begin{remark}
 The modified Hilbert symbol does not depend on the choice of the uniformizer because from the definition of the Artin symbol 
\[
 \binom{\alpha, \pi_{\fq}}{\fq}\sqrt{\alpha} = \left(\pi_{\fq}, K_{\fq}(\gamma)/K_{\fq}\right)_A\sqrt{\alpha}.
\]
and the Artin Symbol does not depend on uniformizers.
\end{remark}

Now we prove a crucial property of the modified symbol:
it recovers the quadratic Hecke character at the primes.

\begin{lemma}
\label{lemma: to check}
For $u\in \cO_K^*$ and $\tau\in L(u)$, let $\gamma$ be the regular elliptic class corresponding to the pair $(\tau,u)$.
Let $\fq$ be any prime of $\cO_K$.
Then,
\[
 \chi_{\gamma}(\fq) 
 = \binom{\tau^2 - 4u\varepsilon, S_{\gamma}}{\fq} =\binom{(\tau^2 - 4u\varepsilon)\pi_{\fq}^{-2\val_{\fq}(S_{\gamma})}}{\fq}_{rH}.
\]
\end{lemma}

\begin{proof}
First, consider the case that $\fq\nmid 2$.
Since
\[
(\tau^2 - 4u\det(\gamma)) = (\tau^2 - 4u\varepsilon) = S_{\gamma}^2\Delta_{\gamma},
\]
we deduce
\[
\val_{\fq}(\tau^2 - 4u\varepsilon) - 2\val_{\fq}(S_{\gamma}) = \val_{\fq}(\Delta_{\gamma}).
\]
The upper entry of the modified Hilbert symbol has $\fq$-valuation equal to that of $\Delta_{\gamma}$.

Note that if $\fq$ is a ramified prime, then both sides of the equality (we wish to prove) are $0$.
If $\fq$ is not ramified, then $\alpha:= (\tau^2 - 4u\varepsilon){\pi_{\fq}^{-2\val_{\fq}(S_{\gamma})}}$ is an element of $\cO^*_{K_{\fq}}$.
By Hensel's Lemma (applied to the polynomial $X^2 - \alpha$) we know that $\alpha$ is a square in $\cO_{K_\fq}$ if and only if $\alpha$ is a square modulo $\fq$.
Recall that $\alpha$ is a square in $\cO_{K_\fq}$ precisely when the extension $K(\sqrt{\alpha})/K$ is either degree $1$ or $2$; i.e.~$\fq$ is either split or inert.
On the other hand, the condition $\alpha$ is a square (resp.~non-square) modulo $\fq$ is equivalent to the symbol being $1$ (resp.~$-1$) by definition.
This concludes the proof for odd primes.

Now, suppose that $\fq\mid 2$.
If $\fq$ is ramified, the proof is the same as before.
So, we only consider the case that $\fq\mid 2$ and is unramified.
It follows from the definition of the Artin symbol that
\[
 \binom{\alpha, \pi_{\fq}}{\fq}\sqrt{\alpha} = \left(\pi_{\fq}, K_{\fq}(\gamma)/K_{\fq}\right)_A\sqrt{\alpha}.
\]
Since $\pi_{\fq}$ is a uniformizer, the Artin symbol $\left(\pi_{\fq}, K_{\fq}(\gamma)/K_{\fq}\right)_A$ is the Frobenius element of $\fq$.
This takes the value $1$ or $-1$ according as $\fq$ is split or inert.
This is precisely the definition of $\chi_{\gamma}(\fq)$.
Hence,
\[
\binom{\alpha, \pi_{\fq}}{\fq} = \chi_{\gamma}(\fq).
\]
This completes the proof.
\end{proof}

\begin{theorem}
\label{thm: 7-9}
Let $u$ be an element of $\cO^*_K$ and let $\tau\in L(u)$.
Denote by $\gamma$ the regular elliptic class that corresponds to the pair $(\tau,u)$.
Let $\fd$ be an ideal of $\cO_K$ satisfying $\fd\mid S_{\gamma}$ and let $\fa$ be any ideal of $\cO_K$.
Write
\[
\fa = \prod \fq^{v_{\fq}}, \; \fd = \prod \fq^{r_{\fq}},
\]
where $v_{\fq}$ or $r_{\fq}$ may be equal to 0.
Then,
\[
\chi_{\fd}(\fa) = \binom{\tau^2 - 4u\varepsilon, \fd}{\fa}.
\]
\end{theorem}

\begin{proof}
By definition
\[
 \binom{\tau^2 - 4u\varepsilon, \fd}{\fa} = \prod_{\fq\mid \fa}\binom{(\tau^2 - 4u\varepsilon)\pi_{\fq}^{-2r_{\fq}}}{\fq}^{v_{\fq}}_{rH}.
\]
As before, set $\fd' = S_\gamma/\fd$.
Recall that the definition of $\chi_{\fd}$ is
\[
\chi_{\fd} = \left\{
	\begin{array}{ll}
		\chi_{\gamma}(\fa) & \mbox{if } (\fa, \fd') = 1 \\
		0 & \mbox{if } (\fa, \fd') \neq 1.
	\end{array}
\right.
\]
Hence, we divide the proof into two cases depending on $(\fa, \fd')$.

\underline{\emph{Case 1:}}
If $(\fa, \fd')\neq 1$ then $\chi_{\fd}(\fa) = 0$ by definition and for some $\fq$ both $r_{\fq}$ and $v_{\fq}$ are non-zero.
Hence
\[
\val_{\fq}((\tau^2 - 4u\varepsilon)\pi_{\fq}^{-2r_\fq}) = \val_{\fq}(\fd'\Delta_{\gamma})\ge 1.
\]
Hence, $(\tau^2 - 4u\varepsilon)\pi_{\fq}^{-2r_\fq}$ is not a unit in $\cO_{K_{\fq}}$; and since $v_{\fq}\ge 1$
\[
 \binom{(\tau^2 - 4u\varepsilon)\pi_{\fq}^{-2r_{\fq}}}{\fq}^{v_{\fq}}_{rH} = 0.
\]
By the definition of the modified Hilbert symbol, we have
\[
 \binom{\tau^2 - 4u\varepsilon, \fd}{\fa} = 0.
\]

\underline{\emph{Case 2:}} Let us now suppose $(\fa, \fd') = 1$ and pick a prime $\fq\mid \fa$.
When $\fq$ is ramified, i.e.~$\fq\mid \Delta_{\gamma}$,
\[
\binom{\tau^2 - 4u\varepsilon, \fd}{\fa} = 0.
\]
This is because
\[
\val_{\fq}((\tau^2 - 4u\varepsilon)\pi_{\fq}^{-2r_\fq}) = \val_{\fq}(\fd'\Delta_{\gamma})\ge 1.
\]
On the other hand, since a ramified prime divides $\fa$,
\[
\chi_{\fd}(\fa) = \chi_{\gamma}(\fa) = 0.
\]
Thus we are left to consider the case in which all the primes that divide $\fa$ are not ramified.
In this situation, $(\tau^2 - 4u\varepsilon)\pi_{\fq}^{-2r_\fq}$ is a unit for each $\fq\mid \fa$.
We have
\[
\val_{\fq}(S_{\gamma}) = \val_{\fq}(\fd) + \val_{\fq}(\fd') = \val_{\fq}(\fd) = r_\fq.
\]
Hence, 
\begin{align*}
\binom{\tau^2 - 4u\varepsilon, \fd}{\fa} 
 &= \prod_{\fq\mid \fa}\binom{(\tau^2 - 4u\varepsilon)\pi_{\fq}^{-2r_{\fq}}}{\fq}_{rH}^{v_{\fq}}\\
 &= \prod_{\fq\mid \fa}\binom{(\tau^2 - 4u\varepsilon)\pi_{\fq}^{-2\val_{\fq}(S_{\gamma})}}{\fq}_{rH}^{v_{\fq}}\\
 &= \prod_{\fq\mid \fa}\chi_{\gamma} (\fq)^{v_{\fq}} \quad \text{by Lemma}~\ref{lemma: to check}\\
 &= \chi_{\gamma}(\fa).\qedhere
\end{align*}
\end{proof}

Theorem~\ref{thm: 7-9} together with the congruence conditions from \S\ref{section: congruence condition} will allow us to clarify how every term of the trace formula depends on $\tau$.
We now prove that the modified Hilbert symbol is periodic.

\begin{lemma}
\label{prop: periodicity 4ad2}
Let $\tau_1\equiv\tau_2 \pmod{4\fa\fd^2}$ in $\cO_K$.
Then
\[
\binom{\tau_1^2 - 4u\varepsilon, \fd}{\fa} = \binom{\tau_2^2 - 4u\varepsilon, \fd}{\fa}.
\]
 \end{lemma}

\begin{proof}
Let $\fq\mid \fa$ and set $r = \val_{\fq}(\fd)$.
Let $\pi_{\fq}$ be a uniformizer of $K_\fq$.
By definition of the modified Hilbert symbol, it is enough to prove
\[
\binom{(\tau_1^2 - 4u\varepsilon)\pi_{\fq}^{-2r}}{\fq} = \binom{(\tau_2^2 - 4u\varepsilon)\pi_{\fq}^{-2r}}{\fq}.
\]
Notice that $\pi_{\fq}^{2r}\mid \tau_1^2 - \tau_2^2$.
Hence $(\tau_1^2 - 4u\varepsilon)\pi_{\fq}^{-2r}$ is an integer if and only if $(\tau_2^2 - 4u\varepsilon)\pi_{\fq}^{-2r}$ is an integer.
Since $\fa\mid \tau_1 - \tau_2$, we get an extra factor $\pi_{\fq}$ than those than come from $\fd^2$.
Whence
\[
 \pi_{\fq}^{2r + 1}\mid \tau_1^2 - \tau_2^2 = (\tau_1^2 - 4u\varepsilon) - (\tau_2^2 - 4u\varepsilon).
\]
So, in case both `numerators' of the symbols are integers, we have
\[
(\tau_1^2 - 4u\varepsilon)\pi_{\fq}^{-2r} \equiv (\tau_1^2 - 4u\varepsilon)\pi_{\fq}^{-2r} \pmod{\pi_{\fq}}.
\]
 
Lemma~\ref{periodicity hilbert symbols}\ref{case 1 q not divides 2 periodicity} asserts that if $\fq\nmid 2$, then 
\[
\binom{(\tau_1^2 - 4u\varepsilon)\pi_{\fq}^{-2r}}{\fq} = \binom{(\tau_2^2 - 4u\varepsilon)\pi_{\fq}^{-2r}}{\fq}.
\]
On the other hand, if $\fq\mid 2$, then 
\[
\val_{\fq}(4\fa\fd^2) = 2e_{\fq} + \val_{\fq}(\fa) + r\ge 2e_{\fq} + 1, 
\]
where $e_{\fq}$ is the ramification of $\fq$ above $2$.
Equality can occur if $r = 0$ and $\val_{\fq}(\fa) = 1$.
Hence
\[
\pi_{\fq}^{2e_{\fq} + 1}\mid (\tau_1^2 - 4u\varepsilon)\pi_{\fq}^{-2r} - (\tau_2^2 - 4u\varepsilon)\pi_{\fq}^{-2r}.
\]
By Lemma~\ref{periodicity hilbert symbols}\ref{case 2 q divides 2 rami eq} we conclude the proof.
\end{proof}

Next, we prove the periodicity of the congruence conditions.
\begin{lemma}
\label{prop: periodicities}
Let $\fa$ and $\fd$ be ideals in $\cO_K$.
If $\tau_1\equiv \tau_2 \pmod{4\fa\fd^2}$ then 
\[
\fd^2 \mid \tau_1^2-4u\varepsilon \Longleftrightarrow \fd^2 \mid \tau_2^2-4u\varepsilon.
\]
Moreover, if $\fd^2 \mid \tau_i^2-4u\varepsilon$, then
\[
\frac{\tau_1^2-4u\varepsilon}{\fd^2}\equiv \frac{\tau_2^2-4u\varepsilon}{\fd^2} \pmod{4}.
\]
\end{lemma}

\begin{proof}
Recall that $4\fd^2\mid \tau_1^2 - \tau_2^2$.
Therefore
\[
4\fd^2\mid (\tau_1^2 - 4u\varepsilon) - (\tau_2^2 - 4u\varepsilon).
\]
The result now follows.
\end{proof}

\subsection{Completion of the Lattice}

Using the modified Hilbert symbol, we can rewrite \eqref{eqn: RU in terms of PhiU and PsiU} as
\begin{small}
\[
R(f)
= \kappa \abs{D_K} \sp^{-\sk/2}\sum_{u\in \cO_K^*}\sum_{\tau\in L(u)}\sum_{\fd^2\mid \tau^2-4u\varepsilon}^{'}\sum_{\fa}\binom{\tau^2 - 4u\varepsilon, \fd}{\fa}\left(\frac{\Phi_{\pm}(\tau, u)}{\Norm_K(\fd)\Norm_K(\fa)} + \Norm_K(\fd)\Psi_{\pm}(\tau, u)\right).
\]
\end{small}
Notice that the expression
\[
\binom{\tau^2 - 4u\varepsilon, \fd}{\fa}\left(\frac{\Phi_{\pm}(\tau, u)}{\Norm_K(\fd)\Norm_K(\fa)} + \Norm_K(\fd)\Psi_{\pm}(\tau, u)\right)
\]
makes sense for any value of $\tau, u, \fd, \fa$ regardless of whether there is a regular elliptic matrix $\gamma$ (corresponding to $\tau$ and $u$) whose quadratic character produces the divisors $\fd\mid S_{\gamma}$.
In particular, we can consider
\[
\kappa \abs{D_K} \sp^{-\sk/2}\sum_{u\in \cO_K^*}\sum_{\tau\notin L(u)}\sum_{\fd^2\mid \tau^2-4u\varepsilon}^{'}\sum_{\fa}\binom{\tau^2 - 4u\varepsilon, \fd}{\fa}\left(\frac{\Phi_{\pm}(\tau, u)}{\Norm_K(\fd)\Norm_K(\fa)} + \Norm_K(\fd)\Psi_{\pm}(\tau, u)\right).
\]
Introducing the congruence condition in place of the condition $\fd\mid S_{\gamma}$ allowed us to replace the $\fd$-sum in \eqref{eqn: RU in terms of PhiU and PsiU} with a sum whose condition depends on $\tau$ and $u$ (but not on whether there is associated to such $\fd$ a quadratic extension which determines $S_{\gamma}$).

\begin{definition}
Define
\begin{small}
\begin{align*}
\Sigma(\square) &:= \kappa \abs{D_K} \sp^{-\sk/2} \sum_{u\in \cO_K^*}\sum_{\tau\notin L(u)}\sum_{\fd^2\mid \tau^2-4u\varepsilon}^{'}\sum_{\fa}\binom{\tau^2 - 4u\varepsilon, \fd}{\fa}\left(\frac{\Phi_{\pm}(\tau, u)}{\Norm_K(\fd)\Norm_K(\fa)} + \Norm_K(\fd)\Psi_{\pm}(\tau, u)\right).
\end{align*}
\end{small}
\end{definition}

For this definition to make sense, we need to prove the following theorem.

\begin{theorem}
\label{thm: absolute convergence}
The series defining $\Sigma(\square)$ is an absolutely convergent sum.
\end{theorem}

\begin{proof}
The absolute value of the character is bounded by $1$.
Thus, what we must prove is that
\begin{small}
\[
\sum_{u\in \cO_K^*}\sum_{\tau\notin L(u)}\sum_{\fd^2\mid \tau^2-4u\varepsilon}^{'}\sum_{\fa}\frac{\abs{\Phi_{\pm}(\tau, u)}}{\Norm_K(\fd)\Norm_K(\fa)} = \sum_{u\in \cO_K^*}\sum_{\tau\notin L(u)}\abs{\theta^{\pm}(\tau, u)}\sum_{\fd^2\mid \tau^2-4u\varepsilon}^{'}\sum_{\fa}\frac{\abs{F\left(\frac{\Norm_K(\fd)^2\Norm_K(\fa)}{\Norm_K(\partial)^{\alpha}\abs{D_K}^\alpha} \right)}}{\Norm_K(\fd)\Norm_K(\fa)}
\]
\end{small}
and
\begin{small}
\begin{multline*}
\sum_{u\in \cO_K^*}\sum_{\tau\notin L(u)}\sum_{\fd^2\mid \tau^2-4u\varepsilon}^{'}\sum_{\fa} \Norm_K(\fd)\abs{\Psi_{\pm}(\tau, u)} 
= \\
\sum_{u\in \cO_K^*}\sum_{\tau\notin L(u)}\abs{\frac{\theta^{\pm}(\tau, u)}{\abs{D_K}^{1/2}\Norm_K(\partial)^{1/2}}}\sum_{\fd^2\mid \tau^2-4u\varepsilon}^{'}\sum_{\fa} \Norm_K(\fd)\abs{H_{\gamma}\left(1,\frac{\Norm_K(\fd)^2\Norm_K(\fa)}{\Norm_K(\partial)^{1-\alpha}\abs{D_K}^{1 - \alpha}}\right)},
\end{multline*}
\end{small}
\normalsize
are both convergent.
If we prove the inner sums over $\fd$ and $\fa$ are absolutely convergent we are done, because the outer sums over $u$ and $\tau$ are finite (since $\theta^{\pm}$ are of compact support).
Thus they contribute a finite nonzero terms.

Let us study the interior sums separately.
We know that for $\fd, \fa$ large enough
\[
\abs{F\left(\frac{\Norm_K(\fd)^2\Norm_K(\fa)}{\Norm_K(\partial)^{\alpha}\abs{D_K}^\alpha} \right)} \ll \exp\left(-\frac{\Norm_K(\fd)^2\Norm_K(\fa)}{\Norm_K(\partial)^{\alpha}\abs{D_K}^\alpha}\right) \ll \left(\frac{\Norm_K(\fd)^2\Norm_K(\fa)}{\Norm_K(\partial)^{\alpha}\abs{D_K}^\alpha}\right)^{-1}.
\]
Thus 
\begin{small}
\[
\sum_{\fd^2\mid \tau^2-4u\varepsilon}^{'}\sum_{\fa}\frac{\abs{F\left(\frac{\Norm_K(\fd)^2\Norm_K(\fa)}{\Norm_K(\partial)^{\alpha}\abs{D_K}^\alpha} \right)}}{\Norm_K(\fd)\Norm_K(\fa)} 
\ll 
\sum_{\fd^2\mid \tau^2-4u\varepsilon}^{'}\sum_{\fa} \dfrac{1}{\Norm_K(\fd)^3\Norm_K(\fa)^2} < \infty.
\]
\end{small}
Using Lemma~\ref{lemma:boundH}, we have
\begin{small}
\[
\Norm_K(\fd)\abs{H_{\gamma}\left(1,\frac{\Norm_K(\fd)^2\Norm_K(\fa)}{\Norm_K(\partial)^{1-\alpha}\abs{D_K}^{1 - \alpha}}\right)} 
\ll 
\frac{\Norm_K(\fd)}{\frac{\Norm_K(\fd)^2\Norm_K(\fa)}{\Norm_K(\partial)^{1-\alpha}\abs{D_K}^{1 - \alpha}}}\exp\left(-\left(\frac{\Norm_K(\fd)^2\Norm_K(\fa)}{\Norm_K(\partial)^{1-\alpha}\abs{D_K}^{1 - \alpha}}\right)^{1/n+1}\right)
\]
\end{small}
For large enough $\Norm_K(\fa), \Norm_K(\fd)$ the above expression can be bounded by
\[
\ll \frac{1}{\Norm_K(\fd)\Norm_K(\fa)\left( \Norm_K(\fd)^2\Norm_K(\fa)\right)^{1/n+1}},
\]
where the implied constant only depends on $\tau, u$, and by the finiteness of their contribution, it can be taken to be absolute.
This completes the proof.
\end{proof}

\begin{remark}
The proof works exactly the same way if we instead use $\tau\in L(u)$, even though we already know that that sum converges absolutely by the approximate functional equation (since it comes from the Zagier Zeta Function which is entire).
We cannot use the approximate functional equation directly for the cases of $\Sigma(\square)$ because the character in that case returns the zeta function, and thus has poles itself.
\end{remark}

\begin{definition}
We define the \emph{completed regular elliptic part} of the trace formula as
\begin{footnotesize}
\begin{align*}
\overline{R}(f) &= \kappa \abs{D_K} \sp^{-\sk/2} \sum_{u\in \cO_K^*}\sum_{\tau\in\cO_K}\sum_{\fd^2\mid \tau^2-4u\varepsilon}^{'}\sum_{\fa}\binom{\tau^2 - 4u\varepsilon, \fd}{\fa}\left(\frac{\Phi_{\pm}(\tau, u)}{\Norm_K(\fd)\Norm_K(\fa)} + \Norm_K(\fd)\Psi_{\pm}(\tau, u)\right) \\
& = R(f) + \Sigma(\square). \notag
\end{align*}
\end{footnotesize}
\end{definition}

\begin{remark}
The name comes from the fact that we have added those $\tau\in\cO_K$ that were missing from the sum.
We have \textit{completed} the lattice of integers.
\end{remark}

In the following result we show that the order of summation in the completed regular elliptic part can be swapped.

\begin{proposition}
\label{prop 7.16}
The order of summation in the completed regular elliptic part can be swapped.
More precisely,
\begin{small}
\[
\overline{R}(f)= \kappa \abs{D_K} \sp^{-\sk/2} 
\sum_{\fa}\sum_{\fd} \sum_{u\in \cO_K^*}
\sum_{\fd^2\mid \tau^2-4u\varepsilon}^{'}
\binom{\tau^2 - 4u\varepsilon, \fd}{\fa}\left(\frac{\Phi_{\pm}(\tau, u)}{\Norm_K(\fd)\Norm_K(\fa)} + \Norm_K(\fd)\Psi_{\pm}(\tau, u)\right).
\]
\end{small}
\end{proposition}

\begin{proof} 
We have defined $\overline{R}(f)$ as the sum of two absolutely convergent series.
Thus we can arrange the series in any alternative order without altering the convergence properties or values.
\end{proof}

\begin{remark}
In the expression defining $\overline{R}(f)$, we see that there are no conditions on $\tau$ and $\fd$ runs over the possible ideals such that $\fd^2\mid \tau^2 - 4u\varepsilon$ and satisfies certain congruence conditions.
After changing the order of summation, $\fd$ affords no conditions and the innermost sum is now over the $\tau$ such that $\tau^2 - 4u\varepsilon$ has $\fd^2$ as divisor, with the appropriate congruence conditions.
Both sums look the same but it is pertinent to note that they have different free and dependent variables.
\end{remark}

\section{Poisson Summation}

\subsection{Explicitly writing the lattice}
Recall from \S\ref{section: blending step} that we fixed the fundamental units $\beta_1,\dots,\beta_{n-1}$ in $\cO_K^*$ (which are embedded into $W$ under the function $S$) satisfying \ref{ass: positivity}.
Any unit $u\in \cO_K^*$ can be written as
\[
u=\pm\beta_1^{y_1}\cdots\beta_{n-1}^{y_{n-1}}.
\]
In particular, $u = u(\pm, y)$ where $y=(y_1, \ldots, y_{n-1})\in\Z^{n-1}$.

\begin{notation}
In what follows, we write $u$ in place of $u(\pm,y)$ for a unit in $\cO_K^*$.
\end{notation}

The expression obtained in Proposition~\ref{prop 7.16} can be rewritten as
\begin{small}
\begin{equation}
\label{eqn: explicit CRE part}
\overline{R}(f) = \kappa \abs{D_K} \sp^{-\sk/2} \sum_{\pm}\sum_{\fa}\sum_{\fd} \sum_{y\in \mathbb{Z}^{n-1}}\sum_{\fd^2\mid \tau^2-4u\varepsilon}^{'}\binom{\tau^2 - 4u\varepsilon, \fd}{\fa}\left(\frac{\Phi_{\pm}(\tau, u)}{\Norm_K(\fd)\Norm_K(\fa)} + \Norm_K(\fd)\Psi_{\pm}(\tau, u)\right).
\end{equation}
\end{small}
Given ideals $\fd$ and $\fa$ in $\cO_K$, the periodicity of the congruence conditions proved in Lemma~\ref{prop: periodicities} implies that for a given congruence class modulo $4\fa\fd^2$, either every term of the arithmetic progression appears in the sum or no term appears.
Therefore, the expression \eqref{eqn: explicit CRE part} can be rewritten as
\begin{small}
\[
\kappa \abs{D_K} \sp^{-\sk/2} \sum_{\pm}\sum_{\fa}\sum_{\fd}\sum_{y\in\mathbb{Z}^{n-1}}\sum_{\mu\in \cO_K/4\fa\fd^2}^{'}\sum_{\tau\equiv\mu} \binom{\tau^2 - 4u\varepsilon, \fd}{\fa}\left(\frac{\Phi_{\pm}(\tau, u)}{\Norm_K(\fd)\Norm_K(\fa)} + \Norm_K(\fd)\Psi_{\pm}(\tau, u)\right).
\]
\end{small}
By Lemma~\ref{prop: periodicities} we know that the symbol is periodic so the above expression can be further rewritten as
\begin{small}
\[
\kappa \abs{D_K} \sp^{-\sk/2} \sum_{\pm}\sum_{\fa}\sum_{\fd}\sum_{y\in\mathbb{Z}^{n-1}}\sum_{\mu\in \cO_K/4\fa\fd^2}^{'}\binom{\mu^2 - 4u\varepsilon, \fd}{\fa}\sum_{\tau\equiv\mu} \left(\frac{\Phi_{\pm}(\tau, u)}{\Norm_K(\fd)\Norm_K(\fa)} + \Norm_K(\fd)\Psi_{\pm}(\tau, u)\right).
\]
\end{small}
Within the lattice $\mathbb{Z}^{n-1}$ there is the sub-lattice $(2\mathbb{Z})^{n-1}$ and the quotient is $\mathbb{F}_2^{n-1}$, where $\mathbb{F}_2$ is the field of two elements.
Let $Y$ be a variable that runs over this quotient.
We can rewrite the above expression as
\begin{footnotesize}
\[
\kappa \abs{D_K} \sp^{-\sk/2} \sum_{\pm}\sum_{\fa}\sum_{\fd}\sum_{Y\in\mathbb{F}_2^{n-1}}\sum_{y\equiv Y}\sum_{\mu\in \cO_K/4\fa\fd^2}^{'}\binom{\mu^2 - 4u\varepsilon, \fd}{\fa}\sum_{\tau\equiv\mu} \left(\frac{\Phi_{\pm}(\tau, u)}{\Norm_K(\fd)\Norm_K(\fa)} + \Norm_K(\fd)\Psi_{\pm}(\tau, u)\right).
\]
\end{footnotesize}
Fix the sign $\pm$ and a vector $Y=(Y_1, \dots, Y_{n-1})\in\mathbb{F}_2^{n-1}$.
If $y\in \Z^{n-1}$ satisfies $y\equiv Y$, there exists another unit $v\in \cO_K^*$ such that $u=u(e, y) = u(Y)v^2$ with
\[
u(Y) = \beta_1^{Y_1} \cdots \beta_{n-1}^{Y_{n-1}},
\]
where $e$ is the fixed sign.
This unit $v$ is uniquely determined by $Y$ and $y$, if we impose the extra condition that it lies in $W$, which is to say, that the two units have the same sign.
The condition $y\equiv Y$ implies the existence of $y' = (y'_1,\dots, y'_{n-1})\in \Z^{n-1}$ such that for each $1\leq t \leq n-1$,
\[
y_t = Y_t + 2y'_t.
\]
This vector $y'\in \Z^{n-1}$ is in one-to-one correspondence with $v=\beta_1^{y'_1}\cdots\beta_{n-1}^{y'_{n-1}}$.
In other words, the unit $v$ can be associated with an $(n-1)$-dimensional vector given by
\[
\frac{1}{2}\left(y_1 - Y_1,\dots, y_{n-1}-Y_{n-1}\right).
\]
We can then rewrite the completed regular elliptic part as 
\begin{footnotesize}
\[
\kappa \abs{D_K} \sp^{-\sk/2} \sum_{\pm}\sum_{\fa}\sum_{\fd}\sum_{Y\in\mathbb{F}_2^{n-1}}\sum_{y\equiv Y}\sum_{\mu\in \cO_K/4\fa\fd^2}^{'}\binom{\mu^2 - 4 {u(Y)}v^2\varepsilon, \fd}{\fa}\sum_{\tau\equiv\mu} \left(\frac{\Phi_{\pm}(\tau, u)}{\Norm_K(\fd)\Norm_K(\fa)} + \Norm_K(\fd)\Psi_{\pm}(\tau, u)\right).
\]
\end{footnotesize}
Since
\[
\binom{\mu^2 - 4 {u(Y)}v^2\varepsilon, \fd}{\fa} = \binom{v^2(\mu v^{-1})^2 - 4 {u(Y)}v^2\varepsilon, \fd}{\fa} = \binom{(\mu v^{-1})^2 - 4 {u(Y)}\varepsilon, \fd}{\fa},
\]
the completed regular elliptic part is
\begin{footnotesize}
\[
\kappa \abs{D_K} \sp^{-\sk/2} \sum_{\pm}\sum_{\fa}\sum_{\fd}\sum_{Y\in\mathbb{F}_2^{n-1}}\sum_{y\equiv Y}\sum_{\mu\in \cO_K/4\fa\fd^2}^{'}\binom{(\mu v^{-1})^2 - 4 {u(Y)}\varepsilon, \fd}{\fa}\sum_{\tau\equiv\mu} \left(\frac{\Phi_{\pm}(\tau, u)}{\Norm_K(\fd)\Norm_K(\fa)} + \Norm_K(\fd)\Psi_{\pm}(\tau, u)\right).
\]
\end{footnotesize}
Since $v$ is a unit, multiplication by it only permutes the classes modulo $4\fa\fd^2$.
Relabeling $\mu v^{-1}$ as $\mu$, we may remove $v^{-1}$ from the symbol at the expense of changing the inner sum to $\tau\equiv \mu v$.
Therefore,
\begin{tiny}
\[
\overline{R}(f) = \kappa \abs{D_K} \sp^{-\sk/2} \sum_{\pm}\sum_{\fa}\sum_{\fd}\sum_{Y\in\mathbb{F}_2^{n-1}}\sum_{y\equiv Y}\sum_{\mu\in \cO_K/4\fa\fd^2}^{'}\binom{\mu^2 - 4 {u(Y)}\varepsilon, \fd}{\fa}\sum_{\tau\equiv\mu v} \left(\frac{\Phi_{\pm}(\tau, u)}{\Norm_K(\fd)\Norm_K(\fa)} + \Norm_K(\fd)\Psi_{\pm}(\tau, u)\right).
\]
\end{tiny}

We can record the main theorem of this section.

\begin{theorem}
\label{thm: regular elliptic part before poisson}
The completed regular elliptic part $\overline{R}(f)$ can be rewritten as
\begin{footnotesize}
\[
\kappa \abs{D_K} \sp^{-\sk/2} \sum_{\pm}\sum_{\fa}\sum_{\fd}\sum_{Y\in\mathbb{F}_2^{n-1}}\sum_{\mu\in \cO_K/4\fa\fd^2}^{'}\binom{\mu^2 - 4 {u(Y)}\varepsilon, \fd}{\fa}\sum_{y\equiv Y}\sum_{\tau\equiv\mu v} \left(\frac{\Phi_{\pm}(\tau, u)}{\Norm_K(\fd)\Norm_K(\fa)} + \Norm_K(\fd)\Psi_{\pm}(\tau, u)\right).
\]
\end{footnotesize}
\end{theorem}

\begin{remark}
When $K = \Q$, the unit group is a finite set - there are two units, namely, the two roots of unity $\pm 1$.
Hence, the sums over $Y$ and $y$ are both empty.
The ideals correspond to the integers in this case.
For each ideal, we always pick the positive representative, and we pick $p$ to be the representative of $\fp$.
Therefore, Theorem~\ref{thm: regular elliptic part before poisson} asserts that the completed regular elliptic part of the trace formula can be written as
\[
\sum_{\pm}\sum_{d}\sum_{a}\sum_{\mu\mod{4ad^2}}\binom{(\mu^2 \pm 4p^k)/d^2}{a}\sum_{\tau\equiv\mu}^{'} \left(\frac{\Phi_u(\tau)}{da} + d\Psi_u(\tau)\right).
\]
This is exactly what appears in \cite[proof of Theorem~4.2]{AliI} upon changing our $a$ for his $l$, our $d$ for his $f$, our $\mu$ for his $a$, and our $\tau$ for his $m$, before performing Poisson summation.
\end{remark}

We must study the indexing set of the two inner sums appearing in the expression in Theorem~\ref{thm: regular elliptic part before poisson} and rewrite them as sums over a lattice.
The following lemma will ensure that the indexing sets form an honest lattice inside $\R^{2n-1}$ which will then allow us to perform the Poisson summation in a future step.

\begin{lemma}
Fix a set of totally positive fundamental units $\beta_1, \dots, \beta_{n-1}$ of $\cO_K^*$, a choice of sign, and a vector $Y=(Y_1, \dots, Y_{n-1})\in \mathbb{F}_2^{n-1}$.
Let $u = u(\pm, Y)$ be the unit corresponding to this choice.
Fix an integral basis $\omega_1,\ldots, \omega_n$ for the ideal $4\fa\fd^2$.
Finally, fix a representative $\mu$ of some class of $\cO_K/4\fa\fd^2$.

For each embedding $\sigma_i:K\hookrightarrow\mathbb{R}$, set $A_i = \sigma_i(u\mu)$ and $\omega_{j, i} = \sigma_i(\omega_j)$.
Then the pairs $(\tau, u(y))$ that appear in Theorem~\ref{thm: regular elliptic part before poisson} are those with coordinates
\begin{multline*}
\left(A_1\beta_{1, 1}^{y_1}\cdots\beta_{n-1, 1}^{y_n-1} + r_1\omega_{1, 1} + \cdots + r_n\omega_{n,1},\ldots, A_n\beta_{1, n}^{y_1}\cdots\beta_{n-1, n}^{y_n-1} + \right. \\ \left. r_1\omega_{1, n} + \cdots + r_n\omega_{n,n}, Y_1 + 2y_1,\dots, Y_{n-1} + 2y_{n-1}\right),
\end{multline*}
for some $r = (r_1,\dots, r_n)\in\Z^n$ and $(y_1,\ldots, y_{n-1})\in\Z^{n-1}$.
\end{lemma}

\begin{proof}
The condition $y\equiv Y$ implies the existence of $y' = (y'_1,\dots, y'_{n-1})\in \Z^{n-1}$ such that for each $1\leq t \leq n-1$,
\[
y_t = Y_t + 2y'_t.
\]
This vector $y'\in \Z^{n-1}$ is in one-to-one correspondence with $v=\beta_1^{y'_1}\cdots\beta_{n-1}^{y'_{n-1}}$.
Therefore, for $\mu\in \cO_K/4\fa\fd^2$, the equivalence $\tau \equiv \mu v$ can be rewritten (component-wise) as
\[
\tau_i = \sigma_i(\beta_1^{y'_1}\cdot\ldots\cdot\beta_{n-1}^{y'_{n-1}}\mu) + z_i,
\]
for some integer $z_i$, where $1\leq i \leq n$.
The result now follows (by relabeling $y_t'$ by $y_t$).
\end{proof}

In view of the above discussion, we make the following definition.

\begin{definition}\label{def: transformation A}
Keep the notation introduced above; except now suppose that $(r,y)\in \R^n\times\R^{n-1}$.
Define
\begin{align*}
C_i(y) &:= \beta_{1, i}^{y_1}\cdots\beta_{n-1, i}^{y_{n-1}} \textrm{ and }\\
A(r, y) &:= \left(A_1C_1(y) + r_1\omega_{1, 1} + \cdots + r_n\omega_{n,1},\ldots, \right. \\
& \left. \hspace{1.8in} A_nC_n(y) + \cdots + r_n\omega_{n,n}, Y_1 + 2y_1,\dots, Y_{n-1} + 2y_{n-1}\right).
\end{align*}
\end{definition}

\begin{proposition}[Change of Variables]
\label{prop: properties A}
Keep the notation introduced above and let $(r,y)\in \R^{n}\times \R^{n-1}$.
Then
\begin{enumerate}[label = \textup{(\roman*)}]
\item \label{exercise invertible} $A(r,y)$ is an invertible function.
\item \label{det Jacobian} The absolute value of the determinant of the Jacobian matrix of $A(r,y)$ is 
\[
2^{3n-1}N(\fa)N(\fd)^2\abs{D_K}^{1/2}.
\]
\end{enumerate}
\end{proposition}

\begin{proof}
\begin{enumerate}[label = \textup{(\roman*)}]
\item One can write explicitly the inverse of this function.
This is left as an exercise for the reader.
\item By computing partial derivatives the reader can check that the Jacobian matrix has the form 
\[
\begin{pmatrix}
\omega_{1, 1} &\cdots & \omega_{n,1} & * & \cdots & *\\
\vdots &\ddots & \vdots & \vdots & \ddots & \vdots\\
\omega_{n, 1} &\cdots & \omega_{n,n} & * & \cdots & *\\
0 & \cdots &0 & 2 & \cdots & 0\\
0 & \cdots &0 & \vdots & \ddots & \vdots\\
0 & \cdots &0 & 0 & \cdots & 2\\
\end{pmatrix}.
\]
\end{enumerate}
Thus, the determinant of the above matrix is 
\[
2^{n-1}\det
\begin{pmatrix}
\omega_{1, 1} &\cdots & \omega_{n,1} \\
\vdots &\ddots & \vdots \\
\omega_{n, 1} &\cdots & \omega_{n,n} \\
\end{pmatrix}
\]
By definition, the determinant of the $n\times n$ matrix is $N(4\fa\fd^2)\abs{D_K}^{1/2} = 4^nN(\fa)N(\fd)^2\abs{D_K}^{1/2}$.
\end{proof}

\begin{notation}
The transformation $A$ depends on $\fa$, $\fd$, $\pm$, $Y$, and $\mu$ but we omit this from the notation.
\end{notation}

We now rewrite the inner two sums of the completed regular elliptic part.

\begin{proposition}
\label{prop: before Poisson with A}
The inner two sums of the expression in Theorem~\ref{thm: regular elliptic part before poisson} are rewritten as
\[
\sum_{y\in\mathbb{Z}^{n-1}}\sum_{r\in \mathbb{Z}^n} \left(\frac{\Phi_{\pm}\left( A(r,y) \right)}{\Norm_K(\fd)\Norm_K(\fa)} + \Norm_K(\fd)\Psi_{\pm}\left(A(r,y) \right)\right).
\]
\end{proposition}

\subsection{Performing Poisson Summation.}

For $(t,s)\in\mathbb{R}^{n}\times\mathbb{R}^{n-1}$, let $G(t,s)$ be a smooth function with compact support.
Set $\widehat{G}$ to denote the \emph{Fourier transform} of $G$, i.e.,
\[
\widehat{G}(\xi, \eta) := \int_{\mathbb{R}^n} \int_{\mathbb{R}^{n-1}}G(t, s) \exp(-2\pi i(t\cdot\xi + s\cdot\eta )) dsdt,
\]
where $(\xi, \eta)\in \R^{n} \times \R^{n-1}$.
Then, Poisson summation is the following assertion
\[
\sum_{(t,s)\in\Z^{n}\times\Z^{n-1}} G(t,s) = \sum_{(\xi, \eta)\in\Z^{n}\times\Z^{n-1}} \widehat{G}(\xi, \eta).
\]

The following important result essentially follows from unraveling the definitions.
\begin{proposition}
\label{hat G expresion}
Let $A$ be the transformation from Definition~\ref{def: transformation A}.
Let $G(r,y) = H(A(r,y))$ for some smooth function $H$ with compact support, then
\[
\widehat{G}(\xi, \eta) = \frac{1}{2^{3n-1}N(\fa)N(\fd)^2\abs{D_K}^{1/2}} \int \int H(t,s)\exp(- 2\pi i(\xi, \eta)\cdot A^{-1}(t,s) )dsdt.
\]
\end{proposition}

\begin{proof}
This follows from definition of Fourier transform and the change of variable formula using Proposition~\ref{prop: properties A}.
\end{proof}

\begin{definition}
For the functions $\Phi_{\pm}$ and $\Psi_{\pm}$ introduced in \eqref{defn: Phi} and \eqref{defn: Psi}, define the \emph{modified Fourier transforms}
\begin{align*}
\Phi_{\pm}^1(\xi, \eta)&:= \int \int \Phi_{\pm}(t,s)\exp(- 2\pi i (\xi, \eta)\cdot A^{-1}(t,s) )ds dt,\\
\Psi_{\pm}^1(\xi, \eta)&:= \int \int \Psi_{\pm}(t,s)\exp(- 2\pi i (\xi, \eta)\cdot A^{-1}(t,s) )ds dt.
\end{align*}
\end{definition}

\begin{remark}
We worked with these modified Fourier transforms because it allowed us to keep track of the constants after Poisson summation more easily.
\end{remark}

Applying Poisson summation, we get the following result.

\begin{theorem}
\label{thm: after Poisson}
The completed regular elliptic part is 
\begin{small}
\begin{multline*}
\frac{\kappa \abs{D_K}^{1/2} \sp^{-\sk/2}}{2^{3n-1} } \sum_{\pm}\sum_{\fa}\sum_{\fd}\sum_{Y\in\mathbb{F}_2^{n-1}}\sum_{\mu\in \cO_K/4\fa\fd^2}^{'}\binom{\mu^2 - 4 {u(Y)}\varepsilon, \fd}{\fa} \\
\sum_{\xi\in \mathbb{Z}^{n}}\sum_{\eta\in\mathbb{Z}^{n-1}} \frac{1}{\Norm_K(\fa)\Norm_K(\fd)^2}\left(\frac{\Phi_{\pm}^{1}(\xi, \eta)}{\Norm_K(\fd)\Norm_K(\fa)} + \Norm_K(\fd)\Psi_{\pm}^{1}(\xi, \eta)\right).
\end{multline*}
\end{small}
\end{theorem}

\begin{proof}
This is a direct consequence of Propositions~\ref{prop: before Poisson with A} and \ref{hat G expresion}.
\end{proof}

Concretely, the Poisson Summation that has been performed is
\begin{align*}
 &\sum_{y\in\mathbb{Z}^{n-1}}\sum_{r\in \mathbb{Z}^n} \left(\frac{\Phi_{\pm}\left( A(r,y) \right)}{\Norm_K(\fd)\Norm_K(\fa)} + \Norm_K(\fd)\Psi_{\pm}\left(A(r,y) \right)\right)\\
 &=
 \sum_{\xi\in \mathbb{Z}^{n}}\sum_{\eta\in\mathbb{Z}^{n-1}} \frac{1}{2^{3n-1}\abs{D_K}^{1/2}\Norm_K(\fa)\Norm_K(\fd)^2}\left(\frac{\Phi_{\pm}^{1}(\xi, \eta)}{\Norm_K(\fd)\Norm_K(\fa)} + \Norm_K(\fd)\Psi_{\pm}^{1}(\xi, \eta)\right).
\end{align*}
Notice that the right hand side of the above equation is absolutely convergent.
Indeed, because $\Phi_{\pm}$ and $\Psi_{\pm}$ are smooth and of compact support, by Proposition~\ref{prop: smoothing}, the bounds 
\[
\abs{\Phi_{\pm}^{1}(\xi, \eta)} \ll \dfrac{1}{(1 + \abs{(\xi, \eta)})^N} \quad \text{
and } \quad 
\abs{\Psi_{\pm}^{1}(\xi, \eta)} \ll \dfrac{1}{(1 + \abs{(\xi, \eta)})^N}, 
\]
hold for any $N$.
If $N$ is big enough these imply the absolute convergence.
In particular, we can rearrange the interior sum on Theorem~\ref{thm: after Poisson} in any way without changing its value.
In other words, the sum of the aforementioned theorem is the sum of
\begin{small}
\begin{multline*}
\frac{\kappa \abs{D_K}^{1/2} \sp^{-\sk/2}}{2^{3n-1} } \sum_{\pm}\sum_{\fa}\sum_{\fd}\sum_{Y\in\mathbb{F}_2^{n-1}}\sum_{\mu\in \cO_K/4\fa\fd^2}^{'}\binom{\mu^2 - 4 {u(Y)}\varepsilon, \fd}{\fa} \\
 \frac{1}{\Norm_K(\fa)\Norm_K(\fd)^2}\left(\frac{\Phi_{\pm}^{1}(0, 0)}{\Norm_K(\fd)\Norm_K(\fa)} + \Norm_K(\fd)\Psi_{\pm}^{1}(0, 0)\right).
\end{multline*}
\end{small}
and
\begin{small}
\begin{multline*}
\frac{\kappa \abs{D_K}^{1/2} \sp^{-\sk/2}}{2^{3n-1} } \sum_{\pm}\sum_{\fa}\sum_{\fd}\sum_{Y\in\mathbb{F}_2^{n-1}}\sum_{\mu\in \cO_K/4\fa\fd^2}^{'}\binom{\mu^2 - 4 {u(Y)}\varepsilon, \fd}{\fa} \\
\sum_{(\xi, \eta)\neq (0, 0)} \frac{1}{\Norm_K(\fa)\Norm_K(\fd)^2}\left(\frac{\Phi_{\pm}^{1}(\xi, \eta)}{\Norm_K(\fd)\Norm_K(\fa)} + \Norm_K(\fd)\Psi_{\pm}^{1}(\xi, \eta)\right).
\end{multline*}
\end{small}
In particular, we have not reordered the sums. We have just distributed the outer sums over the inner one. Our analysis will now focus only on the sum obtained when $(\xi, \eta) = 0$. This is the same that is done in \cite{AliI}, where the corresponding zero term of Poisson summation is analyzed. The rest are studied in \cite{AliII} with different means.

\begin{notation}
We write $\overline{\Sigma_0}(f)$ to denote the expression obtained in Theorem~\ref{thm: after Poisson} with $\xi=\eta=0$.
\end{notation}

Note that the term
\[
\frac{1}{\Norm_K(\fa)\Norm_K(\fd)^2}\left(\frac{\Phi_{\pm}^1(0, 0)}{\Norm_K(\fd)\Norm_K(\fa)} + \Norm_K(\fd)\Psi_{\pm}^1(0, 0)\right)
\]
is independent of the value of $Y$ and $\mu$.
Thus, we can reorder and simplify the expression.

\begin{proposition}
\label{prop overline Sigma 0 f}
With notation as above
\begin{footnotesize}
\[
\overline{\Sigma_0}(f) = \frac{\kappa \abs{D_K}^{1/2} \sp^{-\sk/2}}{2^{3n-1} } \sum_{\pm}\sum_{\fa}\sum_{\fd}\left(\frac{\Phi_{\pm}^1(0, 0)}{\Norm_K(\fd)^3\Norm_K(\fa)^2} + \frac{\Psi_{\pm}^1(0, 0)}{\Norm_K(\fa)\Norm_K(\fd)}\right)\left(\sum_{Y\in\mathbb{F}_2^{n-1}}\sum_{\mu\in \cO_K/4\fa\fd^2}^{'}\binom{\mu^2 - 4 {u(Y)}\varepsilon, \fd}{\fa} \right).
\]
\end{footnotesize}
\end{proposition}

To further simplify notation, we introduce the following definition.

\begin{definition}
Given ideals $\fa, \fd$ of $\cO_K$ and a unit $u\in \cO_K^*$, define the \emph{Kloosterman-type sum} as
\[
K_{\fa, \fd}(u) := \sum_{\mu\in \cO_K/4\fa\fd^2}^{'}\binom{\mu^2 - 4u\varepsilon, \fd}{\fa}.
\]
\end{definition}

\begin{remark}
The above definition requires no hypothesis on $K$.
\end{remark}

Proposition~\ref{prop overline Sigma 0 f} can now be rewritten as follows.
\begin{theorem}
\label{prop:dominanttermwith(tau,u)=0}
With notation as before,
\[
\overline{\Sigma_0}(f) = \frac{ \kappa \abs{D_K}^{1/2} \sp^{-\sk/2}}{2^{3n-1} } \sum_{\pm}\sum_{Y\in\mathbb{F}_2^{n-1}}\sum_{\fa}\sum_{\fd}\left(\frac{\Phi_{\pm}^1(0, 0)}{\Norm_K(\fd)^3\Norm_K(\fa)^2} + \frac{\Psi_{\pm}^1(0, 0)}{\Norm_K(\fa)\Norm_K(\fd)}\right)K_{\fa, \fd}(\pm u(Y)).
\]
\end{theorem}

\begin{proof}
This follows from the above discussion.
We only notice that we can exchange the $Y$ and the $\fa, \fd$ sum because the former sum is finite.
\end{proof}

Our goal in the following section is to evaluate these sums precisely.
Though the sum over $Y$ is straightforward, the Kloosterman-type sum over $\mu$ presents challenges.

\section{Evaluation of Kloosterman-type Sums} 
\label{sec: Kloosterman}

The evaluation of the Kloosterman-type sums is central to the paper, but the calculations are long and technical.
For ease of following the argument, we have deferred them to Appendix~\ref{Appendix_A}.
Here, we only state the main result.

Throughout this section we invoke the following hypothesis.
\begin{equation}
\tag*{\textup{\textbf{(Hyp-split)}}}
\label{ass: split}
\begin{minipage}{0.8\textwidth}
Let $K$ be a number field such that 2 splits completely in $K$.
\end{minipage}
\end{equation}
Let $\fp_1,\dots,\fp_n$ be the primes above $2$ in $\cO_K$, that is,
\[
2\cO_K = \fp_1\cdots\fp_n.
\]
By assumption, $K_{\fp_i} \cong \Q_2$.
We assume $\pi_{\fp_i}$ corresponds to $2$ under the above isomorphism.
This fixes the uniformizer uniquely as there is only one field isomorphism $\mathbb{Q}\longrightarrow\mathbb{Q}$.

\begin{remark}
The splitting condition at the prime 2 is a technical hypothesis we need to impose to evaluate the Kloosterman-type sums.
We believe that the main result of this section should be true in general.
We note that in this section we do not require the hypothesis that $K$ is totally real.
\end{remark}

\begin{definition}
Let $u\in \cO_K^*$ be a unit and let $\fd$ and $\fa$ be arbitrary ideals of $\cO_K$.
For a complex parameter $z = \sigma + it$ with $\sigma > 1$, define
\[
\D(z,u):= \sum_{\fd}\frac{1}{\Norm_K(\fd)^{2z+1}} \sum_{\fa}\frac{K_{\fa,\fd}(u)}{\Norm_K(\fa)^{z+1}}.
\]
\end{definition}
The main result of this section is the following theorem.
We provide a proof of this result in Appendix~\ref{Appendix_A}.

\begin{theorem}
\label{main theorem of the kloosterman section}
Let $K$ be a number field satisfying \ref{ass: split} and let $(\fp,\sk)$ be the pair of prime ideal and corresponding positive integer satisfying \ref{ass: class number} fixed in \S\ref{sec: Definition of the test function and its implications}.
Write $\mathsf{p}=\Norm_{K}(\fp)$.
Let $u\in \cO_K^*$ be a unit.
For a complex parameter $z$ with $\Re(z) > 1$,
\[
\D(z,u)=4^n\frac{\zeta_K(2z)}{\zeta_K(z+1)}\cdot\frac{1-1/\sp^{z(\sk+1)}}{1-1/\sp^{z}}.
\]
In particular, $\D(z,u)$ is independent of $u$.
Moreover, it admits an analytic continuation to a meromorphic function in the whole complex plane with a pole at $z=\frac12$ with residue $\frac{2^{2n+r_1}h_KR_K} 
{\zeta_K(3/2)\sqrt{\abs{D_K}} \omega_K}\frac{1-\sp^{-(\sk+1)/2}}{1-\sp^{-1/2}}$ using the terms detailed in Proposition~\ref{prop: volume manipulation}.
\end{theorem}

\begin{notation}
In Theorem~\ref{main theorem of the kloosterman section} we proved that $\D(z,u(Y))$ is independent of $u(Y)$.
To avoid confusion, we will henceforth drop the dependence on $u$ from the notation.
\end{notation}

\section{Isolation of the Contribution of the Trivial and Special Representations}
\label{sec: Isolation of the Contribution of the Trivial and Special Representations}

In this section of the paper, the goal is to isolate the contribution of the special representation and finish the proof of the main theorem.
Recall from Theorem~\ref{prop:dominanttermwith(tau,u)=0} that the dominant term is,
\begin{small}
\begin{equation*}
\overline{\Sigma_0}(f) = \frac{ \kappa \abs{D_K}^{1/2} \sp^{-\sk/2}}{2^{3n-1} } \sum_{\pm}\sum_{Y\in \mathbb{F}_2^{n-1}}\sum_{\fa}\sum_{\fd}\left(\frac{\Phi_{\pm}^1(0, 0)}{\Norm_K(\fd)^3\Norm_K(\fa)^2} + \frac{\Psi_{\pm}^1(0, 0)}{\Norm_K(\fa)\Norm_K(\fd)}\right)K_{\fa, \fd}(\pm u(Y)).
\end{equation*}
\end{small}
At this point it is worthwhile to compute the constant in front of the sum more explicitly.
We have
\[
\frac{\kappa\abs{D_K}^{1/2}\sp^{-\sk/2}}{2^{3n-1}} 
 = \frac{2^{n-1} h_KR_K}{\abs{D_K}^{1/2}} \cdot \frac{\abs{D_K}^{1/2}\sp^{-\sk/2}}{2^{3n-1}} 
 = \dfrac{h_KR_K\sp^{-\sk/2}}{4^n},
\]
and thus the dominant term is
\begin{small}
\begin{equation}
\label{eq:xi=0new}
\overline{\Sigma_0}(f) = \dfrac{h_KR_K\sp^{-\sk/2}}{4^n} \sum_{\pm}\sum_{Y\in \mathbb{F}_2^{n-1}}\sum_{\fa}\sum_{\fd}\left(\frac{\Phi_{\pm}^1(0, 0)}{\Norm_K(\fd)^3\Norm_K(\fa)^2} + \frac{\Psi_{\pm}^1(0, 0)}{\Norm_K(\fa)\Norm_K(\fd)}\right)K_{\fa, \fd}(\pm u(Y)).
\end{equation}
\end{small}

\subsection{Some algebraic number theory and a discussion}
\label{subsec: A lemma from algebraic number theory and a discussion}

Let $\gamma\in\GL(2, K)$ be a regular elliptic matrix with trace $\tau$ and determinant $\delta$.
For
\[
(\gamma_1, \dots, \gamma_n) \in G_{\infty},
\]
where $\gamma_i$ is the embedding of $\gamma$ in the $i$-th real embedding of $K$, define
\[
C(\gamma) := C(\CH((\gamma_1, \dots, \gamma_n))),
\]
where $C$ is the indicator function defined in Definition~\ref{indicator function}.

\begin{lemma}
\label{lemma: Limit with C}
The following equality holds
\[
\lim_{z\to0}\frac{L_{\infty}( z, \chi_{\gamma})}{L_{\infty}(1- z, \chi_{\gamma})} \frac{\zeta_K(2z)}{\zeta_K(z+1)} = -C(\gamma)\abs{D_K}^{1/2}2^{n -1}
\]
where $L_{\infty}( z, \chi_{\gamma}) = \prod_{\nu}L_{\nu}(z, \chi_{\gamma})$ and $\nu$ varies over all real places.
\end{lemma}

\begin{notation}
In this proof we write $\chi$ in place of $\chi_{\gamma}$ for ease of notation.
\end{notation}

\begin{proof}
Recall that for a totally real field $K$, the completed Dedekind zeta function is defined as
\[
\xi_K(z) := \abs{D_K}^{z/2}L_{\R}(z)^n\zeta_K(z).
\]
Then
\[
\frac{\xi_K(2z)}{\xi_K(z+1)} = \frac{\abs{D_K}^{z}L_\R(2z)^{n} \zeta_K(2z)}{\abs{D_K}^{(z+1)/2}L_\R(z+1)^{n}{\zeta_K(z+1)}} =\abs{D_K}^{(z-1)/2} \frac{L_\R(2z)^{n}\zeta_K(2z)}{L_\R(z+1)^{n} \zeta_K(z+1)}
\]
whose limit at $z=0$ is equal to $-1/2$.
Rearranging terms we write
\[
\frac{\zeta_K(2z)}{\zeta_K(z+1)}
= \frac{\xi_K(2z)}{\xi_K(z+1)} \abs{D_K}^{(1-z)/2} \frac{L_\R(z+1)^{n}}{L_\R(2z)^{n}}.
\]
We can therefore obtain
\[
\frac{L_{\infty}( z, \chi)}{L_{\infty}(1- z, \chi)} \frac{\zeta_K(2z)}{\zeta_K(z+1)}
=\abs{D_K}^{(1-z)/2}\prod_{\nu}\frac{L_\R(z+\delta_{\nu})L_\R(z+1)}{L_\R(1-z+\delta_{\nu})L_\R(2z)}\frac{\xi_K(2z)}{\xi_K(z+1)},
\]
where as before, $\delta_\nu = \delta_{\nu, \gamma} = 1$ if $\chi$ is ramified at $\nu$ and 0 otherwise.
As $z\to 0$, we have
\[
\lim_{z\to 0} \frac{L_{\infty}( z, \chi)}{L_{\infty}(1- z, \chi)} \frac{\zeta_K(2z)}{\zeta_K(z+1)} = - \frac12\abs{D_K}^{1/2}\lim_{z\to 0}\left(\prod_{\nu}\frac{L_\R(z+\delta_\nu)L_\R(z)}{L_\R(1-z+\delta_\nu)L_\R(2z)}\right),
\]
and it remains to compute the limit explicitly.
We have
\begin{align*}
\lim_{z\to 0}\frac{L_\R(z+\delta_\nu)L_\R(1+z)}{L_\R(2z)L_\R(1-z+\delta_\nu)}
&= \lim_{z\to 0}\pi^{z}\frac{\Gamma(\frac{z+\delta_\nu}{2})\Gamma(\frac{z+1}{2})}{\Gamma(\frac{1-z+\delta_\nu}{2})\Gamma(z)} \\
&=\lim_{z\to 0} 2^{1-z}\pi^{z+1/2}\frac{\Gamma(\frac{z+\delta_\nu}{2})}{\Gamma(\frac{1-z+\delta_\nu}{2})\Gamma(\frac{z}{2})} = \begin{cases}
2, & \textrm{ if } \delta = 0 \\ 0, & \textrm{ if } \delta = 1
\end{cases} 
\end{align*}
Altogether then
\[
\lim_{z\to0}\frac{L_{\infty}( z, \chi)}{L_{\infty}(1- z, \chi)} \frac{\zeta_K(2z)}{\zeta_K(z+1)} = -\abs{D_K}^{1/2}2^{n -1}
\]
when $\chi$ is unramified \textit{at every real place} and $0$ otherwise.
By the Theorem~\ref{thm: reinterpretation of C}, this dichotomy is equivalent to $C(\gamma) = 1$ or $0$, respectively.
\end{proof}
From the above and proposition \eqref{prop: extension H} we get immediately

\begin{corollary}
\label{no dep on gamma}
The following limit holds for any $(x, y)\in\mathbb{R}^{2n-1}$
\[
\lim_{z\to0}\frac{L_{\infty}( z, \chi)}{L_{\infty}(1- z, \chi)} \frac{\zeta_K(2z)}{\zeta_K(z+1)} = -C(x, y)\abs{D_K}^{1/2}2^{n -1}.
\]
\end{corollary}

\begin{remark}
Recall that the $L$-factors at infinity only depend on the regions, i.e.~the behaviour at infinity, and thus we can disregard the dependence on any particular $\gamma$.
\end{remark}

\subsection{Rewriting the dominant term} 

We do this calculation in two steps.
First, we rewrite the $\Phi^1_{\pm}(0,0)$-part:
\begin{equation}
\begin{split}
\label{eq:contour1new}
& \dfrac{h_KR_K\sp^{-\sk/2}}{4^n} \sum_{\pm} \sum_{Y\in \mathbb{F}_2^{n-1}} \sum_{\fa}\sum_{\fd}\left(\frac{\Phi_{\pm}^1(0, 0)}{\Norm_K(\fd)^3\Norm_K(\fa)^2}\right)K_{\fa, \fd}(\pm u(Y)) \\
 &= \dfrac{h_KR_K\sp^{-\sk/2}}{4^n} \sum_{\pm}\sum_{Y\in \mathbb{F}_2^{n-1}}\sum_{\fd}\frac{1}{\Norm_K(\fd)^3}\sum_{\fa} \\
 & \qquad \qquad \int \int\theta^{\pm}(x,y) F\left(\frac{\Norm_K(\fd)^2\Norm_K(\fa)\sp^{-\sk\alpha }}{\abs{D_K}^\alpha} \mathcal{P}^{\pm}(x,y)^{2\alpha}\right)\left(\frac{K_{\fa, \fd}(\pm u(Y))}{\Norm_K(\fa)^2}\right) dy dx \\
&= \dfrac{h_KR_K\sp^{-\sk/2}}{4^n} \sum_{\pm}\sum_{Y\in \mathbb{F}_2^{n-1}}\sum_{\fd}\frac{1}{\Norm_K(\fd)^3}\sum_{\fa}\frac{K_{\fa, \fd}(\pm u(Y))}{\Norm_K(\fa)^2} \\
& \qquad \qquad \times \left( \int \int\frac{\theta^{\pm}(x,y)}{2\pi i} \int_{\Re{(z)}=1}\widetilde{F}(z)\left(\frac{\Norm_K(\fd)^2\Norm_K(\fa)\sp^{-\sk\alpha }}{\abs{D_K}^\alpha} \mathcal{P}^{\pm}(x,y)^{2\alpha}\right)^{-z}dz dy dx\right) \\
&= \dfrac{h_KR_K\sp^{-\sk/2}}{4^n} \sum_{\pm}\sum_{Y\in \mathbb{F}_2^{n-1}}\left( \int \int\frac{\theta^{\pm}(x,y)}{2\pi i} \int_{\Re{(z)}=1}\widetilde{F}(z)\left(\frac{\sp^{-\sk\alpha }}{\abs{D_K}^\alpha} \mathcal{P}^{\pm}(x,y)^{2\alpha}\right)^{-z} \right. \\
& \qquad \qquad \left. \times \sum_{\fd}\frac{1}{\Norm_K(\fd)^{2(z + 1)+1}}\sum_{\fa}\frac{K_{\fa, \fd}(\pm u(Y))}{\Norm_K(\fa)^{(z + 1) + 1}}dz dy dx\right) \\
 &= \dfrac{h_KR_K\sp^{-\sk/2}}{4^n} \sum_{\pm}\sum_{Y\in\mathbb{F}_2^{n-1}}\left( \int \int\frac{\theta^{\pm}(x,y)}{2\pi i} \int_{\Re{(z)}=1}\widetilde{F}(z)\left(\frac{\sp^{-\sk\alpha }}{\abs{D_K}^\alpha} \mathcal{P}^{\pm}(x,y)^{2\alpha}\right)^{-z} \D(z+1, u) dz dy dx \right) \\
&= \dfrac{h_KR_K\sp^{-\sk/2}}{2^{n+1}} \sum_{\pm}\left( \int \int\frac{\theta^{\pm}(x,y)}{2 \pi i} \int_{\Re{(z)}=1}\widetilde{F}(z)\left(\frac{\sp^{-\sk\alpha }}{\abs{D_K}^\alpha} \mathcal{P}^{\pm}(x,y)^{2\alpha}\right)^{-z} \D(z+1,u) dz dy dx \right) .
\end{split}
\end{equation}

\normalsize
The last line is obtained by by evaluating the sum over $Y \in \mathbb{F}^{n-1}_2$.
Note that after evaluating the Kloosterman-type sums, nothing in the integrals depends on $Y$.
Similarly for the second part of the sum, we have

\begin{equation*}
\begin{split}
 & \dfrac{h_KR_K\sp^{-\sk/2}}{4^n} \sum_{\pm} \sum_{Y\in \mathbb{F}_2^{n-1}} \sum_{\fa}\sum_{\fd}\left(\frac{\Psi_{\pm}^1(0, 0)}{\Norm_K(\fa)\Norm_K(\fd)}\right)K_{\fa, \fd}(\pm u(Y))\\
&= \dfrac{h_KR_K\sp^{-\sk/2}}{4^n} \sum_{\pm} \sum_{Y\in \mathbb{F}_2^{n-1}} \sum_{\fd}\frac{1}{\Norm_K(\fd)^3}\sum_{\fa}\left(\frac{\Norm_K(\fd)^2\Norm_K(\fa)K_{\fa, \fd}(\pm u(Y))}{\Norm_K(\fa)^2}\Psi_{\pm}^{1}(0, 0)\right)\\
 &= \dfrac{h_KR_K\sp^{-\sk/2}}{4^n} \sum_{\pm}\sum_{Y\in \mathbb{F}_2^{n-1}}\sum_{\fd}\frac{1}{\Norm_K(\fd)^3}\sum_{\fa}\left(\frac{\Norm_K(\fd)^2\Norm_K(\fa)K_{\fa, \fd}(\pm u(Y))}{\Norm_K(\fa)^2}\right)\\
& \quad\times \int \int\theta^{\pm}(x, y) \mathcal{P}^{\pm}(x,y)\abs{D_K}^{-1/2}\sp^{-\sk/2}H_{\gamma}\left(1, \frac{\Norm_K(\fd)^2\Norm_K(\fa)\sp^{-\sk(1-\alpha)}}{\abs{D_K}^{1 - \alpha}}\mathcal{P}^{\pm}(x,y)^{2(1-\alpha)}\right) dy dx\\
 &= \dfrac{h_KR_K\sp^{-\sk/2}}{4^n} \sum_{\pm}\sum_{Y\in \mathbb{F}_2^{n-1}}\sum_{\fd}\frac{1}{\Norm_K(\fd)^3}\sum_{\fa}\left(\frac{\Norm_K(\fd)^2\Norm_K(\fa)K_{\fa, \fd}(\pm u(Y))}{\Norm_K(\fa)^2}\right)\\
 & \times \int\int\frac{\theta^{\pm}(x,y)}{2 \pi i} \mathcal{P}^{\pm}(x,y)\abs{D_K}^{-1/2}\sp^{-\sk/2}\int_{\Re(z)=1}\widetilde{F}(z)\left( \frac{\Norm_K(\fd)^2\Norm_K(\fa)\sp^{-\sk(1-\alpha)}}{\abs{D_K}^{1 - \alpha}}\mathcal{P}^{\pm}(x,y)^{2(1-\alpha)}\right)^{-z} \\
 & \hspace{4.25in} \frac{L_{\infty}(z,\chi)}{L_{\infty}(1-z,\chi)} dz dy dx \\
&= \dfrac{h_KR_K\sp^{-\sk/2}}{4^n} \sum_{\pm} \sum_{Y\in\mathbb{F}_2^{n-1}} \int \int\frac{\theta^{\pm}(x,y)}{2 \pi i} \mathcal{P}^{\pm}(x,y)\abs{D_K}^{-1/2}\sp^{-\sk/2} \\
& \hspace{1in} \int_{\Re(z)=1}\widetilde{F}(z)\left( \frac{\sp^{-\sk(1-\alpha)}}{\abs{D_K}^{1 - \alpha}}\mathcal{P}^{\pm}(x,y)^{2(1-\alpha)}\right)^{-z}{}\frac{L_{\infty}(z,\chi)}{L_{\infty}(1-z,\chi)} \D(z) dz dy dx.
\end{split}
\end{equation*}

\normalsize 
Once again, we apply Theorem~\ref{main theorem of the kloosterman section} to the above expression and evaluating the sum over $Y \in \mathbb{F}_2^{n-1}$.
This cancels the $2^{3n-1}$ in the denominator and the above expression simplifies to
\begin{multline}
\label{eq:contour2new}
\dfrac{h_KR_K\abs{D_K}^{-1/2}\sp^{-\sk}}{2^{n+1}} \sum_{\pm} \int \int\frac{\theta^{\pm}(x,y)}{2 \pi i} \mathcal{P}^{\pm}(x,y) \\
\times\int_{\Re(z)=1}\widetilde{F}(z)\left( \frac{\sp^{-\sk(1-\alpha)}}{\abs{D_K}^{1 - \alpha}}\mathcal{P}^{\pm}(x,y)^{2(1-\alpha)}\right)^{-z} \frac{L_{\infty}(z,\chi)}{L_{\infty}(1-z,\chi)} \mathbb{D}(z,u) dz dy dx.
\end{multline}

Note that \eqref{eq:xi=0new} is the sum of \eqref{eq:contour1new} and \eqref{eq:contour2new}.
However, we write the summands separately as we work with each piece one-at-a-time.
The trivial representation will come from \eqref{eq:contour1new}, whereas the special representation will come from \eqref{eq:contour2new} after contour shifting.
In both cases, they will arise from the residue at $z=0$.

\subsection{Calculating the residue}
While our main interest is the residue at $z=0$, we shall follow \cite{AliI} to shift past the line $\Re(z)=-1/2$.
More precisely, we shift the first contour to the line $\mathcal{C}_1$ defined by $\Re(z)=\sigma$ for $-1<\sigma<-1/2$, and the second to the line $\mathcal{C}_2$ given by the line $\Re(z)=\sigma$ for $-1/2<\sigma<0$.

\begin{lemma}
\label{lem:Contribution1new}
The expression \eqref{eq:contour1new} is equal to the sum of
\begin{small}
\begin{align}
\label{eq:triv res new}
& 2^{n-1}h_KR_K \sp^{-\sk/2}\frac{1-\sp^{-(\sk+1)}}{1-{\sp^{-1}} } \sum_{\pm} \int \int \theta^{\pm}(x,y) dy dx,\\
\label{eq:F(1/2)new}
& 2^{n-1}h_KR_K \dfrac{\sp^{-\sk/2}\kappa}{\zeta_K(3/2)}\frac{1-\sp^{-(\sk+1)/2}}{1-\sp^{-1/2}} \widetilde{F}\left(-\frac{1}{2}\right)\sum_{\pm}\left( \int \int{\theta^{\pm}(x,y)}\left(\frac{\sp^{-\sk\alpha }}{\abs{D_K}^\alpha} \mathcal{P}^{\pm}(x,y)^{2\alpha}\right)^{1/2} dy dx \right), \textrm{ and }\\
\label{eq:C(1)new}
&\frac{ h_KR_K \sp^{-\sk/2}}{2^{n+1}} \sum_{\pm}\left( \int \int\frac{\theta^{\pm}(x,y)}{2 \pi i} \int_{\mathcal{C}_1}\widetilde{F}(z)\left(\frac{\sp^{-\sk\alpha }}{\abs{D_K}^\alpha} \mathcal{P}^{\pm}(x,y)^{2\alpha}\right)^{-z}\D(z+1) dz dy dx \right).
\end{align}
\end{small}
\end{lemma}

\begin{proof}
The constant in front of \eqref{eq:contour1new} is $\dfrac{h_KR_K\sp^{-\sk/2}}{2^{n+1}}$.

To shift the contour of integration to the line $\mathcal{C}_1$, we describe the poles in the region that is traversed.
The integral over $(x,y)$ can be taken over the regular set, so that $\mathcal{P}^{\pm}(x,y)$ is finite and non-zero.
The Mellin transform $\widetilde{F}(z)$ has a simple pole at $z=0$ with residue equal to $1$.
So, the total residue at $z=0$ is equal to \eqref{eq:triv res new}.

The auxiliary Dirichlet series $\D(z)$ has a simple pole at $z=1/2$ with residue equal to
\begin{equation}
\label{aux residuenew}
\lim_{z\rightarrow 1/2}(z - 1/2)\D(z) = \lim_{z\rightarrow 1/2}(z - 1/2)4^n\frac{\zeta_K(2z)}{\zeta_K(z+1)}\cdot\frac{1-1/\sp^{z(\sk+1)}}{1-1/\sp^{z}}=\frac{4^n\kappa} 
{\zeta_K(3/2)}\frac{1-\sp^{-(\sk+1)/2}}{1-\sp^{-1/2}}.
\end{equation}
By a change of variables, we have the same for $\D(z+1)$ at $z=-1/2$.
Thus the total residue there is equal to \eqref{eq:F(1/2)new}.
Notice that the residue at $z = -1/2$ of the quotient of zeta functions is
\begin{equation*}
 \displaystyle\lim_{z\longrightarrow -1/2}(z + 1/2)\dfrac{\zeta_K(2z + 2)}{\zeta_K(z + 2)} = \dfrac{\kappa}{\zeta_K(3/2)},
\end{equation*}
which explains the constant in front of equation \eqref{eq:F(1/2)new}.
The two residues together yield the main term; the remainder \eqref{eq:C(1)new} is described by the integral over $\mathcal{C}_1$.
\end{proof}

\begin{lemma}
\label{lemma 9.2}
The expression \eqref{eq:contour2new} is equal to the sum of the following three expressions
\begin{align}
\label{eq:spec residue2new}
& -4^{n-1}h_KR_K(\sk + 1)\sp^{-\sk} \sum_{\pm} \int \int\theta^{\pm}(x,y) \mathcal{P}^{\pm}(x,y) C(x, y) dx dy,\\ 
\label{F(1/2)new2}
& 2^{n-1}h_KR_K \sp^{-\sk/2}\frac{\kappa} 
{\zeta_K(3/2)}\frac{1-\sp^{-(\sk+1)/2}}{1-\sp^{-1/2}}\widetilde{F}\left(\dfrac{1}{2}\right) \\
& \hspace{1.5in} \times \sum_{\pm} \int \int\theta^{\pm}(x,y) \mathcal{P}^{\pm}(x,y) \left( \frac{\sp^{-\sk \alpha}}{\abs{D_K}^{\alpha}}\mathcal{P}^{\pm}(x,y)^{2\alpha}\right)^{1/2} dy dx, \textrm{ 
and } \notag \\ 
\label{C2new}
& \dfrac{h_KR_K\sp^{-\sk}\abs{D_K}^{-1/2}}{2^{n+1}} \sum_{\pm} \int \int\frac{\theta^{\pm}(x,y)}{2 \pi i} \mathcal{P}^{\pm}(x,y) \abs{D_K}^{1/2}\sp^{-\sk/2} \\
& \hspace{1.3in} \times \int_{\mathcal{C}_2}\widetilde{F}(z)\left( \frac{\sp^{-\sk(1-\alpha)}}{\abs{D_K}^{1 - \alpha}}\mathcal{P}^{\pm}(x,y)^{2(1-\alpha)}\right)^{-z} \frac{L_{\infty}(z,\chi)}{L_{\infty}(1-z,\chi)}\D(z)dz dy dx. \notag
\end{align}
\end{lemma}

\begin{proof}
Shifting the $dz$ contour of integration to the left, we see that at $z=1/2$ the Dirichlet series $\D(z)$ has a pole with the same residue \eqref{aux residuenew}, so the contribution there is equal to 
\begin{multline}
\label{eqn: partial residue calculation}
\dfrac{4^nh_KR_K\abs{D_K}^{-1/2}\sp^{-\sk}}{2^{n+1}} \times \frac{\kappa} 
{\zeta_K(3/2)} \times \frac{1-\sp^{-(\sk+1)/2}}{1-\sp^{-1/2}}\widetilde{F}\left(\frac{1}{2}\right) \\ 
\times \sum_{\pm} \int \int\theta^{\pm}(x,y) \mathcal{P}^{\pm}(x,y)
\left( \frac{\sp^{-\sk(1-\alpha)}}{\abs{D_K}^{1 - \alpha}}\mathcal{P}^{\pm}(x,y)^{2(1-\alpha)}\right)^{-1/2} dy dx.
\end{multline}
Notice that
\[
\sp^{-\sk} \mathcal{P}^{\pm}(x,y)
\left( \frac{\sp^{-\sk(1-\alpha)}}{\abs{D_K}^{1 - \alpha}}\mathcal{P}^{\pm}(x,y)^{2(1-\alpha)}\right)^{-1/2} = \sp^{-\sk/2}\abs{D_K}^{1/2}\dfrac{\sp^{-\sk\alpha/2}\mathcal{P}^{\pm}(x,y)^{\alpha}}{\abs{D_K}^{\alpha/2}}.
\]
Thus the residue \eqref{eqn: partial residue calculation} can be rewritten as
\[
2^{n-1}h_KR_K \sp^{-\sk/2}\frac{\kappa} 
{\zeta_K(3/2)}\frac{1-\sp^{-(\sk+1)/2}}{1-\sp^{-1/2}}\widetilde{F}\left(\frac{1}{2}\right) \sum_{\pm} \int \int\theta^{\pm}(x,y) 
 \left( \frac{\sp^{-\sk \alpha}}{\abs{D_K}^{\alpha}}\mathcal{P}^{\pm}(x,y)^{2\alpha}\right)^{1/2} dy dx.
\]
This is \eqref{F(1/2)new2} above.
Then shifting the contour past to $z=0$, we note that
\[
\frac{1-\sp^{-z(\sk+1)}}{1-\sp^{-z}} = \sum_{n=0}^{\sk}\sp^{-nz},
\]
which equals $\sk + 1$ at $z = 0$.
Recall from Lemma~\ref{lemma: Limit with C} that
\[
\lim_{z\to0}\frac{L_{\infty}( z, \chi)}{L_{\infty}(1- z, \chi)} \frac{\zeta_K(2z)}{\zeta_K(z+1)} = -C(x, y)\abs{D_K}^{1/2}2^{n -1}
\]
and that the residue at $z = 0$ of $\widetilde{F}(z)$ is $1$.
Thus the residue at $z = 0$ is
\[
\dfrac{h_KR_K\abs{D_K}^{-1/2}\sp^{-\sk}}{2^{n+1}} \sum_{\pm} \int \int\theta^{\pm}(x,y)\mathcal{P}^{\pm}(x,y) \cdot (-C(x, y)\abs{D_K}^{1/2}2^{n -1})\cdot (\sk + 1) dy dx,
\]
which becomes
\[
-\dfrac{4^nh_KR_K(\sk + 1)\sp^{-\sk}}{4}\sum_{\pm} \int \int\theta^{\pm}(x,y)\mathcal{P}^{\pm}(x,y) \cdot C(x, y) dy dx,
\]
Thus the total residue is given by \eqref{eq:spec residue2new} as desired.
The two residues together yield the main term, while the remainder \eqref{C2new} is described by the integral over $\mathcal{C}_2$.
Once more, the $4^n$ appears precisely because we have put $\D(z)$ again.
\end{proof}

Recall that $\widetilde{F}$ is an odd function which allows us to conclude that \eqref{eq:F(1/2)new} + \eqref{F(1/2)new2} = 0.
We obtain the following result.

\begin{proposition}
\label{prop: computation of 23 + 27 +´25 + 29}
The dual sum \eqref{eq:xi=0new} is equal to $\eqref{eq:triv res new}+\eqref{eq:spec residue2new}+\eqref{eq:C(1)new}+\eqref{C2new}$.
\end{proposition}

\section{Computation of the trace of the trivial and the special representations}
\label{sec: Computation of the trace of the trivial and the special representations}

In this section, we compute the traces of the representations we desire to cancel; namely the trivial representation and the special representations.
The core of the argument lies the computing the trace at the archimedean places.

\subsection{The main statement} We begin with the following

\begin{definition}
The special representation $\xi$ is the induced representation from the Borel subgroup $B(\A)$ of the trivial representation.
\end{definition}

\begin{remark}
In \cite{LanBE04, Lan80} the special representation is reached at in a different way.
The difference stems from how the Jacquet--Langlands trace formula is constructed in \cite{LanBE04}.
For our purposes, it suffices to know that the special representation enters the Jacquet--Langlands trace formula in the so-called term (vi) in \cite[\S16]{JacquetLanglands1970}, where it is verified to be a constant factor of the trace of the above defined $\xi$. 
\end{remark}

Our main result of this section states the trace of the trivial and the special representations.
We have

\begin{theorem}
\label{main theorem section 9}
Let $f$ be the fixed test function.
Then
\begin{enumerate}[label = \textup{(\roman*)}]
\item \label{prop: trivial rep value}
The trace of the trivial representation is 
\[
\Tr({\bf1}(f)) = {\sp}^{-\sk/2}R_K\frac{1-{\sp}^{-(\sk+1)}}{1-{\sp}^{-1}} \sum_{\pm} \int_{\R^{2n -1}} \theta^{\pm}(x, y) dxdy,
\]
which equals \eqref{eq:triv res new}
\item \label{thm: Weyl Int special functions}
The trace of the special representation is
\[
\Tr(\xi(f)) = (\sk+1)\sp^{-\sk}2^{n-2} R_K \sum_{\pm}\int_{\R^{2n-1}}C(x, y)\theta^{\pm}(x, y)\mathcal{P}^{\pm}(x,y)dx dy,
\]
which is $-1/2$ times \eqref{eq:spec residue2new}.
\end{enumerate}
In these equations, $\theta^{\pm}$ is the interpolation function defined in \S\ref{section: blending step} and $C(x, y)$ is the indicator function defined in Definition~\ref{indicator function}.
\end{theorem}

As we will see, the main computational task is the calculation of the archimedean part of the traces.
Consequently, we begin by proving the general statements involved.

\subsection{Integral and measure calculations}

\subsubsection{Change of variables in the regions}

Recall that we have defined the following maps in Section~\ref{subsec: spl behaviour at inf}
\begin{align*}
\CH:= \ch \otimes \cdots \otimes \ch: G_{\infty} & \longrightarrow \R^2\times\cdots\times \R^2\\
 \gamma_1,\ldots, \gamma_n & \longmapsto (r_1, N_1, \dots, r_n, N_n).
\end{align*}

Restricted to the tori representatives, we have $\CH(T_J) = S_J$ and it is a $2:1$ covering map.
Each tori carries a product measure induced by the measures in $T_{\el}$ and $T_{\spl}$.
Thus to understand the change of measures under this map it is enough to do it for $\GL(2, \R)$.

\begin{proposition}
\label{prop: change of variable in Tori}
Let $T$ be one of $T_{\el}$ or $T_{\spl}$, and $S$ be the corresponding region.
Then
\[
\int_T h(\ch(t)) dt = 2\displaystyle\int_S h(r, N) \dfrac{dr dN}{\abs{N}\abs{r^2 - N}^{1/2}}, 
\]
where $h$ is any integrable function on $S$.
\end{proposition}
\begin{proof}
The $2$ that appears in front of the integral is because the map $\ch$ is a $2:1$ covering map.
The rest is a calculation with differential forms.

For the split torus we have that
\begin{align*}
r &= \alpha + \beta,\\
N &= 4\alpha\beta.
\end{align*}
Taking derivatives we obtain
\begin{align*}
dr &= d\alpha + d\beta,\\
dN &= 4\left(\beta d\alpha + \alpha d\beta\right).
\end{align*}
Thus,
\[
dr\wedge dN = 4(\alpha - \beta)d\alpha\wedge d\beta.
\]
Using that $4\alpha\beta = N$ and that $\abs{\alpha - \beta}^2 = \abs{r^2 - N}$, we get
\begin{equation*}
 \dfrac{d\alpha\wedge d\beta}{\alpha\beta} = \dfrac{dr\wedge dN}{\abs{N}\abs{r^2 - N}^{1/2}},
\end{equation*}
as desired.

For the elliptic torus, we have
\begin{align*}
 r &= 2a\\
 N &= 4(a^2 + b^2).
\end{align*}
Taking derivatives we get
\begin{align*}
 dr &= 2 da\\
 dN &= 8(ada + bdb).
\end{align*}
This implies
\begin{equation*}
 dr\wedge dN = 16 da\wedge(ada + bdb) = 16 b da\wedge db.
\end{equation*}
Notice that
\begin{equation*}
 b^2 = \dfrac{N}{4} - a^2 = \dfrac{N - 4a^2}{4} = \dfrac{N - r^2}{4}.
\end{equation*}
We conclude
\begin{equation*}
 b = \dfrac{\abs{r^2 - N}^{1/2}}{2}.
\end{equation*}
Substituting this we obtain
\begin{equation*}
 \dfrac{2 da\wedge db}{a^2+b^2} = \dfrac{drdN}{N\abs{r^2 - N}^{1/2}}
\end{equation*}
Taking absolute values to recover the measure associated to the differential form finishes the proof.
\end{proof}

The following corollary is an immediate consequence of the above result.

\begin{corollary}
\label{cor: change of variables general index}
Let $j_1, \ldots, j_n \in \{{\spl}, {\el}\}$ and $J = (j_1, \ldots , j_n)$.
Then
\[
\int_{T_J} h(\CH(t)) dt = 2^n\displaystyle\int_{S_J} h(r_1, N_1, \ldots , r_n, N_n) \displaystyle\prod_{i = 1}^n\dfrac{dr_i dN_i}{\abs{N_i}\abs{r_i^2 - N_i}^{1/2}}, 
\]
where $h$ is any integrable function on $S_J$.
\end{corollary}

\subsubsection{Fundamental domain of the action of $Z_+$ and quotient integration}

Recall that $Z_+$ acts on $G_{\infty}$ by matrix multiplication.
Each torus is invariant under this action.
As we have seen, there is a corresponding action in $\R^{2n}$ given by
\[
z\cdot (r_1, N_1,\ldots, r_n, N_n) = (zr_1, z^2N_1,\ldots,zr_n, z^2N_n).
\]
Each region $S_J$ is invariant under this action.

\begin{definition}
\label{def: fundamental domain}
Let $j_1, \ldots, j_n \in \{{\spl}, {\el}\}$ and $J = (j_1, \ldots , j_n)$.
Define
\[
\cF_J^{\pm} = S_J \cap \{N_1 = \pm 1\}.
\]
\end{definition}
These sets will allow us to do the integration under the action of $Z_+$.
The reason is the following result.

\begin{proposition}
\label{prop: fundamental domain}
With notation as above, $\cF_J^{\pm}$ is a fundamental domain of the action of $Z_+$ on $S_J$.
\end{proposition}

\begin{proof}
Let $(r_1, N_1,\ldots, r_n, N_n)\in S_J$ and define $z = \dfrac{1}{\sqrt{\abs{N_1}}}$.
Then, under the action of $Z_+$, we have
\[
z\cdot (r_1, N_1,\ldots, r_n, N_n) = (zr_1, z^2N_1,\ldots,zr_n, z^2N_n) = \left(\dfrac{r_1}{\sqrt{\abs{N_1}}}, \dfrac{N_1}{\abs{N_1}}, \ldots , \dfrac{r_n}{\sqrt{\abs{N_1}}}, \dfrac{N_n}{\abs{N_1}}\right).
\]
The latter point is in $\cF_J^{\pm}$.
It is clear no two such points are equivalent under $Z_+$.
\end{proof}

Integrating in a quotient space is the same as integrating on one of its fundamental domains.
We shall use in the regions $S_J$ the measures obtained in Corollary~\ref{cor: change of variables general index}, We need to understand integration of functions in this fundamental domain.

\begin{lemma}
\label{lem: integration on quotient}
Let $j_1, \ldots, j_n \in \{{\spl}, {\el}\}$ and $J = (j_1, \ldots , j_n)$.
Let $\phi:R_J\rightarrow\C$ be an integrable function invariant under $Z_+$.
Then
\[
\int_{Z_+\backslash R_J}\phi(\dot{r}) \dfrac{dr}{dz} = int_{\cF_J^{\pm}}
 \phi\left(r_1, \pm 1, r_2, N_2, \ldots , r_n, N_n\right) 
 \dfrac{dr_1}{\abs{r_1^2 \pm 1}^{1/2}} \prod_{i = 2}^n\dfrac{dr_i dN_i}{\abs{N_i}\abs{r_i^2 - N_i}^{1/2}}.
\]
\end{lemma}
\begin{proof}
By Lemma~\ref{lemma: regions are invariant}, $\left(r_1, \pm 1, r_2, N_2, \ldots , r_n, N_n\right) \in \cF_J$ are coordinates for the manifold $Z_+\backslash S_J$.
Thus, the current lemma states how to integrate in coordinates.

We have to integrate over a product domain (of integration) which in the direction of the variable $N_1$ integrates to $1$.
This product domain is
\[
(1, e)\times \cF_J^{\pm}.
\]
As we will see, integrating over this domain allows us to uncouple the variable $N_1$ appropriately.
Indeed, we have
\begin{equation}
\label{eqn: integral over domain times e}
\int_1 ^e\displaystyle\int_{\cF_J^{\pm}}\phi\left(\dfrac{r_1}{\sqrt{\abs{N_1}}}, \dfrac{N_1}{\abs{N_1}}, \ldots , \dfrac{r_n}{\sqrt{\abs{N_1}}}, \dfrac{N_n}{\abs{N_1}}\right) \displaystyle\prod_{i = 1}^n\dfrac{dr_i dN_i}{\abs{N_i}\abs{r_i^2 - N_i}^{1/2}},
\end{equation}
where the outer integral is over $N_1$.
Consider the change of variables, in the inner integral, given by
\begin{align*}
r_j &= \tilde{r}_j \sqrt{\abs{N_1}}, j = 1, \ldots , n\\
N_j &= \tilde{N}_j \abs{N_1}, j = 2, \ldots , n.
\end{align*}
This depends on the fixed value of $N_1$, which is the outer integral variable.
Taking derivatives we get
\begin{align*}
dr_j &= d\tilde{r}_j \sqrt{\abs{N_1}}, j = 1, \ldots , n\\
dN_j &= d\tilde{N}_j \abs{N_1}, j = 2, \ldots , n.
\end{align*}
By direct calculation, we get that for $j=2, \ldots , n$
\begin{align*}
dr_jdN_j &= d\tilde{r}_jd\tilde{N}_j\abs{N_1}^{3/2},\\
\abs{N_j}\abs{r_j^2 - N_j}^{1/2} &= \abs{\tilde{N}_j}\abs{\tilde{r}_j^2 - \tilde{N}_j}^{1/2}\abs{N_1}^{3/2}.
\end{align*}
Consequently, the factor $\abs{N_1}^{3/2}$ cancels and the quotient 
\[
\dfrac{dr_j dN_j}{\abs{N_j}\abs{r_j^2 - N_j}^{1/2}},
\]
is invariant for $j = 2, \ldots , n$.
For $j = 1$ we have
\begin{align*}
\dfrac{dr_1}{\abs{r_1^2 - N_1}^{1/2}} 
&= \dfrac{d\tilde{r}_1 \sqrt{\abs{N_1}}}{\abs{N_1}^{1/2}\abs{(r_1/\abs{N_1}^{1/2})^2 - N_1/\abs{N_1}}^{1/2}} = \dfrac{d\tilde{r}_1}{\abs{\tilde{r}_1^2 \pm 1}^{1/2}}.
\end{align*}
The $\pm$ depends on which part of the fundamental domain $\cF_J^{\pm}$ we are integrating over.
As a consequence, if we relabel the variables again without tildes, we have that integral \eqref{eqn: integral over domain times e} equals
\[
\int_{\cF_J^{\pm}} \phi\left(r_1, \pm 1, r_2, N_2, \ldots , r_n, N_n\right) \dfrac{dr_1}{\abs{r_1^2 \pm 1}^{1/2}}\displaystyle\prod_{i = 2}^n\dfrac{dr_i dN_i}{\abs{N_i}\abs{r_i^2 - N_i}^{1/2}} \cdot \displaystyle\int_1^e\dfrac{dN_1}{\abs{N_1}}.
\]
The latter integral over $N_1$ equals $1$.
Thus integral \eqref{eqn: integral over domain times e} equals
\[
\int_{\cF_J^{\pm}} \phi\left(r_1, \pm 1, r_2, N_2, \ldots , r_n, N_n\right) \dfrac{dr_1}{\abs{r_1^2 \pm 1}^{1/2}}\displaystyle\prod_{i = 2}^n\dfrac{dr_i dN_i}{\abs{N_i}\abs{r_i^2 - N_i}^{1/2}}.
\]
This concludes the proof.
\end{proof}

\subsection{The Weyl integration formula in coordinates}

We now establish the Weyl integration formula in coordinates, which is the technical tool to prove our main result.
We shall first carry out all the technical manipulations and then we shall discuss our motivations for the whole process.
We defer the motivation because we would need to introduce the notation anyway in order to discuss it.

\subsubsection{The integral over the quotient tori}

We begin by understanding the Weyl discriminant as it appears in the Weyl integration formula.
The factor that appears in the Weyl integration formula is
\[
\abs{D(t)} = \abs{D(\lambda_1, \lambda_2)}^{1/2} = \dfrac{\abs{\alpha_1 - \alpha_2}}{\abs{\alpha_1\alpha_2}^{1/2}},
\]
where $\alpha_1, \alpha_2$ are the corresponding eigenvalues of the torus element $t$.
Changing to the $(r, N)$ coordinates this becomes
\[
2\abs{\dfrac{r^2}{N} - 1}^{1/2}.
\]
Under the action of the center $Z_+$ this is invariant.
Indeed,
\[
2\abs{\dfrac{(zr)^2}{z^2N} - 1}^{1/2} = 2\abs{\dfrac{r^2}{N} - 1}^{1/2}.
\]
This observation implies that if, for some index $J$, the function $\phi:S_J\rightarrow\C$ is invariant under the action of $Z_+$ then so is 
\[
2\phi(r, N)\abs{\dfrac{r^2}{N} - 1}^{1/2}.
\]
The same discussion applies for the tori of $G_{\infty}$ because their Weyl discriminant is the product of those of the component tori.
Consequently, Lemma~\ref{lem: integration on quotient} can be used to compute its integral on the quotient space.

\begin{proposition}
\label{prop: equality integrals}
Let $j_1, \ldots, j_n \in \{{\spl}, {\el}\}$ and $J = (j_1, \ldots , j_n)$.
Let $T_J$ be the corresponding maximal tori of $G_{\infty}$.
Let $\phi:S_J\rightarrow\C$ be an integrable function invariant under $Z_+$.
Then
\[
\int_{Z_+\setminus T_{J}} \phi(\CH(t)) \abs{D(t)}dt = 2^n \int_{\cF_J^{\pm}} \phi\left(r_1, \pm 1, r_2, N_2, \ldots , r_n, N_n\right) \dfrac{dr_1dr_2\cdots dr_ndN_2\cdots dN_n}{\abs{N_2\cdots N_n}^{3/2}}.
\]
\end{proposition}

\begin{proof}
 Applying Lemma~\ref{lem: integration on quotient}
 to
 \[
 2^n\phi(r_1, N_1, \ldots , r_n, N_n) \displaystyle\prod_{j = 1}^n \abs{\dfrac{r_j^2}{N_j} - 1}^{1/2},
 \]
 as we explained above, we get that the quotient integral is
 \[
 2^n \int_{\cF_J^{\pm}}
 \phi\left(r_1, \pm 1, r_2, N_2, \ldots , r_n, N_n\right) 
 \abs{r_1^2 \pm 1}^{1/2} \prod_{j = 2}^n \abs{\dfrac{r_j^2}{N_j} - 1}^{1/2} 
 \dfrac{dr_1}{\abs{r_1^2 \pm 1}^{1/2}} \prod_{i = 2}^n\dfrac{dr_i dN_i}{\abs{N_i}\abs{r_i^2 - N_i}^{1/2}}.
 \]
 Upon simplification this becomes
 \[ 2^n
 \displaystyle\int_{\cF_J^{\pm}}
 \phi\left(r_1, \pm 1, r_2, N_2, \ldots , r_n, N_n\right) \dfrac{dr_1dr_2\cdots dr_ndN_2\cdots dN_n}{\abs{N_2\cdots N_n}^{3/2}}. \qedhere
 \]
\end{proof}

We are still one step away from our desired formula.
In our previous computations, we have not used $N_2, \ldots , N_n$ as our variables.
Rather, we used the totally positive fundamental units to get the lattice that corresponds to the units.
Recall that $\beta_1, \ldots ,\beta_{n-1}$ is our system of totally positive fundamental units and that 
\[
\beta_{i, j} = \sigma_j(\beta_i),
\]
for $i = 1, \ldots , n - 1$ and $j = 1, \ldots , n$.

\begin{proposition}
\label{prop: change of variable B}
Define the map $B: \{\pm 1\} \times \R^{n - 1} \longrightarrow \{\pm 1\} \times \R^{n - 1}$ given by
\begin{equation*}
 B(\pm 1, y_1, \ldots , y_{n - 1}) = \left(\pm 1, \pm\dfrac{\beta_{1,2}^{y_1}\cdots \beta_{n-1, 2}^{y_{n-1}}\varepsilon_2}{\beta_{1,1}^{y_1}\cdots \beta_{n-1, 1}^{y_{n-1}}\varepsilon_1}, 
 \cdots,
 \pm\dfrac{\beta_{1,n}^{y_1}\cdots \beta_{n-1, n}^{y_{n-1}}\varepsilon_n}{\beta_{1,1}^{y_1}\cdots \beta_{n-1, 1}^{y_{n-1}}\varepsilon_1}\right).
\end{equation*}
This map is a diffeomorphism onto its image, which consists of the vectors whose entries are nonzero.
\end{proposition}
\begin{proof}
Consider a fix point $ P = (s, N_2, \ldots , N_n)$ where $s\in\{1, -1\}$.
The product
\begin{equation*}
 sN_2\cdots N_n 
\end{equation*}
has a fixed sign, which we denote by $t\in\{1, -1\}$.
With this convention, there exists a unique real number $z> 0$ such that
\[
z^{2n}\cdot (sN_2\cdots N_n) = t\sp^k.
\]
Redistributing the above product, and taking absolute values, we see that
\begin{equation*}
 |(z^2s)(z^2N_2)\cdots (z^2N_n)| = \sp^k.
\end{equation*}
Recall that $\varepsilon$ is the generator of the ideal $\fp^\sk$ with embeddings (into the different real places) denoted by $\varepsilon_1, \dots , \varepsilon_n$.
By the product formula
\[
\abs{\varepsilon_1} \cdots \abs{\varepsilon_n} = \sp^{\sk}.
\]
Comparing the above equalities, we deduce that
\[
\abs{z^2s\varepsilon_1^{-1}} \cdots \abs{z^2N_n\varepsilon_n^{-1}} = 1.
\]
Taking logarithms we get
\[
\log\left(\abs{z^2s\varepsilon_1^{-1}}\right) + \cdots + \log\left(\abs{z^2N_n\varepsilon_n^{-1}}\right) = 0.
\]
By definition, we can find real numbers $y_1, \dots , y_{n-1}$ and a sign $\pm$ such that 
\begin{equation}
\label{eqn: z2 equation}
 z^2s = \pm\beta_{1,1}^{y_1}\cdots \beta_{n-1, 1}^{y_{n-1}}\varepsilon_1
\end{equation}
and for $j = 2, \dots , n$ we have 
\[
z^2N_j = \pm\beta_{1,j}^{y_1}\cdots \beta_{n-1, j}^{y_{n-1}}\varepsilon_j.
\]
From equation \eqref{eqn: z2 equation} we deduce that the uniform sign is $s$, which explains the first entry our map $B$.
Solving for $z^2$ we obtain, for $j = 2, \ldots , n$ that
\[
N_j = s\dfrac{\beta_{1,j}^{y_1}\cdots \beta_{n-1, j}^{y_{n-1}}\varepsilon_j}{\beta_{1,1}^{y_1}\cdots \beta_{n-1, 1}^{y_{n-1}}\varepsilon_1}.
\]
This proves that the map is surjective.

We must now verify it is injective.
To this end, for $j=2, \ldots ,n$ we express everything in terms of the exponential function.
We get
\[
N_j = \exp(y_1\log(\beta_{1,j}/\beta_{1,1} ) + \cdots + y_{n-1}\log(\beta_{n-1,j}/\beta_{n-1,1}))\frac{s\varepsilon_j}{\varepsilon_1}.
\]
This is the same as
\begin{equation*}
 \dfrac{\varepsilon_1 N_j}{s\varepsilon_j} = \exp(y_1\log(\beta_{1,j}/\beta_{1,1} ) + \cdots + y_{n-1}\log(\beta_{n-1,j}/\beta_{n-1,1})).
\end{equation*}
Notice that, by construction, the left hand side is positive.
In particular, we are justified in taking logarithms and we get
\begin{equation*}
 y_1\log(\beta_{1,j}/\beta_{1,1} ) + \cdots + y_{n-1}\log(\beta_{n-1,j}/\beta_{n-1,1}) 
 = \log\left( \dfrac{\varepsilon_1 N_j}{s\varepsilon_j} \right).
\end{equation*}
This is a square system of equations of dimension $n - 1$.
The matrix associated to this system is
\begin{equation*}
 \begin{pmatrix}
 \log\left(\beta_{1,2}/\beta_{1, 1}\right) & \cdots & \log\left(\beta_{1,n}/\beta_{1, 1}\right)\\
 \vdots & \ddots & \vdots\\
 \log\left(\beta_{n-1,2}/\beta_{n-1, 1}\right) & \cdots & \log\left(\beta_{n-1,n}/\beta_{n-1, 1}\right)\\
 \end{pmatrix}
\end{equation*}
Because
\[
\log\left(\beta_{i,j}/\beta_{i, 1}\right) = \log\left(\beta_{i,j}\right) - \log\left(\beta_{i, 1}\right),
\]
we see that
\[
\left(\log\left(\beta_{i,j}/\beta_{i, 1}\right)\right)_{i, j} = \left(\log\left(\beta_{i,j}) - \log(\beta_{i, 1}\right)\right)_{i, j}.
\]
The standard properties of determinants then imply that
\[
\det\left(\log\left(\beta_{i,j}/\beta_{i, 1}\right)\right) = \det\left(\log\left(\beta_{i,j}\right)\right) = R_K,
\]
where $R_K$ is the regulator of $K$.
So, this system always has a unique solution $y_1, \ldots , y_n$.

In the process of the previous analysis we have seen that $B$ has an inverse, obtained by solving the system of equation.
This proves that both $B$ and its inverse, by inspection, are smooth.
This concludes the proof.
\end{proof}

In Proposition~\ref{prop: equality integrals} we have reduced our calculations over a quotient to the integral in coordinates
\[
\int_{\cF_J^{\pm}} \phi\left(r_1, \pm 1, r_2, N_2, \ldots , r_n, N_n\right) \dfrac{dr_1dr_2\cdots dr_ndN_2\cdots dN_n}{\abs{N_2\cdots N_n}^{3/2}}.
\]Our next goal is to change of variables.
We begin with some preparations.
Upon rearranging the coordinates, 
\[
(r_1, \pm, r_2, N_2,\ldots, r_n, N_n)
\]
as
\[
(\pm 1, r_1, \ldots, r_n, N_2,\ldots, N_n).
\]
The previous proposition allowed us to obtain the exponents that correspond to the determinant variables.
We now extend $B$ to a map that also modifies appropriately the trace variables.

\begin{definition}
\label{def: complete map change of coordinates N}
Define the map $N: \{\pm 1\} \times \R^{2n - 1} \longrightarrow \{\pm 1\} \times \R^{2n - 1}$ by sending
\[
(\pm 1, x_1, \ldots , x_n, y_1, \ldots ,y_{n-1})
\]
to
\begin{small}
\[
\left(\pm 1, \dfrac{x_1}{\beta_{1,1}^{y_1/2}\cdots \beta_{n-1, 1}^{y_{n-1}/2}\varepsilon_1 ^{1/2}}, \ldots ,\dfrac{x_n}{\beta_{1,1}^{y_1/2}\cdots \beta_{n-1, 1}^{y_{n-1}/2}\varepsilon_1 ^{1/2}},\pm\dfrac{\beta_{1,2}^{y_1}\cdots \beta_{n-1, 2}^{y_{n-1}}\varepsilon_2}{\beta_{1,1}^{y_1}\cdots \beta_{n-1, 1}^{y_{n-1}}\varepsilon_1}, 
 \cdots,
 \pm\dfrac{\beta_{1,n}^{y_1}\cdots \beta_{n-1, n}^{y_{n-1}}\varepsilon_n}{\beta_{1,1}^{y_1}\cdots \beta_{n-1, 1}^{y_{n-1}}\varepsilon_1}\right).
\]
\end{small}
\end{definition}

\begin{proposition}
\label{prop: change of variable N}
The map $N$ is a diffeomorphism onto its image, which consists of the vectors whose last $n-1$ entries are nonzero.
\end{proposition}
\begin{proof}
Consider a point $(\pm 1, r_1, \ldots , r_n, N_2, \ldots , N_n)$.
The last $n-1$ entries of the map are, by definition, $B$.
This implies there exists unique $y_1, \ldots , y_{n-1}$ with
\begin{equation*}
 B(y_1, \ldots , y_{n-1}) = (\pm 1, N_2, \ldots , N_n).
\end{equation*}
Once these values are known, one can solve for $x_1, \ldots , x_n$ uniquely (as they are just scaled by a vector depending on $\pm 1, y_1, \ldots , y_{n-1}$.

That it is a diffeomorphism onto its image follows from the fact that $B$ is and that the trace variables are scaled by an invertible factor.
\end{proof}

\begin{notation}
Let $N$ be the map defined in Proposition~\ref{prop: change of variable N}.
Let $j_1, \ldots, j_n \in \{{\spl}, {\el}\}$ and $J = (j_1, \ldots , j_n)$.
Define
\[
\cD^{\pm}_J := N^{-1}(\cF^{\pm}_J).
\]
\end{notation}
\begin{remark}
Notice that $\cD^{\pm}_J$ is an open subset of $\{\pm 1\} \times \mathbb{R}^{2n-1}$ and that 
 \begin{equation*}
 \{\cD^{\pm}_J\}_J,
 \end{equation*}
is an open partition of $\{\pm 1\} \times \mathbb{R}^{2n-1}$ up to a set of measure zero.
This set of measure zero corresponds to the singular set (i.e.~the parabolas in the different $(r, N)$ coordinates).
\end{remark}

\begin{proposition}
\label{prop:change of variables N}
Let $j_1, \ldots, j_n \in \{{\spl}, {\el}\}$ and $J = (j_1, \ldots , j_n)$.
Let $T_J$ be the corresponding maximal tori of $G_{\infty}$.
Let $\phi:S_J\rightarrow\C$ be an integrable function invariant under $Z_+$.
Then the integral
\[
\int_{\cF_J^{\pm}} \phi\left(\pm 1, r_1, \ldots, r_n, N_2, \ldots , N_n\right) \dfrac{dr_1dr_2\cdots dr_ndN_2\cdots dN_n}{\abs{N_2\cdots N_n}^{3/2}}, 
\]
equals
\[
\sp^{-k/2}R_K \int_{\cD_J^{\pm}}\phi\left(N(\pm, x_1, \ldots , x_n, y_1, \ldots ,y_{n-1})\right) dx_1dx_2\cdots dx_ndy_1\cdots dy_{n-1}.
\]
\end{proposition}
\begin{proof}
This follows from the change of variable formula applied to the map $B$.
We will understand the determinant variables first.

We saw within the proof of Proposition~\ref{prop: change of variable B} that
\[
N_j = \exp(y_1\log(\beta_{1,j}/\beta_{1,1} ) + \cdots + y_{n-1}\log(\beta_{n-1,j}/\beta_{n-1,1}))\frac{s\varepsilon_j}{\varepsilon_1}.
\]
Using the chain rule we get
\[
dN_j = N_j\left(\log(\beta_{1,j}/\beta_{1,1})dy_1 + \cdots + \log(\beta_{n-1,j}/\beta_{n-1,1})dy_{n-1}\right).
\]
Taking the wedge product (and recalling that we are considering only the associated measure) we get
\[
dN_2\wedge \cdots \wedge dN_{n} = \abs{N_2} \cdots \abs{N_{n}} R_K dy_1\wedge \cdots \wedge dy_{n-1}.
\]
Here, we have used the fact that the wedge product associated to a change of variable brings out the determinant of the associated derivative linear map.
We have calculated this derivative in the proof of Proposition~\ref{prop: change of variable B} and realized it is the regulator.

We now study the trace variables.
From the definition of the map $N$ we have
\[
r_j = \dfrac{x_j}{\beta_{1,1}^{y_1/2}\cdots \beta_{n-1, 1}^{y_{n-1}/2}\varepsilon_1 ^{1/2}}.
\]
To simplify the computations, define
\[
w(y_1, \ldots , y_n) := \dfrac{1}{\beta_{1,1}^{y_1/2}\cdots \beta_{n-1, 1}^{y_{n-1}/2}\varepsilon_1^{1/2}}.
\]
Then
\[
r_j = x_jw(y_1, \ldots y_n).
\]
Taking derivatives we find
\[
dr_j = w(y_1, \ldots ,y_n)\cdot dx_j + x_jdw.
\]
The final change of variables, as it respects the involved differential form, is
\begin{multline*}
dr_1\wedge \ldots \wedge dr_n \wedge dN_2 \wedge \cdots \wedge dN_{n} = \\
\abs{N_2} \cdots \abs{N_{n}} R_K\cdot (w dx_1 + x_1dw) \wedge \cdots \wedge (w dx_n + x_ndw) \wedge dy_1\wedge \cdots \wedge dy_{n-1}.
\end{multline*}
Notice that $dw$ is a linear combination of $dy_1, \ldots, dy_n$ (with coefficients the partial derivatives with respect to $y_1, \ldots, y_{n-1}$).
Thus,
\[
dw\wedge dy_1\wedge \cdots \wedge dy_{n-1} = 0.
\]
We conclude the change of variable formula is
\[
dr_1\wedge \cdots \wedge dr_n \wedge dN_2 \wedge \cdots \wedge dN_{n} 
 = \abs{N_2} \cdots \abs{N_{n}} R_Kw^{n} dx_1\wedge \cdots \wedge dx_n \wedge dy_1\wedge \cdots \wedge dy_{n-1}.
\]
Which is to say, in measures we have
\[
dr_1 \ldots dr_n dN_2 \ldots dN_{n} = \abs{N_2} \cdots \abs{N_{n}} R_K w^{n} dx_1 \ldots dx_n dy_1 \cdots dy_{n-1}.
\]
We conclude that the integrand 
\[
\phi\left(\pm 1, r_1, \ldots, r_2, N_2, \ldots , N_n\right) \dfrac{dr_1dr_2\cdots dr_ndN_2\cdots dN_n}{\abs{N_2\cdots N_n}^{3/2}}
\]
transformed as
\begin{equation*}
 \phi\left(N(\pm, x_1, \ldots , x_n, y_1, \ldots ,y_{n-1})\right)
 \cdot \dfrac{R_K w^n}{\abs{N_2\cdots N_n}^{1/2}} 
 \cdot dx_1dx_2\cdots dx_ndy_1\cdots dy_{n-1}.
\end{equation*}

 Recall from the proof of Proposition~\ref{prop: change of variable B} that 
\begin{align*}
 z^{2n}\abs{N_2 \ldots N_n} & = \sp^k\\
 z^2 & = \beta_{1,1}^{y_1}\cdots \beta_{n-1, 1}^{y_{n-1}}\varepsilon_1.
\end{align*}
Thus $z = w^{-1}$.
It follows that
\[
\dfrac{w^n}{\abs{N_2\cdots N_n}^{1/2}} = \dfrac{1}{z^n\abs{N_2\cdots N_n}^{1/2}} = \dfrac{1}{\sp^{k/2}}.
\]
We thus conclude that the integrand transformed as 
\[
dr_1 \ldots dr_n dN_2 \ldots dN_{n} = \sp^{-k/2}R_K \phi\left(N(\pm, x_1, \ldots , x_n, y_1, \ldots ,y_{n-1})\right) dx_1 \ldots dx_n dy_1 \cdots dy_{n-1}.
\]
This concludes the proof.
\end{proof}

\subsubsection{The interpolation function and $\fP^{\pm}$}

We now connect all our computations with the formulas we have studied in the previous sections.
We begin by recalling that the interpolation function is 
\begin{small}
\begin{multline*}
\theta^{\pm}(r_1, \dots, r_n, y_1, \dots, y_{n-1}) = \\ 
\sp^{\frac{\sk}{2}}\prod_{i = 1}^n\left(2\abs{\frac{r_i^2}{\pm4\beta_{1,i}^{y_1}\cdots \beta_{n-1,i}^{y_{n-1}}\varepsilon_i}-1}^{1/2} g_{\nu_i,1}( r_i, \pm4\beta_{1,i}^{y_1}\cdots \beta_{n-1,i}^{y_{n-1}}\varepsilon_i) + g_{\nu_i,2}(r_i, \pm4\beta_{1,i}^{y_1}\cdots \beta_{n-1,i}^{y_{n-1}}\varepsilon_i)\right),
\end{multline*}
\end{small}
As discussed in Section~\ref{section: blending step}, this function is the product of the orbital integral and the Weyl discriminant as it comes from the volume term.
In particular, it is a function invariant under the action of $Z_+$.

Recall the element
\[
w = \dfrac{1}{\beta_{1,1}^{y_1/2}\cdots \beta_{n-1, 1}^{y_{n-1}/2}\varepsilon_1 ^{1/2}} > 0,
\]
from the proof of Proposition~\ref{prop: change of variable N}.
By the invariance just mentioned, we have
\begin{footnotesize}
\begin{align*}
&\theta^{\pm}(r_1, \dots, r_n, y_1, \dots, y_{n-1})\\
&= \sp^{\frac{\sk}{2}}\prod_{i = 1}^n\left(2\abs{\frac{r_i^2}{\pm4\beta_{1,i}^{y_1}\cdots \beta_{n-1,i}^{y_{n-1}}\varepsilon_i}-1}^{1/2} g_{\nu_i,1}( r_i, \pm4\beta_{1,i}^{y_1}\cdots \beta_{n-1,i}^{y_{n-1}}\varepsilon_i) + g_{\nu_i,2}(r_i, \pm4\beta_{1,i}^{y_1}\cdots \beta_{n-1,i}^{y_{n-1}}\varepsilon_i)\right)\\
&= \sp^{\frac{\sk}{2}}\prod_{i = 1}^n\left(2\abs{\frac{(r_iw)^2}{\pm4w^2\beta_{1,i}^{y_1}\cdots \beta_{n-1,i}^{y_{n-1}}\varepsilon_i}-1}^{1/2} g_{\nu_i,1}( wx_i, \pm4w^2\beta_{1,i}^{y_1}\cdots \beta_{n-1,i}^{y_{n-1}}\varepsilon_i) + g_{\nu_i,2}(x_iw, \pm4w^2\beta_{1,i}^{y_1}\cdots \beta_{n-1,i}^{y_{n-1}}\varepsilon_i)\right).
\end{align*}
\end{footnotesize}
By inspecting the last line of the previous chain of equations and using the definition, we see 
\begin{equation}
\label{eqn: theta invariance under N}
 \theta^{\pm}(r_1, \dots, r_n, y_1, \dots, y_{n-1}) = \theta^{\pm}(N(\pm 1, r_1, \ldots , r_n, y_1, \ldots , y_{n-1})).
\end{equation}
We see that our function is precisely created to capture the appropriate change of variable in the Weyl integration formula.
We now proceed to do the analogous process for $\fP^{\pm}(x, y)$.
Recall that this function is
\begin{align*}
\mathcal{P}^{\pm}(r, y) &:= \frac{1}{2^n}\prod_{i = 1}^n \left( \frac{\abs{r_i^2 \pm 4 \beta_{1,i}^{y_1}\dots \beta_{n-1,i}^{y_{n-1}}\varepsilon_i}^{-1/2}}{\abs{4\beta_{1,i}^{y_1}\dots \beta_{n-1,i}^{y_{n-1}}\varepsilon_i}^{-1/2}} \right).
\end{align*}
By inspection we see this function is invariant under the action of $Z_+$ (which we know must be the case, because this a multiple of the Weyl discriminant in coordinates).
Consequently, using the same $w$ as above and repeating the same process, we deduce
\[
\mathcal{P}^{\pm}(r, y) = \mathcal{P}^{\pm}(N(\pm 1, r, y)).
\]

\subsubsection{The Weyl Integration Formula}

We are now ready to write the Weyl integration formula in our coordinates.
\begin{theorem}
\label{thm: weyl Integration Formula in Coordinates}
Let $h:G_{\infty}\longrightarrow \C$ be a function in $C_c(Z_+\backslash G_{\infty})$, then
\[
\int_{Z_+\backslash G_{\infty}} h(\dot{g}) \dfrac{dg}{dz} = \sp^{-k}R_K\displaystyle\sum_{\pm} \displaystyle\int_{\R^{2n-1}}\theta^{\pm}\left(x_1, \ldots , x_n, y_1, \ldots ,y_{n-1}\right) dx_1\cdots dx_ndy_1\cdots dy_{n-1}.
\]
where $\dot{g}$ denotes the variable of integration in the quotient $Z_+\backslash G_{\infty}$.
\end{theorem}
\begin{proof}
The Weyl integration formula is
\[
\int_{Z_+\backslash G_{\infty}} f(g) dg = \dfrac{1}{2^n}\displaystyle\sum_T \int_{Z_+\backslash T}\int_{T\setminus G_{\infty}} f(g^{-1}tg)dg\abs{D(t)}^2dt,
\]
A set of representatives of maximal tori of $G_{\infty}$ is $T_J$ where $J = (j_1, \ldots , j_n)$ and $j_1, \ldots, j_n \in \{{\spl}, {\el}\}$.
Let $J$ be a fixed subindex.
We can define
\[
\phi_J(t) = \int_{T\setminus G_{\infty}} f(g^{-1}tg)dg.
\]
This is the orbital integral over $T_J$, which is invariant under the action of the center by the choice of our function $f$.
Then Proposition~\ref{prop: equality integrals} implies
\[
\int_{Z_+\backslash T} \phi_J(t)\abs{D(t)}^2 dt = 2^n \int_{\cF_J^{\pm}}
 \phi_J\left(\pm 1, r_1, r_2, N_2, \ldots , r_n, N_n\right) \dfrac{dr_1dr_2\cdots dr_ndN_2\cdots dN_n}{\abs{N_2\cdots N_n}^{3/2}}.
\]
Then, Proposition~\ref{prop:change of variables N} applied to the latter integral implies
\[
2^n \sp^{-k/2}R_K \int_{\cD_J^{\pm}}\phi_J\left(N(\pm, x_1, \ldots , x_n, y_1, \ldots ,y_{n-1})\right) dx_1dx_2\cdots dx_ndy_1\cdots dy_{n-1}.
\]
Substituting in the Weyl integration formula gives 
\begin{multline*}
\int_{Z_+\backslash G_{\infty}} f(\dot{g}) \dfrac{dg}{dz} =\\ \sp^{-k/2}R_K \sum_{\pm} \sum_{J} \int_{\cD_J^{\pm}}\phi_J\left(N(\pm, x_1, \ldots , x_n, y_1, \ldots ,y_{n-1})\right) dx_1dx_2\cdots dx_ndy_1\cdots dy_{n-1}.
\end{multline*}
Notice that the factor $2^n$ cancels.
By definition, if $(\pm 1, r_1, \ldots , r_n, y_1, \ldots , y_{n-1})$ corresponds to a point in $S_J$, and consequently to some element $t(\pm 1, r_1, \ldots , r_n, y_1, \ldots , y_{n-1})$in the tori $T_J$, then
\begin{multline*}
\theta^{\pm}(r_1, \ldots , r_n, y_1, \ldots , y_{n-1}) = \\ \sp^{-k/2}\phi_J(t(\pm 1, r_1, \ldots , r_n, y_1, \ldots , y_{n-1}))\abs{D(t(\pm 1, r_1, \ldots , r_n, y_1, \ldots , y_{n-1}))}.
\end{multline*}
In particular,
\[
\sum_{\pm} \displaystyle\sum_{J}\displaystyle\int_{\cD_J^{\pm}}\phi_J\left(N(\pm, x_1, \ldots , x_n, y_1, \ldots ,y_{n-1})\right) dx_1dx_2\cdots dx_ndy_1\cdots dy_{n-1},
\]
equals
\[
\sp^{-k/2}\displaystyle\sum_{\pm} \sum_{J} \int_{\cD_J^{\pm}}\theta^{\pm}\left(N(\pm, x_1, \ldots , x_n, y_1, \ldots ,y_{n-1})\right) dx_1dx_2\cdots dx_ndy_1\cdots dy_{n-1}
\]
Using the invariance \eqref{eqn: theta invariance under N} and that $\cD_J^{\pm}$ are a partition of $\R^{2n-1}$ up to a set of measure zero, gives the desired result.
\end{proof}

\subsection{The trace of the Trivial Representation}
\label{subsec: The Weyl integration formula and the trace of the Trivial Representation}

We now compute the trace of the trivial representation.
Following Langlands' calculation \cite[p.~644]{LanBE04}, the trace of the \emph{trivial representation} at the non-archimedean place $\fp$ is 
\[
\int_{G(K_\fp)} f(g) dg = \sp^{\sk/2}\frac{1-\sp^{-(\sk+1)}}{1-\sp^{-1}}.
\]
By our choice of the test function, the contribution at all other non-archimedean places is 1.
So we have 
\begin{equation}
\label{spectral1}
\Tr({\bf1}(f)) = \sp^{\sk/2}\frac{1-\sp^{-(\sk+1)}}{1-\sp^{-1}}\Tr({\bf1}(f_\infty)),
\end{equation}
where $\bf1$ is the trivial representation and $f$ is the test function.
We can break the integral as a product of local integrals because the function factorizes and the ad\`{e}lic measure we have picked is the restricted product measure.

The above formula explains the need to study the archimedean trace of the trivial representation.
Theorem~\ref{thm: weyl Integration Formula in Coordinates} implies
\begin{equation*}
 \Tr({\bf1}(f_\infty)) = \sp^{-k}R_K\displaystyle\sum_{\pm} \displaystyle\int_{\R^{2n-1}}\theta^{\pm}\left(x_1, \ldots , x_n, y_1, \ldots ,y_{n-1}\right) dx_1dx_2\cdot dx_ndy_1\cdots dy_{n-1}.
\end{equation*}
Substituting this into \eqref{spectral1} proves Theorem~\ref{main theorem section 9}\ref{prop: trivial rep value}.

\subsection{The trace of the special representation at the archimedean places}
\label{subsec: The trace of the special representation at the archimedean places}

\subsubsection{The case of $\mathbb{R}$}
Langlands computed the character $\theta_{\pm, \pm}$ of the representation $\pi_{\pm, \pm}$ in \cite[pp.~630--631]{LanBE04}.
The values of the character of the representation are given by the regions in which the parabola and the axis divide the plane (notice that this makes sense, because characters only depend on conjugacy class, so they are genuine functions of the variables $(r, N)$).

\begin{lemma}
\label{lemma: char one real place}
The character $\theta_{+,+}$ satisfy the following properties
 \begin{enumerate}[label = \textup{(\roman*)}]
 \item The character is equal to $0$ in region $S_{\el}$ (that is, in the elliptic torus).
 \item In $S_{\spl}$, the character is $\abs{\frac{r^2}{N}-1}^{-1/2}$.
 More precisely, for $\pi_{+, +}$ the constant $C_{+,+}$ equals $1$ in regions $S_{\spl}$.
 \end{enumerate}
\end{lemma}

\begin{proof}
 See the discussion in \cite[pp.~630--631]{LanBE04}.
\end{proof}

Langlands and Altu\u{g} perform change of variables related to quotienting by $Z_+$, and reinterpret the above result in terms of the trace $r$.
We avoid doing this, as the action of $Z_+$ enters the picture slightly differently for us.

\subsubsection{Totally Real Fields}
\label{subsubsec: totally real field sec 9}
Starting with the representation $\chi_{+,+}\otimes\cdots\otimes\chi_{+,+}$ of $T_{\spl}\times \cdots \times T_{\spl}$ we can obtain the induced representation $\pi_{+,+}\otimes\cdots\otimes\pi_{+,+}$.
We now prove the analogue of Lemma~\ref{lemma: char one real place} in our setting.

\begin{proposition}
\label{prop: big CH lemma}
Let $\CH:= \ch \otimes \cdots \otimes \ch: G_{\infty} \longrightarrow \R^2\times\cdots\times \R^2$ be given by
\[
\CH(\gamma_1,\ldots, \gamma_n) = (r_1, N_1, \dots, r_n, N_n).
\]
Let $\theta_{+,+ , \dots , +,+}$ be the character of $\pi_{+,+}\otimes \cdots \otimes\pi_{+,+}$ and $C$ be the indicator function defined in Definition~\ref{indicator function}.
Then
\[
\CH(\theta_{+,+ , \dots ,+,+}(r_1, N_1, \dots, r_n, N_n)) = C(r_1, N_1, \dots, r_n, N_n) \prod_{i = 1}^n \abs{\frac{r_i^2}{N_i}-1}^{-1/2}.
\]
\end{proposition}

\begin{proof}
We begin with the observation that the trace of a tensor product is the product of the traces.
Therefore, the character splits into the characters of each individual archimedean place.

Lemma~\ref{lemma: char one real place} asserts that at each place the character is a multiple of $\abs{\frac{r_i^2}{N_i}-1}^{-1/2}$.
This constant multiple depends on the region and takes a value in $\{0, 1\}$.
The result follows by multiplying over all real places.
\end{proof}

We now complete the proof of Theorem~\ref{main theorem section 9}\ref{thm: Weyl Int special functions}.
Our objective is to find the value of $\Tr(\pi_{+,+}\otimes\cdots\otimes\pi_{+,+}(f))$.
The steps are the same as in \S\ref{subsec: The Weyl integration formula and the trace of the Trivial Representation}, via Theorem~\ref{thm: weyl Integration Formula in Coordinates}, except that we substitute 
\[
h(r_1,\ldots , r_n, N_2,\ldots , N_n) = \theta_{+,+,\ldots , +,+}(r_1,\ldots , r_n, N_2,\ldots , N_n)\theta^{\pm}(r_1,\ldots , r_n, N_2,\ldots , N_n).
\]
From this we get
\[
\int_{Z_+\backslash G_{\infty}} f(g)\theta_{+, +, \dots,+, +}(g) dg = \sp^{-\sk}2^nR_K\sum_{\pm}\int_{\R^{2n-1}}C(x, y)\theta^{\pm}(x, y)\mathcal{P}^{\pm}(x,y)dx dy.
\]
where $C(x, y)$ is the indicator function given by $C$ in the variables $x, y$.

To explain the appearance of $\mathcal{P}^{\pm}(x,y)$ and the factors of $2$, recall that Proposition~\ref{prop: big CH lemma} introduces 
\[
\prod_{i = 1}^n \abs{\frac{r_i^2}{N_i}-1}^{-1/2}.
\]
Using Definition~\ref{def: P pm}, this product in the variables $(x, y)$ is precisely $2^n\mathcal{P}^{\pm}(x,y)$.

\subsection{The computation of the one dimensional traces}

We now put everything together.
In Theorem~\ref{thm: condn on k and h} we proved that if \ref{ass: class number} is not satisfied, then the regular elliptic part of the trace formula for the test function $f$ is zero.
We now prove that on the spectral side a similar implication holds.
More precisely, we prove the following result.

\begin{theorem}
\label{thm: ass holds or spectral zero}
The following dichotomy holds for the test function $f$
\begin{enumerate}[label = \textup{(\roman*)}]
 \item if \ref{ass: class number} holds then
 \begin{align*}
 \displaystyle\sum_{\chi} \Tr((\boldsymbol{1}\otimes\chi)(f)) &= 2^{n-1}h_K\Tr(f),\\
 \displaystyle\sum_{\chi} \Tr((\xi_0\otimes\chi)(f)) &= 2^{n-1}h_K\Tr(\xi_0(f)).
 \end{align*}
 \item if \ref{ass: class number} does not hold, then
 \begin{align*}
 \displaystyle\sum_{\chi} \Tr((\boldsymbol{1}\otimes\chi)(f)) &= 0,\\
 \displaystyle\sum_{\chi} \Tr((\xi_0\otimes\chi)(f)) &= 0.
 \end{align*}
\end{enumerate}
In each case, the sum is over characters $\chi:\A_K\rightarrow \Q$ which are trivial on $K^*\displaystyle\prod_{v}\R^*$.
\end{theorem}

The result follows from several observations.
We remind the reader that as a character on $\GL(2, \A_K)$, it is possible to identify $\chi$ with $\chi\circ \det$.
We begin with a local lemma.

\begin{lemma}
\label{lem: nonspherical implies zero}
Let $\fq$ be any finite prime of $K$ and $\chi_{\fq}:G(K_{\fq})\rightarrow\mathbb{C}$ be a character.
If $\chi_{\fq}$ is not spherical then
\[
\int_{G(K_{\fq})} f_{\fq}(x)\chi_{\fq}(x) dx = 0.
\]
\end{lemma}

\begin{proof}
For notational simplicity, we write $\chi=\chi_{\fq}$ in this proof.
Recall that a character $\chi: K_\fq^* \rightarrow \C^*$ is called spherical if its restriction to the ring of units $\cO_{K_\fq}^*$ is the identity.
Therefore, if $\chi$ is not spherical then there exists $a\in\cO_{K_\fq}^*$ such that $\chi(a)\neq 1$.
Using this element, define the matrix
\[
A := \begin{pmatrix}
 a & 0\\
 0 & 1\\
 \end{pmatrix}
 \in \GL(2, K_\fq).
\]
We have $\abs{\det(A)}_\fq = \abs{a}_\fq = 1$.
Recall the definition of $f_{\fq}$ from Definition~\ref{defi: ffq(k) defi}
\[
f_\fq = \begin{cases}
 f_{\fq}^{(0)}
 = \mathbbm{1}\left(X\in\Mat(\cO_{K_{\fq}})\mid \abs{\det(X)}_{\fq} = 1\right) & \text{ if }\fq\neq\fp\\
 \Norm_{K}(\fp)^{-k/2}f_{\fp}^{(k)}
 =\Norm_{K}(\fp)^{-k/2}\mathbbm{1}\left(X\in\Mat(\cO_{K_{\fp}})\mid\; \abs{\det(X)}_{\fp} = \Norm_{K}(\fp)^{-k}\right) & \text{ if }\fq=\fp;
 \end{cases}
\]
we can thus conclude $f_\fq(xA) = f_\fq(x)$.
Using this we have
 \begin{align*}
 \displaystyle\int_{G(K_{\fq})} f_{\fq}(x)\chi(x) dx
 &= \displaystyle\int_{G(K_{\fq})} f_{\fq}(xA)\chi(xA) dx\\
 &= \displaystyle\int_{G(K_{\fq})} f_{\fq}(x)\chi(x)\chi(a) dx\\
 &= \chi(a)\displaystyle\int_{G(K_{\fq})} f_{\fq}(x)\chi(x) dx.
 \end{align*}
Notice we have used the invariance of the Haar measure.
In particular,
\[
(1 - \chi(a)) \displaystyle\int_{G(K_{\fq})} f_{\fq}(x)\chi(x)dx = 0.
\]
The result follows.
\end{proof}

The sum of the traces of the representations constituting the residual spectrum is
\begin{equation*}
 \sum_{\chi} \Tr((\boldsymbol{1}\otimes\chi)(f)).
\end{equation*}
We see that for each individual term we can apply Lemma~\ref{lem: nonspherical implies zero}.
Concretely, we have
\[
\Tr((\boldsymbol{1}\otimes\chi)(f)) 
 = \prod_\fq \int_{G(K_{\fq})} f_{\fq}(x_\fq)\chi_{\fq}(x_\fq) dx_\fq \times \int_{Z_+\backslash G_\infty} f_{\infty}(x_\infty)\chi_{\infty}(x_\infty) dx_\infty.
\]
Lemma~\ref{lem: nonspherical implies zero} asserts that only those characters that are spherical at every finite place may have non-zero contribution to the sum.
In other words, the sum really runs over the characters of 
\[
K^*\displaystyle\prod_{v}\R^*_{>0}\displaystyle\prod_\fq \cO_{K_\fq}^*\backslash \A_K.
\]
We know this quotient is finite and equals the narrow class group.
In particular, the sum really runs over the characters of the narrow class group and hence it is finite.

The finiteness of this sum allows us to exchange the sum and the integral, that is,
\[
\sum_{\chi} \Tr((\boldsymbol{1}\otimes\chi)(f)) 
 = \sum_{\chi} \int_{Z_+\backslash G(\A_K)} f(g)\chi(g) dg
 = \int_{Z_+\backslash G(\A_K)} f(g) \sum_{\chi}\chi(\det(g)) dg.
\]
Analyzing the character sum carefully leads to the proof of Theorem~\ref{thm: ass holds or spectral zero}.

\begin{proof}[Proof of Theorem~\ref{thm: ass holds or spectral zero}]
We carry out a detailed argument for the trivial representation.
The argument for the special representation is the same.

\underline{Part (b):}
Define the set
\[
U := \prod_{\fq\neq \fp} \cO_\fq^* \times \pi_\fp^k\cO_\fp^* \times \displaystyle\prod_v \R^*,
\]
where $\pi_\fp$ is a uniformizer at $\fp$.

By definition of $f$, if for some $g\in \GL(2, \A_K)$ we have $\det(g)\notin U$ then $f(g) = 0$.
Indeed, if $g = (g_\fp)$ then
\[
\det(g) = (\det(g_{\fp}))_{\fp},
\]
and thus for some of its finite components the condition defining the indicator is not satisfied.
In this case, 
\[
\sum \chi(\det(g)) = \sum \chi(\fp^k) = 0,
\]
On the other hand, every id{\`e}le of $U$ corresponds to the ideal $\fp^k$.
Hence, for $g$ in the support of $f$, we have 
\[
\sum \chi(\det(g)) = \sum \chi(\fp^k) = 0,
\]
because $\fp^k\neq 1$ in the class group (i.e.~not principal) since \ref{ass: class number} does not hold.
Thus, for all $g$, the product 
\[
f(g)\displaystyle\sum_{\chi}\chi(\det(g)) = 0.
\]
This proves (b).
We emphasize that we have used that the sum runs over characters that come from the narrow class groups and thus their value at an id{\`e}le depends only on their finite part.

\underline{Part (a):} The same analysis we did above proves that, restricted to $U$, we have
\[
f(g)\displaystyle\sum_{\chi}\chi(\det(g)) = h_K^+f(g),
\]
because this time $\fp^k = 1$ in the narrow class group.
In here, $h_K^+$ is the narrow class number.
Upon integrating, we get
\[
\sum_{\chi} \Tr((\boldsymbol{1}\otimes\chi)(f)) 
 = \int_{Z_+\backslash G(\A_K)} f(g) \sum_{\chi}\chi(\det(g)) dg
 = h_K^+ \int_{Z_+\backslash U} f(g) dg
 = h_K^+ \Tr((\boldsymbol{1}\otimes\chi)(f)).
\]
Finally, we recall that the narrow class number satisfies (see \cite[Proposition~2.1]{childress2008CFT})
\[
h_K^+ = \dfrac{h_K2^{r_1}}{[U_K:U_K^+]},
\]
where $r_1$ is the number of real embeddings and $[U_K:U_K^+]$ refers to the index of the totally positive units inside the units.
Under the conditions we have put to our number field, we have $[U_K:U_K^+] = 2$.
This completes the proof of (a).
\end{proof}

Finally we conclude with the main result.
 
\begin{theorem}
\label{THE THEOREM}
The following equality holds
\[
\overline{\Sigma_0}(f) = 
\displaystyle\sum_{\chi} \Tr((\boldsymbol{1}\otimes\chi)(f))
- 2 \displaystyle\sum_{\chi} \Tr((\xi_0\otimes\chi)(f)) +\eqref{eq:C(1)new}+\eqref{C2new}.
\]
\end{theorem}
\begin{proof}
By Theorem~\ref{thm: ass holds or spectral zero} we have
\[
\sum_{\chi} \Tr((\boldsymbol{1}\otimes\chi)(f)) = 2^{n-1}h_K\Tr(f),
\]
and by Theorem~\ref{main theorem section 9},
\[
\Tr({\bf1}(f)) = {\sp}^{-\sk/2}R_K\frac{1-{\sp}^{-(\sk+1)}}{1-{\sp}^{-1}} \sum_{\pm} \int_{\R^{2n -1}} \theta^{\pm}(x, y) dxdy.
\]
Hence, by substituting, we get
\[
\sum_{\chi} \Tr((\boldsymbol{1}\otimes\chi)(f)) = 2^{n-1}h_KR_K{\sp}^{-\sk/2}\frac{1-{\sp}^{-(\sk+1)}}{1-{\sp}^{-1}} \sum_{\pm} \int_{\R^{2n -1}} \theta^{\pm}(x, y) dxdy,
\]
which we recognize as \eqref{eq:triv res new}.
In exactly the same way, Theorem~\ref{thm: ass holds or spectral zero} implies
\[
\displaystyle\sum_{\chi} \Tr((\xi_0\otimes\chi)(f)) = 2^{n-1}h_K\Tr(\xi_0(f))
\]
and by Theorem~\ref{main theorem section 9},
\[
\Tr(\xi_0(f)) = (\sk+1)\sp^{-\sk}2^{n-2} R_K \sum_{\pm}\int_{\R^{2n-1}}C(x, y)\theta^{\pm}(x, y)\mathcal{P}^{\pm}(x,y)dx dy
\]
Consequently,
\[
\sum_{\chi} \Tr((\xi_0\otimes\chi)(f)) = \dfrac{1}{2}(\sk+1)\sp^{-\sk}4^{n-1} h_KR_K \sum_{\pm}\int_{\R^{2n-1}}C(x, y)\theta^{\pm}(x, y)\mathcal{P}^{\pm}(x,y)dx dy,
\]
which we recognize as $-\dfrac{1}{2}$ of \eqref{eq:spec residue2new}.

Finally, by Proposition~\ref{prop: computation of 23 + 27 +´25 + 29} we have
\[
\overline{\Sigma_0}(f) = \eqref{eq:triv res new}+\eqref{eq:spec residue2new}+\eqref{eq:C(1)new}+\eqref{C2new}.
\]
Using Theorem~\ref{main theorem section 9} above, we get the result.
\end{proof}

\appendix

\section{Kloosterman Sum Calculations}
\label{Appendix_A}
In this appendix, we compute the Kloosterman sums and provide a proof of Theorem~\ref{main theorem of the kloosterman section}.
For the calculations in this section, we assume that $K$ is a number field satisfying \ref{ass: class number} and \ref{ass: split}.
For the convenience of the reader, we repeat the statement below.
\begin{theorem*}
Let $K$ be a number field satisfying \ref{ass: split} and let $(\fp,\sk)$ be the pair of prime ideal and corresponding positive integers satisfying \ref{ass: class number} fixed in \S\ref{sec: Definition of the test function and its implications}.
Write $\mathsf{p}=\Norm_{K}(\fp)$.
Let $u\in \cO_K^*$ be a unit.
For a complex parameter $z$ with $\Re(z) > 1$,
\[
\D(z,u)=4^n\frac{\zeta_K(2z)}{\zeta_K(z+1)}\cdot\frac{1-1/\sp^{z(\sk+1)}}{1-1/\sp^{z}}.
\]
Hence, $\D(z, u)$ admits an analytic continuation to a meromorphic function in the whole complex plane with a pole at $z=\frac12$.
\end{theorem*}

\begin{proof}
The assertion follows from Lemma~\ref{lemma:kl1} and Propositions~\ref{kloosc1},~\ref{kloosc3},~\ref{kloosc2},~\ref{kloosc4}, and~\ref{more kloosc} whose proofs will occupy the remainder of this section.
\end{proof}

The remainder of this section is dedicated to proving the aforementioned propositions.

\subsection{Factorization into local factors.}

At each prime ideal $\fq$ (which may or may not equal $\fp$), we fix a uniformizer $\pi_{\fq}$.
When $\fq\mid 2$, we assume that $\pi_{\fq}=2$.

\begin{proposition}
\label{chinese remainder theorem}
Let $(\fp,\sk)$ be the pair of prime ideal and corresponding positive integer satisfying \ref{ass: class number} fixed in \S\ref{sec: Definition of the test function and its implications}.
Let $\fq$ be any prime ideal of $\cO_K$ (possibly equal to $\fp$).
Fix a generator $\rho$ of $\fp^{h_\fp}$ and fix $\sk' = \sk/h_{\fp}$.
Let $\fa,\fd$ be ideals of $\cO_K$ and let $u$ be a unit.
An integer class $x$ modulo $4\fa\fd^{2}$ satisfies the congruence conditions for $K_{\fa,\fd}(u)$ if and only if for \emph{each} prime $\fq$ its restriction $x_{\fq}$ modulo $\fq^{\val_{\fq}(4\fa\fd^2)}$ satisfies the following conditions:
 \begin{itemize}
 \item If $\fq\nmid 2$, then 
\[
x_{\fq}^2\equiv 4u\rho^{\sk'}\pmod{\fq^{\val_{\fq}(\fd^2)}},
\]
\item If $\fq\mid 2$, then 
\[
x_{\fq}^2\equiv 4u\rho^{\sk'} \pmod{\fq^{\val_{\fq}(\fd^2)}},
\]
and 
\[
\frac{x_{\fq}^2 - 4u\rho^{\sk'}}{\pi_{\fq}}\equiv 0, 1 \pmod{4}.
\]
\end{itemize}
\end{proposition}

\begin{proof}
This is the Chinese Remainder Theorem.
\end{proof}

Based on the above proposition we introduce the following local Kloosterman sum.

\begin{definition}
With notation as in Proposition~\ref{chinese remainder theorem},
For $\rho$ a generator of $\fp^{h_\fp}$ and non-negative integers $\upsilon,r$ define
\[
\widetilde{K}_{\fq^\upsilon, \fq^r}(u) := \begin{cases}
\displaystyle\sum_{\substack{\mu\mod \fq^{\upsilon+2r}\\
\mu^2\equiv 4u\rho^{\sk'}\mod \fq^{2r}}} \binom{\mu^2 - 4u\rho^{\sk'}, \fq^r}{\fq^\upsilon} \quad \text{ if } \fq\nmid 2\\
\displaystyle\sum_{\substack{\mu\mod \fq^{2+\upsilon+2r}\\\mu^2\equiv 4u\rho^{\sk'}\mod \fq^{2r}\\ \frac{\mu^2 - 4u\rho^{\sk'}}{\pi_{\fq}^{2r}}\equiv 0,1 \pmod{4} }} \binom{\mu^2 - 4u\rho^{\sk'}, \fq^r}{\fq^\upsilon} \quad \text{ if } \fq\mid 2.
\end{cases}
\]
\end{definition}

\begin{remark}
The definitions coincide at the primes above 2 with $K_{\fq^\upsilon, \fq^r}(u)$.
However, for the primes not above $2$, the definition is slightly different because the modulus in $K_{\fq^\upsilon, \fq^r}(u)$ is still $4\fq^{\upsilon + 2r}$, which cannot be recovered as a local object because of the factor $4$.
More precisely, when $\fq\nmid 2$ the corresponding local object arise without the factor of $4$ because it is a unit.
\end{remark}

\begin{proposition}\label{factoring kloosterman}
With notation as above, suppose that the ideals $\fa$ and $\fd$ have the following factorization
\[
\fa = \prod \fq^{\upsilon_{\fq}}, \quad \fd = \prod\fq^{r_{\fq}}.
\]
Then
\[
K_{\fa, \fd}(u) = \prod_{\fq} \widetilde{K}_{\fq^{\upsilon_{\fq}}, \fq^{r_{\fq}}}(u).
\]
\end{proposition}

\begin{proof}
This is a restatement of Proposition~\ref{chinese remainder theorem}.
\end{proof}

With this at hand, we make the following definition.

\begin{definition}
Let $\fq$ be a prime ideal of $\cO_K$ and let $z$ be a complex parameter with $\Re(z)>1$.
Define
\[
\D_{\fq}(z,u):= 
\sum_{r=0}^{\infty}\frac{1}{\Norm_K(\fq)^{r(2z+1)}} \sum_{\upsilon=0}^{\infty}\frac{\widetilde{K}_{\fq^{\upsilon}, \fq^{r}}(u)}{\Norm_K(\fq)^{\upsilon(z+1)}}.
\]
\end{definition}

We now record some useful properties of $\D(z,u)$ and $\D_{\fq}(z,u)$ which we will be used without further mention.

\begin{proposition}
\label{abs conv}
Let $\fq$ be any prime ideal of $\cO_K$.
The series defining $\D(z, u)$ and $\D_{\fq}(z, u)$ converge absolutely and uniformly in compact subsets of the right half plane $\Re(z) > 1$.
\end{proposition}

\begin{proof}
For ideals $\fa$ and $\fd$ we have the bound
\[
\abs{K_{\fa, \fd}(u)} \le 4^n\Norm_K(\fa)\Norm_K(\fd)^2,
\]
where $n=[K:\Q]$.
Writing $\Norm_{K}(\fq) = \sq$, we can analogously show that

\[
\abs{\widetilde{K}_{\fq^\upsilon, \fq^r}(u)} \leq \begin{cases}
 \sq^{\upsilon + 2r} &\text{ if } \fq\nmid 2\\
 \sq^{2 + \upsilon + 2r} & \text{ if } \fq\mid 2.
\end{cases}
\]
Writing $z=\sigma + it$, we have
\begin{align*}
 \sum_{\fd}\frac{1}{\Norm_K(\fd)^{2z+1}} \sum_{\fa}\frac{\abs{K_{\fa,\fd}(u)}}{\Norm_K(\fa)^{z+1}}
 &\le \sum_{\fd}\frac{1}{\Norm_K(\fd)^{2\sigma+1}} \sum_{\fa}\frac{4^n\Norm_K(\fa)\Norm_K(\fd)^2}{\Norm_K(\fa)^{\sigma+1}}\\
 &\le 4^n\sum_{\fd}\frac{1}{\Norm_K(\fd)^{2\sigma-1}} \sum_{\fa}\frac{1}{\Norm_K(\fa)^{\sigma}}\\
 &= 4^n\zeta_K(2\sigma-1)\zeta_K(\sigma).
\end{align*}
Notice that for $\sigma > 1$ we have $2\sigma - 1 > 1$, so the bound implies the absolute convergence and uniformity in compact sets.
The result for $\D_{\fq}(z, u)$ is completely analogous.
\end{proof}

\begin{lemma}
\label{lemma:kl1}
The following local Euler decomposition holds for a complex parameter $z$ with $\Re(z) > 1$, i.e.,
\[
\D(z,u)= \prod_{\fq}\D_{\fq}(z,u).
\] 
Further, the convergence is absolute and uniform in compact sets.
\end{lemma}

\begin{proof}
This follows from uniqueness of factorization into prime ideals of $\cO_K$ and Proposition~\ref{factoring kloosterman}.
By Lemma~\ref{abs conv} we have the convergence of the product absolutely and uniformly in compact sets.
\end{proof}

\subsection{The primes not above $2$}

Using the definition of the modified Hilbert symbol we have
\[
 \widetilde{K}_{\fq^\upsilon, \fq^r}(u) 
 = \sum_{\substack{\mu\mod \fq^{\upsilon+2r}\\\mu^2\equiv 4u\rho^{\sk'}\mod \fq^{2r}}} \binom{\mu^2 - 4u\rho^{\sk'}, \fq^r}{\fq^\upsilon}
 = \sum_{\substack{\mu\mod \fq^{\upsilon+2r}\\\mu^2\equiv 4u\rho^{\sk'}\mod \fq^{2r}}} \binom{(\mu^2 - 4u\rho^{\sk'})\pi_{\fq}^{-2r}}{\fq}_{rH}^\upsilon.
\]

We quote the following facts, which will be used repeatedly in all the subsequent computations.
\begin{enumerate}[label = (\textbf{Fact}~\arabic*) ]
\item
\label{fact1}
When $\upsilon = 0$, this is a sum of 1's, so the value is the number of terms in the sum.
\item
\label{fact2}
When $\upsilon\neq 0$, the only thing that matters from $\upsilon$ is its parity.
\end{enumerate} 

The following lemma is proven over $\Q$ in \cite{LanBE04}.
We provide a proof in our case.
\begin{lemma}
\label{Langlands lemma}
Let $\fq\nmid 2$ and let $m$ be any rational integer relatively prime to $\fq$.
Then
\[
\sum_{x\pmod{\fq}}\binom{x^2 - m}{\fq} = -1.
\]
\end{lemma}

\begin{proof}
The number of $y$ such that $x^2 - m = y^2$ is given by $\binom{x^2 - m}{\fq} + 1$.
Hence, the number of solutions modulo $\fq$ to the curve $x^2 - m \equiv y^2 \pmod{\fq}$ is
\[
\sum_{x\pmod{\fq}}\left(\binom{x^2 - m}{\fq} + 1\right).
\]
On the other hand, factoring the curve as
\[
(x -y)(x + y) = m,
\]
we can parameterize all solutions as $x = \frac{a + a^{-1}m}{2}$ and $y = \frac{a - a^{-1}m}{2}$ for some $a$ which is invertible modulo $\fq$.
Hence, there are $\Norm_K(\fq) - 1$ solutions.
The result follows.
\end{proof}

\subsubsection{Case \texorpdfstring{$\fq\nmid 2\rho$}{}}
We begin with the computations of $\widetilde{K}_{\fq^\upsilon, \fq^r}(u)$.

\begin{lemma}
\label{q odd q not rho Kvals}
Let $(\fp,\sk)$ be the pair of prime ideal and corresponding positive integer satisfying \ref{ass: class number} fixed in \S\ref{sec: Definition of the test function and its implications}.
Set ${\sk'} = \sk/h_{\fp}$ and denote by $\rho$ the generator of $\fp^{h_{\fp}}$.
Let $\fq$ be a prime ideal of $\cO_K$ such that $\sq = \Norm_K(\fq)$ and $\fq\nmid 2\rho$, i.e.~$\fq$ is coprime to $2$ and $\fp$.
Then the values of $\widetilde{K}_{\fq^\upsilon, \fq^r}(u)$ are given as follows:
\begin{enumerate}[label = \textup{(\roman*)}]
\item \label{case 1 q not divide rho} Case $\upsilon= 0, r = 0$: 
 $\widetilde{K}_{1, 1}(u) = 1$
\item \label{case 2 q not divide rho} Case $\upsilon= 0, r > 0 $:
 $\widetilde{K}_{1, \fq^r}(u) = 1 + \binom{u\rho^{\sk'}}{\fq}$
\item \label{case 3 q not divide rho} Case $\upsilon> 0$ even , $r = 0$: $\widetilde{K}_{\fq^\upsilon, 1}(u) = \sq^{\upsilon}-\sq^{\upsilon-1}\left(1+\binom{u\rho^{\sk'}}{\fq}\right)$
\item \label{case 4 q not divide rho} Case $\upsilon> 0$ odd , $r = 0$: $\widetilde{K}_{\fq^\upsilon, 1}(u) = -\sq^{\upsilon- 1}$
\item \label{case 5 q not divide rho} Case $\upsilon> 0$ even , $r > 0$: $\widetilde{K}_{\fq^\upsilon, \fq^r}(u) = \sq^{\upsilon-1}(\sq-1)\left( 1 + \binom{\mu\rho^{\sk'}}{\fq} \right)$
\item \label{case 6 q not divide rho} Case $\upsilon> 0$ odd , $r > 0$: $\widetilde{K}_{\fq^\upsilon, \fq^r}(u) = 0$
 \end{enumerate}
\end{lemma}

\begin{proof} We divide into the cases as in the statement.

\textit{\underline{Case $\upsilon= r = 0$.}}
Then,
\begin{equation}
\begin{split}
\widetilde{K}_{(1),(1)}(u) 
&= \sum_{\substack{\mu\mod \fq^{0}\\\mu^2\equiv 4u\rho^{\sk'}\mod \fq^{0}}} \binom{(\mu^2 - 4u\rho^{\sk'})\pi_{\fq}^{0}}{\fq}^0 
=1
\end{split}
\end{equation}
because the sum goes over the representatives of $\cO_K/(1)=1$.
This settles \ref{case 1 q not divide rho}.

\textit{\underline{Case $\upsilon=0$, $r >0$.}}
Then,
\begin{equation}
\begin{split}
\widetilde{K}_{(1),(\fq^r)}(u)
&= \sum_{\substack{\mu\mod \fq^{2r}\\\mu^2\equiv 4u\rho^{\sk'}\mod \fq^{2r}}} \binom{(\mu^2 - 4u\rho^{\sk'})\pi_{\fq}^{-2r}}{\fq}^0\\
&= \sum_{\substack{\mu\mod \fq^{2r}\\\mu^2\equiv 4u\rho^{\sk'}\mod \fq^{2r}}}1\\
 &=1 + \binom{u\rho^{\sk'}}{\fq}.
\end{split}
\end{equation}
We used \ref{fact1} above because $\upsilon= 0$.
Notice for the third equality we used the representatives run over $\cO_K/\fq^{2r}$ and that the number of solutions to the congruence
\[
\mu^2\equiv 4u\rho^{\sk'}
\]
is given by $1 + \binom{u\rho^{\sk'}}{\fq}$ since $\fq\nmid 2$ and by the use of Hensel's Lemma.
This settles \ref{case 2 q not divide rho}.

\textit{\underline{Case $\upsilon>0$, $r=0$:}}
Then,
\begin{align*}
\widetilde{K}_{(\fq^\upsilon),(1)}(u)
&= \sum_{\substack{\mu\mod \fq^{\upsilon}\\\mu^2\equiv 4u\rho^{\sk'}\mod \fq^{0}}} \binom{\mu^2 - 4u\rho^{\sk'}}{\fq}^\upsilon\\
&= \sum_{\mu\mod \fq^\upsilon} \binom{\mu^2-4u\rho^{\sk'}}{\fq}^\upsilon\\
&=q^{\upsilon-1} \sum_{\mu'\mod{\fq}}\binom{{\mu'}^2-4u\rho^{\sk'}}{\fq}^\upsilon \textrm{ by periodicity of the modified Hilbert symbol modulo }\fq.
\end{align*}
When $\upsilon$ is even, each term contributes $1$ except for the solutions of ${\mu'}^2 \equiv 4u\rho^{\sk'} \pmod{\fq}$, which are $1+\binom{u\rho^{\sk'}}{\fq}$ in quantity.
Hence, for even $\upsilon$ we have
\[
\sum_{\mu'\mod{\fq}}\binom{{\mu'}^2-4u\rho^{\sk'}}{\fq}^\upsilon= \sq - \left(1+\binom{u\rho^{\sk'}}{\fq}\right).
\]
We use Lemma~\ref{Langlands lemma} for the case when $\upsilon$ is odd to get the value of the sum.
Therefore,
\begin{equation}
K_{(\fq^\upsilon),(1)}(u)=\begin{cases}
\sq^{\upsilon}-\sq^{\upsilon-1}\left(1+\binom{u\rho^{\sk'}}{\fq}\right) &\text{ if }\upsilon \text { is even }\\
-\sq^{\upsilon-1} &\text{ if }\upsilon \text{ is odd }.
\end{cases} 
\end{equation}
This settles \ref{case 3 q not divide rho} and \ref{case 4 q not divide rho}.

\textit{\underline{Case $\upsilon>0$, $r>0$:}}
Then,
\begin{align*}
\widetilde{K}_{\fq^\upsilon,\fq^r}(u)
&= \sum_{\substack{\mu\mod \fq^{\upsilon+2r}\\\mu^2\equiv 4u\rho^{\sk'}\mod \fq^{2r}}} \binom{(\mu^2 - 4u\rho^{\sk'})\pi_{\fq}^{-2r}}{\fq}^\upsilon\\
&=q^{\upsilon-1} \sum_{\substack{\mu\mod 
\fq^{1+2r}\\\mu^2\equiv 4u\rho^{\sk'}\mod \fq^{2r}}} \binom{(\mu^2 - 4u\rho^{\sk'})\pi_{\fq}^{-2r}}{\fq}^\upsilon \textrm{ by periodicity of modified Hilbert symbol mod }\fq.
\end{align*}
If $\upsilon$ is even we have to count the number of solutions to the congruence
\[
\mu^2\equiv 4u\rho^{\sk'}\mod \fq^{2r}
\]
modulo $\fq^{1+2r}$ because each one contributes a $1$ to the sum.
As we have said before, the number of solutions is $1 + \binom{u\rho^{\sk'}}{\fq}$.
Each such solution has $\sq$ lifts modulo $\fq$, so there are
\[
\sq\left(1 + \binom{u\rho^{\sk'}}{\fq}\right)
\]
solutions to this congruence we seek.
However, we must subtract from our count those for which
\[
\mu^2\equiv 4u\rho^{\sk'}\mod \fq^{2r + 1},
\]
because for those $\binom{(\mu^2 - 4u\rho^{\sk'})\pi_{\fq}^{-2r}}{\fq}$ is 0, since $(\mu^2 - 4u\rho^{\sk'})\pi_{\fq}^{-2r} = 0$ is a multiple of $\pi_{\fq}$.
We conclude, for even $\upsilon$,
\[
\sum_{\substack{\mu\mod 
\fq^{1+2r}\\\mu^2\equiv 4u\rho^{\sk'}\mod \fq^{2r}}} \binom{(\mu^2 - 4u\rho^{\sk'})\pi_{\fq}^{-2r}}{\fq}^\upsilon= \sq\left(1 + \binom{u\rho^{\sk'}}{\fq}\right) - \left(1 + \binom{u\rho^{\sk'}}{\fq}\right) = (\sq - 1)\left(1 + \binom{u\rho^{\sk'}}{\fq}\right).
\]
On the other hand, for odd $\upsilon$, we must evaluate the sum
\[
\sum_{\substack{\mu\mod 
\fq^{1+2r}\\\mu^2\equiv 4u\rho^{\sk'}\mod \fq^{2r}}} \binom{(\mu^2 - 4u\rho^{\sk'})\pi_{\fq}^{-2r}}{\fq}.
\]
The terms that contribute to the sum are again those congruences modulo $\fq^{2r + 1}$ such that 
\[
\mu^2\equiv 4u\rho^{\sk'}\mod \fq^{2r}
\]
but not
\[
\mu^2\equiv 4u\rho^{\sk'}\mod \fq^{2r + 1}.
\]
If $4u\rho^{\sk'}$ is a quadratic residue modulo $\fq$, it follows from Hensel's lemma that we can find a solution to $\mu^2\equiv 4u\rho^{\sk'}\mod \fq^{2r}$ with any prescribed congruence modulo $\fq$ of $(\mu^2 - 4u\rho^{\sk'})\pi_{\fq}^{-2r}$.
Since the number of quadratic residues and non-residues is the same modulo $\fq$, we conclude half the terms contribute $1$ and the other half contribute $-1$, i.e.
\[
\sum_{\substack{\mu\mod 
\fq^{1+2r}\\\mu^2\equiv 4u\rho^{\sk'}\mod \fq^{2r}}} \binom{(\mu^2 - 4u\rho^{\sk'})\pi_{\fq}^{-2r}}{\fq}.
\]
In conclusion,
\begin{align*}
\widetilde{K}_{\fq^\upsilon,\fq^r}(u)
&=\begin{cases}
\sq^{\upsilon-1}(\sq-1)\left( 1 + \binom{\mu\rho^{\sk'}}{\fq} \right) &\text{ if }\upsilon \text { is even}
\\
0 & \text{ otherwise.}
\end{cases}
\end{align*}
This settles the remaining cases \ref{case 5 q not divide rho} and \ref{case 6 q not divide rho}.
\end{proof}

With this at hand, the following result is a (tedious) algebraic computation.

\begin{proposition}
\label{kloosc1}
With the same notation introduced in Lemma~\ref{q odd q not rho Kvals},
\[
\D_{\fq}(z,u) = \frac{1-1/\sq^{z+1}}{1-1/\sq^{2z}}.
\]
\end{proposition}

\begin{proof}
We divide the sum according to the six cases of Proposition~\ref{q odd q not rho Kvals}.
Note that the term corresponding to \ref{case 6 q not divide rho} vanishes.
For the cases where $\upsilon$ is assumed to be even or odd we will write $2\upsilon$ or $2\upsilon+ 1$, respectively.

We have
\begin{align*}
\D_{\fq}(z,u)
&= \sum_{r=1}^{\infty}\frac{1}{\sq^{r(2z+1)}} \sum_{\upsilon=1}^{\infty}\frac{\widetilde{K}_{\fq^\upsilon,\fq^r}(u)}{\sq^{\upsilon(z+1)}}\\
&=1+ \sum_{r=1}^{\infty}\frac{\widetilde{K}_{(1),\fq^r}(u)}{\sq^{r(2z+1)}} + \sum_{\upsilon=1}^{\infty}\frac{\widetilde{K}_{\fq^\upsilon,(1)}(u)}{\sq^{\upsilon(2z+1)}} + \sum_{r=1}^{\infty}\frac{1}{\sq^{r(2z+1)}} \sum_{\upsilon=1}^{\infty}\frac{K_{\fq^\upsilon,\fq^r}(u)}{\sq^{\upsilon(z+1)}}\\
&=1+\sum_{r=1}^{\infty}\frac{\widetilde{K}_{(1),\fq^r}(u)}{\sq^{r(2z+1)}} + \sum_{\upsilon=1}^{\infty}\frac{\widetilde{K}_{\fq^{2\upsilon},(1)}(u)}{\sq^{2\upsilon(2z+1)}} 
+ \sum_{\upsilon=1}^{\infty}\frac{\widetilde{K}_{\fq^{2\upsilon+1},(1)}(u)}{\sq^{(2\upsilon+1)(2z+1)}} + \sum_{r=1}^{\infty}\frac{1}{{\sq}^{r(2z+1)}} \sum_{\upsilon=1}^{\infty}\frac{\widetilde{K}_{{\sq}^{2\upsilon},{\sq}^r}(u)}{\sq^{2\upsilon(z+1)}}.
\end{align*}
We simplify each of the terms appearing in the final expression of $\D_{\fq}(z,u)$ We proceed in order of appearance.

The first one, corresponding to $\upsilon= 0, r > 0$, is
\begin{align*}
\sum_{r=1}^{\infty}\frac{\widetilde{K}_{(1),\fq^r}(u)}{\sq^{r(2z+1)}}
&=\left(1+\binom{u\rho^{\sk'}}{\fq}\right) \sum_{r=1}^{\infty}\frac{1}{\sq^{r(2z+1)}.}
\end{align*}
Next we have the ones corresponding to $\upsilon> 0$ and $r = 0$.
When $\upsilon$ is odd we have
\[
\sum_{\upsilon=1}^{\infty}\frac{\widetilde{K}_{\fq^{2\upsilon+1},(1)}(u)}{\sq^{(2\upsilon+1)(2z+1)}}
=-\frac{1}{\sq} \sum_{\upsilon=0}^{\infty}\frac{\sq^{2\upsilon+1}}{\sq^{(2\upsilon+1)(z+1)}}
=-\frac{1}{\sq} \sum_{\upsilon=0}^{\infty}\frac{1}{\sq^{z(2\upsilon+1)}}.
\]
On the other hand, when $\upsilon$ is even we get
\begin{align*}
 \sum_{\upsilon=1}^{\infty}\frac{\widetilde{K}_{\fq^{2\upsilon},(1)}(u)}{\sq^{2\upsilon(2z+1)}} &=\left(\sq^{2\upsilon}-\sq^{2\upsilon-1}\left(1+\binom{u\rho^{\sk'}}{\fq}\right)\right) \sum_{\upsilon=1}^{\infty}\frac{1}{{\sq^{2\upsilon(z+1)}}}\\
 &=\left( 1-\frac{1}{\sq} \left( 1+\binom{u\rho^{\sk'}}{\fp}\right)\right) \sum_{\upsilon=1}^{\infty}\frac{\sq^{2\upsilon}}{\sq^{2\upsilon(z+1)}}\\
 &=\left(1-\frac{1}{\sq}\left(1+\binom{u\rho^{\sk'}}{\fq}\right)\right) \sum_{\upsilon=1}^{\infty}\frac{1}{\sq^{2\upsilon z}}.
\end{align*}
Finally, when both $\upsilon, r$ are positive:
\begin{align*}
 \sum_{r=1}^{\infty}\frac{1}{\sq^{r(2z+1)}} \sum_{\upsilon=1}^{\infty}\frac{\widetilde{K}_{\fq^{2\upsilon},\fq^r}(u)}{\sq^{2\upsilon(z+1)}}
&= \sum_{r=1}^{\infty}\frac{1}{\sq^{r(2z+1)}} \sum_{\upsilon=1}^{\infty}\frac{\sq^{2\upsilon}-\sq^{2\upsilon-1}\left(1+\binom{u\rho^{\sk'}}{\fq}\right)}{\sq^{2\upsilon(z+1)}}\\
&= \sum_{r=1}^{\infty}\frac{1}{\sq^{r(2z+1)}} \sum_{\upsilon=1}^{\infty}\frac{\sq^{2\upsilon}(1-\frac{1}{\sq})\left(1+\binom{u\rho^{\sk'}}{\fq}\right)}{\sq^{2\upsilon(z+1)}}\\
&= \sum_{r=1}^{\infty}\frac{1}{\sq^{r(2z+1)}} \sum_{\upsilon=1}^{\infty}\frac{(1-\frac{1}{\sq})\left(1+\binom{u\rho^{\sk'}}{\fq}\right)}{\sq^{2\upsilon z}}\\
&=\left(1-\frac{1}{\sq}\right)\left(1+\binom{u\rho^{\sk'}}{\fq}\right) \sum_{r=1}^{\infty}\frac{1}{\sq^{r(2z+1)}} \sum_{\upsilon=1}^{\infty}\frac{1}{\sq^{2\upsilon z}}.
\end{align*}

So we consider the following sum
\begin{tiny}
\begin{align*}
1+\left(1-\frac{1}{\sq}\left(1+\binom{u\rho^{\sk'}}{\fq}\right)\right)& \sum_{\upsilon=1}^{\infty}\frac{1}{\sq^{2\upsilon z}}-\frac{1}{\sq} \sum_{\upsilon=0}^{\infty}\frac{1}{\sq^{(2\upsilon+1)z}}+\left(1+\binom{u\rho^{\sk'}}{\fq}\right) \sum_{r=1}^{\infty}\frac{1}{\sq^{zr(2z+1)}}\\
&\ +\left(1-\frac{1}{\sq}\right)\left(1+\binom{u\rho^{\sk'}}{\fq}\right) \sum_{r=1}^{\infty}\frac{1}{\sq^{r(2z+1)}} \sum_{\upsilon=1}^{\infty}\frac{1}{\sq^{2\upsilon z}}\\
&=1+ \sum_{\upsilon=1}^{\infty}\frac{1}{\sq^{2\upsilon z}}-\frac{1}{\sq} \sum_{\upsilon=0}^{\infty}\frac{1}{\sq^{(2\upsilon +1)z}}- \left[\frac{1}{\sq}\left(1+\binom{u\rho^{\sk'}}{\fq}\right) \sum_{\upsilon=1}^{\infty}\frac{1}{\sq^{2\upsilon z}}\right.\\
&\ \left.+\left(1+\binom{u\rho^{\sk'}}{\fq}\right) \sum_{r=1}^{\infty}\frac{1}{\sq^{r(2z+1)}}+\left(1-\frac{1}{\sq}\right)\left(1+\binom{u\rho^{\sk'}}{\fq}\right) \sum_{r=1}^{\infty}\frac{1}{\sq^{r(2z+1)}} \sum_{\upsilon=1}^{\infty}\frac{1}{\sq^{2\upsilon z}}\right]\\
&=1+\frac{1/\sq^{2z}}{1-1/\sq^{2z}}-\frac{1}{\sq} \cdot \frac{1/\sq^{z}}{1-1/\sq^{2z}}\\
&=1+\frac{1/\sq^{2z}}{1-1/\sq^{2z}}- \frac{1/\sq^{(z+1)}}{1-1/\sq^{2z}}\\
&=\frac{1-1/\sq^{z+1}}{1-1/\sq^{2z}}.
\end{align*}
\end{tiny}

To justify the second equality we show that the terms in the square brackets sum to zero.
Indeed,

\begin{tiny}
\begin{align*}
 -\frac{1}{\sq}\left(1+\binom{u\rho^{\sk'}}{\fq}\right)& \sum_{\upsilon=1}^{\infty}\frac{1}{\sq^{2\upsilon z}}+\left(1+\binom{u\rho^{\sk'}}{\fq}\right) \sum_{r=1}^{\infty}\frac{1}{\sq^{r(2z+1)}} +\left(1-\frac{1}{\sq}\right)\left(1+\binom{u\rho^{\sk'}}{\fq}\right) \sum_{r=1}^{\infty}\frac{1}{\sq^{r(2z+1)}} \sum_{\upsilon=1}^{\infty}\frac{1}{\sq^{2\upsilon z}}\\
 &= -\frac{1}{\sq}\left(1+\binom{u\rho^{\sk'}}{\fq}\right) \sum_{\upsilon=1}^{\infty}\frac{1}{\sq^{2\upsilon z}}+\left(1+\binom{u\rho^{\sk'}}{\fq}\right) \sum_{r=1}^{\infty}\frac{1}{\sq^{r(2z+1)}}\\
 &+\left(1+\binom{u\rho^{\sk'}}{\fq}\right) \sum_{r=1}^{\infty}\frac{1}{\sq^{r(2z+1)}} \sum_{\upsilon=1}^{\infty}\frac{1}{\sq^{2\upsilon z}}-\frac{1}{\sq}\left(1+\binom{u\rho^{\sk'}}{\fq}\right) \sum_{r=1}^{\infty}\frac{1}{\sq^{r(2z+1)}} \sum_{\upsilon=1}^{\infty}\frac{1}{\sq^{2\upsilon z}} \\
 &=\left(1+\binom{u\rho^{\sk'}}{\fq}\right) \sum_{r=1}^{\infty}\frac{1}{\sq^{r(2z+1)}}\left(1+ \sum_{\upsilon=1}^{\infty}\frac{1}{\sq^{2\upsilon z}}\right) - \frac{1}{\sq}\left(1+\binom{u\rho^{\sk'}}{\fq}\right)\left(1+ \sum_{r=1}^{\infty}\frac{1}{\sq^{r(2z+1)}} \right) \sum_{\upsilon=1}^{\infty}\frac{1}{\sq^{2\upsilon z}}\\
 &=\left(1+\binom{u\rho^{\sk'}}{\fq}\right) \sum_{r=1}^{\infty}\frac{1}{\sq^{r(2z+1)}} \sum_{\upsilon=0}^{\infty}\frac{1}{\sq^{2\upsilon z}}-\frac{1}{\sq}\left(1+\binom{u\rho^{\sk'}}{\fq}\right) \sum_{r=0}^{\infty}\frac{1}{\sq^{r(2z+1)}} \sum_{\upsilon=1}^{\infty}\frac{1}{\sq^{2\upsilon z}}\\
 &=\left(1+\binom{u\rho^{\sk'}}{\fq}\right)\left(\frac{1/\sq^{(2z+1)}}{1-1/\sq^{(2z+1)}}\right)\left(\frac{1}{1-1/\sq^{2z}}\right) -\frac{1}{\sq}\left(1+\binom{u\rho^{\sk'}}{\fq}\right)\left(\frac{1}{1-1/\sq^{(2z+1)}}\right)\left(\frac{1/\sq^{2z}}{1-1/\sq^{2z}}\right)\\
 &=\left(\frac{1}{(1-1/\sq^{(2z+1)})}\frac{1}{(1-1/\sq^{2z})}\right)\left(\frac{\left(1+\binom{u\rho^{\sk'}}{\fq}\right)}{\sq^{(2z+1)}} - \frac{\left(1+\binom{u\rho^{\sk'}}{\fq}\right)}{\sq^{2z+1}} \right)\\
 &=\left(\frac{1}{(1-1/\sq^{(2z+1)})}\frac{1}{(1-1/\sq^{2z})}\right)\left(\frac{\left(1+\binom{u\rho^{\sk'}}{\fq}\right)}{\sq^{2z+1}} - \frac{\left(1+\binom{u\rho^{\sk'}}{\fq}\right)}{\sq^{2z+1}} \right)\\
 &=0.\qedhere
\end{align*}
\end{tiny}
\end{proof}

\subsubsection{Case \texorpdfstring{$\fq\nmid 2$ and $\fq\mid \rho$}{}}
Here, we have to further consider whether $\sk$ is even or odd.
Recording the following lemma will be useful in proving the next case of Kloostermn type sums (and we can avoid repeating certain arguments).

\begin{lemma}
\label{ntfacts}
Let $\fq\nmid 2$ and fix an \emph{odd} positive integer $k$.
Let $r > k$ be any integer.
Then for any $a\in\cO_K$ with $\val_{\fq}(a) = k$, there are no solutions to the congruence
\[
 x^2\equiv a \pmod{\fq^r}.
\]
\end{lemma}

\begin{proof}
Suppose $x$ is a potential solution to
\[
x^2 \equiv a \pmod{\fq^r}.
\]
Since $r > k$ we have $x^2\equiv 0\pmod{\fq^k}$.
This implies
\[
2\val_{\fq}(x)\ge k.
\]
Since $k$ is odd this inequality is actually strict.
That is,
\[
\val_{\fq}(x^2)\ge k + 1.
\]
Since $r\ge k + 1$, the congruence $x^2\equiv a \pmod{\fq^{k + 1}}$ holds.
Hence we conclude
\[
 a\equiv 0 \pmod{\fq}^{k + 1},
\]
which contradicts $\val_{\fq}(a)= k$.
Thus no solution can exist.
\end{proof}

\begin{lemma}
\label{q odd q div rho Kvals}
Let $\fq\nmid 2$ and suppose that $\fq^\sk\Vert \rho^{\sk'}$.
In other words, $\fq=\fp$ and $\Norm_K(\fq)=\Norm_K(\fp)=\sp$.
If $\sk = 2\sk_0 + 1$ is odd, then the values of $\widetilde{K}_{\fq^\upsilon, \fq^r}(u)$ are as below:
 \begin{enumerate}[label = \textup{(O-\roman*)}]
\item \label{case 1 q odd q div rho k odd} Case $\upsilon= 0, r = 0$: 
 $\widetilde{K}_{1, 1}(u) = 1$,
\item \label{case 2 q odd q div rho k odd} Case $\upsilon> 0$, $r = 0$: $\widetilde{K}_{\fq^\upsilon, 1}(u) = \sp^\upsilon- \sp^{\upsilon- 1}$,
\item \label{case 3 q odd q div rho k odd} Case $\upsilon= 0, 1 \le r \le \sk_0 $:
 $\widetilde{K}_{1, \fq^r}(u) = \sp^r$,
\item \label{case 4 q odd q div rho k odd} Case $\upsilon= 0, r \ge \sk_0 + 1 $:
 $\widetilde{K}_{1, \fq^r}(u) = 0$,
\item \label{case 5 q odd q div rho k odd} Case $\upsilon> 0$, $1 \le r \le \sk_0 $: $\widetilde{K}_{\fq^\upsilon, \fq^r}(u) = \sp^{r +\upsilon} - \sp^{r +\upsilon - 1}$,
\item \label{case 6 q odd q div rho k odd} Case $\upsilon> 0$, $r\ge \sk_0 + 1 $: $\widetilde{K}_{\fq^\upsilon, \fq^r}(u) = 0$.
\end{enumerate}
On the other hand, if $\sk = 2\sk_0$ is even and $\rho^{\sk'} = \pi_{\fq}^{\sk}w$ for some unit $w$, then
\begin{enumerate}[label = \textup{(E-\roman*)}]
\item \label{case 1 q odd q div rho k even} Case $\upsilon= 0, r = 0$: 
 $\widetilde{K}_{1, 1}(u) = 1$
\item \label{case 2 q odd q div rho k even} Case $\upsilon= 0, 1\le r \le \sk_0 $:
 $\widetilde{K}_{1, \fq^r}(u) = \sp^r$
\item \label{case 3 q odd q div rho k even} Case $\upsilon= 0, r \ge \sk_0 + 1$:
 $\widetilde{K}_{1, \fq^r}(u) = \sp^{\sk_0}\left(1 + \binom{4uw}{\fq}\right)$
\item \label{case 4 q odd q div rho k even} Case $\upsilon> 0$, $0 \le r \le \sk_0 - 1$: $\widetilde{K}_{\fq^\upsilon, 1}(u) = \sp^{r +\upsilon} - \sp^{r +\upsilon - 1}$
\item \label{case 5 q odd q div rho k even} Case $\upsilon> 0$ even, $r = \sk_0$: $\widetilde{K}_{\fq^\upsilon, 1}(u) =\sp^{k_0 +\upsilon} - \sp^{\sk_0 +\upsilon - 1}\left(1 + \binom{4uw}{\fq}\right)$
\item \label{case 6 q odd q div rho k even} Case $\upsilon> 0$ odd, $r = \sk_0$: $\widetilde{K}_{\fq^\upsilon, 1}(u) = -\sp^{\sk_0 +\upsilon -1}$
\item \label{case 7 q odd q div rho k even} Case $\upsilon> 0$ even , $r \ge \sk_0 + 1$: $\widetilde{K}_{\fq^\upsilon, \fq^r}(u) = \left(1 - \frac{1}{\sp}\right)\sp^{\sk_0 +\upsilon}\left(1 + \binom{4uw}{\fq}\right)$
\item \label{case 8 q odd q div rho k even} Case $\upsilon> 0$ odd, $r \ge \sk_0 + 1$: $\widetilde{K}_{\fq^\upsilon, \fq^r}(u) = 0$
\end{enumerate}
\end{lemma}

\begin{proof}
The proof is very similar to the one of Lemma~\ref{q odd q not rho Kvals} above.
We deal with each case separately.\newline

\textbf{Case $\sk$ is odd:} 

\textit{\underline{Case $\upsilon= r = 0$}:} Just as before,
\[
\widetilde{K}_{(1),(1)}(u) = 1.
\]
This settles \ref{case 1 q odd q div rho k odd}.
 
 \textit{\underline{Case $\upsilon>0, r = 0$}:} The sum becomes
 \begin{equation}
 \begin{split}
 \widetilde{K}_{\fq^\upsilon, (1)}(u) 
 &= \sum_{\substack{\mu\mod \fq^{\upsilon}\\\mu^2\equiv 4u\rho^{\sk'}\mod \fq^{0}}} \binom{\mu^2 - 4u\rho^{\sk'}}{\fq}^\upsilon\\
 &= \sum_{\substack{\mu\mod \fq^{\upsilon}\\\mu^2\equiv 4u\rho^{\sk'}\mod \fq^{0}}} \binom{\mu^2}{\fq}^\upsilon\\
 &= \sp^\upsilon- \sp^{\upsilon- 1},
\end{split}
\end{equation}
where we used that $\fq\mid \rho$ to deduce
\[
\binom{\mu_i^2-4u\rho^{\sk'}}{\fq} = \binom{\mu_i^2}{\fq},
\]
and the involved symbols are $1$ only when $\fq\nmid \mu_i$, which does not occur exactly $\sp^{\upsilon- 1}$ times, and otherwise are $0$.

\textit{\underline{Case $\upsilon= 0, r > 0$:}} The sum becomes
\begin{align*}
 \widetilde{K}_{(1), \fq^r}(u) 
 &= \sum_{\substack{\mu\mod \fq^{2r}\\\mu^2\equiv 4u\rho^{\sk'}\mod \fq^{2r}}} \binom{(\mu^2 - 4u\rho^{\sk'})\pi_{\fq}^{-2r}}{\fq}^0\\
 &=\sum_{\substack{\mu \pmod{\fq^{2r}}\\ \mu^2\equiv 4u\rho^{\sk'}\pmod{\fq^{2r}}}} 1,
\end{align*}
where we have used \ref{fact1} above since $\upsilon= 0$.

Since $\sk$ is odd, the last computation breaks down into the cases 
$r\le \sk_0$ and $r>\sk_0$.
The first case reduces to counting modulo $\fq^{2r}$ the solutions to $\mu^2\equiv 0 \pmod{\fq^{2r}}$, which has $\sp^r$ solutions.
The latter case by Proposition~\ref{ntfacts} has no solutions.
 Concretely, we get
\[
\widetilde{K}_{\fq^\upsilon, (1)}(u) = 
 \left\{
 \begin{array}{ll}
 0 & \text{if } r > \sk_0,\\
 \sp^r & \text{if } r \le \sk_0.
 \end{array}
 \right.
\]
This settles \ref{case 3 q odd q div rho k odd} and \ref{case 4 q odd q div rho k odd}.

\textit{\underline{Case $\upsilon, r > 0$}:} The sum is
\[
\widetilde{K}_{\fq^\upsilon, \fq^r}(u) 
 = \sum_{\substack{\mu\mod \fq^{\upsilon+2r}\\\mu^2\equiv 4u\rho^{\sk'}\mod \fq^{2r}}} \binom{(\mu^2 - 4u\rho^{\sk'})\pi_{\fq}^{-2r}}{\fq}^\upsilon.
\]
The is non-zero only if $\mu^2\equiv 4u\rho^{\sk'} \pmod{\fq^{2r}}$ but $\mu^2\not\equiv 4u\rho^{\sk'} \pmod{\fq^{2r + 1}}$.
By Proposition~\ref{ntfacts} this congruence has no solution if $2r > \sk$, because $\sk$ is odd.
In the other case we can split
\[
(\mu^2-4u\rho^{\sk'})\pi_{\fq}^{-2r} = (\mu\pi_{\fq}^{-r})^2 - 4u\rho^{\sk'}\pi_{\fq}^{-2r},
\]
and each term is an integer, with the second one still multiple of $\pi_{\fq}$ (because $\sk$ is odd!).
Hence, using the periodicity modulo $\fq$ of the symbol,
\[
\binom{(\mu^2-4u\rho^{\sk'})\pi_{\fq}^{-2r}}{\fq} = \binom{(\mu\pi_{\fq}^{-r})^2 - 4u\rho^{\sk'}\pi_{\fq}^{-2r}}{\fq} = \binom{(\mu\pi_{\fq}^{-r})^2}{\fq},
\]
which is $1$ or $0$ according to whether $\val_{\fq}(\mu) = r$ or bigger.
That is, this sum counts the integers mod $\fq^{2r +\upsilon}$ that are in $\fq^r$ but not in $\fq^{r + 1}$.
This has cardinality
\[
\sp^{r +\upsilon} - \sp^{r +\upsilon - 1}.
\]
This concludes \ref{case 5 q odd q div rho k odd} and \ref{case 6 q odd q div rho k odd}.

\textbf{Case $\sk$ is even:}

\textit{\underline{Case $\upsilon= r = 0$:}} Exactly as before.

\textit{\underline{Case $\upsilon= 0$, $r>0$:}}
The sum becomes 
\[
\widetilde{K}_{(1), \fq^r}(u) 
 = \sum_{\substack{\mu\mod \fq^{2r}\\\mu^2\equiv 4u\rho^{\sk'}\mod \fq^{2r}}} \binom{(\mu^2 - 4u\rho^{\sk'})\pi_{\fq}^{-2r}}{\fq}^0.
\]
Using \ref{fact1} and because $\upsilon= 0$, the value it has is the number of solutions modulo $\fq^{2r}$ of the congruence $\mu^2\equiv 4u\rho^{\sk'}\pmod{\fq^{2r}}$.
When $2r\le \sk$, this reduces to $\mu^2\equiv 0 \pmod{\fq^{2r}}$, which has $q^r$ solutions.
On the other hand, if $2r > \sk$ then our sought congruence implies $\mu^2\equiv 0 \pmod{\fq^{2\sk_0}}$.
Hence $\mu = \pi^{\sk_0}\alpha$ for some $\alpha$.
So the congruence $\mu^2\equiv 4u\rho^{\sk'}\pmod{\fq^{2r}}$ becomes
\[
\alpha^2\equiv 4u\rho^{\sk'}\pi_{\fq}^{-\sk} \equiv 4uw\pmod{\pi_{\fq}^{2r - \sk}} .
\]
This has $1 + \binom{4uw}{\fq}$ solutions modulo $\fq^{2r - k}$.
This corresponds to $\sp^{\sk_0}\left(1 + \binom{4uw}{\fq}\right)$ lifts modulo $\fq^{2r - \sk_0}$, any of which can be the value of $\alpha$.
We conclude 
\[
\widetilde{K}_{(1), \fq^r}(u) = 
 \left\{
 \begin{array}{ll}
 \sp^r & \text{if } r \le \sk_0,\\
 \sp^{\sk_0}\left(1 + \binom{4uw}{\fq}\right)& \text{if } r \ge \sk_0 + 1.
 \end{array}
 \right.
\]
This settles cases \ref{case 2 q odd q div rho k even} and \ref{case 3 q odd q div rho k even}.

\textit{\underline{Case $\upsilon\ge 0$, $0\le r \le \sk_0 - 1$}:}
This is solved exactly as case \ref{case 5 q odd q div rho k odd}.

\textit{\underline{Case $\upsilon> 0$, $r = \sk_0$}:} The sum becomes
\[
\widetilde{K}_{\fq^\upsilon, \fq^{\sk_0}}(u) = \sum_{\substack{\mu\mod \fq^{\upsilon+\sk}\\\mu^2\equiv 4u\rho^{\sk'}\mod \fq^{\sk}}} \binom{(\mu^2 - 4u\rho^{\sk'})\pi_{\fq}^{-\sk}}{\fq}^\upsilon= q^\upsilon\sum_{\substack{\mu\mod \fq^{\sk}\\\mu^2\equiv 0\mod \fq^{\sk}}} \binom{(\mu^2 - 4u\rho^{\sk'})\pi_{\fq}^{-\sk}}{\fq}^\upsilon,
\]
where we have used the periodicity of the modified Hilbert symbol.

When $\upsilon$ is even, we are counting the solutions to $\mu^2\equiv 0 \pmod{\fq^\sk}$ such that $\mu^2 \not\equiv 4u\rho^{\sk'} \pmod{\fq^{\sk + 1}}$.
The first congruence has $\sp^{\sk_0}$ solutions.
On the other hand, counting similarly to the case \ref{case 3 q odd q div rho k even}, the second congruence has
\[
\sp^{\sk_0 - 1}\left(1 + \binom{4uw}{\fq}\right),
\]
By subtracting, we deduce
\[
\widetilde{K}_{\fq^{\upsilon}, \fq^{\sk_0}} = \sp^{\sk_0 +\upsilon} - \sp^{\sk_0 +\upsilon - 1}\left(1 + \binom{u}{\fq}\right).
\]
This settles case \ref{case 5 q odd q div rho k even}.

When $\upsilon$ is odd, we get 
\[
\sum_{\substack{\mu\mod \fq^{\sk}\\\mu^2\equiv 0\mod \fq^{\sk}}} \binom{(\mu^2 - 4u\rho^{\sk'})\pi_{\fq}^{-\sk}}{\fq}^\upsilon
= \sum_{\substack{\mu\mod \fq^{\sk}\\\mu^2\equiv 0\mod \fq^{\sk}}} \binom{(\mu^2 - 4u\rho^{\sk'})\pi_{\fq}^{-\sk}}{\fq}
= q^{\sk_0 - 1} \sum_{\substack{\alpha\mod \fq}} \binom{\alpha^2 - 4uw}{\fq},
\]
where we have used that $\mu^2\equiv 0 \pmod{\sk}$ is equivalent to $\mu = \pi_{\fq}^{\sk_0}\alpha$ where alpha can be any class modulo $\fq^{k_0}$.
These classes form groups of size $\sp^{\sk_0 - 1}$ as lifts of the classes modulo $\fq$, which explains the factor $\sp^{\sk_0 - 1}$ that comes out from the periodicity of the symbol.
This last sum is $-1$ by Lemma~\ref{Langlands lemma}.
We conclude
\[
\widetilde{K}_{\fq^{\upsilon}, \fq^{\sk_0}} = -\sp^{\upsilon+ \sk_0 - 1},
\]
This settles case \ref{case 6 q odd q div rho k even}.

\textit{\underline{Case $\upsilon> 0$, $r \ge \sk_0 + 1$}:} This is solved in exactly the same way as cases \ref{case 5 q not divide rho} and \ref{case 6 q not divide rho} of Lemma~\ref{q odd q not rho Kvals}.
\end{proof}


\begin{proposition}
\label{kloosc3}
Keep the notation introduced in Lemma~\ref{q odd q div rho Kvals}.
Then 
\[
\D_{\fq}(z,u)= \D_{\fp}(z,u)=
\frac{(1-\sp^{z(\sk+1)})(1-1/\sp^{z+1})}{(1-1/\sp^{2z})(1-1/\sp^z)} 
\]
\end{proposition}

\begin{proof}
We do each case separately by separating the series into the cases given by Proposition~\ref{q odd q div rho Kvals}.

\underline{\textit{Case $\sk$ is odd}:}
We write $\sk = 2\sk_0 + 1$.
Then
\begin{align*}
 \D_{\fq}(z,u) 
 &= \sum_{r=0}^{\infty}\frac{1}{\sp^{r(2z+1)}} \sum_{\upsilon=0}^{\infty}\frac{\widetilde{K}_{\fq^\upsilon,\fq^r}(u)}{\sp^{\upsilon(z+1)}}\\
 &= 1 + \sum_{\upsilon= 1}^\infty \frac{\sp^\upsilon- \sp^{\upsilon- 1}}{\sp^{z(\upsilon+ 1)}} + \sum_{r = 1}^{\sk_0}\frac{\sp^r}{\sp^{r(2z + 1)}} + \sum_{\upsilon= 1}^{\infty}\sum_{r = 1}^{\sk_0}\frac{\sp^{r +\upsilon} - \sp^{r +\upsilon - 1}}{\sp^{z(\upsilon+ 1)}\sp^{r(2z + 1)}}\\
 &= 1 + \left(1 - \frac{1}{\sp}\right)\sum_{\upsilon= 1}^\infty \frac{1}{\sp^{z\upsilon}} + \sum_{r = 1}^{\sk_0}\frac{1}{\sp^{2rz}} + \left(1 - \frac{1}{\sp}\right)\sum_{\upsilon= 1}^\infty \frac{1}{\sp^{z\upsilon}}\sum_{r = 1}^{\sk_0}\frac{1}{\sp^{2rz}}\\
 &= \left(\left(1 - \frac{1}{\sp}\right)\sum_{\upsilon= 1}^\infty \frac{1}{\sp^{z\upsilon}} + 1 \right)\sum_{r = 0}^{\sk_0}\frac{1}{\sp^{2rz}}\\
 &= \left(\left(1 - \frac{1}{\sp}\right)\frac{\sp^{-z}}{1 - \sp^{-z}} + 1\right)\frac{1 - \sp^{2z(\sk_0 + 1)}}{1 - \sp^{-2z}}\\
 &=\frac{1 - \sp^{-(z + 1)}}{1 - \sp^{-2z}}\frac{1 - \sp^{-z(\sk + 1)}}{1 - \sp^{-z}}.
\end{align*}

\underline{\textit{Case $\sk$ is even}:} Writing $\sk = 2\sk_0$, we have 

\begin{align*}
\D_{\fq}(z,u) 
 &= \sum_{r=0}^{\infty}\frac{1}{\sp^{r(2z+1)}} \sum_{\upsilon=0}^{\infty}\frac{\widetilde{K}_{\fq^\upsilon,\fq^r}(u)}{\sp^{\upsilon(z+1)}}\\
 &= 1 
 + \sum_{\upsilon= 1}^{\infty}\frac{\sp^\upsilon- \sp^{\upsilon- 1}}{\sp^{\upsilon(z + 1)}} 
 + \sum_{r = 1}^{\sk_0}\frac{\sp^r}{\sp^{r(2z + 1)}}
 + \sum_{r = \sk_0 + 1}^{\infty}\frac{\sp^{\sk_0}\left(1 + \binom{4uw}{\fq}\right)
 }{\sp^{r(2z + 1)}}\\
 &+ \sum_{\upsilon= 1}^{\infty}\frac{1}{\sp^{\upsilon(z + 1)}}\sum_{r = 1}^{\sk_0 - 1}\frac{\sp^{r +\upsilon} - \sp^{r +\upsilon - 1}}{\sp^{r(2z + 1)}}\\
 &+\frac{1}{\sp^{k_0(2z + 1)}}\sum_{\upsilon= 1}^{\infty}\frac{\sp^{\sk_0 + 2\upsilon } - \sp^{\sk_0 + 2\upsilon- 1}\left(1 + \binom{4uw}{\fq}\right)}{\sp^{(2\upsilon )(z + 1)}}
 +\frac{1}{\sp^{\sk_0(2z + 1)}}\sum_{\upsilon= 0}^{\infty}\frac{-\sp^{\sk_0 + (2\upsilon +1) -1}}{\sp^{(2\upsilon+ 1)(z + 1)}}\\
 &+\sum_{\upsilon= 1}^{\infty}\sum_{r = \sk_0 + 1}^{\infty}\frac{\left(1 - \frac{1}{\sp}\right)\sp^{\sk_0 + 2\upsilon }\left(1 + \binom{4uw}{\fq}\right)}{\sp^{2\upsilon (z + 1)}}.
\end{align*}

We evaluate separately and then we substitute back into the previous sums.
The second term is 
\[
\sum_{\upsilon= 1}^{\infty}\frac{\sp^\upsilon- \sp^{\upsilon- 1}}{\sp^{\upsilon(z + 1)}} = \left(1 - \frac{1}{\sp}\right)\frac{\sp^{-z}}{1 - \sp^{-z}}.
\]The third and the fourth term together contribute
\[
\sum_{r = 1}^{\sk_0}\frac{\sp^r}{\sp^{r(2z + 1)}}
 + \sum_{r = \sk_0 + 1}^{\infty}\frac{\sp^{\sk_0}\left(1 + \binom{4uw}{\fq}\right)}{\sp^{r(2z + 1)}}
 = \sp^{-2z}\frac{1 - \sp^{-2z\sk_0}}{1 - \sp^{-2z}} + \sp^{\sk_0}\left(1 + \binom{4uw}{\fq}\right)\frac{\sp^{-(2z + 1)(\sk_0 + 1)}}{1 - \sp^{-(2z + 1)}}.
\]
The next term in the sum becomes
\begin{align*}
 \sum_{\upsilon= 1}^{\infty}\frac{1}{\sp^{\upsilon(z + 1)}}\sum_{r = 1}^{\sk_0 - 1}\frac{\sp^{r +\upsilon} - \sp^{r +\upsilon - 1}}{\sp^{r(2z + 1)}} 
 &= \left(1 - \frac{1}{\sp}\right)\sum_{\upsilon= 1}^{\infty}\sum_{r = 1}^{\sk_0 -1}\frac{\sp^{\upsilon+ r}}{\sp^{\upsilon(z + 1)}{\sp^{r(2z + 1)}}}\\ 
 &= \sp^{-2z}\left(1 - \frac{1}{\sp}\right)\frac{\sp^{-z}}{1 - \sp^{-z}}\frac{1 - \sp^{-2z(\sk_0-1)}}{1 - \sp^{-2z}},
\end{align*}
since the $r$ and $\upsilon$ sums are independent.
The sixth term in the sum is
\[
\frac{1}{\sp^{\sk_0(2z + 1)}}\sum_{\upsilon= 1}^{\infty}\frac{\sp^{\sk_0 + 2\upsilon } - \sp^{\sk_0 + 2\upsilon- 1}\left(1 + \binom{4uw}{\fq}\right)}{\sp^{(2\upsilon )(z + 1)}}
 =\sp^{-2\sk_0z}\frac{\sp^{-2z}}{1 - \sp^{-2z}} - \frac{\sp^{\sk_0 - 1}}{\sp^{\sk_0(2z+1)}}\left(1 + \binom{4uw}{\fp}\right)\frac{\sp^{-2z}}{1 - \sp^{-2z}},
\]
while the seventh is
\[
\frac{1}{\sp^{\sk_0(2z + 1)}}\sum_{\upsilon= 0}^{\infty}\frac{-\sp^{\sk_0 + (2\upsilon +1) -1}}{\sp^{(2\upsilon+ 1)(z + 1)}} 
 = -\sp^{-2\sk_0z - 1}\sum_{\upsilon= 0}^{\infty}\frac{1}{\sp^{(2\upsilon+ 1)z}}
 = -\sp^{-2\sk_0z}\frac{\sp^{-(z+1)}}{1 - \sp^{-2z}}.
\]
Finally the last term is
\begin{align*}
 \sum_{\upsilon= 1}^{\infty}\sum_{r = \sk_0 + 1}^{\infty}\frac{\left(1 - \frac{1}{\sp}\right)\sp^{\sk_0 + 2\upsilon }\left(1 + \binom{4uw}{\fq}\right)}{\sp^{2\upsilon (z + 1)}\sp^{r(2z + 1)}} 
 &= \sp^{\sk_0}\left(1 - \frac{1}{\sp}\right)\left(1 + \binom{4uw}{\fq}\right) \sum_{\upsilon= 1}^{\infty}\sum_{r = \sk_0 + 1}^{\infty}\frac{1}{\sp^{2\upsilon z}\sp^{r(2z + 1)}}\\
 &= \sp^{\sk_0}\left(1 - \frac{1}{\sp}\right)\left(1 + \binom{4uw}{\fq}\right) \frac{\sp^{-2z}}{1 - \sp^{-2z}}\frac{\sp^{-(\sk_0 + 1)(2z+1)}}{1 - \sp^{-(2z + 1)}}.
\end{align*}

By collecting the terms that share the factor $\left(1 + \binom{4uw}{\fq}\right)$ we obtain from them
\begin{tiny}
\begin{align*}
 &\;\sp^{\sk_0}\left(1 + \binom{4uw}{\fq}\right)\left(\frac{\sp^{-(2z + 1)(\sk_0 + 1)}}{1 - \sp^{-(2z + 1)}} -\frac{\sp^{-(2z + 1)(\sk_0 + 1)}}{1 - \sp^{-2z}} + \left(1 - \frac{1}{\sp}\right)\frac{\sp^{-2z}}{1 - \sp^{-2z}}\frac{\sp^{-(\sk_0 + 1)(2z+1)}}{1 - \sp^{-(2z + 1)}}\right)\\
 &= \sp^{\sk_0}\left(1 + \binom{4uw}{\fq}\right)\left(\sp^{-(2z + 1)(\sk_0 + 1)}\left(\frac{1}{1 - \sp^{-(2z + 1)}} -\frac{1}{1 - \sp^{-2z}} \right) +\left(1 - \frac{1}{\sp}\right)\frac{\sp^{-2z}}{1 - \sp^{-2z}}\frac{\sp^{-(\sk_0 + 1)(2z+1)}}{1 - \sp^{-(2z + 1)}}\right)\\
 &= \sp^{\sk_0}\left(1 + \binom{4uw}{\fq}\right)\left(\sp^{-(2z + 1)(\sk_0 + 1)}\left(\frac{\sp^{-(2z+1)}-\sp^{-2z}}{(1-\sp^{-2z})(1-\sp^{-(2z+1)})} \right) + \left(1 - \frac{1}{\sp}\right)\frac{\sp^{-2z}}{1 - \sp^{-2z}}\frac{\sp^{-(\sk_0 + 1)(2z+1)}}{1 - \sp^{-(2z + 1)}}\right)\\
 &= \sp^{\sk_0}\left(1 + \binom{4uw}{\fq}\right)\left(-\left(1 - \frac{1}{\sp}\right)\frac{\sp^{-2z}}{1 - \sp^{-2z}}\frac{\sp^{-(\sk_0 + 1)(2z+1)}}{1 - \sp^{-(2z + 1)}} + \left(1 - \frac{1}{\sp}\right)\frac{\sp^{-2z}}{1 - \sp^{-2z}}\frac{\sp^{-(\sk_0 + 1)(2z+1)}}{1 - \sp^{-(2z + 1)}}\right)\\
 &=0.
\end{align*}
\end{tiny}

Plugging the terms in order the terms that remain are
\begin{align*}
 \D_{\fq}(z, u) &= 1 
 + \left(1 - \frac{1}{\sp}\right)\frac{\sp^{-z}}{1 - \sp^{-z}} 
 + \sp^{-2z}\frac{1 - \sp^{-2z\sk_0}}{1 - \sp^{-2z}}\\ 
 &+ \sp^{-2z}\left(1 - \frac{1}{\sp}\right)\frac{\sp^{-z}}{1 - \sp^{-z}}\frac{1 - \sp^{-2z(\sk_0-1)}}{1 - \sp^{-2z}}\\
 &+\sp^{-2\sk_0z}\frac{\sp^{-2z}}{1 - \sp^{-2z}}
 -\sp^{-2\sk_0z}\frac{\sp^{-(z+1)}}{1 - \sp^{-2z}}.
\end{align*}
Notice that
\[
\left(1 - \frac{1}{\sp}\right)\frac{\sp^{-z}}{1 - \sp^{-z}} + \sp^{-2z}\left(1 - \frac{1}{\sp}\right)\frac{\sp^{-z}}{1 - \sp^{-z}}\frac{1 - \sp^{-2z(\sk_0-1)}}{1 - \sp^{-2z}}
 =\left(1-\frac{1}{\sp}\right)\frac{\sp^{-z}}{1-\sp^{-z}}\frac{1-\sp^{-2\sk_0z}}{1-\sp^{-2z}},
\]
and so we get finally
\begin{align*}
 \D_{\fp}(z, u) &= 1 
 + \sp^{-2z}\frac{1 - \sp^{-2z\sk_0}}{1 - \sp^{-2z}}
 + \left(1-\frac{1}{\sp}\right)\frac{\sp^{-z}}{1-\sp^{-z}}\frac{1-\sp^{-2\sk_0z}}{1-\sp^{-2z}}
 +\sp^{-2\sk_0z}\frac{\sp^{-2z}}{1 - \sp^{-2z}}
 -\sp^{-2\sk_0z}\frac{\sp^{-(z+1)}}{1 - \sp^{-2z}}\\
 &= \frac{1- \sp^{-2\sk_0z -(z + 1)}}{1 - \sp^{-2z}}
 + \left(1-\frac{1}{\sp}\right)\frac{\sp^{-z}}{1-\sp^{-z}}\frac{1-\sp^{-2\sk_0z}}{1-\sp^{-2z}}\\
 &= \frac{1}{(1-\sp^{-2z})(1-\sp^{-z})}\left((1- \sp^{-2\sk_0z -(z + 1)})(1 - \sp^{-z}) + \left(1 - \frac{1}{\sp}\right)\sp^{-z}(1-\sp^{-2\sk_0z})\right)\\
 &= \frac{(1-\sp^{-(z + 1)})(1 - \sp^{-z(2\sk_0 + 1)})}{(1-\sp^{-2z})(1-\sp^{-z})}\\
 &= \frac{(1-\sp^{-(z + 1)})(1 - \sp^{-z(\sk + 1)})}{(1-\sp^{-2z})(1-\sp^{-z})}.
\end{align*}
\end{proof} 

\subsection{The primes above $2$}

Assumption~\ref{ass: split} that $2$ is totally split in $K$ means that the computations are the same ones as the ones for $\Z$ with the Kronecker Symbol.
Unfortunately, these computations were neither written in \cite{AliI} nor in \cite{altug2013beyond}.
We include the computations for the sake of completeness.
For ease of notation, we write $2$ instead of $\fq$ and $\pi_{\fq}$.

In this case, the Kloosterman-type sums look like
\[
\widetilde{K}_{2^\upsilon, 2^r}(u) = \sum_{\substack{\mu\mod 2^{2+\upsilon+2r}\\\mu^2\equiv 4u\rho^{\sk'}\mod 2^{2r}\\ \frac{\mu^2 - 4u\rho^{\sk'}}{2^{2r}}\equiv 0,1 \pmod{4} }} \binom{\mu^2 - 4u\rho^{\sk'}, 2^r}{2^\upsilon} 
 = \sum_{\substack{\mu\mod 2^{2+\upsilon+2r}\\\mu^2\equiv 4u\rho^{\sk'}\mod 2^{2r}\\ \frac{\mu^2 - 4u\rho^{\sk'}}{2^{2r}}\equiv 0,1 \pmod{4} }} \binom{(\mu^2 - 4u\rho^{\sk'})2^{-2r}}{2}_K^\upsilon.
\]

\begin{notation}
For the remainder of this section we drop the `K' from the notation of the symbol as we always work with the Kronecker symbol and there is no chance of confusion.
\end{notation}
The steps to follow are similar in essence to what we did before.
The main differences are the following
\begin{enumerate}[label = \textup{(\alph*)}]
\item the congruence conditions.
\item the Kronecker Symbol is not the quadratic character.
\end{enumerate}
These complicate the computations significantly.

\begin{remark}
Set $c:\Z_2\rightarrow \Z_2/8\Z_2$ to denote the reduction map modulo $8$.
By Hensel's Lemma we know that for a unit $w$ the quantity $c(w) = 1$ if and only if $w$ is a square in $\Z_2$.
The main complication with the following computations is that the Kronecker Symbol does not detect this: it assigns $1$ also to the class $7$ modulo $8$, even though $7$ is not a square in $\Z_2$.
\end{remark}

\subsubsection{Case $2\nmid \rho$}
We begin with a series of lemmas that will make arguments clearer later.
The purpose of the following lemma is to understand when do have solutions to the system of congruence conditions
\begin{equation}
\label{eqn: system of cong cond}
\begin{split}
\mu^2 &\equiv 4u\rho^{\sk'} \pmod{4^r},\\
\frac{\mu^2 - 4u\rho^{\sk'}}{2^{2r}} &\equiv 0,1 \pmod{4}.
\end{split}
\end{equation}

\begin{lemma}
\label{solutions kloos 1}
Let $(\fp,\sk)$ be the tuple fixed in \S\ref{sec: Definition of the test function and its implications} such that the prime ideal and corresponding positive integer satisfy \ref{ass: class number}.
Set ${\sk'} = \sk/h_\fp$ and denote by $\rho$ the generator of $\fp^{h_\fp}$.
Suppose that $2\nmid \rho$ (equivalently $\fp\neq 2$).
Then
\begin{enumerate}[label = \textup{(\roman*)}]
\item \label{r=1} 

When $r=1$, the solutions to the system of congruence conditions \eqref{eqn: system of cong cond} can be read off from the table below\footnote{We explain how to read the first line of the table.
Note that $\mu^2 \equiv 4 \equiv 0\pmod{4}$ has a solution precisely when $\mu\equiv 0,2\pmod{4}$.
Suppose that $c(u\rho^{\sk'})\equiv 1\pmod{8}$.
If $\mu\equiv 2\pmod{4}$ then $\frac{\mu^2 - 4u\rho^{\sk'}}{4}\equiv 0\pmod{4}$; whereas $\mu\equiv 0\pmod{4}$ gives $\frac{\mu^2 - 4u\rho^{\sk'}}{4}\equiv -1\pmod{4}$.}:
\begin{center}
 \begin{tabular}{|c||c|c|c|c|} 
 \hline
 \; & \multicolumn{4}{|c|}{$\frac{\mu^2 - 4u\rho^{\sk'}}{4} \pmod{4}$} \\
 \hline
 $c(u\rho^{\sk'})$ & 0& 1 & 2 & 3\\
 \hline
 1& $\mu \equiv 2$& - &- & $\mu \equiv 0$ \\
 \hline
 3&-& $\mu \equiv 0$ &$\mu \equiv 2$ & - \\
 \hline
 5& $ \mu \equiv 2$& - & -& $ \mu \equiv 0$ \\
 \hline
 7& -& $ \mu \equiv 0$ & $ \mu \equiv 2$ & - \\
 \hline
 \end{tabular}
 \end{center}
\item \label{r=2}
When $r = 2$ and $c(u\rho^{\sk'}) \equiv 1, 5\pmod{8}$ the solutions to the system \eqref{eqn: system of cong cond} can be read off from the table:
 \begin{center}
 \begin{tabular}{|c||c|c|c|c|} 
 \hline
 \; & \multicolumn{4}{|c|}{$\frac{\mu^2 - 4u\rho^{\sk'}}{16} \pmod{4}$} \\
 \hline
 $c(u\rho^{\sk'}) \pmod{16}$ & 0& 1 & 2 & 3\\
 \hline
 $1$& $ \mu\equiv 2, 14$ & - & $ \mu\equiv 6, 10$ & -\\
 \hline
 $9$& $ \mu\equiv 6, 10$ & - & $ \mu\equiv 2, 14$ & -\\
 \hline
 $5$& - & $ \mu\equiv 6, 10$ & & $ \mu\equiv 2, 14$\\
 \hline
 $13$& - & $ \mu\equiv 2, 14$ & & $ \mu\equiv 6, 10$\\
 \hline
 \end{tabular}
 \end{center}
\item \label{reduction is 3 or 7 n r>1} If $c(u\rho^{\sk'}) \equiv 3, 7\pmod{8}$ and $r\ge 2$, then there are no solutions modulo $4^r$ to the system \eqref{eqn: system of cong cond}.
\item \label{reduction is 3,5 or 7 n r>2} If $c(u\rho^{\sk'}) \equiv 3, 5, 7\pmod{8}$ and $r\ge 3$, then there are no solutions modulo $4^r$ to the system \eqref{eqn: system of cong cond}.
\end{enumerate}
\end{lemma}

\begin{proof}
\begin{enumerate}[label = \textup{(\roman*)}]
\item is straightforward counting, so we omit the details.
\item is also straightforward counting.
We mention that we have to consider both lifts, since each case of $c(u\rho^{\sk'})$ produces potential solutions representatives modulo $16$ which have to be checked.
\item when $r\ge 2$, the existence of a solution to the system (either of \ref{reduction is 3 or 7 n r>1} or \ref{reduction is 3,5 or 7 n r>2}) implies that $\mu = 2\beta$ for some $\beta\in\Z_2$.
Hence, canceling a $4$ we obtain that
\[
\beta^2\equiv u\rho^{\sk'} \pmod{4^{r - 1}}.
\]
If $r\ge 2$, then $u\rho^{\sk'}$ is an odd square modulo $4$, i.e.~$1$.
Therefore, modulo $8$ it must be $1$ or $5$.
\item if $r\ge 3$, then we can conclude
\[
\beta^2\equiv u\rho^{\sk'} \pmod{8}.
\]
which implies $c(u\rho^{\sk'}) = 1$.
\end{enumerate}
\end{proof}

Lemma~\ref{solutions kloos 1} shows that when $c(u\rho^{\sk'})\not\equiv 1\pmod{8}$, the Kloosterman-type sum will have most of its term equal to $0$ because the congruence conditions are mostly not satisfied.
For $c(u\rho^{\sk'}) \equiv 1\pmod{8}$ the situation is different.

\begin{lemma}
\label{solutions kloos c=1 r ge 3 }
Keep the notation introduced in Lemma~\ref{solutions kloos 1}.
Let $r\ge 3$ and $c(u\rho^{\sk'}) \equiv 1\pmod{8}$ with $u\rho^{\sk'} = \beta^2$ for some $\beta\in\Z_2^*$.
Then the solutions to 
\begin{equation}
\label{eqn: first congruence}
\mu^2 \equiv 4u\rho^{\sk'} \pmod{4^r},
 \end{equation}
are of the form $2\beta + 4\delta$ where $\delta$ is a lift of $0$ or $-\beta$ modulo $4^{r - 1}$.
Of these, only \emph{four} satisfy the other congruence
\begin{equation}
\label{eqn: second congruence}
\frac{\mu^2 - 4u\rho^{\sk'}}{4^r} \equiv 0, 1 \pmod{4}.
\end{equation}
Among these solutions, \emph{two} give $0$ modulo $4$ while the other \emph{two} give $1$ modulo $4$.
\end{lemma}

\begin{proof}
Notice that $2\beta$ satisfies the congruence system.
Another potential solution is $2\beta + \gamma$.
The first congruence \eqref{eqn: first congruence} can be rewritten as
\[
\gamma(4\beta + \gamma) \equiv 0 \pmod{4^r}.
\]
This implies that $2\mid \gamma$; say $\gamma = 2\alpha$.
Then, the above congruence becomes
\[
\alpha(2\beta + \alpha) \equiv 0 \pmod{4^{r-1}}.
\]
This implies that $2$ divides $\alpha$ (we are using that $r - 1\ge 1)$; say $\alpha = 2\delta$.
Then
\[
\delta(\beta + \delta) \equiv 0 \pmod{4^{r-2}}.
\]
Since $\beta$ is relatively prime to $2$, 
\[
\delta \equiv 0, -\beta \pmod{4^{r - 2}}.
\]
These classes together produce \emph{eight} lifts modulo $4^{r - 1}$.
Since $\gamma = 4\delta$ is a class modulo $4^{r}$, these are all the possibilities.

We now check when the second congruence condition \eqref{eqn: second congruence} is satisfied.
The lifts modulo $4^{r-1}$ of $0 \pmod{4^{r - 2}}$ are $0, 4^{r-2}, 2\cdot 4^{r-2}, 3\cdot 4^{r-2}$.
Hence
\[
 \frac{(2\beta + k4^{r-1})^2 - 4u\rho^{\sk'}}{4^r} 
 = \frac{k\beta4^{r} + k^24^{2r-2}}{4^r} 
 = k\beta + k^24^{r-2} \equiv k\beta \pmod{4}.
\]
Hence only when $k = 0$ or $k = \beta^{-1}\pmod{4}$ we get $0$ or $1$.
Similarly, the lifts of $-\beta \pmod{4^{r-2}}$ are $-\beta + k4^{r - 2}$ with $k = 0, 1, 2, 3$.
The congruence condition is
\[
 \frac{(2\beta - 4\beta + k4^{r-1})^2 - 4u\rho^{\sk'}}{4^r}
 = \frac{(-2\beta + k4^{r-1})^2 - 4u\rho^{\sk'}}{4^r}
 = -\beta k + k^24^{r - 2} \equiv -\beta k\pmod{4}.
\]
Once more, the ones that satisfy it are when $k = 0$ or when $k = -\beta^{-1}$.
Notice that for the choices $k = 0$, we get the congruence $0$ and there are \emph{two} of them.
This concludes the proof.
\end{proof}

With all the previous results at hand, we now proceed to the computations of $\widetilde{K}_{2^\upsilon, 2^r}(u)$.
\begin{lemma}
\label{kloos values even not rho}
Keep the notation introduced in Lemma~\ref{solutions kloos 1}.
Let $\upsilon,r$ be any non-negative integers.
Then the values of $\widetilde{K}_{2^\upsilon, 2^r}$ are given as follows:
\begin{enumerate}[label = \textup{(\roman*)}]
\item \label{case 1 2 not div rho} Case $\upsilon= 0, r = 0$: 
 $\widetilde{K}_{1, 1}(u) = 4$
\item \label{case 2 2 not div rho} Case $\upsilon= 0, r = 1 $:
 $\widetilde{K}_{1, 2}(u) = 4$
\item \label{case 3 2 not div rho} Case $\upsilon= 0, r = 2 $:
 $\widetilde{K}_{1, 4}(u) = 
 \left\{
	\begin{array}{ll}
		8 & \mbox{if } c(u\rho^{\sk'}) = 1, 5 \\
		0 & \mbox{if } c(u\rho^{\sk'}) = 3, 7
	\end{array}
 \right.$
\item \label{case 4 2 not div rho} Case $\upsilon= 0, r \ge 3 $:
 $\widetilde{K}_{1, 2^r}(u) = 
 \left\{
	\begin{array}{ll}
		16 & \mbox{if } c(u\rho^{\sk'}) = 1 \\
		0 & \mbox{if } c(u\rho^{\sk'}) = 3, 5, 7
	\end{array}
 \right.$
\item \label{case 5 2 not div rho} Case $\upsilon> 0$ even , $r = 0$: $\widetilde{K}_{2^\upsilon, 1}(u) = 2^{\upsilon+ 1}$
\item \label{case 6 2 not div rho} Case $\upsilon> 0$ odd , $r = 0$: $\widetilde{K}_{2^\upsilon, 1}(u) = -2^{\upsilon+ 1}$
\item \label{case 7 2 not div rho} Case $\upsilon> 0$ even , $r = 1 $: $\widetilde{K}_{2^\upsilon, 2}(u) = \left\{
	\begin{array}{ll}
		0 & \mbox{if } c(u\rho^{\sk'}) = 1, 5 \\
		2^{\upsilon+ 2} & \mbox{if } c(u\rho^{\sk'}) = 3, 7
	\end{array}
 \right.$
\item \label{case 8 2 not div rho} Case $\upsilon> 0$ even , $r = 2$: $\widetilde{K}_{2^\upsilon, 4}(u) = \left\{
	\begin{array}{ll}
		2^{\upsilon+ 3} & \mbox{if } c(u\rho^{\sk'}) = 5 \\
		0 & \mbox{if } c(u\rho^{\sk'}) = 1, 3, 7
	\end{array}
 \right.$
\item \label{case 9 2 not div rho} Case $\upsilon> 0$ even , $r \ge 3$: $\widetilde{K}_{2^\upsilon, 2^r}(u) = 
 \left\{
	\begin{array}{ll}
		2^{\upsilon+ 3} & \mbox{if } c(u\rho^{\sk'}) = 1 \\
		0 & \mbox{if } c(u\rho^{\sk'}) = 3, 5, 7
	\end{array}
 \right.$
\item \label{case 10 2 not div rho} Case $\upsilon> 0$ odd , $r = 1 $: $\widetilde{K}_{2^\upsilon, 2}(u) = 0$
\item \label{case 11 2 not div rho} Case $\upsilon> 0$ odd , $r = 2 $: $\widetilde{K}_{2^\upsilon, 4}(u) = 0$
\item \label{case 12 2 not div rho} Case $\upsilon> 0$ odd , $r \ge 3 $: $\widetilde{K}_{2^\upsilon, 2^r}(u) = 0$
 \end{enumerate}
\end{lemma}

\begin{proof}
First, we explain the counting strategy.
Lemmas~\ref{solutions kloos 1} and \ref{solutions kloos c=1 r ge 3 } give the number of solutions to the system \eqref{eqn: system of cong cond} for the different values of $c(u\rho^{\sk'})$.
On the other hand, the Kloosterman-type sums we are evaluating consider the solutions to this system modulo $2^{2 +\upsilon + 2r}$.
The system itself remains invariant, but the value of the symbol might not (since it can exchange $1$ and $5$).

Note that if $x$ is a solution to the system modulo $4^r$, then $x + 4^r$ is another a lift modulo $4^{r + 1}$.
We get
\[
\frac{(x + 4^r)^2 - 4u\rho^{\sk'}}{4^r} = \frac{x^2 - 4u\rho^{\sk'}}{4^r} + 2x +4^r.
\]
Recall that all solutions to the system are even.
Hence, this equation implies the following: if the quotient is $0$, then its lifts are also $0$ modulo $4$ (since all lifts are obtained by adding $4^r$ consecutively).
Hence, they can only count when $\upsilon= 0$ as otherwise the symbol has value $0$.

If $\frac{x^2 - 4u\rho^{\sk'}}{4^r} + 2x +4^r \equiv 1\pmod{4}$, the term $2x + 4^r$ can swap between the two lifts $1,5 \pmod{8}$.
Using Lemma~\ref{solutions kloos 1}, we see that (for quotient having value $1$) when $r = 1$, $2x$ is a multiple of $4$ and so the change is $x$ to $x + 4$, so a swap of sign happens in the symbol.
Similarly, for $r\ge 2$, the change is controlled by $2x$ and we see this is an odd multiple of $4$ every time.
Once more, a swap of sign happens in the symbol.

The conclusion is: when $\upsilon\ge 2$ is \emph{even}, all lifts of the solutions with quotient $1$ contribute a $1$ (to the counting) and it is enough to count the number of lifts.
On the other hand, if $\upsilon\ge 1$ is \emph{odd} then all lifts cancel in pairs, except for $r = 0$.
In this final case, the division is vacuous.
This proves \ref{case 10 2 not div rho}, \ref{case 11 2 not div rho} and \ref{case 12 2 not div rho} having value $0$.

We do each case separately.

\textit{\underline{Case $\upsilon= r = 0$}:}
In this case the congruence conditions are vacuous.
We just had a term, contributing $1$ to the sum, per class $\mu \pmod{4}$.
Hence $\widetilde{K}_{1, 1}(u) = 4$.
This settles \ref{case 1 2 not div rho}.

\textit{\underline{Case $\upsilon= 0, r = 1$}:} In this case 
\[
\widetilde{K}_{(1), 2}(u) = \sum_{\substack{\mu\mod 2^{4}\\\mu^2\equiv 4u\rho^{\sk'}\mod 2^{2}\\ \frac{\mu^2 - 4u\rho^{\sk'}}{2^{2}}\equiv 0,1 \pmod{4} }} \binom{(\mu^2 - 4u\rho^{\sk'})2^{-2}}{2}^0
 = 4 \sum_{\substack{\mu\mod 2^{2}\\\mu^2\equiv 4u\rho^{\sk'}\mod 2^{2}\\ \frac{\mu^2 - 4u\rho^{\sk'}}{2^{2}}\equiv 0,1 \pmod{4} }} \binom{(\mu^2 - 4u\rho^{\sk'})2^{-2}}{2}^0,
\]
where we have used the periodicity of the symbol for the last line.
Since $\upsilon= 0$ we have that we have one term per solution to the congruence solutions.
The table in Proposition~\ref{solutions kloos 1}\ref{r=1} shows there is in each case a single solution.
This settles (2).

\textit{\underline{Case $\upsilon= 0, r = 2$}:}
Since $\upsilon= 0$, we must count the corresponding number of solutions of the given system.
By Proposition~\ref{solutions kloos 1}\ref{r=2} and \ref{reduction is 3 or 7 n r>1}, the result follows.

\textit{\underline{Case $\upsilon= 0, r \ge 3$}:} We count one more time the number of solutions modulo $4^r$ of the system
\begin{align*}
 \mu^2 &\equiv 4u\rho^{\sk'} \pmod{4^r},\\
 \frac{\mu^2 - 4u\rho^{\sk'}}{4^r} &\equiv 0, 1 \pmod{4}.
 \end{align*}
By Proposition~\ref{solutions kloos 1} we get that for $c(u\rho^{\sk'})\neq 1$ there are no solutions.
On the other hand, if $c(u\rho^{\sk'}) = 1$ then there are $4$ solutions.
The results follows.

\textit{\underline{Case $\upsilon> 0, r = 0$}:} In this case 
\[
\widetilde{K}_{2^\upsilon, (1)}(u) = \sum_{\substack{\mu\mod 2^{2+v}\\\mu^2\equiv 4u\rho^{\sk'}\mod 2^{0}\\ \frac{\mu^2 - 4u\rho^{\sk'}}{2^{0}}\equiv 0,1 \pmod{4} }} \binom{(\mu^2 - 4u\rho^{\sk'})2^{0}}{2}^\upsilon
 = 2 \sum_{\substack{\mu\mod 2^{v+1}}} \binom{\mu^2 - 4u\rho^{\sk'}}{2}^\upsilon.
\]
Notice we have used in the last line that the congruences conditions are vacuous, since all squares are $0, 1$ modulo $4$.
The expression 
\[
\binom{(\mu^2 - 4u\rho^{\sk'})2^{0}}{2}
\]
is periodic modulo $4$, and modulo $2^{\upsilon+ 1}$ there are $2^{\upsilon- 1}$ lifts of each class modulo $4$.
Hence
\begin{align*}
 \widetilde{K}_{2^\upsilon, (1)}(u) 
 &= 2^\upsilon\left(\binom{0^2 - 4u\rho^{\sk'}}{2}^\upsilon+ \binom{1^2 - 4u\rho^{\sk'}}{2}^\upsilon+ \binom{2^2 - 4u\rho^{\sk'}}{2}^\upsilon+ \binom{3^2 - 4u\rho^{\sk'}}{2}^\upsilon\right)\\
 &= 2^\upsilon\left(\binom{-4u\rho^{\sk'}}{2}^\upsilon+ \binom{1 - 4u\rho^{\sk'}}{2}^\upsilon+ \binom{4 - 4u\rho^{\sk'}}{2}^\upsilon+ \binom{1 - 4u\rho^{\sk'}}{2}^\upsilon\right)\\
 &= 2^\upsilon\left(\binom{1 - 4u\rho^{\sk'}}{2}^\upsilon+ \binom{1 - 4u\rho^{\sk'}}{2}^\upsilon\right),\\
 &= 2^{v+1}\binom{1 - 4u\rho^{\sk'}}{2}^\upsilon,
\end{align*}
where we have used in the last equality that the symbol is $0$ because $-4u\rho^{\sk'}$ and $4 - 4u\rho^{\sk'}$ are multiples of $2$.
If $\upsilon$ is even, the symbol equals $1$ and we obtain statement \ref{case 5 2 not div rho}.

On the other hand, if $\upsilon$ is odd then
\[
\widetilde{K}_{2^\upsilon, (1)}(u) = 2^{v+1}\binom{1 - 4u\rho^{\sk'}}{2}.
\]
Notice that $4u\rho^{\sk'} \equiv 4 \pmod{8}$, since $u\rho^{\sk'}$ is not a multiple of $2$.
Hence,
\[
\widetilde{K}_{2^\upsilon, (1)}(u) = 2^{v+1}\binom{1 - 4u\rho^{\sk'}}{2} = 2^{\upsilon+ 1}\binom{-3}{2} = -2^{\upsilon+ 1},
\]
which is statement \ref{case 6 2 not div rho}.

\textit{\underline{Case $\upsilon> 0$ even , $r = 1 $}:} In this case 
\[
\widetilde{K}_{2^\upsilon, 2}(u) = \sum_{\substack{\mu\mod 2^{2+\upsilon+ 2}\\\mu^2\equiv 4u\rho^{\sk'}\mod 2^{2}\\ \frac{\mu^2 - 4u\rho^{\sk'}}{2^{2}}\equiv 0,1 \pmod{4} }} \binom{(\mu^2 - 4u\rho^{\sk'})2^{-2}}{2}^\upsilon
 = 4 \sum_{\substack{\mu\mod 2^{\upsilon+ 2}\\\mu^2\equiv 4u\rho^{\sk'}\mod 2^{2}\\ \frac{\mu^2 - 4u\rho^{\sk'}}{2^{2}}\equiv 0,1 \pmod{4} }} \binom{(\mu^2 - 4u\rho^{\sk'})2^{-2}}{2}^\upsilon.
\]
This time, since $\upsilon$ is even and positive, we only get a contribution of $1$ from the solutions with 
\[
\frac{\mu^2 - 4u\rho^{\sk'}}{4}\equiv 1 \pmod{4}
\]
We obtain the result using the table of Proposition~\ref{solutions kloos 1}\ref{r=1}.

\textit{\underline{Case $\upsilon> 0$ even , $r = 2 $}:} Result follows from Proposition~\ref{solutions kloos 1}\ref{r=2} and \ref{reduction is 3 or 7 n r>1}.

\textit{\underline{Case $\upsilon> 0$ even , $r \ge 3 $}:} When $c(u\rho^{\sk'}) \neq 1$, the sum has no terms by Proposition~\ref{solutions kloos 1}\ref{reduction is 3,5 or 7 n r>2}.
If $c(u\rho^{\sk'}) = 1$, then Proposition~\ref{solutions kloos c=1 r ge 3 } proves there are \emph{two} solutions.
This settles \ref{case 9 2 not div rho}.

\textit{\underline{Case $\upsilon> 0$ odd , $r = 1 $}:}
We obtain the result using the tables in Propositions~\ref{solutions kloos 1} and \ref{solutions kloos c=1 r ge 3 }.
In each case, we see that there are as many terms whose quotients are $1$ modulo $8$ as cases where it it $5$ modulo $8$.
Hence, their contribution cancels.

\textit{\underline{Case $\upsilon> 0$ odd , $r\ge 2 $}:} Proposition~\ref{solutions kloos 1} implies there are no solutions and the claim follows
\end{proof}


\begin{proposition}
\label{kloosc2}
We keep the notation introduced in Lemma~\ref{solutions kloos 1}.
Then 
\[
\D_{2}(z,u) = 4\frac{1-1/2^{z+1}}{1-1/2^{2z}}
\]
\end{proposition}

\begin{proof}
We separate into cases according to the value of $c(u\rho^{\sk'})$.
Notice that the following sum is common to all the cases (corresponding to \ref{case 1 2 not div rho}, \ref{case 2 2 not div rho}, \ref{case 5 2 not div rho}, and \ref{case 6 2 not div rho} in Lemma~\ref{kloos values even not rho}):
\begin{equation}\label{shared terms}
 4 
 + \frac{4}{2^{2z + 1}} 
 + \sum_{\upsilon= 1}^{\infty}\frac{2^{2\upsilon+ 1}}{2^{2\upsilon (z + 1)}}
 - \sum_{\upsilon= 0}^{\infty}\frac{2^{2\upsilon+ 2}}{2^{(2\upsilon +1)(z + 1)}} 
 = 4 
 + \frac{4}{2^{2z + 1}} 
 +\frac{2^{-2z+1} - 2^{-z+1}}{1-2^{-2z}}
 = \frac{4 - 2^{-4z + 1} - 2^{-z+1}}{1-2^{-2z}}.
\end{equation}
We begin with $c(u\rho^{\sk'}) = 5$ and then we compare all of them to this one.

\textit{\underline{Case $c(u\rho^{\sk'}) = 5$}:} We have
\begin{align*}
 \D_{2}(z, U) 
 &= \sum_{r = 0}^{\infty}\frac{1}{2^{r(2z + 1)}} \sum_{\upsilon= 0}^{\infty}\frac{\widetilde{K}_{2^\upsilon, 2^r}(u)}{2^{v(z + 1)}}\\
 &= 4 
 + \frac{4}{2^{2z + 1}} 
 + \frac{8}{2^{2(2z+1)}} 
 + \sum_{\upsilon= 1}^{\infty}\frac{2^{2\upsilon+ 1}}{2^{2\upsilon (z + 1)}}
 - \sum_{\upsilon= 0}^{\infty}\frac{2^{2\upsilon+ 2}}{2^{(2\upsilon +1)(z + 1)}} 
 +\frac{8}{2^{2(2z + 1)}} \sum_{\upsilon=1}^{\infty}\frac{2^{2\upsilon }}{2^{2\upsilon (z + 1)}}\\
 &= \frac{4 - 2^{-4z + 1} - 2^{-z+1}}{1-2^{-2z}}
 +\frac{8}{2^{2(2z + 1)}} \sum_{\upsilon=0}^{\infty}\frac{2^{2\upsilon }}{2^{2\upsilon (z + 1)}}\\
 &= \frac{4 - 2^{-4z + 1} - 2^{-z+1}}{1-2^{-2z}}
 +\frac{8}{2^{2(2z + 1)}}\frac{1}{1-2^{-2z}}\\
 &= \frac{4 - 2^{-z+1}}{1-2^{-2z}}\\
 &= 4\cdot \frac{1 -2^{-z-1}}{1 - 2^{-2z}}
\end{align*}

\textit{\underline{Case $c(u\rho^{\sk'}) = 1$}:}
Notice that $c(u\rho^{\sk'})=1$ and $c(u\rho^{\sk'})=5$ share almost all terms except for \ref{case 4 2 not div rho} and \ref{case 9 2 not div rho}, which appear when $c(u\rho^{\sk'}) = 1$ but not when $c(u\rho^{\sk'})=5$.
Similarly, \ref{case 8 2 not div rho} appears for $c(u\rho^{\sk'})=5$ but not for $c(u\rho^{\sk'})=1$.

Since we already proved that the case $c(u\rho^{\sk'})=5$ works, we only need to prove that the terms \ref{case 4 2 not div rho} and \ref{case 9 2 not div rho} sum to the same as \ref{case 8 2 not div rho}.
We have
\begin{align*}
 \sum_{r\ge3}\frac{16}{2^{r(2z+1)}} + \sum_{r\ge 3}\frac{1}{2^{r(2z + 1)}} \sum_{\upsilon= 1}^{\infty}\frac{2^{2\upsilon +3}}{2^{2\upsilon (z + 1)}}
 &= 16\cdot\frac{2^{-6z-3}}{1 - 2^{-2z-1}} + 8\cdot\frac{2^{-6z-3}}{1 - 2^{-2z-1}}\cdot \frac{2^{-2z}}{1 - 2^{-2z}}\\
 &= \frac{2^{-6z+1}}{1 - 2^{-2z-1}} + \frac{2^{-6z}}{1 - 2^{-2z-1}}\cdot \frac{2^{-2z}}{1 - 2^{-2z}}\\
 &= \frac{2^{-6z+1}(1 - 2^{-2z}) + 2^{-8z}}{(1 - 2^{-2z-1})(1 - 2^{-2z})}\\
 &= \frac{2^{-6z+1} - 2^{-8z + 1} + 2^{-8z}}{(1 - 2^{-2z-1})(1 - 2^{-2z})}\\
 &= \frac{2^{-6z+1} - 2^{-8z}}{(1 - 2^{-2z-1})(1 - 2^{-2z})}\\
 &= \frac{2^{-6z+1}}{1 - 2^{-2z}}.
\end{align*}
On the other hand, sum \ref{case 5 2 not div rho} is
\[
\frac{1}{2^{2(2z + 1)}} \sum_{\upsilon=1}^{\infty}\frac{2^{2\upsilon+ 3}}{2^{2\upsilon (z + 1)}} = \frac{8}{2^{2(2z + 1)}} \sum_{\upsilon=1}^{\infty}\frac{1}{2^{2\upsilon z}} = \frac{8}{2^{2(2z + 1)}}\frac{2^{-2z}}{1 - 2^{-2z}} = \frac{2^{-6z+1}}{1 - 2^{-2z}}.
\]
This proves that the assertion also holds for the case $c(u\rho^{\sk'}) = 1$.

\textit{\underline{Case $c(u\rho^{\sk'}) = 3, 7$}: }
Notice that all the values of the Kloosterman sums coincide for $c(u\rho^{\sk'}) =3, 7$.
In either case, the sum has one additional term compared to the expression \eqref{shared terms}, namely
\[
\frac{1}{2^{2z+1}} \sum_{\upsilon=1}^{\infty}\frac{2^{2\upsilon +2}}{2^{2\upsilon (z+1)}} = \frac{8}{2^{2(2z+1)}} 
 +\frac{8}{2^{2(2z + 1)}} \sum_{\upsilon=1}^{\infty}\frac{2^{2\upsilon }}{2^{2\upsilon (z + 1)}}.
\]
The result follows from the case $c(u\rho^{\sk'}) = 5$.
\end{proof}

\subsubsection{Case \texorpdfstring{$2\mid \rho$ and $\sk$}{} odd.}

In order to complete this case we need some lemmas.
As before, we write $\sk = 2\sk_0 +1$.

\begin{lemma}
\label{lemma: 2 div rho, k odd, never integer}
For $r > \sk_0 + 1$, $\mu$ an integer, and $u$ a unit, the quantity $\frac{x^2 - 2^{2 + \sk}u}{2^{2r}}$ is never an integer.
Equivalently,
\[
2^{2r} \nmid x^2 - 2^{2 + k}u.
\]
\end{lemma}

\begin{proof}
By assumption, we have that $2r > 2 + 2\sk_0$.
It is clear that $2r > 2 + 2\sk_0 + 1 = 2 + \sk$.
In other words, $2r\ge 3+\sk$.
We have that $2^{2+\sk}\mid 2^{2r}$ or in fact, $2^{3+\sk}\mid 2^{2r}$.

If we are to assume that $2^{2r} \mid \mu^2 - 2^{2+\sk}u$, then the above observation forces that
\[
2^{2+\sk}\mid 2^{2r} \mid \mu^2 -2^{2+\sk}u.
\]
Hence, $2^{2+\sk}\mid \mu^2$.
But, comparing powers (and parity]), we actually have that 
\[
2^{3 +\sk}\mid x^2.
\]
This, together with the above divisibility, implies
\[
2^{3+\sk} \mid 2^{2+\sk}.
\]
which is impossible.
This contradicts our assumption.
\end{proof}

We now have the values of the Kloosterman-type sums:

\begin{lemma}
\label{lemma: 2 divides rho, k odd}
Let $\fq= 2$ and suppose that $2^\sk\Vert \rho^{\sk'}$.
In other words, $\fq=\fp=2$ and in our previous notation $\Norm_K(\fq)=\Norm_K(\fp)=2$.
If $\sk = 2\sk_0 + 1$ is odd, then the values of $\widetilde{K}_{2^\upsilon, 2^r}(u)$ are as given below:
\begin{enumerate}[label = \textup{(O-\roman*)}]
\item \label{Malors O-i 13-Jul case} Case $\upsilon= 0, r = 0$: 
 $\widetilde{K}_{1, 1}(u) = 4$
 
\item \label{Malors O-ii 13-Jul case} Case $\upsilon> 0, r = 0 $:
 $\widetilde{K}_{2^\upsilon, 1}(u) = 2^{v+2} - 2^{v+1}$
 
\item \label{Malors O-iii 13-Jul case} Case $\upsilon= 0$ , $r \le \sk_0$: $\widetilde{K}_{1, 2^r}(u) = 2^{2+r}$
 
\item \label{Malors O-iv 13-Jul case} Case $\upsilon= 0$ , $r = 1 + \sk_0$: $\widetilde{K}_{1, 2^r}(u) = 0$
 
\item \label{Malors O-v 13-Jul case} Case $\upsilon= 0$ , $r > 1 + \sk_0$: $\widetilde{K}_{1, 2^r}(u) = 0$
 
\item \label{Malors O-vi 13-Jul case} Case $\upsilon> 0$ even , $r \le \sk_0$: $\widetilde{K}_{2^\upsilon, 2^r}(u) = 2^{v-1}(2^{r+3}-2^{r+2})$

\item \label{Malors O-vii 13-Jul case} Case $\upsilon> 0$ even , $r = \sk_0 + 1$: $\widetilde{K}_{2^\upsilon, 2^r}(u) = 0$

\item \label{Malors O-viii 13-Jul case} Case $\upsilon> 0$ even , $r > \sk_0 + 1$: $\widetilde{K}_{2^\upsilon, 2^r}(u) = 0$

\item \label{Malors O-ix 13-Jul case} Case $\upsilon> 0$ odd , $r \le \sk_0$: $\widetilde{K}_{2^\upsilon, 2^r}(u) = 2^{\upsilon+ r + 1}$

\item \label{Malors O-x 13-Jul case} Case $\upsilon> 0$ odd , $r = \sk_0 + 1$: $\widetilde{K}_{2^\upsilon, 2^r}(u) = 0$

\item \label{Malors O-xi 13-Jul case} Case $\upsilon> 0$ odd , $r > \sk_0 + 1$: $\widetilde{K}_{2^\upsilon, 2^r}(u) = 0$
\end{enumerate}
\end{lemma}

\begin{proof}
We do each case separately.

\textit{\underline{Case $\upsilon= r = 0$}:}
The calculations are the same as Lemma~\ref{kloos values even not rho}\ref{case 1 2 not div rho}.

\underline{\textit{Case $\upsilon>0, r=0$}:} The congruence condition is vacuously satisfied in this case.
By the periodicity property of the symbol at the even prime (modulo $4$), we have
\[
\widetilde{K}_{2^\upsilon, 1} 
= \sum_{\mu\mod{2^{\upsilon+ 2}}}\binom{\mu^2 - 2^{2 + \sk}u}{2}^{v}
= 2^\upsilon \sum_{\mu\mod{4}}\binom{\mu^2 - 2^{2 + \sk}u}{2}^{v}
\]
Notice that for $\sk \ge 1$ we always have
\[
\mu^2 - 2^{2 + \sk}u\equiv \mu^2 \pmod{8}.
\]
Hence,
\[
\sum_{\mu\mod{4}}\binom{\mu^2 - 2^{2 + \sk}u}{2}^{v} = \sum_{\mu\mod{4}}\binom{\mu^2}{2}^{v} = \sum_{\mu\mod{4}}\binom{\mu^2}{2}.
\]
The final equality follows from the fact that the symbol only take the value $0$ or $1$ on squares.
To evaluate the Kloosterman sum, we only sum over the odd values of $\mu$ (which contribute a 1 to the sum).
Since there are $2^{\upsilon+ 1}$-many even terms, we get that 
\[
\widetilde{K}_{2^\upsilon, 1} = 2^{\upsilon+ 2} - 2^{\upsilon+ 1}.
\]

\underline{\textit{Case $\upsilon= 0$ , $r \le \sk_0$}:}
The definition of the Kloosterman-type sum becomes
\[
\widetilde{K}_{1, 2^r} 
= \sum_{\substack{\mu\mod{2^{2r + 2}}\\ \mu^2 \equiv 2^{2 + k} \pmod {2^{2r}}\\ (\mu^2 - 2^{2 + k}u)2^{-2r} \equiv 0, 1\mod{4}}}\binom{(\mu^2 - 2^{2 + k}u)2^{-2r}}{2}^0.
\]
More precisely, a term contributes a $1$ to the sum if and only if it is a class $\mu\pmod{2^{2r + 2}}$ such that 
\[
\mu^2\equiv 2^{2 + \sk}u\pmod{2^{2r}},
\]
and that satisfies the congruence condition.
Otherwise it contributes a $0$.

Since $r\le \sk_0$, we have $2r \le 2\sk_0 < 2 + k$, so that this last congruence is actually
\[
\mu^2\equiv 2^{2 + \sk}u\equiv 0\pmod{2^{2r}}.
\]
Equivalently, $\mu\equiv 0 \pmod{2^r}$.
There are $2^r$ solutions to this congruence modulo $2^{2r}$, and each of them lifts to \emph{four} solutions modulo $2^{2 + 2r}$.

Suppose that $\mu = 2^r y$ is such a solution.
Then
\[
\frac{\mu^2 - 2^{2 + \sk}u}{2^{2r}} = \frac{y^2\cdot2^{2r} - 2^{2 + \sk}u}{2^{2r}} = y^2 - 2^{2 + \sk -2r}u \equiv y^2 \pmod{4}.
\]
Where, we have used the fact that
\[
2 + \sk - 2r = 2 + 1 + 2(\sk_0 - r)> 2.
\]
The congruence condition is therefore satisfied and we conclude
\[
\widetilde{K}_{1, 2^r} = 2^{2 + r}.
\]

\underline{\textit{Case $\upsilon= 0$ , $r = 1 + \sk_0$}:}
In this situation 
\[
\widetilde{K}_{1, 2^{1+\sk_0}} 
= \sum_{\substack{\mu\mod{2^{2r + 2}}\\ \mu^2 \equiv 2^{2 + \sk}u \pmod {2^{2r}}\\ (\mu^2 - 2^{2 + \sk}u)2^{-2r} \equiv 0, 1\mod{4}}}\binom{(\mu^2 - 2^{2 + \sk}u)2^{-2r}}{2}^0.
\]

Here, we will show that the congruence conditions are never satisfied simultaneously.
Note that if we require $2^{2r} \mid \mu^2 - 2^{2+\sk}u$ then $\mu = 2^r x$, for some $x$.
Upon simplification and using the fact that $r = \sk_0 + 1$,
\[
\frac{\mu^2 - 2^{2+\sk}u}{2^{2r}} = x^2 - 2u.
\]
Observe that $u\equiv \pm 1\pmod{4}$ and that
\[
x^2 \pm 2 \not\equiv 0,1 \pmod{4}.
\]
This calculation concludes this case.

\underline{\textit{Case $\upsilon= 0$ , $r > 1 + \sk_0$}:}
Lemma~\ref{lemma: 2 div rho, k odd, never integer} asserts that the congruence conditions can never be satisfied.
This is an empty sum and the Kloosterman sum is 0.

\underline{\textit{Case $\upsilon> 0$ even , $r \le \sk_0$}:}
Using the periodicity of the symbol, we get that \[
\widetilde{K}_{2^\upsilon, 2^r} 
= 2^{\upsilon- 1} \sum_{\mu\mod{2^{2r + 3}}}\binom{(\mu^2 - 2^{2 + \sk}u){2^{-2r}}}{2}^{v}.
\]
The assumption on $r$ guarantees that
\[
2r < 2r + 1 \le 2\sk_0 + 1 = \sk < 2 + \sk.
\]
Therefore, the terms that contribute to the Kloosterman-type sum are those $\mu\pmod{2^{2r + 3}}$ such that 
\begin{itemize}
 \item $\mu^2\equiv 2^{2 + \sk}u \equiv 0\pmod{2^{2r}}$,
 \item $\mu^2\not\equiv 2^{2 + \sk}u\pmod{2^{2r + 1}}$ (equivalently, $\mu^2\not\equiv 0\pmod{2^{2r + 1}}$)
 \item they satisfy the congruence condition modulo $4$.
\end{itemize}
Now observe that modulo $2^{2r + 1}$, there are $2^{r + 1}$ solutions to the first congruence; of those, $2^r$ satisfy the second one.

Write $\mu = 2^ry$ for some solution $\mu$ that satisfies the first two the congruence conditions.
Then
\[
\frac{\mu^2 - 2^{2 + \sk}u}{2^{2r}} = \frac{y^2\cdot2^{2r} - 2^{2 + \sk}u}{2^{2r}} = y^2 - 2^{2 + \sk -2r}u \equiv y^2 \pmod{4},
\]
where we have used the fact that
\[
2 + \sk - 2r = 3 + 2(\sk_0 - r) > 2.
\]
The calculations show that the third congruence condition is always satisfied and \[
\widetilde{K}_{2^\upsilon, 2^r} = 2^{\upsilon- 1}(2^{r + 3} - 2^{r+2}).
\]

\underline{\textit{Case $\upsilon> 0$ even , $r = \sk_0 + 1$}:}
As in the case \ref{Malors O-iv 13-Jul case}, the congruence conditions are not satisfied simultaneously and the sum is 0.

\underline{\textit{Case $\upsilon> 0$ even , $r > \sk_0 + 1$}:} 
We can once again use Lemma~\ref{lemma: 2 div rho, k odd, never integer} to show that the sum is 0.

\underline{\textit{Case $\upsilon> 0$ odd , $r \le \sk_0$}:} 
As in \ref{Malors O-vi 13-Jul case} above, the condition $r \le \sk_0$ implies that $2r < 2 + \sk$.
The quotient $\frac{\mu^2 - 2^{2 + \sk}u}{2^{2r}}$ is an integer if and only if $2^r\mid \mu$.
Writing $\mu = 2^ry$ for some $y$ and using the periodicity of the symbol, we can replace the sum over the $\mu$ to that over the $y$.
We get 
\begin{align*}
K_{2^\upsilon, 2^r} 
&= \sum_{\substack{\mu\mod{2^{\upsilon+ 2r + 2}}\\ \mu^2 \equiv 2^{2 + \sk}u \pmod {2^{2r}}\\ (\mu^2 - 2^{2 + \sk}u)2^{-2r} \equiv 0, 1\mod{4}}}\binom{(\mu^2 - 2^{2 + \sk}u)2^{-2r}}{2}^{v}.\\
&= 2^{v} \sum_{\substack{y\mod{2^{r + 2}}\\ y^2 - 2^{2 + \sk -2r}u \equiv 0, 1\mod{4}}}\binom{(2^{2r}y^2 - 2^{2 + \sk}u)2^{-2r}}{2}.
\end{align*}
The congruence condition
\[
y^2 - 2^{2 + \sk -2r}u \equiv 0, 1 \pmod{4}
\]
holds because $2 + \sk - 2r \geq 2$.
Actually, more is true; because $r\le \sk_0$, we have
\[
2 + \sk - 2r = 3 + 2(\sk_0 - r)\ge 3.
\]
Invoking the periodicity modulo $8$ of the symbol and the fact that the congruences condition hold, the sum becomes:
\begin{align*}
K_{2^\upsilon, 2^r} 
 &= 2^{\upsilon+ r} \sum_{y\mod{4}}\binom{y^2 - 2^{2 + \sk - 2r}u}{2}\\
 &= 2^{\upsilon+ r} \sum_{y\mod{4}}\binom{y^2}{2}\\
 &= 2^{\upsilon+ r + 1} \quad \text{ since contribution comes from $y = \pm 1$ and the sum equals $2$}.
\end{align*}

\underline{\textit{Case $\upsilon> 0$ odd , $r = \sk_0 + 1$}:} As in cases \ref{Malors O-iv 13-Jul case} and \ref{Malors O-vii 13-Jul case} above, the sum is 0.

\underline{\textit{Case $\upsilon> 0$ odd , $r > \sk_0 + 1$}:} For this case, we once again use Lemma~\ref{lemma: 2 div rho, k odd, never integer} to show that the sum is 0.
\end{proof}

\begin{proposition}
\label{kloosc4}
With notation introduced in Lemma~\ref{lemma: 2 divides rho, k odd},
\[
\D_{2}(z,u)=
4\frac{(1-2^{z(\sk+1)})(1-1/2^{z+1})}{(1-1/2^{2z})(1-1/2^z)} 
\]
\end{proposition}

\begin{proof}
The sum $\D_2(z, U)$ splits into several parts according to the values of $r$ and $\upsilon$ given by Lemma~\ref{lemma: 2 divides rho, k odd}.
Cases \ref{Malors O-iv 13-Jul case}, \ref{Malors O-v 13-Jul case}, \ref{Malors O-vii 13-Jul case}, \ref{Malors O-viii 13-Jul case}, \ref{Malors O-x 13-Jul case}, and \ref{Malors O-xi 13-Jul case} 
do not contribute to the sum.
On the other hand,
\begin{itemize}
\item \ref{Malors O-i 13-Jul case} contributes $4$,
\item \ref{Malors O-ii 13-Jul case} contributes $\frac{2^{-z+1}}{1 - 2^{-z}}$,
\item \ref{Malors O-iii 13-Jul case} contributes $4\cdot\frac{2^{-2z}(1 - 2^{-2\sk_0z})}{1 - 2^{-2z}}$,
\item \ref{Malors O-vi 13-Jul case} contributes $2\cdot\frac{2^{-2z}(1 - 2^{-2\sk_0z})}{1 - 2^{-2z}}\cdot\frac{2^{-2z}}{1 - 2^{-2z}}$,
\item \ref{Malors O-ix 13-Jul case} contributes $2\cdot\frac{2^{-2z}(1 - 2^{-2\sk_0z})}{1 - 2^{-2z}}\cdot\frac{2^{-z}}{1 - 2^{-2z}}$.
\end{itemize}
The sum of all these terms gives
\[
4\frac{(1-1/2^{z(\sk+1)})(1-1/2^{z+1})}{(1-1/2^{2z})(1-1/2^z)}.
\]
We leave the details to the reader.
\end{proof} 

\subsubsection{Case \texorpdfstring{$2\mid \rho$ and $\sk$}{} even.}
In this section, we only record the main result.

\begin{proposition}
\label{more kloosc}
Let $\fq= 2$ and suppose that $2^\sk\Vert \rho^{\sk'}$ with $\sk$ even; 
i.e.~$\fq=\fp=2$ and $\Norm_K(\fq)=\Norm_K(\fp)=2$.
Then
\[
\D_{2}(z,u)=
4^n\frac{(1-2^{z(\sk+1)})(1-1/2^{z+1})}{(1-1/2^{2z})(1-1/2^z)} .
\]
\end{proposition}

\begin{proof}
This can be proven in exactly the same way as before.
This proof is left as an exercise for the reader.
\end{proof}

\textbf{Conflict of interest/Competing interest:} The authors hereby declare that no situation of real, potential or apparent conflict of interest is known to us.
In particular, there are not financial or non-financial interests that are directly or indirectly related to the work submitted for publication.

\textbf{AI Usage:}
There was no use of AI during preparation or revision of this article.

\bibliographystyle{alpha}

\bibliography{bibliography}

\end{document}